\documentclass[a4paper,95pt]{article}
\usepackage{amssymb,amsmath,amsthm}
\usepackage{graphicx}

\usepackage{amsfonts}
\usepackage{latexsym}
\usepackage[english]{babel}
\usepackage{hyperref}

\numberwithin{equation}{section}

\newcommand{\Om}{\Omega}

\newenvironment{pf}{\noindent{\sc Proof}.\enspace}{\rule{2mm}{2mm}\smallskip}

\newtheorem{theorem}{Theorem}[section]
\newtheorem{proposition}{Proposition}[section]
\newtheorem{lemma}{Lemma}[section]
\newtheorem{corollary}{Corollary}[section]

\newtheorem{remark}{Remark}[section]
\newtheorem{remarks}{Remark}[section]
\newtheorem{definition}{Definition}[section]
\newcommand{\be}{\begin{equation}}
\newcommand{\ee}{\end{equation}}
\newcommand{\teta}{\theta}
\newcommand{\om}{\omega}

\newcommand{\e}{\varepsilon}

\newcommand{\ph}{\varphi}

\newcommand{\ov}{\overline}
\newcommand{\wtilde}{\widetilde}

\newcommand{\R}{\mathbb R}
\newcommand{\C}{\mathbb C}

\newcommand{\Z}{\mathbb Z}

\newcommand{\N}{\mathbb N}
\newcommand{\T}{\mathbb T}

\renewcommand{\a }{\alpha }
\renewcommand{\b }{\beta }
\newcommand{\s }{\sigma }
\newcommand{\ii }{{\rm i} }
\renewcommand{\d }{\delta }

\newcommand{\g }{\gamma}

\newcommand{\vphi}{\varphi }

\renewcommand{\t }{\tau }
\renewcommand{\o }{\omega }
\renewcommand{\O }{\Omega }

\newcommand{\pa}{\partial}

\newcommand{\io}{\iota}
\newcommand{\inv}{^{-1}}

\newcommand{\Lipg}{{\gamma\rm{lip}}}
\newcommand{\lip}{{\rm lip}}
\newcommand{\lla}{\langle \! \langle}
\newcommand{\rra}{\rangle \! \rangle}

\begin{document}

\title{{\bf Large KAM tori for perturbations of \\ 
the dNLS equation}}


\author{
Massimiliano Berti\footnote{PRIN 2012 "Variational and perturbative aspects of nonlinear
differential problems".} , 
Thomas Kappeler\footnote{Supported in part by the Swiss National Science Foundation.}  ,
Riccardo Montalto\footnote{Supported in part by the Swiss National Science Foundation.} 
}

\maketitle

\noindent
{\bf Abstract.}
We prove that small, semi-linear Hamiltonian perturbations of the defocusing nonlinear Schr\"odinger (dNLS) equation on the circle
have an abundance of invariant tori of any size and (finite) dimension which support quasi-periodic solutions.
When compared with previous results the novelty consists in considering perturbations which do not satisfy any  symmetry
condition (they may depend on  $x$ in an arbitrary way) 
and  need not be analytic. The main difficulty is posed by pairs of almost resonant dNLS frequencies.  
The proof is based on the integrability of the dNLS equation, in particular the fact that the nonlinear part of the 
Birkhoff coordinates is one smoothing.
We implement a Newton-Nash-Moser iteration scheme to construct the invariant tori. 
The key point is the reduction of 
linearized operators, coming up in the iteration scheme, to  $ 2 \times 2 $ block diagonal ones with constant coefficients together with sharp asymptotic  estimates of their eigenvalues. 
\\[1mm]



\tableofcontents

\section{Introduction}\label{introduzione paper}
\label{1. Introduction}

Consider the defocusing nonlinear Schr\"odinger (dNLS) equation in one
space dimension
   \begin{equation}
   \label{1.1} \ii \partial _t u = - \partial ^2_x u + 2|u|^2u
   \end{equation}
on the standard Sobolev space $H^\s \equiv H^\sigma ({\mathbb T}_1, {\mathbb C})$ of complex
valued functions on ${\mathbb T}_1 : = {\mathbb R} / {\mathbb Z}$.
It is well known that for $\s \ge 0$, \eqref{1.1} is wellposed  and for $\s \ge 1$, it 
is a Hamiltonian PDE with Poisson bracket and Hamiltonian given by
\begin{equation}\label{Poisson brackets}
\{ {\mathcal F}, {\mathcal G} \} (u_1, u_2) 
= - \ii \int ^1_0 (\nabla_u {\mathcal F} \nabla_{\bar u} {\mathcal G} - 
    \nabla_{\bar u} {\mathcal F} \nabla_u {\mathcal G})dx, \qquad
     {\mathcal H}^{nls}(u_1, u_2) = \int ^1_0 (\partial _x  u \partial _x \bar{u} + u^2 \bar u^2)dx \, .
\end{equation}
Here $u_1, u_2$ are the real valued functions, defined in terms of $u$ by 
$u_1 = \sqrt{2} {\rm Re}(u)$, $ u_2 = - \sqrt{2}{\rm Im}(u), $  the $L^2-$gradients
$ \nabla_{ u}, \nabla_{\bar u}$ are given by
$\nabla_{ u} := (\nabla_{u_1} + \ii \nabla_{u_2})/\sqrt{2}$,   
$ \nabla_{\bar u} := (\nabla_{u_1} - \ii \nabla_{u_2})/\sqrt{2},$
and
$\mathcal F, \mathcal G$,
viewed as functions of $u_1$ and $u_2$, are $ {\cal C}^1$-smooth, real valued functionals on $H^\s$ with sufficiently regular $L^2$-gradients. 
The Hamiltonian vector field corresponding to
${\mathcal H}^{nls}$ can then be computed to be
$- \ii \nabla_{\bar u}{\mathcal H}^{nls}$ and when written in Hamiltonian form, equation \eqref{1.1}
becomes $\partial _tu = - \ii \nabla_{\bar u} {\mathcal H}^{nls}$. 
According to \cite{GK}, \eqref{1.1} is an integrable PDE  in the strongest
possible sense, meaning that it admits global Birkhoff
coordinates on $H^\sigma$, $\sigma \in {\mathbb Z}_{\geq 0}$ -- see Subsection~\ref{section birkhoff coordinates} for more details.
 In these coordinates, equation \eqref{1.1}
can be solved by quadrature and the phase space $H^\sigma$ is the union of compact, connected tori,
invariant under the flow of \eqref{1.1}. All the solutions are periodic, quasi-periodic or almost periodic in time. 
These invariant  tori are parametrized by the action variables 
$I = (I_k)_{k\in {\mathbb Z}}$, 
the latter being defined in terms of the Birkhoff coordinates and filling out the whole
positive quadrant $\ell^{1, 2\sigma}_+$ of the weighted 
sequence space  $\ell^{1,2 \s} \equiv \ell^{1, 2 \sigma}(\mathbb Z, \mathbb R)$. 
The dimension of such a torus, denoted by ${\mathcal T}_I$, coincides with the cardinality of the
index set $S \equiv S_I \subseteq {\mathbb Z}$, given by
$S = \{ k \in {\mathbb Z} \big\arrowvert I_k > 0 \} $. 
In case $|S| < \infty $, it can be shown that elements in ${\mathcal T}_I$ 
are $ {\cal C}^\infty-$smooth
and that solutions of \eqref{1.1} with inital data in ${\mathcal T}_I$
wrap around ${\mathcal T}_I$ with speed, defined in terms of the 
 frequencies $\omega ^{nls}_k (I),$ $k \in S $. 
They are called $S$-gap solutions. 

Our aim is to prove that for Hamiltonian perturbations 
   \begin{equation}   \label{1.3} 
   \ii \partial _tu  = - \partial ^2_x u + 2 |u|^2 u +  \e    \mbox{\em f } (x, u ) 
   \end{equation}
of equation \eqref{1.1}, many of these finite dimensional tori persist, provided that $\varepsilon $ 
is sufficiently small. The perturbation 
$\mbox{\em f }$ is assumed to be given by 
$ \mbox{\em f} (x, u) = \nabla_{\bar u}{\mathcal P}$ where ${\mathcal P}$
is a real valued Hamiltonian of the form
   \begin{equation}
   \label{1.4} {\mathcal P}(u) = \int^1_0 \mbox{\em p}(x,u_1(x) , u_2(x))dx
   \end{equation}
and {\em p} a real valued function
 $$
  \mbox{\em p } : {\mathbb T_1} \times {\mathbb R}^2
               \rightarrow {\mathbb R} , \ (x ,  \zeta _1, \zeta _2)
               \mapsto \mbox{\em p}(x, \zeta _1, \zeta _2)\, 
 $$
 which is then related to $f: \mathbb T_1 \times \mathbb C \to \mathbb C$ by
the identity, valid for any $\zeta =  (\zeta_1 - \ii \zeta_2 )/ \sqrt 2$ with $\zeta_1, \zeta_2 \in \R$,
 \be\label{nonlin:f}
  \mbox{\em f }(x, \zeta) 
= \partial _{\bar \zeta } \mbox{\em p }(x,  \zeta_1, \zeta_2),
\qquad
\partial _{\bar \zeta } :=
\big( \partial_{\zeta _1} - {\ii} \partial _{\zeta _2} \big)/\sqrt{2} \,.
  \ee
We assume that $f$ is $ {\cal C}^{\sigma, s_*}$-smooth, meaning that
\begin{equation}\label{regolarita di p}
\partial_x^\alpha \partial_{\zeta_1}^{\beta_1} \partial_{\zeta_2}^{\beta_2} f
\in {\cal C}(\T_1 \times \C, \, \C)\,, \quad \forall \,\,0 \leq \alpha \leq\sigma,
\quad \forall \,\,0 \leq  \beta_1, \beta_2 \leq s_*\,.
\end{equation}   
Note that $f(x, \zeta)$ need not be complex differentiable in $\zeta$.
To state our result in detail, introduce 
for any given $S \subseteq {\mathbb Z}$ with cardinality $|S|
< \infty $, the parameter space 
   \[ \Pi _S : = \{ (\xi_k)_{k \in {\mathbb Z}} \subset \mathbb R
      \big\arrowvert \xi_k = 0 \,\, \forall k \in {\mathbb Z} \backslash S ;
	\,\,\,  \xi_k > 0 \,\, \forall k \in S \} \,,
   \]
   which we identify with $\R_{> 0}^S$. 
The elements of $S$ are referred to as tangential sites. 
By the non-degeneracy property \eqref{Kolmogorov-c} of Proposition \ref{Proposition 2.3}, 
the action-to-frequency map
\be\label{def:AFM}
\omega ^S : \Pi _S \rightarrow {\mathbb R}^S , \ I \mapsto
      (\omega ^{nls}_k (I))_{k \in S}
\ee
is a local diffeomorphism on an open, dense subset of $\Pi _S$.
Finally, let $\T := \R/ (2 \pi \Z)$.
 The main result of this paper is the following one. 

\begin{theorem}
\label{Theorem 1.1} Let $\sigma \in {\Z}_{\geq 4}$ and $S \subset
{\mathbb Z}$ with $|S| < \infty $, $0 \in S$, and $-S=S$ be given and assume that $\Pi \subseteq
\Pi _S$ is a compact subset of positive Lebesgue measure, $ \mbox{\rm meas} (\Pi ) > 0$,  with the property that the action-to-frequency map $ \om^{nls} : \Pi \rightarrow
{\mathbb R}^S, \ I \mapsto (\omega ^{nls}_k(I))_{k \in S},$ is a
bi-Lipschitz homeomorphism onto its image $\Omega $. Then there is an integer
$s_* > \max \big(\sigma, |S|/2 \big) $ so that for any
Hamiltonian ${\mathcal P}$ of the form \eqref{1.4} with 
 $ f = \nabla_{\bar u}{\mathcal P}$ of class 
$\mathcal {\cal C}^{\sigma, s_*}$, there exist $\varepsilon _0 > 0$ and 
$|S|/2 < s < s_*$ so that for any $0 <\varepsilon  \leq \varepsilon _0$
the following holds: there exist
 a closed subset  $\Omega _\varepsilon \subseteq \Omega$, satisfying
\be\label{measure estimate Omega}
 \lim_{\e\to 0} \, \frac{{\rm meas}(\Omega_\e)}{{\rm meas}(\Omega )} = 1\, ,
\ee
 and a Lipschitz family of maps
 $ \io _\omega :  \mathbb T^S \to H^\sigma,$ $\omega \in \Omega_\e  ,$
 so that $\io _\omega$ are $H^s$-smooth embeddings with the property that for any
initial data $ \varphi \in {\mathbb T}^S$, the curves
   \[  \quad t \mapsto
      \io _\omega (\varphi + t\omega)
   \]
are quasi-periodic solutions of \eqref{1.3}.
The torus described by the map $\io _\omega$ is invariant under the flow of the perturbed Hamiltonian ${\mathcal H}^{nls} + \varepsilon {\mathcal P}$.  
\end{theorem}

In Theorem \ref{main theorem} we will show in addition 
that, for $\omega \in \Omega_\e$, the distance of the invariant torus $ \io _\omega (\T^S)  $ to the unperturbed torus ${\mathcal T}_{\xi (\omega )}$ is of
the order $O \big( \varepsilon \gamma^{- 2 } \big)$ where $0 <
\gamma < 1$ is the constant appearing in the diophantine condition
of $\omega $ introduced in \eqref{Omega o Omega gamma tau}. Here $\xi(\omega)$ denotes
the element in $\Pi$, corresponding to $\omega$ by the action-to-frequency map defined in \eqref{def:AFM}.
Expressing equation \eqref{1.3} in suitable coordinates, one sees that actually the 
distance of the invariant torus to the unperturbed one is $ O( \e \g^{-1}) $, 
see Corollary \ref{stima optimal size torus}. 
Note that the frequency vector
$\omega $ 
of the quasi-periodic solution $\io_\omega ( \varphi + t\omega ) $ of \eqref{1.3} is the same 
as the one of the quasi-periodic solutions 
on the invariant torus ${\mathcal T}_{\xi (\omega )}$ of  \eqref{1.1}.

\medskip

\noindent {\em Comments: } 
\begin{enumerate} 
\item 
Using a covering argument one can show that Theorem~\ref{Theorem 1.1} actually holds for any compact subset
 $\Pi \subseteq \Pi _S $ with ${\rm meas } (\Pi ) > 0 $. See the comment after Theorem \ref{main theorem}. 
\item 
 In Theorem \ref{measure estimate} we prove that for some $\nu > 0$,
 ${\rm meas}(\Omega \setminus \Omega_\e) = O(\e^\nu)$ 
as $\varepsilon \to 0.$ 
\item The assumption $ 0 \in S $ and $ S = - S $ are introduced just for simplicity, so that  all elements in the 
complement  $ \Z \setminus S $ of $ S $
come in pairs,
so that in the reduction procedure in section \ref{sec:redu}
 we only have to deal with  $ 2 \times 2 $ blocks. 
\item 
By \eqref{regolarita di p} the perturbation $ f $ is assumed to be $ {\cal C}^{\sigma, s_*} $-smooth where 
a lower bound for $ s_* $ is given in
Theorem \ref{iterazione-non-lineare} (Nash-Moser). 
Note that the regularity with respect to the space variable is just $ \sigma \in \Z_{\geq 4} $. 
No special effort has been made to get optimal  lower bounds for $ s_* $ and $ \sigma $. 
\end{enumerate}

\noindent {\em Outline of the proof of Theorem~\ref{Theorem 1.1} :} 
The starting point of our proof is to write  the perturbed dNLS equation \eqref{1.3} in complex Birkhoff coordinates $(w_k)_{k \in \Z}$, 
the latter being briefly reviewed in Subsection~\ref{section birkhoff coordinates}.
The dNLS-Hamiltonian ${\mathcal H}^{nls}$, expressed in these coordinates, is a real analytic function $H^{nls}$ 
of the actions $I_k = w_k\bar w_k, \ k \in {\mathbb Z}$, and the dNLS frequencies $\omega ^{nls}_k$ 
are given by
\[
\omega ^{nls}_k = \partial _{I_k} H^{nls} , \quad k \in {\mathbb Z}. 
\]
Denoting by $P$ the Hamiltonian $\mathcal P$, expressed in these coordinates, equation \eqref{1.3} then becomes 
 the following infinite dimensional Hamiltonian system
\be\label{NLS original}
\ii \dot w_k =  \om_k^{nls} \, w_k + \e \pa_{\bar w_k} P  \, , \ k \in \Z  \, ,  
\ee
on the  phase space
$h^\sigma \equiv h^\sigma(\Z, \C)$, $\sigma \in \mathbb Z_{\geq 4},$ where
\be\label{space:h-sigma}
h^\sigma := \Big\{ w = (w_k)_{k \in \Z}  \subset \mathbb C \ | \,
 \| w \| _\sigma < \infty \Big\} \, , \quad 
\| w \| _\sigma  := \big( \sum _{k \in {\mathbb Z}} \langle k \rangle^{2\sigma }| w_k|^2 \big)^{1/2},  \,\, \langle k \rangle := (1 + |k|^2)^{\frac12}\,.
\ee
The sequence space $h^\sigma$ is endowed with the  symplectic form 
$  \ii \sum_{k \in \Z } d w_k \wedge d \bar w_k $. 
Given a finite subset $ S \subset \Z $, introduce the space of $S-$gap potentials,
$$
M_S := \{ w =  (w_k)_{k \in \Z}  \subset \mathbb C \, | \, w_k = 0 \,  \mbox{ iff } k \in S^\bot \} \subset h^\sigma   \, , \quad S^\bot := \Z \setminus S\, ,
$$
which is symplectic.
Note that this space is invariant under the flow of \eqref{NLS original} with $\e = 0$.
On $ M_S$, we introduce the  angle-action variables 
$ (\teta, I) := (\teta_k, I_k)_{k \in S} \in  \T^S \times \R_{>0}^S$, 
defined by
$$
I_k := w_k \bar w_k \, , \quad w_k = \sqrt{ I_k} \, e^{- \ii \theta_k} \, , \ \  k \in S \,    
$$
and consider the symplectic space
$$ 
\T^S \times \R^S_{>0} \times h_\bot^{\s}\, , \quad h_\bot^{\s}  := \{ z := (z_k)_{k \in S^\bot} \in h^\sigma(S^\bot, \C) \} \, , 
$$ 
referring to the coordinates $z_k := w_k$, $k \in S^\bot,$ as normal coordinates.
On $\T^S \times \R^S_{>0} \times h_\bot^{\s}$,
 the symplectic form $  \ii \sum_{k \in \Z } d w_k \wedge d \bar w_k $ then
becomes 
\be\label{symplectic-2-form}
\Lambda := \sum_{k \in S} d \teta_k \wedge d I_k + \ii \sum_{k \in S^\bot} d z_k \wedge d \bar z_k
\ee 
and the Hamiltonian system \eqref{NLS original} reads
\be\label{HS-action-angle}
\dot \teta =  \om^{nls}  + \e \nabla_{I} P \, , \qquad 
\dot I = - \e \nabla_{\teta}  P \, , \qquad  
\ii \dot z_k = \om_k^{nls} z_k + \e \pa_{\bar z_k} P \,,  \quad   \forall k \in S^\bot ,
\ee
where $\omega^{nls} = (\omega^{nls}_k)_{k \in S}$ 
and $\omega_k^{nls} = \omega_k^{nls}(I, z \bar z)$, $k \in \mathbb Z,$ 
with  $z \bar z \equiv \big( z_k \bar z_k \big)_{k \in S^\bot}$.
Here, the Hamiltonian $P$ is viewed as a function of the new coordinates $\theta, I, z$
and by a slight abuse of terminology, also made in the sequel in other contexts, 
$(I, z \bar z)$ denotes the conveniently regrouped sequence of actions 
$(w_k \bar w_k)_{k \in \mathbb Z}$. 
Note that for any $\xi := (\xi_k)_{k \in S} \in \R_{> 0}^S$, the torus
\be\label{torus unp}
{\cal T}_\xi := \T^S \times \{ I = \xi \} \times \{ z = 0 \}\, ,
\quad \xi \in \mathbb R_{>0}^S\, ,
\ee
is invariant under the flow of the unperturbed system. In fact,  
the solutions of \eqref{NLS original} with $\e = 0$ are of the form
\be\label{unp-quasi-periodic-action angle}
t \mapsto  (\teta + \om^{nls}(\xi,0) t, \, \xi, \, 0) \, . 
\ee
Here $\teta \in \mathbb T^S$ parametrizes the initial data and
$\om^{nls}_k (\xi,0)$, $k \in S$, are referred to as the unperturbed {\it tangential} frequencies of ${\cal T}_\xi $.
Our aim is to prove that for $ \e >0$ sufficiently small,  most of the tori ${\cal T}_\xi$
persist. This is 
a small divisors problem. 
To be able to apply KAM type techniques requires that for $\e= 0$,   
the  Hamiltonian system \eqref{HS-action-angle},  linearized at the quasi-periodic solution
  \eqref{unp-quasi-periodic-action angle} of the unperturbed system,
has constant coefficients. Indeed this is the case since this linearized system is given by 
\be\label{Lin-Floquet}
\dot {\widehat\teta} =   ( \pa_I \om^{nls} ( \xi,0  )) \, \widehat I \, , \quad 
\dot{\widehat I}  = 0  \, , \quad  
\ii \dot {\widehat z}_k = \om_k^{nls} ( \xi , 0 ) \widehat z_k \, , \ k \in S^\bot \, .
\ee
Since the linearization of \eqref{1.3} at a $S-$gap solution is {\em not} 
a linear PDE with constant coefficients,
this is one of the main reasons to express equation \eqref{1.3} in Birkhoff coordinates.
System \eqref{Lin-Floquet} shows that each torus $ {\cal T}_\xi $ is {\it elliptic}.  
Furthermore it can be proved (cf Subsection \ref{section birkhoff coordinates} ; \cite{KST1}) 
that the dNLS frequencies have the asymptotics 
\be\label{asymptotics nls frequencies} 
\om_k^{nls} (\xi,0 ) = 4 \pi^2 k^2 + 4 \sum_{j \in S} \xi_j + O\Big(\frac1k \Big)\, , \quad |k| \to \infty\,,  
\ee
implying that  $\omega_k^{nls}(\xi, 0) - \omega_{- k}^{nls}(\xi, 0)$ cannot be bounded away from $0$ uniformly in $k$. 
However bounds of such type are part of a set of non resonance conditions, referred to as second order Melnikov conditions
which are one of the main assumptions in the KAM perturbation theory for elliptic tori as developed in \cite{K}, \cite{KP2},  \cite{P}.
Hence the latter does not apply. 

It turns out to be convenient to study \eqref{HS-action-angle} in the canonical coordinates 
$ (\teta, y, z )$ where $y$ is in a neighborhood $U_0 \subset \R^S$ of $0$ chosen such that
$\Pi +U_0 \Subset \R^S_{>0} $, where  $\Pi  \subset \R^S_{>0} $ is the compact set of 
 actions  in Theorem 
\ref{Theorem 1.1}.
The Hamiltonian system \eqref{HS-action-angle} then reads
\be\label{HS teta,y,z}
\dot \teta =  \nabla_y H_\e \, , \quad 
\dot y = - \nabla_{\teta}  H_{\e} \, , \quad  
\ii \dot z =  \nabla_{\bar z} H_{\e}   
\ee
where the Hamiltonian $H_\e$ is given by
\be\label{HamiltonianHep}
H_\e (\teta, y, z) \equiv H_\e (\teta, y, z; \xi ) = H^{nls}( \xi + y, z \bar z) + \e P (\teta, y, z)
\ee
and, by a slight abuse of notation, $P$ is now viewed as a function of $\teta, y, z,$ given by
$ P (\teta, \xi + y , z)  $. 
We want to find invariant tori of \eqref{HS teta,y,z}
 close to the tori $ {\cal T}_\xi $ of \eqref{torus unp}, 
admitting quasi-periodic solutions with frequency vector $ \om $. 
It amounts to solve the equation 
\be\label{definitionFep}
   {F}_\omega(\io) = 0\, ,  \quad 
{F}_\omega(\io) :=   
\big( \omega \cdot \partial _\varphi  \theta - \nabla_y H_\varepsilon\circ \breve \io ,
 \ \omega \cdot \partial _\varphi y + \nabla_\theta H_\varepsilon \circ \breve \io, 
\ \omega \cdot \partial _\varphi z + \ii \nabla_{\bar z} H_\varepsilon \circ \breve \io     \big) 
 \ee
where the unknown is the torus embedding $\breve \io(\vphi) = (\vphi, 0, 0) + \io(\vphi)$ with
$\io$ being the map
 $$
 \io : {\mathbb T}^S \rightarrow M^\s, \quad \varphi \mapsto \big(   \theta(\vphi) - \vphi,\
  y(\varphi ),  \ z(\vphi )\big) \,,
$$
and the phase space
\be\label{def:M-sigma}
M^\s \equiv M^\s_S := {\mathbb T}^S \times U_0 \times h^\sigma _\bot \, , \quad \sigma \geq 4 \, . 
\ee
In this paper we fix the space regularity $ \sigma $. 
In the sequel we will always choose the vector $ \xi $ in \eqref{HamiltonianHep} \eqref{definitionFep} to be the function of the 
parameter $ \om \in \O$ given by
\be\label{xi-function-om}
\xi = (\om^{nls})^{-1} (\om) \, . 
\ee
Note that  other  KAM theorems, such as in \cite{K}, \cite{P},  are formulated for perturbations of  parameter dependent families of isochronous systems,  with $ \xi $ being the independent parameter.

Due to the small divisors problem coming up in the course of the proof, we will look for 
quasi-periodic solutions whose frequencies are  diophantine, 
 namely 
 $\om \in \Omega_{\gamma, \tau} $ where 
\begin{equation}\label{Omega o Omega gamma tau}
\Omega_{\gamma, \tau} := 
\Big\{ \omega \in \Omega : |\omega \cdot \ell| \geq \frac{\gamma}{| \ell|^\tau}\, \,\,\, \forall \ell \in \Z^S \setminus \{ 0 \} \Big\}\, \subset \Omega
\quad\mbox{with} \quad  0 < \gamma < 1\, , \,\,\, \tau \ge |S| +1\, .
\end{equation}
In addition, in order to control the resonant interactions between the tangential and the normal frequencies
of such solutions, 
we will impose on $ \omega $ also first and second order Melnikov non resonance conditions. 
At the starting point of the iteration, we choose 
finite-gap solutions of the unperturbed system
which satisfy  first and second order Melnikov conditions  of the type 
\begin{align*}
& |\omega \cdot \ell + \omega_k^{nls}(\xi(\omega), 0) | \geq 
\frac{\gamma k^2}{\langle \ell \rangle^\tau} \, , \quad \forall (\ell, k) \in \Z^S \times S^\bot \, , \\
& |\omega \cdot \ell + \omega_k^{nls}(\xi(\omega), 0) - \omega_j^{nls}(\xi(\omega), 0) | \geq 
\frac{\gamma \langle k^2 - j^2\rangle }{\langle \ell \rangle^\tau} \, , \quad 
\forall (\ell, k, j) \in \Z^S \times S^\bot \times S^\bot \, , \  (\ell, k, j) \neq  (0,  k,  \pm k )\, ,  \\
& |\omega \cdot \ell + \omega_k^{nls}(\xi(\omega), 0) + \omega_j^{nls}(\xi(\omega), 0) | \geq 
\frac{\gamma \langle k^2 + j^2\rangle }{\langle \ell \rangle^\tau} \, , \quad 
\forall (\ell, k, j) \in \Z^S \times S^\bot \times S^\bot  \, . 
\end{align*}
Using the asymptotics \eqref{asymptotic expansion NLS} 
of the dNLS frequencies in Theorem \ref{Corollary 2.2}  and the non-degeneracy conditions  \eqref{Melnikov-1-2}
in  Proposition \ref{Proposition 2.3}, the above conditions are fulfilled for most values of the parameter $ \om $.
We will then need  to impose   conditions of this type  at each step of the  iteration. 
In the setup chosen in  this paper they take the form \eqref{prime melnikov off diagonali finali matrici}
 and  \eqref{seconde melnikov off diagonali finali matrici} - \eqref{seconde melnikov diagonali finali matrici}. 

\smallskip
 Let us  now explain the main parts  of the proof of Theorem \ref{Theorem 1.1}. 
In view of our non analytic setup, we use a Newton-Nash-Moser iteration scheme 
for solving ${F}_\omega(\io) =0$.
At each step of the scheme, the subsequent approximation is constructed 
with the help of an approximate right inverse of the differential $dF_\om$ 
using a smoothing procedure to counterbalance the loss of regularity of the latter. 
The construction of an approximate right inverse
of $d {F}_\omega$ at an embedding $\breve \io$ near 
$\breve \io_0 (\vphi) = (\vphi, 0 , 0) $ and the proof of tame estimates for it 
are at the core of the implementation of such a scheme.  
Following  the strategy developed in \cite{BB1}, \cite{BBM3}, \cite{BBM4} 
the task of getting such right inverses 
can be reduced to construct an approximate right inverse  
of the part of $dF_\omega$, acting (as an unbounded operator) on $h^\sigma_\bot$
(cf Section~\ref{3. Set up}).
It amounts to solve a $ \vphi$-dependent linear system of the form 
\begin{align}\label{eq:linearized}
 \omega \cdot \partial_\vphi h_k(\vphi)  + \ii  \om_k^{nls}  h_k(\vphi) 
& + \ii  \sum_{j \in S^\bot} \pa_{I_j} \om_k^{nls}  \, z_k (\vphi) 
\Big( \bar z_j (\vphi) h_j(\vphi) + z_j  (\vphi) \bar h_j(\vphi)  \Big)  \nonumber\\
 &  +  \ii  \e  \sum_{j \in S^\bot} \Big( \partial_{z_j} \partial_{\bar z_k} P(\breve \io (\vphi)) h_j(\vphi) + \partial_{\bar z_j} \partial_{z_k} P(\breve \io (\vphi)) \bar h_j(\vphi) \Big) = 0 \,, \quad  k \in S^\bot 
\end{align}
where $\om_k^{nls}$ and $\pa_{I_j} \om_k^{nls}$ are evaluated at 
$(\xi + y(\vphi), z (\vphi)\bar z (\vphi))$. 
We analyze such systems in detail in Section~\ref{sec:5} and Section~\ref{sec:redu}.
In view of the small divisors problems, we would like to apply a KAM scheme to reduce it 
to a linear system in diagonal form with $\vphi$-independent coefficients.  
However, since according to \eqref{asymptotics nls frequencies}, 
the dNLS frequencies do not satisfy the second order  Melnikov conditions
with $ (\ell, k, j) = (0, k, \pm k ) $, this is not possible.  
Instead we reduce the corresponding linear operator to a  self-adjoint, $ 2 \times 2 $ block diagonal operator with $\vphi$-independent coefficients, by grouping together the variables $z_{-k}$ and  $z_k$. 
 For small amplitude solutions of  nonlinear wave (NLW) equations with an external potential,
such a scheme has been  successfully implemented by Chierchia-You \cite{CY},  
using that the NLW equation can be written as a symmetric first order Hamiltonian system, for which the nonlinear part of the Hamiltonian vector field is one smoothing. It implies that the non constant part of the asymptotic expansion of  the normal frequencies is of the size 
$O(\e / |k| ) $ as $ |k| \to + \infty $, where $\e$ is related to the amplitude of the (small) solution. 
In contrast, for the dNLS equation, according to \eqref{asymptotics nls frequencies}, the non-constant part of the asymptotic expansion of the frequencies $\omega^{nls}_k(\xi,0)$ is of size $O(1)$ 
and the nonlinear part of the perturbative Hamiltonian vector field  is not regularizing so that the 'perturbed normal frequencies', denoted by $\om_k$, $k \in S^\bot$,
 will behave asymptotically as $ 4 \pi^2 k^2 + O(1)$. 
This information alone does not allow to verify that along the KAM iteration scheme,
 for any $\ell \neq 0$ and most values of $\xi,$ one has 
$ |\om \cdot \ell + \om_{k} - \om_{-k}| \geq \gamma \langle \ell \rangle^{-\tau} $.  
However such non resonance conditions are needed to eliminate  along the KAM scheme
the $\vphi$-dependent monomials 
$ e^{\ii \ell \cdot \vphi} z_{k} \bar z_{-k} $ and 
 $ e^{\ii \ell \cdot \vphi} z_{-k} \bar z_{k} $ in the perturbed Hamiltonian. 
One of the main tasks in our proof of Theorem~\ref{Theorem 1.1} is to derive for the perturbed normal frequencies
an asymptotic expansion of the form (cf \eqref{asintotica autovalori finali misura})
\be\label{Ome-exp}
 \om_{k}^{nls} (\xi,0 ) + c + O(\e \gamma^{- 2} |k|^{-1})  
\, , \quad |k| \to \infty\,,  
\ee
 where $ c \in \R $ satisfies $ c = O(\e \gamma^{- 2} )$, see Lemma \ref{limitazioni indici risonanti}.
 It  allows to show that the required second order Melnikov non resonance conditions
hold true for a large set of $\omega$'s -- see the arguments of section \ref{sec:measure}. 
It turns out that in \eqref{Ome-exp} the constant $ c $ is independent of the sign of $ k $, but
this fact is irrelevant for the applicability of this approach. 

The asymptotic expansion \eqref{Ome-exp} is achieved by adapting the strategy of
\cite{BBM1} - \cite{BBM3}, developed for quasi-linear perturbations of the KdV equation. 
The main idea is to perform a symplectic transformation which 
reduces the linearized operator to a diagonal operator with $\vphi$-independent coefficients up to a one smoothing remainder.
This is achieved in three steps in Subsections \ref{sec1} - \ref{sec-gauge}. 
One of  the key ingredients  is that,  by \cite{KST},  the Birkhoff map is a perturbation of the 
Fourier transform by a $1-$smoothing nonlinear map.
Thus the  highest order term of the linearized equation, expressed in the 
Birkhoff coordinates, is 
the same as the one in the original coordinates. 
In contrast to the KdV equation, treated in \cite{BBM1}, \cite{BBM3}, \cite{BBM4},
the NLS equation is a vector valued system, requiring to analyze  commutators 
of  matrix valued  pseudodifferential operators. 
Actually, strictly speaking, the operators involved are not pseudodifferential since their symbols are not $ {\cal C}^\infty $.
The regularity assumption \eqref{regolarita di p} on the perturbation allows
 to perform the Nash-Moser iteration  in Sobolev spaces of fixed regularity 
with respect to the space variable. As a consequence  
we have to choose the transformations in Sections \ref{sec1} - \ref{sec2} with care.  
After these preliminary changes of coordinates have been performed, 
we apply a KAM type scheme, described in detail in Section \ref{sec:redu}, to reduce,  
for $\om $'s satisfying the second order Melnikov non-resonance conditions, the above linear operator 
to a $ 2 \times 2 $ block diagonal infinite dimensional matrix with $\vphi$-independent coefficients.
We express the set of $ \omega $'s satisfying 
the second order Melnikov non-resonance conditions at each step of the induction 
in terms of the  reduced operator only, see  \eqref{Omegainfty} as well as Lemma  \ref{inclusion of cantor sets}.   
The measure estimates for these sets are performed  in section \ref{sec:measure}.

\medskip

\noindent {\em Related results}: The first KAM theorem for analytic perturbations of the dNLS
equation was established by Kuksin and P\"oschel \cite{KP2} for
finite dimensional tori near zero. To avoid the difficulties
caused by the near resonances of $\omega ^{nls}_k$ and $\omega
^{nls}_{-k}$ for $|k| \rightarrow \infty $, they considered the
dNLS equation on the dNLS invariant subspace of $H^\sigma$ of odd functions, requiring the perturbation to be odd.
Further results of this kind can be found for instance in \cite{LY}. 
Using the integrability
of the dNLS equation this result was shown in Gr\'ebert and Kappeler \cite{GK1} to hold for finite
dimensional tori of arbitrary size contained in one of the
subspaces defined by the fixed point sets of the maps
$ R_\alpha : u(x) \mapsto e^{\ii \alpha }u(1 - x), \,
 \alpha \in {\mathbb R} / 2 \pi {\mathbb Z}.$
Again, these subspaces are invariant under the dNLS flow and the KAM
result holds for perturbations which preserve this symmetry.
For $\alpha = 0$, or $ \alpha = \pi $,  it is the subspace of even, respectively odd, 
functions in $H^\sigma$. In another approach, Geng and You \cite{GY1}
proved a KAM result for  the dNLS equation for tori near zero in case the 
perturbation $f(u)$ in \eqref{1.3} is analytic and does not explicitly depend on $x$,
see  also \cite{GY}. In this case, the momentum is an additional integral for the perturbed PDE, 
allowing to deal with the difficulties  caused by the near resonances
of $\omega ^{nls}_k$ and $\omega ^{nls}_{-k}$. 
It can be shown that this result actually holds
for perturbations of finite gap solutions of arbitrary size, see Liang and Kappeler \cite{KL}.

The difficulty posed by resonant frequencies has been also solved for analytic perturbations of the dNLS
equation in $ 1 $-space dimension by
Craig and Wayne \cite{CW} for small periodic solutions, and  by Bourgain \cite{B1} for small quasi-periodic solutions 
by an approach which does not require second order Melnikov conditions. 
These results do not prove the linear stability of the  quasi-periodic solutions. In higher space dimensions
this approach has been extended in \cite{B2},  \cite{B4},  \cite{BB}, \cite{Wang}. 
A KAM theorem  with second order Melnikov non-resonance conditions for the
Schr\"odinger equation with convolution potential 
and analytic perturbations has been  
developed by Eliasson and Kuksin  in \cite{EK} where they introduced the notion of T\"oplitz-Lipschitz matrices. 
Further KAM results have been proved  by \cite{GY2}, \cite{GXY}, \cite{PP} using the conservation of momentum.

Our approach is completely different from the one of the  KAM result of Eliasson and Kuksin. 
As mentioned above, the key point is the expansion \eqref{Ome-exp} for the frequencies of the perturbed equations, 
which is obtained  by conjugating the linearized equation 
\eqref{eq:linearized}  
to a system of equations decoupled 
up to order $ |k|^{-1} $,  
with leading coefficients 
given by  \eqref{Ome-exp} --
see Section \ref{sec:5}. This allows to verify the
second order Melnikov conditions for perturbations of the $1$-dimensional dNLS equation 
with  periodic boundary conditions. 
Our approach does not require the perturbation to be analytic.
We also mention the recent related work \cite{FP} 
where small quasi-periodic solutions 
for fully nonlinear forced reversible Schr\"odinger equations are constructed.

\medskip


\noindent {\em Organization}: The paper is organized as follows: 
In Section~\ref{2. Preliminaries} and  Section~\ref{Estimates on H_e} we introduce 
additional notation and discuss
auxilary results used throughout the paper.
In Section~\ref{Main result restated} we  restate Theorem~\ref{Theorem 1.1} 
in our functional setup, and outline the organisation of its proof.
In Section~\ref{3. Set up} we analyze the differential of $F_\om$
and prove the results on the approximate right inverse needed 
in the proof of the Nash-Moser iteration scheme, assuming
results on the approximate right inverse of the part of the differential,
acting in normal directions. The latter results are proved in 
Section~\ref{sec:5} (preliminary transformations) and
Section~\ref{sec:redu} (reduction to a constant $2\times 2$ block diagonal operator
by a KAM interation scheme).
In Section~\ref{sec:NM} we construct solutions of $F_\om (\io) =0$ by
the aforementioned Nash-Moser iteration scheme for $\om$'s, satisfying appropriate
non-resonance conditions. Finally, in Section~\ref{sec:measure} 
we obtain the claimed measure estimates of Theorem~\ref{Theorem 1.1} 
of the subset $\Omega_\e$. 

For the convenience of the reader all the above arguments are proved in a self-contained way. 

\medskip

\noindent
{\em Notations:} 
Throughout the paper, for  $\s \in \Z_{\ge 0}$, 
$H^\sigma \equiv H^\sigma (\T_1, \C)$ denotes the Sobolev space 
\be\label{def:H-sigma}
 H^\sigma = \big\{ f \in L^2 (\T_1, \C) : \| f \|_\s < \infty  \big\}\,, \qquad
\| f \|_\s \equiv \| f \|_{ H^\sigma} := \Big( \sum_{n \in \Z}  \langle n \rangle^{2\s} | f_n |^2 \Big)^{1/2}
\ee
where 
\be\label{f:Fourier}
f (x) = \sum_{n \in \Z} f_n e^{\ii 2 \pi n x}, \qquad f_n= \int_0^1 f(x) e^{- \ii 2 \pi n x}\, d x\, , \quad n \in \Z\, ,
\ee
and  $\langle n \rangle  := {\rm max}\{ 1, |n| \}$.
Since the Fourier transform is an isometry between $H^\sigma$ and the sequence space
$h^\s \equiv h^\s(\Z, \C)$, we will not distinguish between the two spaces and  
frequently identify a function $ f(x) = \sum_{n \in \Z} f_n e^{2 \pi \ii n x } $ with the sequence 
of its Fourier coefficients $  (f_n)_{n \in \Z} $.  Similarly, we will identify 
the subspace 
\be\label{H-sigma-bot}
H_\bot^{\s}   := 
\Big\{ f(x) = \sum_{n \in \Z} f_n e^{\ii 2\pi n x} \in H^{\s} : \,\, f_n= 0 \,, \,\,\, \forall n \in S \Big\} 
\ee
of $ H^\sigma $
with the corresponding subspace 
$h^\s_\bot = h^\s (S^\bot, \C)$ of $h^\s$ where,  throughout the paper,
$S^\bot$ denotes the complement $\Z \setminus S$ of a given finite subset $S \subset\Z$.
We denote by $ \pi_\bot $ the standard $L^2$-orthogonal projection of  $H^{\s} $ onto $H_\bot^\s  $,  
\be\label{def:pi0-bot}
\pi_\bot : H^{\s} \to H_\bot^\s \, .
\ee
 Let
\be\label{complex-real-scalar-pr}
\langle f, g \rangle := \int_{\T_1} f(x) \bar g(x) \, dx \, , \qquad \langle f, g \rangle_r := \int_{\T_1} f(x)  g(x) \, dx \, . 
\ee
For a linear operator $ A $ acting in $ L^2 (\T_1)  $ we denote by $ A^* $ its adjoint  
with respect to the complex inner product 
$ \langle \ , \ \rangle $ 
and by $ A^t $ the one with respect to the bilinear form $ \langle \ , \  \rangle_r $. We also denote 
$$
\overline {A} (f) := \overline{A ( \, \overline{f} \, )}
$$
and note that  $  A^* = \overline{A}^t $.  
We shall use the notation $ A^* $, $ A^t $, $ \overline{A} $  also
for an operator $A$ acting on the sequence space $ h^\s $. 
Furthermore, we need to consider maps $f: \T^S \to X$ with values in a 
$\C-$Banach space $ X $. Given any $ L^2-$map $ f : \T^S \to X  $  
(in the sense of Bochner), we define its Fourier coefficients
\be\label{Fourier-coeff-f}
\hat f (\ell) := \frac{1}{(2 \pi)^{|S|} } \int_{\T^S} f(\vphi ) e^{- \ii \ell \cdot \vphi } d \vphi \, \in X , \quad \ell \in \Z^S,
\ee
and  for any $ s \in \Z_{\ge 0} $ the norm
\be\label{spaceHs}
\| f \|_s := \Big( \sum_{\ell \in \Z^S} \| \hat f(\ell) \|_X^2 \langle \ell \rangle^{2s} \Big)^{1/2} \, , 
\ee
where for $\ell = (\ell_k)_{k \in S} \in \Z^S,$    
$$
\langle \ell \rangle := {\rm max}\{ 1, |\ell| \} \, , \qquad |\ell| := \sum_{k \in S} |\ell_k|.
$$
We denote by $L^2 ( \T^S, X) $ the space of $L^2-$maps $f: \T^S \to X$ and introduce
for any $s \in  \Z_{\ge 0}$ the Banach space
\be\label{def:Hs}
 H^s (\T^S, X)  := \Big\{ f  \in L^2 ( \T^S, X) \, : \,  \| f \|_s < \infty \Big\}\, .
\ee
Usually, we write $L^2 (\T^S, X)$ instead of $H^0 (\T^S, X)$. 

For any $ s \in \Z_{\ge 0} $, ${\cal C}^s(\T^S, X)$ denotes the Banach space
of $ {\cal C}^s-$smooth maps on $\T^S$ with values in $X$, equipped with the norm 
\be\label{norm-Cs}
\|f \|_{{\cal C}^s} := \sum_{0 \leq |\alpha| \leq s} \| \partial_\vphi^\alpha f \|^{\sup}_X\, ,
\qquad  \| \partial_\vphi^\alpha f \|^{\sup}_X := \sup_{\vphi \in \T^S} \| \partial_\vphi^\alpha f(\vphi) \|_X
\ee
where we have used the customary multi-index notation, i.e., for any $\alpha = \big( \alpha_k \big)_{k \in S} 
\in \Z^S_{\ge 0},$ $\partial_{\vphi}^\alpha$ is the differential operator given by
$ \prod_{k \in S}(\partial_{\vphi_k})^{\alpha_k}$ and $|\alpha | = \sum_{k \in S} \alpha_k$.
Frequently, we will identify $f: \T^S \to X$ with its
lift $\R^S \to X$, which is periodic with respect to the lattice $(2\pi \Z)^S$.
Furthermore, we define
$$
s_0 :=   [|S|/2] + 1 \in \Z 
$$
so that $ H^s (\T^S, X) \hookrightarrow  {\cal C}^0 (\T^S, X) $ for any $ s \geq s_0 $, cf Lemma \ref{lemma D omega}. 

For a map $f :  \Om \to X $, $ \om \mapsto f_\om$ with domain of definition $ \Omega \subset \R^S$
and target a $\C-$Banach space $X$, we define its sup-norm and its Lipschitz semi-norm by
\be \label{def norma sup lip}
 \| f \|^{\sup}_{X,\Omega} 
:= \sup_{ \om \in \Omega } \| f_{\om} \|_X \, , 
\quad
 \| f \|^{\lip}_{X,\Omega}  
:= \sup_{\begin{subarray}{c} \om_1, \om_2 \in \Omega \\ \om_1 \neq \om_2 \end{subarray}} 
\frac{ \| f_{\om_1} - f_{\om_2} \|_X }{ | \om_1 - \om_2 | }\,,
\ee
and, for $0 < \g < 1 $ as in \eqref{Omega o Omega gamma tau}, the Lipschitz norm
\be \label{def norma Lipg}
 \| f \|^\Lipg_{X,\Om} := \| f \|^{\sup}_{X, \Omega} + \g \| f \|^{\lip}_{X, \Omega}  \, . 
\ee
If $ X = H^s(\T^S, \C) $ or $X= H^s(\T^S, H^\s),$ we simply write $  \| f \|^{\Lipg}_s $ for $ \| f \|^{\Lipg}_{H^s}$. In the sequel we will typically suppress $\Omega$ in the above norms, whenever the context permits. 

Finally, throughout the paper, the expression $ a \leq_s b $ means that
 there exists a constant $ C (s) $
such that $ a \leq C(s)  b $ where $s$ refers to the index of the Sobolev space $H^s(\T^S, X)$.
The constant may depend on data such as $|S|$, $\tau$, $ \Omega $, the perturbation $P,$ $\dots \,\,\, $.
The notation $ a \lessdot b $ means that in addition, the constant $ C $ 
is independent of the Sobolev index $ s $.  The  constants $ C(s)$ and $C$ may change from one argument to another. If a constant $\kappa$ depends only on $|S|$ and $\tau$ such as the number $s_0$, we 
often will write $\lessdot$ for $\leq_{\kappa}$.

\section{Functional analytic prerequisites}
\label{2. Preliminaries}

In this section we introduce additional notation
and discuss some auxiliary results from functional analysis, needed in the sequel.


\subsection{Sobolev spaces}
\label{Functions and norms}

We discuss elementary properties of the Banach spaces $H^s (\T^S, X)$. 

\begin{lemma}\label{lemma D omega}
Let $f$ be an element in $H^{s_0} (\T^S, X)$  with 
$s_0 : = [|S| / 2] + 1$.  Then the following holds:

\noindent
$(i)$ For any $\vphi \in \T^S$, the series 
$\sum_{\ell \in \Z^S} \hat f(\ell) e^{\ii \ell \cdot \vphi}$
converges absolutely and 
$ f(\vphi ) = \sum_{\ell \in \Z^S} \hat f(\ell) e^{\ii \ell \cdot \vphi}\,. $

\noindent
$(ii)$ If $\| f \|_{s + 1} < + \infty$ for some $s \geq s_0$, then for any $\om \in \R^S,$
$$
\| (\omega \cdot \partial_\vphi) f\|_s \lessdot \| f \|_{s + 1}\,
$$
where $\omega \cdot \partial_\vphi = \sum_{k \in S}\omega_k \pa_{\vphi_k}$. 

\noindent
$(iii)$ For any $s \in \Z_{\ge 0}$,
\be\label{trivialembe}
\| f \|_{{\cal C}^s} \leq_s \| f \|_{s+s_0} \, ,   \qquad
\| f \|_{s} \leq_s \| f \|_{{\cal C}^{s + s_0}} \, 
\ee
where the Banach spaces $( {\cal C}^s, \| \cdot \|_{{\cal C}^s} )$ were introduced  at the end of 
Section~\ref{introduzione paper}, see  \eqref{norm-Cs}.
\end{lemma}

If $ (X, \langle \, \cdot  , \, \cdot \rangle) $ is a $ \C $--Hilbert space then 
Plancherel's theorem holds, i.e.  (cf \eqref{Fourier-coeff-f})
$$
\frac{1}{(2 \pi)^{|S|} }\int_{\T^S} \langle f(\vphi),  g(\vphi)\rangle d \vphi =  \sum_{\ell \in \Z^S} 
\langle \hat f(\ell), \hat g(\ell) \rangle \, , 
\quad \forall f, g \in L^2 (\T^S, X )\, , 
$$
implying that for any $ s \geq 0 $,
\be\label{HSHilbert}
\| f \|_s \stackrel{\eqref{spaceHs}} = (2 \pi)^{-|S|/2} 
\Big\|  \sum_{\ell \in \Z^S} \langle \ell \rangle^s \hat f(\ell) e^{\ii \ell \cdot \vphi } \Big\|_{L^2(\T^S, X)}  
\ee
and that in this case, the  $ L^2 $-Fourier theory for scalar valued functions extends in a straightforward way. 

In the iteration schemes considered in this paper,
we will frequently encounter equations of the form
\begin{equation}\label{diophantine equation}
(\omega \cdot \partial_\vphi) f = g
\end{equation}
where $\om \in \R^S$ is assumed to satisfy the diophantine conditions \eqref{Omega o Omega gamma tau} and
$g: \T^S \to X$ the compatibility assumption
$ \hat g(0)  = 0 $. The solution $f = (\omega \cdot \partial_\vphi)^{- 1}g$ is given by 
\begin{equation}\label{omega partial vphi inverso}
 \hat f (0) = 0\,, 
\qquad  \hat f (\ell)
:= \frac{ \hat g(\ell)  }{\ii \omega \cdot \ell}\,, 
\quad \forall \ell \in \Z^S \setminus \{ 0 \} \, , 
\end{equation} 
and satisfies the following standard estimates. 

\begin{lemma}\label{om vphi - 1 lip gamma}
Let $s \ge s_0$ and assume that $\omega \in \R^S$ satisfies the diophantine conditions \eqref{Omega o Omega gamma tau}. 
Then for any $g \in H^{s + \tau}(\T^S, X)$ with $\hat g(0) = 0$, 
the linear equation \eqref{diophantine equation} has a unique solution $f \in H^{s}(\T^S, X)$ with $\hat f(0) = 0$. It satisfies the estimate 
$$
\| f \|_s \lessdot \gamma^{- 1} \| g \|_{s + \tau}\,.
$$
If $g = g_\omega \in H^{s + 2 \tau + 1}(\T^S, X)$ is Lipschitz continuous in $\omega \in \Omega \subseteq \R^S$, then the solution $f = f_\omega \in H^s(\T^S, X)$ is Lipschitz continuous in $\omega$ and satisfies 
\begin{equation}\label{om d vphi Lipg}
\| f \|_s^\Lipg \lessdot \gamma^{- 1} \| g \|_{s + 2 \tau + 1}^\Lipg\,.
\end{equation}
\end{lemma}

For the class of semilinear perturbations considered in \eqref{nonlin:f} -- \eqref{regolarita di p}, 
it is possible to keep the index $\s \geq 4$ of the Sobolev space $H^\s \equiv H^{\s}(\T_1, \C)$ fixed, 
 whereas the index $ s $ of the Sobolev spaces $H^s(\T^S,X)$ varies due to a possible loss of regularity 
in the (time) variable $ \vphi$ along the various iteration schemes.
Nonetheless, since the dNLS equation \eqref{1.1} contains the differential operator 
$ \pa_{x}^2 $, we also will need 
to consider  functions with values in $ H^{\s'} $ with $\s'$ such as $\s - 2$. 
We recall that we identify $ H^{\s'} $ with $ h^{\s'} $ via the Fourier transform.
In the sequel, we will frequently consider the Sobolev space $\big( H^s (\T^S, h^{\s'} ), \| \ \|_{s, \s' } \big)$
of maps with values in the Hilbert space $h^{\s'}$ where $\s' \in \Z_{\ge 0}$ and the norm
$\| u \|_{s, \s' }$ of $u$ is given by
\be\label{norma other s-2} 
\| u \|_{s, \s' }:= \Big( \sum_{\ell \in \Z^S} \| \hat u(\ell )\|_{h^{\s'}}^2 \langle \ell \rangle^{2 s} \Big)^{1/2} \, .
\ee
In the case where $\s' = \s,$ we simply write $\| u \|_{s}$ instead of $\| u \|_{s, \s }$.
For any $\ell \in \Z^S$, the Fourier coefficient $ \hat u(\ell )$ is a sequence in $h^{\s'}$,
which we denote by $\big( \hat u_n(\ell) \big)_{n \in \Z}.$ Note that $ \hat u_n(\ell),$ $\ell \in \Z^S$,
are the Fourier coefficients of the function $\vphi \mapsto u_n(\vphi)$, which is the n'th component
of $u(\vphi) = \big( u_j(\vphi) \big)_{j \in \Z}$, i.e., 
$u_n(\vphi) = \sum_{\ell \in \Z^S}  \hat u_n(\ell) e^{\ii \ell \cdot \vphi}$.
Furthermore, 
\be\label{norma other}
 \| u \|_{s, \s'}^2 = 
 \sum_{ n \in \Z, \ell \in \Z^S} | \hat u_n (\ell ) |^2 \langle n \rangle^{2 \s'} \langle \ell \rangle^{2 s}  =
\sum_{n \in \Z}   \| u_n  \|_{s}^2 \langle n \rangle^{2 \s'} 
\ee
where $  \| u_n  \|_s =  \| u_n  \|_{H^s(\T^S,\C)} $. 
We shall also consider functions 
$ \vphi \mapsto y(\vphi)$ with values in $\R^S$ in the Sobolev space $H^s(\T^S, \R^S)$
whose  norm is also denoted by 
$$ 
\| y \|_s := \| y \|_{H^s(\T^S, \R^S)}\,.
$$
Another class of Sobolev spaces used in this paper are the spaces of operator valued maps,
 $H^s (\T^S, {\cal L} (h^{\s'} ) )$,
where ${\cal L}(h^{\s'} )$ denotes the Banach space of bounded linear  operators on $ h^{\s'} $, endowed
with the operator norm. 
A linear operator $ A $ has a natural matrix representation 
$ (A^j_k)_{j, k \in \Z} $ determined by 
\be\label{matrix:coefficients}
 \qquad (A (h) )_k = \sum_{j \in \Z} A^j_k h_j \in \C \, , \quad k \in \Z \, . 
\ee
We will also consider such Sobolev spaces with
 $ h^{\s'} (\Z, \C) \times  h^{\s'} (\Z, \C)  $ or $ h_\bot^{\s'} $ instead of $h^{\s'}$. 
For an element $\vphi \mapsto A(\vphi)$ in $H^s (\T^S, {\cal L} (h^{\s'} ) )$, 
the correponding norm is conveniently denoted by $|A |_{s, \s' },$ i.e., 
\be\label{def norm1}
|A |_{s, \s' } := 
\Big( \sum_{\ell \in \Z^S} \| \hat A(\ell )\|_{\s'}^2 \langle \ell \rangle^{2s} \Big)^{1/2}\,, 
\qquad \| \hat A(\ell )\|_{\s'} := \| \hat A(\ell ) \|_{{\cal L}( h^{\s'})} \, . 
\ee
In case $\s' = \s,$ we simply write $|A|_{s}$ instead of $|A|_{s, \sigma}$.
We remark that $|A|_{s}$ is a quite strong norm 
but particularly convenient for estimating solutions of homological equations -- see e.g. Lemma \ref{lemma:redu}. 

According to \eqref{def norm1}, \eqref{trivialembe}, \eqref{spaceHs} one has
\be\label{upper-bound-norm-op}
|A |_{s, \s' } \leq_s \| A \|_{{\cal C}^{s + s_0}(\T^S, {\cal L}(h^{\sigma'}))}\, \quad \text{ and } \quad  \| A \|_{{\cal C}^{s }(\T^S, {\cal L}(h^{\sigma'}))} \leq_s |A |_{s + s_0, \s' } \, . 
\ee
To state our next result, let $ D $ be the operator 
defined for  $ h = (h_j)_{j \in \Z} $ 
by setting
\be\label{def:D}
(D h)_j  := 2 \pi j h_j \, , \quad \forall j \in \Z \, ,
\ee
and let $\lla D  \rra  := (1 + D^2)^{1/2}$, i.e.   
\be\label{def:DD}
(\lla D  \rra h)_j := \lla j \rra  h_j \, ,
\qquad \lla j\rra := (1 + (2 \pi j)^2 )^{1/2}\, \qquad \forall  j \in \Z \, .
\ee
Note that $ D $ is the operator 
corresponding to the Fourier multiplier  $\frac{1}{\ii} \partial_x $. 

\begin{lemma}\label{lemma-reg-D} Let $s \in \Z_{\ge 0}$ and $\s \in \Z_{\ge 2}$ and
assume that $A$ is in $H^s (\T^S, {\cal L} (h^{\s-2}, h^{\s-1} ) )$.
Then the following holds:

\noindent
$(i) \qquad \qquad $  $ | A |_{s, \s-2} \lessdot | A \,  \lla D \rra |_{s, \s-1} \, \quad \mbox{and} 
\quad  | A |_{s, \s-1} \lessdot | A \, \lla D \rra |_{s, \s -1} $. 

\noindent
$(ii)$ If $A = A_\omega$ is Lipschitz continuous in $\omega \in \Omega \subseteq \R^S$ 
then 
$$ 
| A |_{s, \s-2}^\Lipg \lessdot | A \,  \lla D \rra |_{s, \s-1}^\Lipg \,  
\quad \mbox{and}
\quad | A |_{s, \s-1}^\Lipg \lessdot | A \, \lla D \rra |_{s, \s-1}^\Lipg \, .
 $$
\end{lemma}

\begin{proof}
Since for any $\ell \in \Z^S$, $\hat A(\ell) $  satisfies
$$
 \| \hat A(\ell) \|_{\s-2} \leq \| \hat A(\ell) \|_{ {\cal L} (h^{\s-2}, h^{\s-1} )}
\leq \| \hat A(\ell) \, \lla D \rra \|_{\s-1}  
\| \lla D \rra^{-1} \|_{ {\cal L} (h^{\s-2}, h^{\s-1} )}
\lessdot 
\| \hat A(\ell) \, \lla D \rra \|_{\s -1}\, ,
$$ 
and similarly, 
$$ 
\| \hat A(\ell) \|_{\s-1} 
 \leq \| \hat A(\ell) \, \lla D \rra \|_{ {\cal L} (h^{\s}, h^{\s-1} )}  
\|\lla D \rra^{-1} \|_{ {\cal L} (h^{\s-1}, h^{\s} )}
\lessdot 
\| \hat A(\ell) \lla D \rra \|_{\s-1} , 
$$
item $(i)$ holds.
The claimed estimates of item $(ii)$ are an immediate consequence of item $(i)$. 
\end{proof}

Finally, we consider the operator, defined by multiplication with a map. More precisely,
assume that $ q$ is  in  $H^s (\T^S, H^{\s'} ) $ with  $s\ge s_0$ and $\s' \ge 1$.
The latter conditions imply that $H^{\s'}$ and in turn $H^s (\T^S, H^{\s'} ) $ are algebras
and hence the operator $\Lambda_q$ of multiplication by $q$, defined on $H^s (\T^S, H^{\s'} ) $
by setting for any $\vphi \in \T^S,$
$$
 \Lambda_q (\vphi) : H^{\s'} \to H^{\s'}\, , \ \  f \mapsto 
\Lambda_q (\vphi) f (\cdot) : = q(\vphi, \cdot) f(\cdot)
$$
is well defined. In the following lemma we again identify the Hilbert spaces $H^{\sigma'}$ and $h^{\sigma'}$ 
by the Fourier transform.

\begin{lemma}\label{lemma:mult} {\bf (Multiplication and commutator estimates)}
Let $ q \in H^s (\T^S, H^{\s}) $ with $s \ge s_0$ and $\s \ge 4$. Then the following holds:

\noindent
$(i)$ For any $\sigma' \in \{ \sigma, \sigma - 1 , \sigma - 2, \s -3 \}$, 
$| \Lambda_q |_{s, \sigma'}  \lessdot  \| q \|_{s,\s'} $.

\noindent
$(ii)$ For any $\sigma' \in \{ \sigma , \sigma - 1, \s -2  \}$, 
the commutator $\,[\, \lla D \rra, \Lambda_q \, ]\,$ of 
$ \lla D \rra $ with $\Lambda_q$
satisfies
$$
|\, [\, \lla D \rra, \Lambda_q \, ] \, |_{s, \sigma' - 1} 
\lessdot \| q\|_{s,\s'}\,.
$$
\end{lemma}

\begin{proof}
$(i)$ Since $\s \ge 4,$ one has $\s' \ge 1$ for 
$\s'$ in $\{ \sigma, \sigma - 1 , \sigma - 2, \s -3 \}$. 
Furthermore, the Fourier coefficient
 $ \hat \Lambda_q (\ell) : H^{\s' } \to H^{\s'} $, $\ell \in \Z^S$, 
 is the multiplication operator 
by the function $ \hat q (\ell) \in H^{\s'} $. Its operator norm is bounded by 
$ C \| \hat q (\ell) \|_{H^{\s'}} $ with $C \equiv C(\s')$
and thus, recalling \eqref{def norm1}, 
$$
| \Lambda_q |_{s, \s'}  \leq  C \Big( \sum_{\ell \in \Z^S} \| \hat q( \ell ) \|_{H^{\s'}}^2  
\langle \ell \rangle^{2s} \Big)^{1/2}
\leq C \| q \|_{s, \s'}\,.
$$

\noindent
$(ii)$ 
Let $A := [\, \lla D \rra, \Lambda_q \,]$. Then the operator $\hat A(\ell )$ is represented by the matrix
$$
\hat A(\ell)_j^{j'}  = \big( \lla j \rra - \lla j' \rra \big) 
\hat q_{j - j'}(\ell)\,, \quad j, j' \in \Z \, .
$$
Since 
$ \langle j \rangle^{\sigma' -1} 
\lessdot \langle j - j' \rangle^{\sigma' -1} +  \langle j' \rangle^{\sigma' -1}$ 
and $| \lla j \rra - \lla j' \rra | 
\lessdot \langle j - j' \rangle,$
one gets 
 that, for any $h = (h_j)_{j \in \Z}$ in $h^{\s' -1}$, 
\begin{align*}
\| \hat A(\ell) h \|_{H^{\sigma' - 1}}^2 & = \sum_{j \in \Z} \langle j \rangle^{2 (\sigma' - 1)} \big| \sum_{j' \in \Z} \hat A(\ell)_j^{j'} h_{j'}\big|^2 \nonumber\\
& \lessdot \sum_{j \in \Z} \Big( \sum_{j' \in \Z} \langle j - j'\rangle^{\sigma'} |\hat q_{j - j'}(\ell)| |h_{j'}|  \Big)^2  + \sum_{j \in \Z}\Big( \sum_{j' \in \Z} \langle j - j'\rangle |\hat q_{j - j'}(\ell)| \langle j'\rangle^{\sigma' - 1} |h_{j'}|  \Big)^2 =: I + II  \, . 
\end{align*}
Since,  by assumption, $ \s' -1 \ge 1$,  we get, by the Cauchy Schwartz inequality
\begin{align*}
I  & \lessdot \sum_{j \in \Z} \Big( \sum_{j' \in \Z} \langle j - j'\rangle^{\sigma'} |\hat q_{j - j'}(\ell)|
\langle  j'\rangle^{\sigma' -1}  |h_{j'}|  \frac{1}{\langle  j'\rangle^{\sigma' -1}}\Big)^2 \\
& \lessdot 
\sum_{j \in \Z} \Big( \sum_{j' \in \Z} \langle j - j'\rangle^{2 \sigma'} |\hat q_{j - j'}(\ell)|^2 
\langle  j'\rangle^{2 (\sigma' -1)}  |h_{j'}|^2 \Big) \Big( \sum_{j' \in \Z} \frac{1}{\langle  j'\rangle^{2 (\sigma' -1)}}\Big) \\
& \lessdot 
 \sum_{j \in \Z} \langle j - j'\rangle^{2 \sigma'} |\hat q_{j - j'}(\ell)|^2 
\sum_{j' \in \Z}  \langle  j'\rangle^{2 (\sigma' -1)}  |h_{j'}|^2 
\lessdot \| \hat q(\ell)\|_{H^{\sigma'}}^2 \| h \|_{H^{\sigma' - 1}}^2  \, .
\end{align*}
The term $II$ is estimated in the same way, 
yielding altogether  
\begin{equation}\label{widehat A ell widehat q ell}
\| \hat A(\ell)\|_{{\cal L}(H^{\sigma' - 1})} \lessdot \| \hat q(\ell) \|_{H^{\sigma'}} \, .
\end{equation}
Finally 
$$
|A|_{s, \sigma' - 1} = \Big( \sum_{\ell \in \Z^S} \langle \ell \rangle^{2 s} \|\hat A(\ell) \|^2_{{\cal L}(H^{\sigma' - 1})}\Big)^{1/2} 
\ \stackrel{\eqref{widehat A ell widehat q ell}}{\lessdot} \Big( \sum_{\ell \in \Z^S} 
\langle \ell \rangle^{2 s} \| \hat q(\ell) \|_{H^{\sigma'}}^2 \Big)^{1/2} 	
\lessdot \| q\|_{s, \s'} \,,
$$
which is the claimed estimate of item $(ii)$. 
\end{proof}

\subsection{Smoothing operators and interpolation}
In this subsection, we review the notion of families of smoothing operators for scales 
of Banach spaces and discuss specific examples, needed on the sequel.
 Assume that $(X_k)_{k \in \Z_{\geq 0}}$ is a scale of Banach spaces $\, \cdots \subseteq X_{k + 1}\subseteq X_k \subseteq \cdots \subseteq X_1 \subseteq X_0$, with norms $\| \cdot \|_{k} :=\| \cdot \|_{X_k}$, so that for any $0 \leq n \leq k$, $\| \cdot \|_n \leq \| \cdot \|_k$. Let us define $X_\infty := \cap_{k \geq 0} X_k$. 

\begin{definition}[\bf Smoothing operators]\label{definizione astratta smoothing}
A one parameter family of linear operators $S_t : X_0 \to X_\infty$, $t \geq 1$ is said to be a family of smoothing operators for the scale $(X_k)_{k \in \Z_{\geq 0}}$ if the following three conditions are satisfied: 

\noindent
{(SM1)} For any $f \in X_0,$
$$
\lim_{t \to + \infty} \|S_t f - f \|_0 = 0. 
$$ 

\noindent
{(SM2)} For any $k , n \in \mathbb Z_{\ge 0}$ with $ n \leq k$, there exists a constant $C_{k, n} > 0$ such that 
$$
\| S_t f \|_k \leq  C_{k, n} t^n \| f \|_{k - n}\,, \quad \forall f \in X_{k - n}\,, \quad \forall t \geq 1\,.
$$

\noindent
{(SM3)} For any $k, n \in \mathbb Z_{\ge 0}$, there exists a constant $C_{k, n}^{'} > 0$ such that 
$$
\| S_t f - f \|_k \leq C^{'}_{k, n} t^{-n} \| f\|_{k + n}\,, \quad \forall f \in X_{k + n}\,, \quad \forall t \geq 1\,.
$$
\end{definition}
Smoothing operators have the following interpolation property.

\begin{proposition}[\bf Interpolation estimates]\label{interpolation scales}
Given any integers $0 \leq k_1 \leq  k \leq k_2$ with $k_2 - k_1 \geq 1$, there exists a constant $C_{k, k_1, k_2} > 0$ such that 
$$
\| f \|_k \leq C_{k, k_1, k_2} \| f \|_{k_1}^{1 - \lambda} \| f \|_{k_2}^{\lambda} \,, \quad \forall f \in X_{k_2} 
$$
where $0 \leq \lambda \leq 1$ is  $ \lambda := (k - k_1)/(k_2 - k_1).$
\end{proposition}
\begin{proof}
Write $\| f \|_k \leq \| S_t f\|_k + \| S_t f- f \|_k$ and use {\em (SM2) - (SM3)}, to see that the claimed estimate follows by 
choosing $ t $ for minimizing the right hand side. For more details see for instance \cite{BBP3}, Lemma 1.1.
\end{proof}

\noindent
{\it Smoothing operators for scales of Sobolev spaces:} 
Let $H^s(\T^S, X)$, $s \in \Z_{\geq 0}$, be the Banach spaces defined  in \eqref{def:Hs}.
Note that $ {\cal C}^\infty(\T^S, X) = \bigcap_{s \geq 0} H^s(\T^S, X)$. 
We define the one parameter family of operators $\Pi_t$, $t \geq 1$
\begin{equation}\label{smoothing Sobolev}
\Pi_t : L^2(\T^S, X) \to {\cal C}^\infty(\T^S, X)\,,\quad f(\vphi) \mapsto \Pi_t f (\vphi) := \sum_{|\ell| \leq t} \hat f(\ell) e^{\ii \ell \cdot \vphi}\,, \quad \forall t \geq 1\,.
\end{equation}
In the sequel, we will also consider Lipschitz maps $f = f_\omega$, $\om \in \Omega \subset \R^S$, with values in
$ H^s(\T^S, X)$, 
equipped with the norm
$\| f \|_s^\Lipg = \| f \|_s^{\sup} + \gamma \| f \|_{s, \Omega}^{\lip}$ defined in \eqref{def norma Lipg} and \eqref{spaceHs}.
The following lemma can be proved in a straightforward way. 
\begin{lemma}[{\bf Smoothing operators for scales of $H^s$-spaces}]
\label{smoothing sobolev lemma}
The one parameter family of operators $\Pi_t$, $t \geq 1$, defined in \eqref{smoothing Sobolev}, 
is a family of smoothing operators for the scale of Banach spaces $(H^s(\T^S, X), \|\cdot \|_s)$, $s \in \mathbb Z_{\ge 0}$.

\noindent
At the same time, it is also a family of smoothing operators for the scale of Banach spaces of Lipschitz families in $H^s(\T^S, X)$
equipped with the norms $\| \cdot \|_s^\Lipg,$ $s \in \mathbb Z_{\ge 0}$.
\end{lemma}

For later reference, we briefly mention the smoothing operators for the special scales of the spaces 
$H^s(\T^S, {\cal L}(h^{\sigma}))$. For any $t \ge 1$ and  
$A  = \sum_{\ell \in \Z^S} \hat A(\ell) e^{\ii \ell \cdot \vphi}  \in H^s(\T^S, {\cal L}(h^{\sigma}))$,
 $\Pi_t A$ is an operator valued map with Fourier coefficients given by
\be\label{SM}
\widehat {{\Pi_t A }}(\ell)  :=
\begin{cases}
\hat A(\ell)  \qquad \, {\rm if} \  |\ell | \leq t  \\
0 \quad \ \quad \quad {\rm otherwise.}
\end{cases}
\ee
The operator $ \Pi_t^\bot := {\rm Id} - \Pi_t $ satisfies for any $n \in \Z_{\ge 0}$
\be\label{smoothingN}
| \Pi_t^\bot A |_{s} \leq t^{- n} |  A |_{s+n} \, , \qquad
| \Pi_t^\bot A |_{s}^{\gamma {\rm lip}} \leq t^{- n} |  A |_{s+n}^{\gamma 
{\rm lip}} \, .
\ee

\noindent
{\it Smoothing operators for scales of ${\cal C}^s$ spaces:} 
Let us consider the scale of Banach spaces 
${\cal C}^s(\T^S, X)$, $s \in \Z_{\geq 0}$, equipped with the norm  $  \| \cdot \|_{{\cal C}^s}  $ defined in \eqref{norm-Cs}. .
A one parameter family of smoothing operators 
can be constructed as follows (cf e.g. Lemma 6.2.2, Lemma 6.2.4 in \cite{Nirenberg}): 
let $\chi$ be a $ {\cal C}^\infty-$smooth, real valued function on $\R^S$,
which is even and satisfies
$$
 \chi(\xi) = 1 \,, \,\,\, \forall |\xi| \leq 1\,, \quad \mbox{and}
\quad  \chi(\xi) = 0\,, \,\,\, \forall |\xi| \geq 2\,,
$$
and denote by $\rho$ its Fourier transform, 
$$
\rho(\vphi) := \frac{1}{(2 \pi)^{|S|}} \int_{\R^S} \chi(\xi) e^{- \ii \vphi \cdot \xi}\, d \xi\,.
$$
Then $\rho$ is of Schwartz class and, since by assumption $\chi$ is even, real-valued.
Furthermore,
$$
\chi(\xi) = \int_{\R^S} \rho(\vphi) e^{\ii \vphi \cdot \xi}\, d \vphi
$$
implies that $\int_{\R^S} \rho(\vphi)\, d \vphi = \chi(0) = 1\,,$ and for any multi-index $\alpha \in \Z_{\geq 0}^S$, $\int_{\R^S} (\ii \vphi)^\alpha \rho(\vphi)\, d \vphi = \partial_\xi^\alpha \chi(\xi)_{|\xi = 0} = 0$
where $(\ii \vphi)^\alpha =  \prod_{k \in S} (\ii \vphi_k )^{\alpha_k}$.
For any $t \geq 1$, we define the function $\rho_t(\vphi) := t^{|S|} \rho(t \vphi)\,,$ 
which satisfies the identities
$$
\int_{\R^S} \rho_t(\vphi)\, d \vphi = 1\,, \quad \int_{\R^S} (\ii \vphi)^\alpha \rho_t(\vphi)\, d \vphi = 0\,, \quad \forall \alpha \in \Z_{\geq 0}^S\,.
$$
The $\rho_t$'s now yield the following one parameter family of operators, 
\begin{equation}\label{operatori smoothing Cs}
S_t f(\vphi) := (\rho_t \star f)(\vphi) = \int_{\R^S} \rho_t(\vphi - \psi) f(\psi)\, d \psi\,, 
\qquad \forall f \in {\cal C}^0(\T^S, X)\,.
\end{equation}
The maps $S_t f$ are $ {\cal C}^\infty-$smooth and $(2 \pi \Z)^S-$periodic, i.e.,  
$$
S_t : {\cal C}^0(\T^S, X) \to {\cal C}^\infty(\T^S, X) = \bigcap_{s \geq 0} {\cal C}^s(\T^S, X)\,.
$$
The following lemma can be proved in a straightforward way. 
\begin{lemma}[{\bf Smoothing operators for scales of ${\cal C}^s$-spaces}]
\label{smoothing family Cs}
The one parameter family of operators $S_t$, $t \geq 1,$ defined in \eqref{operatori smoothing Cs}, 
is a family of smoothing operators for the scale of Banach spaces $\big( {\cal C}^s(\T^S, X), \| \cdot \|_{{\cal C}^s} \big),$ $s \in \Z_{\geq 0}.$
\end{lemma}

\subsection{Tame estimates} \label{sec:2}

The aim of this subsection is to discuss 
various tame estimates with respect to the $\vphi$-variable.
Since  the class of semilinear perturbations  \eqref{nonlin:f} -- \eqref{regolarita di p}
considered in this paper, do not lose 
regularity with respect to the $ x $-variable, 
tame estimates with respect to the space variable are not needed.  
We begin with establishing tame estimates for the product of maps $u, v$ in $H^s(\T^S, H^\sigma)$.
Recall  that for $s \ge s_0$ and $\s \ge 1$, 
$H^s(\T^S, H^\sigma)$ is an algebra. 
Establishing tame estimates for the product $uv$ 
means to bound the norm $\|u v \|_{s}$ by an
expression which is linear in the high norms $\|u \|_{s}$ and $\|v \|_{s}$. 
More precisely, we have the following result.

\begin{lemma}[\bf Tame estimates for products of maps]
\label{interpolation product}
Let $s \in \Z_{ \ge s_0}$ and $\s \ge 1$. Then there are constants 
$C_{prod}(s) \ge C_{prod}(s_0) \ge 1$ (which also might depend on $\s$),
so that the following holds:

\noindent
($i$) for any $u, v \in H^s(\T^S, H^\sigma)$, 
\be \label{tame for functions}
\|u v \|_{s} \leq C_{prod}(s_0) \|u \|_{s_0} \| v \|_{s} 
+   C_{prod}(s)\| u \|_{s} \| v \|_{s_0}\,;
\ee
($ii$) for any $u\equiv u_\omega$, $v \equiv  v_\omega$
in $H^s(\T^S, H^\sigma)$, which are Lipschitz continuous
 in the parameter $\omega \in \Omega \subseteq \R^S$, 
\be \label{lip tame for functions}
\|u v \|_{s}^\Lipg \leq C_{prod}(s_0) \|u \|_{s_0}^\Lipg \| v \|_{s}^\Lipg 
+   C_{prod}(s)\| u \|_{s}^\Lipg \| v \|_{s_0}^\Lipg\,.
\ee
In the case where $u, v \in H^s(\T^S, \C)$, the same tame estimates hold  with $ \| \ \|_{s} $  
replaced by $\| \ \|_{H^s(\T^S, \C)}$. 
\end{lemma}
\begin{proof}
The proof follows the classical argument,  see e.g. \cite{BBP3}. 
We have to estimate the $ \| \cdot\|_s $-norm of the map 
$$
\vphi \mapsto u(\vphi)  v (\vphi )   = \sum_{\ell \in \Z^S} 
\Big( \sum_{k \in \Z^S} \hat u(k) \hat v (\ell -k) \Big)\, e^{\ii  \ell \cdot \vphi} \, . 
$$
Using that $ H^\s$ is an algebra and that for any two elements $f,g$ in $ H^\s$,
$\| fg \|_\s \le C \|f \|_\s \|g\|_\s$ with $C\equiv C(\s)$, one gets
\be\label{s1s2 for functions}
\| uv \|_s^2 = \sum_{\ell \in \Z^S} 
\Big\| \sum_{k \in \Z^S} 
\hat u (k) \hat v (\ell - k) \Big\|_{\s}^2 \langle \ell \rangle^{2s}
\leq C^2 \sum_{\ell \in \Z^S}  \Big( \sum_{k \in \Z^S} 
\| \hat u(k) \|_{\s}  \| \hat v (\ell -k)\|_{\s} \Big)^2 \langle \ell \rangle^{2s} 
\le 2C^2T_1 + 2C^2T_2 
\ee
where with $ c(s) := 2^{1/s} - 1 $,  
$$
T_1 := \sum_{\ell \in \Z^S}  
\Big( \sum_{{ \langle k \rangle > \langle \ell \rangle \slash (1 + c(s))}} 
\| \hat u(k)\|_{\s}  \| \hat v ( \ell - k) \|_{\s} \Big)^2 
\langle \ell \rangle^{2s} \, , 
$$
and
$$
T_2 := \sum_{\ell \in \Z^S}  \Big( \sum_{{ \langle k \rangle \leq 
\langle \ell \rangle \slash (1 + c(s))}} \| \hat u (k)\|_{\s}  
\| \hat v ( \ell -k) \|_{\s} \Big)^2 \langle \ell \rangle^{2s} \, . 
$$
{\it Estimate of $ T_1 $.} 
We estimate $ T_1 $ using the Cauchy-Schwartz inequality
\begin{align}
T_1 & =   \sum_{\ell \in \Z^S}  
\Big( \sum_{{ \langle k \rangle  >  \langle \ell \rangle \slash (1 + c(s))}}  
\langle k \rangle^{s}
 \| \hat u (k)\|_{\s}  \langle \ell - k \rangle^{s_0}  
\| \hat v ( \ell -k)\|_{\s}  \frac{\langle \ell \rangle^s}{ \langle k \rangle^{s} \langle \ell - k \rangle^{s_0}   } \Big)^2 \nonumber \\ 
& \leq \sum_{\ell \in \Z^S}  \Big( \sum_{{ \langle k \rangle  >  \langle \ell \rangle \slash (1 + c(s))}} 
\langle k \rangle^{s} \| \hat u (k) \|_{\s}  \langle \ell - k \rangle^{s_0}  
\| \hat v  (\ell -k) \|_{\s}  \frac{2}{\langle \ell - k \rangle^{s_0}   } \Big)^2 \nonumber \\
& \leq 4 \sum_{\ell \in \Z^S}  \Big( \sum_{k \in \Z^S} 
\langle k \rangle^{2s} \| \hat u( k) \|_{\s}^2  \langle \ell - k \rangle^{2 s_0}  
\| \hat v ( \ell -k) \|_{\s}^2 \Big)  \sum_{k \in \Z^S} 
 \langle k \rangle^{-2 s_0}  \, .\nonumber 
\end{align}    
Exchanging the order of the sums leads to the bound 
$$
T_1 \leq \tilde C(s_0) \sum_{k, \ell \in \Z^S} 
\langle k \rangle^{2s} \| \hat u ( k) \|_{\s}^2  \langle \ell  \rangle^{2 s_0}  
\| \hat v ( \ell ) \|_{\s}^2  \leq \tilde C(s_0) \|  u \|_s^2 \| v \|_{s_0}^2 
$$
where we emphasize that the constant $ \tilde C(s_0) $ is independent of $ s $. 

\noindent
{\it Estimate of $ T_2 $.} In the sum $ T_2 $ we have 
$ \langle \ell - k \rangle \geq \langle  \ell \rangle - \langle k \rangle \geq \langle  \ell \rangle - \frac{\langle \ell \rangle}{1 + c(s) } $
and so 
$ \frac{\langle \ell \rangle}{\langle \ell - k \rangle} \leq \frac{1 + c(s)}{c(s)}  $. 
Thus, arguing as above, 
$$
T_2  \leq \Big(\frac{1 + c(s)}{ c(s)}\Big)^2 
\sum_{\ell \in \Z^S}  \Big( \sum_{k \in \Z^S} 
\langle k \rangle^{2s_0} \| \hat u ( k) \|_{\s}^2  \langle \ell - k \rangle^{2 s}  
\| \hat v ( \ell -k) \|_{\s}^2 \Big) 
\sum_{k \in \Z^S}   \langle k \rangle^{-2 s_0} \leq \tilde C(s) \| v \|_s^2 \| u \|_{s_0}^2 \, .
$$
The claimed estimate 
\eqref{tame for functions} now follows from \eqref{s1s2 for functions} with
the above bounds for $ T_1 $ and $ T_2 $. 
The  bound \eqref{lip tame for functions} follows by applying \eqref{tame for functions} to
the difference quotient 
$$
\frac{(uv)_{\om_1} - (u v)_{\om_2}}{  \om_1 - \om_2 }  = 
 \frac{ u_{\om_1} - u_{\om_2}}{ \om_1 - \om_2 }  \, v_{\omega_1} +  
u_{\omega_2} \, \frac{ u_{\om_1} - u_{\om_2}}{ \om_1 - \om_2 } 
$$
for any $\omega_1, \omega_2 \in \Omega$.  
\end{proof}

Since for any $\s$, the space of operators $ {\cal L}( H^\s ) $ is an algebra with multiplication given by the
composition of operators and for any two operators $A, B$ in $ {\cal L}( H^\s ) $,
the operator norm  $\|AB\|_\s $ of $AB$ is bounded by $ \|A\|_\s  \|B\|_\s $,
the proof of Lemma~\ref{interpolation product} also shows that
the composition of operator valued maps satisfies tame estimates
with respect to the norm $ | \ |_s = | \ |_{s, \sigma} $
introduced in \eqref{def norm1}. 

\begin{lemma} \label{prodest}
{\bf (Tame estimates for the composition of operator valued maps)}  
Let $s \in \Z_{\ge s_0}$ and $\s \ge 0$. Then there are constants 
$ C_{op}(s) \geq C_{op}(s_0) \geq 1 $ (which also might depend on $\s$), so that the following holds:

\noindent
($i$) for any operator valued maps $A, B$ in $H^s(\T^S, {\cal L}( H^\s ) )$, 
\be \label{interpm}
| B A |_{s} \, , \  | A B|_{s} \leq  
C_{op}(s) |A|_{s_0} |B|_s + C_{op}(s_0) |A|_s |B|_{s_0} \, ; 
\ee
($ii$) for any operator valued maps $A \equiv A_\om$ and $B \equiv B_\om$ 
 in $H^s(\T^S, {\cal L}( H^\s ) )$,  which are Lipschitz continuous 
in the parameter $\om \in \Om \subset \R^S $,
\be
\label{interpm Lip}
|A B|_{s}^{\gamma {\rm lip}} \, , \, |B A|_{s}^{\gamma {\rm lip}}
 \leq C_{op}(s) |A|_{s_0}^{\gamma {\rm lip}} |B|_{s}^{\gamma {\rm lip}}
+ C_{op}(s_0) |A|_{s}^{\gamma {\rm lip}} |B|_{s_0}^{\gamma {\rm lip}} .
\ee
As a  consequence, for any $n \geq 1$, 
\be\label{Mnab}
| A^n |_{s_0} \leq \big( 2C_{op}(s_0) \big)^{n-1} | A |_{s_0}^n \qquad
\text{and}  \qquad   
| A^n |_{s} \leq  
n \cdot \big( 2 C_{op}(s_0) |A|_{s_0} \big)^{n-1} \cdot 
C_{op}(s) | A |_{s} \, , 
\ee
and similar estimates hold for the Lipschitz norm $| \ |_s^\Lipg $. 

\noindent
($iii$) 
The same estimates as in items ($i$)-($ii$)
hold for operator valued maps in  $H^s(\T^S, {\cal L}( h_\bot^\s \times  h_\bot^\s ) )$
where the space $h_\bot^\s = h^\s (S^\bot, \C)$ is introduced in  Notations at the end of Section \ref{introduzione paper}.
\end{lemma}

\begin{remark}\label{generalizzazione prodotto operatori}
Occasionally we need a straightforward generalization of the estimates \eqref{interpm}, \eqref{interpm Lip}. More precisely: 
for 
 $A \in H^s(\T^S, {\cal L}(H^{\sigma_1}, H^{\sigma_2})) $ and $ B \in H^s(\T^S, {\cal L}(H^{\sigma_2}, H^{\sigma_3})) $,  
$BA \in H^s(\T^S, {\cal L}(H^{\sigma_1}, H^{\sigma_3}))$ satisfies the tame estimate 
\begin{align*}
\| B A \|_{H^s(\T^S, {\cal L}(H^{\sigma_1}, H^{\sigma_3}))} & \leq C_{op}(s) \| B \|_{H^s(\T^S, {\cal L}(H^{\sigma_2}, H^{\sigma_3}))}  \| A\|_{H^{s_0}(\T^S, {\cal L}(H^{\sigma_1}, H^{\sigma_2}))} \\
& \quad + C_{op}(s_0) \| B \|_{H^{s_0}(\T^S, {\cal L}(H^{\sigma_2}, H^{\sigma_3}))}  \| A\|_{H^s(\T^S, {\cal L}(H^{\sigma_1}, H^{\sigma_2}))}\,.
\end{align*}
 Moreover if $A = A_\omega$, $B = B_\omega$ are Lipschitz continuous in $\Omega$, then the above estimate holds for the corresponding Lipschitz norms.
\end{remark}

We also need to derive tame estimates for maps of the form $ \vphi \mapsto A (\vphi) u(\vphi) $ 
where  $ \vphi \mapsto u(\vphi)  $ is in the Sobolev space
$H^s (\T^S, h^\s )$ 
and $ \vphi \mapsto A(\vphi)$ is an operator valued map in $H^s(\T^S, {\cal L}( H^\s ) )$.
Writing $A$ and $u$ as Fourier series, 
$A(\vphi) =  \sum_{\ell \in \Z^S} \hat A(\ell) \, e^{\ii  \ell \cdot \vphi} $ respectively
$u(\vphi) = \sum_{\ell \in \Z^S} \hat u(\ell) \, e^{\ii  \ell \cdot \vphi} $, one gets
$$
A(\vphi)  u (\vphi )   = \sum_{\ell \in \Z^S} 
\Big( \sum_{k \in \Z^S} \hat A(\ell - k) \hat u (k) \Big)\, e^{\ii  \ell \cdot \vphi} \, . 
$$
Note that $\hat A(\ell - k) \hat u (k)$ is in $H^\s$ and that its norm can be estimated as
$ \|\hat A(\ell - k) \hat u (k) \|_\s 
\le \| \hat A(\ell - k)\|_\s  \| \hat u (k) \|_\s$
where $\| \hat A(\ell - k)\|_{\s}$ denotes the operator norm of $ \hat A(\ell - k)$ in ${\cal L}( H^\s )$. 
Hence the proof of Lemma~\ref{interpolation product} also shows that
the action of operators on functions satisfies tame estimates in the following sense:

\begin{lemma}[\bf Tame estimates for the action of operators on maps]
\label{lemma:action-Sobolev} 
Let $s \in \Z_{\ge s_0}$ and $\s \ge 0$. Then there are constants 
$ C_{act}(s) \geq C_{act}(s_0) \geq 1 $ 
(which also might depend on $\s$), so that the following holds:

\noindent
($i$) for any operator valued map $A$ in $H^s(\T^S, {\cal L}( H^\s ) )$
and any map $u \in H^s(\T^S, h^\sigma)$  one has
\be
\| A u \|_s
\leq C_{act}(s) |A|_{s_0} \| u \|_s
+ C_{act}(s_0) |A|_{s} \| u \|_{s_0} \, ;
\ee
($ii$) for any operator valued map $A \equiv A_\omega$ and any map $u \equiv u_\omega$,
which are both Lipschitz continuous in the parameter $\omega \in \Omega \subseteq \R^S$, 
\be
\| A u \|_s^{\gamma {\rm lip}}
\leq C_{act}(s) |A|_{s_0}^{\gamma {\rm lip}} \| u \|_s^{\gamma {\rm lip}}
+ C_{act}(s_0) |A|_{s}^{\gamma {\rm lip}} \| u \|_{s_0}^{\gamma {\rm lip}} \, .
\ee
\end{lemma}

 Lemma~\ref{prodest}  can be used to derive tame estimates for the exponential of
an operator valued map. We state them in the specific form needed in Section~\ref{sec:5}
where we consider operator valued maps in 
$H^s (\T^S, {\cal L}(h^\sigma_\bot \times h^\sigma_\bot))$ with 
$h^\sigma_\bot = h^\sigma ( S^\bot, \C)$. 
We introduce the vector valued Fourier multiplier 
\be\label{def:D-doubled}
{\mathfrak D} := 
{\rm diag}( \lla D \rra , \lla D \rra ) \,: 
h^\sigma_\bot \times h^\sigma_\bot\to h^\sigma_\bot \times h^\sigma_\bot 
\ee
where we recall that $\lla D \rra $ is defined in \eqref{def:DD}.
Let  ${\mathbb I}_2$ be the identity operator on $h^\sigma_\bot \times h^\sigma_\bot$.

\begin{lemma}\label{lem:inverti} {\bf (Tame estimates for the exponential of operators)}
Assume that $s \in \Z_{\ge s_0}$ $\s \ge \Z_{\ge 4}$
 and $C_{op}(s_0) \ge 1$ is the constant in Lemma \ref{prodest}-($iii$).
 Then for any  Lipschitz continuous  map  $A  \equiv A_\omega$, $\omega \in \Omega \subset \R^S$,  with values in 
$H^s (\T^S, {\cal L}(h^\sigma_\bot \times h^\sigma_\bot))$, the following holds:

\noindent
$(i)$ if $A$ satisfies the smallness condition 
$2 C_{op}(s_0)|A|_{s_0}^{\gamma{\rm lip}} \leq 1$,  
then $ \Phi : ={\rm exp} (A) $ and its inverse $\Phi^{- 1} = {\rm exp}(- A)$ satisfy
\be\label{PhINV}
| \Phi^{\pm 1} - {\mathbb I}_2 |_s \leq_s | A |_s  \quad \mbox{and} \quad
| \Phi^{ \pm 1} - {\mathbb I}_2 |_{s}^{\Lipg} \leq_s | A |_{s}^{
\Lipg} \, ; \ee

\noindent
$(ii)$ if $A$ satisfies $2C_{op}(s_0) |A {\frak D} |_{s_0}^{\gamma {\rm lip}} \leq 1$ and in addition
$A(\vphi ) \in {\cal L}(h^{\sigma - 1}_\bot \times h^{\sigma -1}_\bot , h^{\sigma}_\bot \times h^\sigma_\bot )$ for any $\vphi \in \T^S$, then 
\begin{equation}\label{PhINV with D}
| (\Phi^{\pm 1} - {\mathbb I}_2) {\frak D} |_{s} \leq_s | A {\frak D} |_{s} \quad \mbox{and} \quad
| (\Phi^{ \pm 1} - {\mathbb I}_2) {\frak D} |_{s}^{\Lipg} \leq_s | A {\frak D} |_{s}^{
\Lipg}\, ;
\end{equation}

\noindent
$(iii)$ if $A$ satisfies $2C_{op}(s_0) |A|_{s_0, \sigma} \leq 1$ and in addition
 for any $\sigma' \in \{\sigma, \sigma- 1, \sigma - 2, \sigma - 3 \}$, $A \in H^s(\T^S, {\cal L} (h^{\sigma'}_\bot \times h^{\sigma'}_\bot))$ with
$|A|_{s, \sigma'} \lessdot |A|_{s, \sigma}$ and
$|A|_{s_0, \sigma'} \lessdot |A|_{s_0, \sigma}$, then 
$$
\Big| \sum_{n \geq 2} \frac{1}{n !}{\frak D}^2 ({\frak D}^{- 1} A {\frak D}^{- 1})^n  {\frak D}\Big|_{s, \sigma - 1}\,,\quad \Big| \sum_{n \geq 2} \frac{1}{n !} ({\frak D}^{- 1} A {\frak D}^{- 1})^n  {\frak D}^3\Big|_{s, \sigma - 1} \leq_s |A|_{s, \sigma } |A|_{s_0, \sigma } \,;
$$

\noindent
$(iv)$ if $A$ satisfies $2C_{op}(s_0) |A|_{s_0, \sigma} \leq 1$ and in addition for any 
$\sigma' \in \{\s + 1, \s, \s- 1, \s - 2, \s - 3, \s -4 \}$, 
$A \in H^s(\T^S, {\cal L} (h^{\sigma'}_\bot \times h^{\sigma'}_\bot))$ with
$|A|_{s, \sigma'}\lessdot |A|_{s, \sigma + 1}$ and
$|A|_{s_0, \sigma'} \lessdot |A|_{s_0, \sigma}$, then 
\begin{align*}
& \Big| \sum_{n \geq 3} \frac{1}{n !}{\frak D}^2 ({\frak D}^{- 1} A )^n  {\frak D}\Big|_{s, \sigma - 1}\,,\,\Big| \sum_{n \geq 3} \frac{1}{n !}{\frak D}^2 ( A {\frak D}^{- 1})^n  {\frak D}\Big|_{s, \sigma - 1} \leq_s |A|_{s, \sigma + 1} |A|_{s_0, \sigma + 1}^2 \,, \\
& 
\Big| \sum_{n \geq 3} \frac{1}{n !} ({\frak D}^{- 1} A )^n  {\frak D}^3\Big|_{s, \sigma - 1}\,,\,\Big| \sum_{n \geq 3} \frac{1}{n !} ( A {\frak D}^{- 1})^n  {\frak D}^3\Big|_{s, \sigma - 1} \leq_s |A|_{s, \sigma + 1} |A|_{s_0, \sigma + 1}^2 \,;
\end{align*}
$(v)$ assume that  $ \Phi_i = {\rm exp}(A_i) $, $i= 1,2$, with 
$A_i \in H^s(\T^S, {\cal L}(h^{\sigma}_\bot \times h^\sigma_\bot ))$  such that 
\begin{equation}\label{smallness A1 A2}
2C_{op}(s_0)| A_i |_{s_0} \leq 1\,.
\end{equation}
Then the difference $\Phi_2^{-1} - \Phi_1^{-1}$ satisfies the estimate
\begin{equation}\label{derivata-inversa-Phi}
\vert \Phi_2^{-1} - \Phi_1^{-1} \vert_{s}
\leq_s
 \vert A_2 - A_1 \vert_{s}
+ \big( \vert A_1 \vert_s + \vert A_2 \vert_s \big)
\vert A_2 -  A_1 \vert_{s_0}  \, .
\end{equation}
Similarly, if
$A_i(\vphi) \in {\cal L}(h^{\sigma - 1}_\bot \times h^{\sigma - 1}_\bot, h^\sigma_\bot \times h^{\sigma }_\bot)$, 
$\vphi \in \T^S$,  and 
$ 2 C_{op}(s_0)| A_i {\frak D} |_{s_0} \leq 1$,
then 
\begin{equation}\label{derivata-inversa-Phi frak D}
\vert (\Phi_2^{-1} - \Phi_1^{-1}) {\frak D} \vert_{s}
\leq_s
 \vert (A_2 - A_1) {\frak D} \vert_{s}
+ \big( \vert A_1 {\frak D} \vert_s + \vert A_2 {\frak D} \vert_s \big)
\vert (A_2 -  A_1) {\frak D} \vert_{s_0}  \, .
\end{equation}
\end{lemma}

\begin{proof}
$(i)$ Let us prove the estimate \eqref{PhINV} for $| \ |_s$. The estimate with the norm $| \ |_s^{\gamma {\rm lip}}$ can be proven similarly. We have, with $C_{op}(s)$, $C_{op}(s_0)$
 given as in Lemma \ref{prodest}-($iii$),
\begin{align*}
| \Phi^{\pm 1} - {\mathbb I}_2 |_s & \leq \sum_{n \geq 1} \frac{|A^n|_s}{n !} 
\stackrel{\eqref{Mnab}}{\leq} 
C_{op}(s) |A|_s \sum_{n \geq 1} \frac{ \big(2C_{op}(s_0) |A|_{s_0} \big)^{n - 1}}{(n - 1)!} = C_{op}(s) |A|_s {\rm exp}(2C_{op}(s_0) |A|_{s_0}) \leq_s |A|_s \,.
\end{align*} 

\noindent
$(ii)$ Now let us prove the inequality \eqref{PhINV with D} for $|\ |_s$. The corresponding estimate with the norm $| \cdot |_s^\Lipg$ is shown in a similar way. 
For any $n \geq 2$, 
\begin{align}
|A^n {\frak D}|_s & \leq C_{op}(s) |A^{n - 1}|_{s_0} |A {\frak D}|_s 
+ C_{op}(s_0) |A^{n - 1}|_s |A {\frak D}|_{s_0} 
 \nonumber\\& \stackrel{\eqref{Mnab}}{\leq_s} 
C_{op}(s) C_{op}(s_0)
\big(n  (2C_{op}(s_0) |A|_{s_0} )^{n - 2} |A|_s |A {\frak D} |_{s_0} 
+  (2C_{op}(s_0) )^{n - 2} |A|^{n - 1}_{s_0} |A {\frak D} |_s   \big) \nonumber\\
& \leq_s (C_{op}(s))^2 n (|A|_s + |A {\frak D} |_s) 
\leq_s 2(C_{op}(s))^2 n |A {\frak D} |_s\,. \nonumber
\end{align}
Hence 
$$
| (\Phi^{\pm 1} - {\mathbb I}_2 ) {\frak D} |_s \leq_s  |A {\frak D} |_s
{\mathop \sum}_{n \geq 1} \frac{1}{ (n - 1)!} \leq_s |A {\frak D} |_s\,.
$$

\noindent
$(iii)$ For any $n \geq 2$, one has 
$$
{\frak D}^2 ({\frak D}^{- 1} A {\frak D}^{- 1})^n {\frak D} = {\frak D} A {\frak D}^{- 1} B^{n - 2} {\frak D}^{- 1} A\,, \qquad B := {\frak D}^{- 1} A {\frak D}^{- 1}\, .
$$
Let us estimate separately the norms of ${\frak D} A {\frak D}^{- 1} $, $B^{n - 2}$, 
and ${\frak D}^{- 1} A$. We have 
$$
|{\frak D} A {\frak D}^{- 1} |_{s, \sigma - 1} 
\leq \| {\frak D} \|_{{\cal L}(h^{\sigma}, h^{\sigma - 1})} |A |_{s, \sigma} \| {\frak D}^{-1} \|_{{\cal L}(h^{\sigma - 1}, h^{\sigma})} \lessdot  |A|_{s, \sigma}\, , \qquad
|{\frak D} A {\frak D}^{- 1} |_{s_0, \sigma - 1}   \lessdot |A|_{s_0, \sigma} .
$$
Since for $n \geq 3$
$$
|B^{n - 2}|_{s_0, \s } \stackrel{\eqref{Mnab}}{\leq} 
(2C_{op}(s_0))^{n - 3} |B|_{s_0, \s}^{n - 2}\, ,
\qquad
|B^{n - 2}|_{s, \s } \stackrel{\eqref{Mnab}}{\leq} 
n C_{op}(s) (2C_{op}(s_0))^{n - 3} |B|_{s_0, \s}^{n - 3} |B|_{s, \s } \, ,
$$
it then follows from
$$
|B|_{s_0, \sigma } = |{\frak D}^{- 1} A {\frak D}^{- 1}|_{s_0, \sigma} 
\leq |A|_{s_0, \sigma}\,, \qquad
|B|_{s, \sigma } = |{\frak D}^{- 1} A {\frak D}^{- 1}|_{s, \sigma} 
\leq |A|_{s, \sigma}\,,
$$
and  $2C_{op}(s_0)|A|_{s_0, \sigma } \leq 1$ that for $n \geq 3,$
$$
|B^{n - 2}|_{s_0, \sigma} \le 1\, ,
\qquad
|B^{n - 2}|_{s, \sigma} \leq n C_{op}(s) |A|_{s, \sigma }\, .
$$
Using that 
$$
|{\frak D}^{- 1} A |_{s, \sigma -1 } \leq |A|_{s, \sigma -1} 
\lessdot |A|_{s, \sigma }\, \quad \mbox{and}\quad
|{\frak D}^{- 1} A |_{s_0, \s -1 } \leq |A|_{s_0, \s -1 } \lessdot |A|_{s_0, \sigma}
$$ 
one then concludes from \eqref{interpm} that for any $n \ge 3,$
\begin{align}
|{\frak D} A {\frak D}^{- 1} B^{n - 2} {\frak D}^{- 1} A |_{s, \sigma- 1} & 
\leq_s n |A|_{s, \sigma} |A|_{s_0, \sigma}\, \nonumber
\end{align}
and in turn 
$$
\Big| \sum_{n \geq 2} \frac{1}{n !}{\frak D}^2 ({\frak D}^{- 1} A {\frak D}^{- 1})^n  {\frak D} \Big|_{s, \sigma - 1} \leq_s |A|_{s, \sigma} |A|_{s_0, \sigma} \sum_{n \geq 2} \frac{n}{n !} \leq_s |A|_{s, \sigma } |A|_{s_0, \sigma}\,.
$$
The estimate for $| \sum_{n \geq 2} \frac{1}{n !} ({\frak D}^{- 1} A {\frak D}^{- 1})^n  {\frak D}^3 |_{s, \sigma- 1}$ follows by similar arguments. 

\noindent
$(iv)$ The four series are estimated in the same way. Let us just comment how to prove the estimate for $\sum_{n \geq 3} \frac{1}{n !} ({\frak D}^{- 1} A )^n  {\frak D}^3$ which we write as the composition $B_1 B_2$ where
$$
 \quad B_1 :=  {\mathop \sum}_{n \geq 3} \frac{1}{n !} ({\frak D}^{- 1} A )^{n - 3}\,, \quad B_2 :=({\frak D}^{- 1} A)^3 {\frak D}^3\,.
$$
The norm $|B_2 |_{s, \sigma - 1}$ is treated separately using Remark \ref{generalizzazione prodotto operatori}, whereas the series $B_1$ is estimated in the same way as the ones of item $(iii)$. To obtain the claimed estimate we then apply Lemma \ref{prodest} to the composition $B_1 B_2$.   

\noindent
$(v)$ Since $ \Phi_i^{-1} = \exp (-A_i) $ the estimate \eqref{derivata-inversa-Phi}  for 
$ \Phi_2^{-1} - \Phi_1^{-1} $ is obtained from the one 
for $ \Phi_2 - \Phi_1 $ by replacing $ A_i $ by $ - A_i $. 
Observe that 
$$
 \Phi_2 - \Phi_1  = \sum_{n \geq 1} \frac{A_2^n - A_1^n}{n !} = 
\sum_{n \geq 1} \frac{1}{n !} \Big( \widehat A A_2^{n - 1} + A_1 \widehat A A_2^{n - 2} +
 \ldots + A_1^{n - 2} \widehat A A_2 + A_1^{n - 1} \widehat A \Big)\,,
 $$
where $\widehat A := A_2 -A_1$. The terms $A_1^k \widehat A A_2^{n - k - 1}$,
$1 \le k \le n-2,$ of the above sum can be estimated as follows 
\begin{align}
|A^k_1 \widehat A A^{n - k - 1}_2 |_s & \stackrel{\eqref{interpm}}{\leq} 
C_{op}(s) C_{op}(s_0) \big( |A_1^k|_s 2 |\widehat A|_{s_0} | A_2^{n - k - 1} |_{s_0} + |A_1^k|_{s_0} |\widehat A|_s | A_2^{n - k - 1} |_{s_0} +  |A_1^k|_{s_0} |\widehat A|_{s_0} | A_2^{n - k - 1} |_{s}\big) \nonumber\\
& \stackrel{\eqref{Mnab}, \eqref{smallness A1 A2}}{\leq} 
n C_{op}(s)^2 \big( (|A_1|_s  + |A_2|_s ) |\widehat A|_{s_0} + |\widehat A|_s \big)\,. \nonumber
\end{align}
The terms $|\widehat A A_2^{n - 1}|_s$ and $|A_1^{n - 1} \widehat A|_s$ can be estimated
in the same way and admit similar bounds. Hence
$$
|\Phi_2 - \Phi_1|_s \leq_s \big({\mathop \sum}_{n \geq 1} \frac{n^2}{n!} \big) 
\big( (|A_1|_s  + |A_2|_s ) |\widehat A|_{s_0} + |\widehat A|_s \big) 
$$
implying \eqref{derivata-inversa-Phi}. The proof of the estimate \eqref{derivata-inversa-Phi frak D} is similar.
\end{proof}

Finally we want to derive tame estimates for the  composed map $ f \circ \breve \io  $ where
$\breve \io$ denotes a map  $\breve \io : \T^S \to M^\s$
and $   f : M^\s \to Y$ takes values in the Banach space $ Y $.

Recall that $M^\s = \T^S \times U_0 \times h^\s_\bot 
$ denotes the phase space introduced in \eqref{def:M-sigma}.
We assume  that $ \breve \io $ has a lift of the form 
$( \vphi, 0, 0) + \io (\vphi) $
where $ \io : \R^S \to \R^S \times U_0 \times h^\s_\bot $ is $ (2\pi\Z)^S$-periodic. 
Whenever the context permits, we will identify  $\breve \io$ with its lift and denote
both by the same letter. Similarly, we will identify maps $\T^S \to Y$ with their lifts
$\R^S \to Y$, which are $(2\pi\Z)^S$-periodic.    

\begin{lemma} \label{Lemma 2.4bis} 
{\bf (Tame estimates for the composition of maps in ${\cal C}^s$-spaces)}
Assume that $ f $ is a map in ${\cal C}^s (\T^S \times V , Y )$ 
where $ V $ is an open neighborhood in $\R^S \times h^\s_\bot$ and $s \in \Z_{\ge 0} $. 
Then for any map $ \breve \io  (\vphi) = ( \vphi, 0, 0) + \io (\vphi)$ with  
$\io \in {\cal C}^s (\T^S, \R^S \times \R^S \times h^\s_\bot)$
and $\breve \io (\T^S) \subset \T^S \times V$, the following holds:

\noindent
$(i)$ The composition $ f \circ \breve \io \in {\cal C}^s (\T^S, Y ) $ satisfies the tame estimate
\be\label{ftamei}
  \| f \circ \breve \io \| _{{\cal C}^s}
      \leq C(s, \|f\|_{{\cal C}^s}, \| \io \|_{{\cal C}^0}) \cdot \big( 1+ \| \io \| _{{\cal C}^s} \big) \, . 
 \ee
 \noindent
 $(ii)$ If  $ f \in {\cal C}^{s + 1} (\T^S \times V , Y )$, then for any $\widehat \io$ in
${\cal C}^s(\T^S, \R^S \times \R^S \times h^\s_\bot)$,
 \begin{equation}\label{differential composition operator}
 \|d f(\breve \io)[\widehat \io] \|_{{\cal C}^s} \leq C(s, \|f\|_{{\cal C}^{s + 1}}, \| \io \|_{{\cal C}^0}) 
\cdot \big( \|\widehat \io\|_{{\cal C}^s} + \|\io \|_{{\cal C}^s}  \| \widehat \io \|_{{\cal C}^0} \big)\,.
 \end{equation} 
 \noindent
$(iii)$ If $ f \in {\cal C}^{s + 1} (\T^S \times V , Y )$ and $V$ is in addition convex, then for any two maps, $ \breve \io^{(a)} (\vphi) = ( \vphi, 0, 0) + \io^{(a)} (\vphi) $ with 
$\io^{(a)} \in  {\cal C}^s (\T^S, \R^S \times \R^S \times h^\s_\bot)$
and $\breve \io^{(a)} (\T^S) \subset \T^S \times V$ , $ a =1,2$, the difference
$\Delta_{12} f = f\circ \breve \io^{(1)} - f\circ \breve \io^{(2)}$ 
satisfies the estimate 
$$
 \| \Delta_{12} f  \|_{{\cal C}^s} \leq  C(s, \|f\|_{{\cal C}^{s+1}}, 
\| \io^{(1)} \|_{{\cal C}^0}, \| \io^{(2)} \|_{{\cal C}^0}) 
\cdot \big( \|\Delta_{12} \io\|_{{\cal C}^s} 
+ (\| \io^{(1)} \|_{{\cal C}^s} + \| \io^{(2)}\|_{{\cal C}^s}) \| \Delta_{12} \io \|_{{\cal C}^0}  \big)
$$
where $\Delta_{12}\io := \io^{(1)} - \io^{(2)}$.

\noindent
$(iv)$ If $ f \in {\cal C}^{s + 1} (\T^S \times V , Y )$ and in addition $V$ is  convex
and $ \io \equiv  \io_\omega$  Lipschitz continuous in the parameter $\omega \in \Omega \subset \R^S,$
 the composition $f \circ \breve \io \in {\cal C}^s (\T^S, Y ) $ is also Lipschitz continuous 
in $\omega$ and satisfies the estimate
\begin{equation}\label{stima-lip-4}
\| f \circ \breve \io \| _{{\cal C}^s}^{ {\rm lip}}
      \leq C(s, \|f\|_{{\cal C}^{s+1}}, \| \io \|^{{\rm sup}}_{{\cal C}^0}) \cdot
\big(\| \io \|^{ {\rm lip}} _{{\cal C}^s} + \| \io \|^{ {\rm sup}} _{{\cal C}^s}
 \| \io \|^{ {\rm lip}} _{{\cal C}^0} \big) \, . 
\end{equation}
\end{lemma}
\begin{proof}
$(i)$
For any multi-index $ \a \in \Z_{\ge 0}^S $ with $ 1 \leq |\a | \leq s $, one computes
$$
\partial^\a_\vphi ( f \circ  \breve \io ) (\vphi ) = \sum_{1 \leq m \leq |\a| \atop \a = \a_1 + \cdots + \a_m}  
c_{\a_1, \cdots , \a_m} \, (d^m f)(  \breve \io (\vphi)) [ \pa^{\a_1}_\vphi  \breve \io (\vphi) , \cdots,   \pa^{\a_m}_\vphi  \breve \io (\vphi) ] 
$$
where $ c_{\a_1, \cdots , \a_m} $ are combinatorial constants and $\alpha_1, \cdots , \alpha_m$ are nonzero integer vectors in  $\Z_{\ge 0}^S $.
Hence
\begin{align}
\| \partial^\a_\vphi ( f \circ \breve \io )  \|_{{\cal C}^0} 
& \leq  
C(s, \|f\|_{{\cal C}^s}) \sum_{1 \leq m \leq |\a| \atop \a = \a_1 + \cdots + \a_m}  
\| \pa^{\a_1}_\vphi  \breve \io \|_{{\cal C}^0}  \cdots \| \pa^{\a_m}_\vphi \breve \io \|_{{\cal C}^0} \nonumber\\
& \leq  
C(s, \|f\|_{{\cal C}^s}) \sum_{1 \leq m \leq |\a| \atop \a = \a_1 + \cdots + \a_m}  
( 1 + \|   \io \|_{{\cal C}^{|\a_1|}})  \cdots ( 1 + \|   \io\|_{{\cal C}^{|\a_m|}})  \, . \label{argomento conv label}
\end{align}
We claim that for any $0 \leq k \leq |\alpha|$, there exists a constant $C_{|\alpha|, k} > 0$ such that 
\begin{equation}\label{claim convessita}
1 + \|  \io \|_{{\cal C}^{k}} \leq C_{|\alpha|, k} (1 + \| \io \|_{{\cal C}^0})^{1 - \frac{k}{|\alpha|}} (1 + \| \io \|_{{\cal C}^{|\alpha|}})^{\frac{k}{|\alpha|}}\,.
\end{equation}
Indeed, by the interpolation estimates for ${\cal C}^s$-spaces (Proposition \ref{interpolation scales}, Lemma \ref{smoothing family Cs}) one has $\|  \io \|_{{\cal C}^{k}}  \lessdot \| \io\|_{{\cal C}^0}^{1 - \frac{k}{|\alpha|}} \| \io \|_{{\cal C}^{|\alpha|}}^{\frac{k}{|\alpha|}}$ yielding 
\begin{align}
1 + \|  \io \|_{{\cal C}^{k}}  \leq C'_{|\alpha|, k}(1 +  \| \io\|_{{\cal C}^0}^{1 - \frac{k}{|\alpha|}} )(1 + \| \io \|_{{\cal C}^{|\alpha|}}^{\frac{k}{|\alpha|}} )\,. \label{4444}
\end{align}
Since for any $0 \leq \lambda \leq 1$, $f_\lambda : \R^+ \to \R\,,\, t \mapsto t^\lambda$ is concave, one has 
$$
\frac12(1 + t^{\lambda} )= \frac12 f_\lambda(1) + \frac12 f_\lambda(t) \leq f_\lambda\Big( \frac{1 + t}{2} \Big) = 2^{- \lambda} (1 + t)^{\lambda}
$$
implying that $(1 + t^{\lambda} ) \leq 2^{1 - \lambda} (1 + t)^{\lambda}$ for any $t \geq 0$. Thus we conclude that 
$$
1 +  \| \io\|_{{\cal C}^0}^{1 - \frac{k}{|\alpha|}} \leq 2^{\frac{k}{|\alpha|}}(1 +  \| \io\|_{{\cal C}^0})^{1 - \frac{k}{|\alpha|}}\,, \qquad \qquad 1 + \| \io \|_{{\cal C}^{|\alpha|}}^{\frac{k}{|\alpha|}} \leq 2^{1 - \frac{k}{|\alpha|}}(1 + \| \io \|_{{\cal C}^{|\alpha|}})^{\frac{k}{|\alpha|}}\,.
$$
Combining this with \eqref{4444} yields \eqref{claim convessita}.   Applying the estimate \eqref{claim convessita} to the products in \eqref{argomento conv label}, one gets 
$$
(1 + \|    \io \|_{{\cal C}^{|\a_1|}})  \cdots 
( 1 + \|   \io \|_{{\cal C}^{|\a_m|}})   \leq C_s \prod_{j=1}^m 
(1 + \|   \io \|_{{\cal C}^{0}})^{1 - \frac{|\a_j|}{|\a |}}  (1 + \|    \io \|_{{\cal C}^{|\a|}})^{\frac{|\a_j|}{|\a |}}  \leq C_s
( 1 + \|   \io \|_{{\cal C}^{0}})^{m-1}  (1 +  \|   \io \|_{{\cal C}^{|\a|}} )
$$
which proves the estimate \eqref{ftamei}.

\noindent
$(ii)$ By the Leibnitz rule, for any multi-index $\beta \in \Z_{\ge 0}^S $ with
 $0 \leq |\beta | \leq s$,
and any $\widehat \io \in {\cal C}^s(\T^S, \R^S \times \R^S \times h^\sigma_\bot)$, one has 
\begin{align*}
\partial_\vphi^\beta d f(\breve \io(\vphi))[\widehat \io(\vphi)] & = \sum_{\beta_1 + \beta_2 = \beta} c_{\beta_1, \beta_2} \partial_\vphi^{\beta_1}(d f (\breve \io(\vphi)))[\partial_\vphi^{\beta_2} \widehat \io(\vphi)]  
\end{align*}
where $c_{\beta_1, \beta_2}$ are combinatorial constants.
Each term in the latter sum is estimated individually.
For the term with $\beta_1 = 0$, $\beta_2 = \beta$ one gets
$$
\| d f(\breve \io)[\partial_\vphi^\beta \widehat \io] \|_{{\cal C}^0} \lessdot  \| f \|_{{\cal C}^1} \| \widehat \io\|_{{\cal C}^{|\beta|}} \lessdot  \| f \|_{{\cal C}^1} \| \widehat \io\|_{{\cal C}^{s}}
$$
whereas in the case $1 \le |\beta_1| \le s$, one has
$$
\partial_\vphi^{\beta_1} (d f(\breve \io (\vphi) ))
[\partial_\vphi^{\beta_2} \widehat \io(\vphi)] = \sum_{\begin{subarray}{c}
1 \leq m \leq |\beta_1| \\
\alpha_1 + \cdots + \alpha_m = \beta_1
\end{subarray}} c_{\a_1, \cdots , \a_m}  d^{m + 1} f(\breve \io(\vphi))[\partial_\vphi^{\alpha_1} \breve \io(\vphi), \cdots, \partial_\vphi^{\alpha_m} \breve \io(\vphi), \partial_\vphi^{\beta_2} \widehat \io(\vphi)]
$$
yielding
\begin{align*}
\| \partial_\vphi^{\beta_1} (d f(\breve \io))[\partial_\vphi^{\beta_2} \widehat \io]  \|_{{\cal C}^0} & \leq C(s, \| f \|_{{\cal C}^{s + 1}}) \sum_{\begin{subarray}{c}
1 \leq m \leq |\beta_1| \\
\alpha_1 + \cdots + \alpha_m = \beta_1
\end{subarray}} (1 + \| \io\|_{{\cal C}^{|\alpha_1|}}) \cdots (1 +  \|  \io\|_{{\cal C}^{|\alpha_m|}}) \| \widehat \io\|_{{\cal C}^{|\beta_2|}}\,. \nonumber
\end{align*}
Since $|\alpha_1| + \cdots + |\alpha_m| + |\beta_2| = |\beta_1| + |\beta_2| = |\beta|$, the interpolation estimates for ${\cal C}^s$-spaces
(Proposition \ref{interpolation scales}, Lemma \ref{smoothing family Cs}) and the estimate \eqref{claim convessita}, then lead to
\begin{align*}
(1 + \|   \io \|_{{\cal C}^{|\a_1|}})  \cdots ( 1 + \|  \io \|_{{\cal C}^{|\alpha_m|}})  \| \widehat \io\|_{{\cal C}^{|\beta_2|}}  & 
\leq C_s \|\widehat \io \|^{1 - \frac{|\beta_2|}{|\beta|}}_{{\cal C}^0} \| \widehat \io\|^{\frac{|\beta_2|}{|\beta|}}_{{\cal C}^{|\beta|}}
\prod_{j = 1}^m
( 1 + \|   \io \|_{{\cal C}^{0}})^{1 - \frac{|\a_j|}{|\beta |}}  ( 1 + \|   \io \|_{{\cal C}^{|\beta|}})^{\frac{|\a_j|}{|\beta |}}\,. 
\end{align*}
Using that $\frac{\sum_{j =1}^m |\alpha_j|}{|\beta|} = \frac{|\beta_1|}{|\beta|} = 1 - \frac{|\beta_2|}{|\beta|} $ it then follows that 
\begin{align*}
(1 + \|    \io \|_{{\cal C}^{|\a_1|}})  \cdots (1 + \| \io \|_{{\cal C}^{|\alpha_m|}})  \| \widehat \io\|_{{\cal C}^{|\beta_2|}}  & \leq  C(s, \| \io \|_{{\cal C}^0})  \cdot
\|\widehat \io \|_{{\cal C}^0}^{\frac{|\beta_1|}{|\beta|}} (1 + \|  \io \|_{{\cal C}^{|\beta|}} )^{\frac{|\beta_1|}{|\beta|}} \cdot
\| \widehat \io\|^{\frac{|\beta_2|}{|\beta|}}_{{\cal C}^{|\beta|}}  
\end{align*}
and by Young's inequality with exponents $|\beta|/ |\beta_1|$, $|\beta|/ |\beta_2|$ we conclude that 
\begin{align*} 
(1 + \|    \io \|_{{\cal C}^{|\a_1|}})  \cdots (1 + \| \io \|_{{\cal C}^{|\alpha_m|}})  \| \widehat \io\|_{{\cal C}^{|\beta_2|}}
& \leq C(s, \| \io \|_{{\cal C}^0}) \big( \| \widehat \io\|_{{\cal C}^{|\beta|}} + \| \io\|_{{\cal C}^{|\beta|}} \| \widehat \io\|_{{\cal C}^0} \big)\,.
\end{align*}
Combining the estimates  obtained so far, the estimate
 \eqref{differential composition operator} follows. 

\noindent
$(iii)$ Since by assumption, $V$ is convex, the claimed estimates for $\Delta_{12} f$ can be derived from the estimates of item $(ii)$ by the mean value theorem. 

\noindent
$(iv)$ The estimate \eqref{stima-lip-4} 
directly follows from the estimates of item $(iii)$.
\end{proof}
When combined with the inequalities \eqref{trivialembe},
Lemma~\ref{Lemma 2.4bis} leads to tame estimates in the case where
$\breve \io$ are maps in Sobolev spaces. We state them in the form needed in the sequel.

\begin{lemma}\label{composition in Sobolev}
{\bf (Tame estimates for the composition of maps in $H^s$-spaces)}
Assume that  $ f $ is in ${\cal C}^{s + s_{0}} (\T^S \times V , Y )$, 
where $ V $ is an open subset contained in $\R^S \times h^\s_\bot$
and $s \in \Z_{\ge 0} $. 
Then the following holds:

\smallskip

\noindent
$(i)$ There exists a constant $C(s) > 0$ (depending on $\|f\|_{{\cal C}^{s+ s_{0} }}$) so that for any  map $ \breve \io  (\vphi) = ( \vphi, 0, 0) + \io (\vphi)$ 
 with  $\io \in H^{s + 2s_{0}} (\T^S, \R^S \times \R^S \times h^\s_\bot) $, 
$\|\io\|_{s_{0}} \le 1$, and $ \breve \io  (\T^S) \subset \T^S \times V$,
the composition $ f \circ \breve \io$ is in $H^s (\T^S, Y ) $ and satisfies the tame estimate
\be\label{composition formula}
\| f \circ \breve \io \| _{s, Y}  \leq C(s) (1+ \| \io \| _{s +  2 s_0 } ) \, . 
\ee
$(ii)$ Assume in addition that $ f \in {\cal C}^{s+ s_{0} + 1} (\T^S \times V , Y )$ and $V$ is convex.
 Then there exists a constant $C(s) > 0$ (depending on $\|f\|_{{\cal C}^{s+ s_{0} +1}}$)
so that for any two maps,
$ \breve \io^{(a)} (\vphi) = ( \vphi, 0, 0) + \io^{(j)} (\vphi) $ with 
$\io^{(a)} \in  H^{s+2s_{0}} (\T^S, \R^S \times \R^S \times h^\s_\bot )$, 
$\|\io^{(a)}\|_{s_{0}} \le 1$, and $\breve \io^{(a)}(\T^S) \subset \T^S \times V$, $ a = 1 , 2$,
  the difference
$\Delta_{12} f = f\circ \breve \io^{(1)} - f\circ \breve \io^{(2)}$ 
satisfies the tame estimate 
$$
 \| \Delta_{12} f  \|_{s,Y} \leq  C(s) 
\cdot \big( \|\Delta_{12} \io\|_{s + 2s_{0}} + (\| \io^{(1)} \|_{s + 2s_{0}} + \| \io^{(2)}\|_{s + 2s_{0}}) \| \Delta_{12} \io \|_{s_{0}}  \big)
$$
where $\Delta_{12}\, \io := \io^{(1)} - \io^{(2)}$.

\noindent
$(iii)$ Assume in addition that $ f \in {\cal C}^{s+ s_{0} + 1} (\T^S \times V , Y )$ 
and $V$ is convex. Then
there exists a constant $C(s) > 0$ (depending on $\|f\|_{{\cal C}^{s+ s_{0} +1}}$) so that 
 for any map $ \breve \io  (\vphi) = ( \vphi, 0, 0) + \io (\vphi)$ with
$ \breve \io  (\T^S) \subset \T^S \times V$ and $\io \equiv  \io_\omega \in 
H^{s+2s_{0}} (\T^S, \R^S \times \R^S \times h^\s_\bot)$ having the property that it is
Lipschitz continuous in the parameter $\omega \in \Omega \subset \R^S$
and satisfies $\|\io\|^{\sup}_{s_{0}} \le 1$,
 the composition $f \circ \breve \io$ is in $H^s (\T^S, Y ) $, is Lipschitz continuous 
in $\omega$, and admits the tame estimate
$$
\| f \circ \breve \io \| _s^{ {\rm lip}}
      \leq C(s) \cdot
\big(\| \io \|^{ {\rm lip}} _{s + 2s_{0}} + \| \io \|^{ {\rm sup}} _{s + 2s_{0}}
 \| \io \|^{ {\rm lip}} _{s_{0}} \big) \, . 
$$
\end{lemma}

\section{Setup and preliminary estimates}
\label{Estimates on H_e}

In this section we review properties of the Birkhoff coordinates, constructed in \cite{GK},
discuss asymptotic estimates of the dNLS frequencies, and describe the Hamiltonian setup
for the perturbation of the dNLS equation.
Furthermore we provide (tame) estimates of the composition and its derivatives
of torus embeddings with the dNLS Hamiltonian $H^{nls}$ and with the perturbation $ P$, 
needed in the sequel. 

\subsection{Normal form of the dNLS equation}\label{section birkhoff coordinates} 

Introduce the  $\R$-subspaces $H^\sigma_r$ of $H^\sigma \times H^\sigma$
and $h^\sigma_r$ of $h^\sigma \times h^\sigma$, defined by
 $$
 H^\sigma_r := \big\{ (u, \bar u) : u \in H^\sigma \big\}\, , \qquad
 h^\sigma_r := \Big\{ \big( (w_k)_{k \in \Z}, (\bar w_k)_{k \in \Z} \big) : (w_k)_{ k \in \Z} \in h^\sigma \Big\}
 $$
 with $ H^\sigma $ and $ h^\sigma $ defined  in \eqref{def:H-sigma} and \eqref{space:h-sigma}. 
 Denote by $ F_{nls} $  the following version of the  Fourier transform in the space variable
introduced in \cite{GK} 
 \begin{equation}\label{Fourier transform F nf}
 F_{nls} : H^0 \times H^0 \to h^0 \times h^0 \,, 
 \quad ( u^{(1)}, u^{(2)}) \to \Big( (- u^{(1)}_{- k})_{k \in \Z}\,,\, (- u^{(2)}_k)_{k \in \Z} \Big) 
 \end{equation}
 where the Fourier coefficients $ u^{(1)}_k $, $ u^{(2)}_k $ are defined as in \eqref{f:Fourier}. 
 Note that for $(u^{(1)}, u^{(2)}) \in H^0_r$, one has $  u^{(2)} = \overline u^{(1)}$, 
 implying that for any $ k \in \Z$, $ u_k^{(2)} = \overline  u_{- k}^{(1)}$. 
 Hence $F_{nls}$ maps $H_r^0$ into $h_r^0$. In fact, for any $\sigma \geq 0$, $F_{nls} : H^\sigma_r \to h^\sigma_r$ is a linear isomorphism. 
  The definition of $ F_{nls}$ in   \eqref{Fourier transform F nf}  is related to the specific choices made
  in the construction  of the Birkhoff coordinates in \cite{GK} -- see Theorem  \ref{Theorem Birkhoff coordinates} below. 

 In addition we introduce the bilinear bounded map 
 $$
 I : h^\sigma \times h^\sigma \to \ell^{1, 2 \sigma}\,, \quad \big( (z_k)_{k \in \Z}, (w_k)_{k \in \Z} \big) \to (z_k w_k)_{k \in \Z}\,,
 $$
 where $\ell^{1, 2 \sigma} \equiv \ell^{1, 2\sigma}(\Z, \C)$ denotes the weighted $\ell^1$ sequence space
 \be\label{norm-l12}
 \ell^{1, 2 \sigma} := \Big\{ (y_k)_{k \in \Z} \subseteq \C \, : \  
 \sum_{k \in \Z}\langle k \rangle^{2 \sigma}|y_k | < + \infty \Big\}\,.
 \ee
Clearly, for $ \s' \leq \s $ we have the continuous embedding $  \ell^{1, 2 \sigma} \hookrightarrow  \ell^{1, 2 \sigma' } $.
Note that for $(w_k)_{k \in \Z}$ in $h^\sigma_r$, 
$(I_k)_{k \in \Z} = (w_k \bar w_k)_{k \in \Z}$ is in the positive quadrant
 $$
 \ell^{1, 2 \sigma}_+ = \big\{ (y_k)_{k \in \Z} \in \ell^{1, 2 \sigma} : y_k \geq 0\,,\, \, \forall k \in \Z  \big\}\,.
 $$
The following theorem summarizes the pertinent
properties of the Birkhoff coordinates for the dNLS equation, used in the sequel. 

\begin{theorem}[\cite{GK}, \cite{KST}] {\bf (Birkhoff coordinates)}\label{Theorem Birkhoff coordinates}
$(i)$ There exists a neighbhourhood ${\cal W}$ in $H^0 \times H^0$ and an analytic map $\Phi^{nls} : {\cal W} \to h^0 \times h^0$ with the following properties: 
\begin{description}
\item[{\rm (BC1)}] For any $\sigma \in \Z_{\geq 0}$, $\Phi^{nls}(H_r^\sigma) \subseteq h^\sigma_r$ and $\Phi^{nls} : H^\sigma_r \to h^\sigma_r$
is a real analytic diffeomorphism. 
\item[{\rm (BC2)}] The map $\Phi ^{nls}$ is canonical on $H^0_r$ with respect to the Poisson bracket \eqref{Poisson brackets}, i.e., 
$\{ w_k, \bar w_k \} = - \ii$ for any $ k \in \Z$, whereas all other Poisson brackets between coordinate functions vanish.  
\item[{\rm (BC3)}] The Hamiltonian ${\cal H}^{nls}$ of dNLS, when expressed in Birkhoff coordinates on $h^1_r$, is a function of the actions $I = (I_k)_{k \in \Z} \in \ell_+^{1, 2}$ only and $H^{nls} = {\cal H}^{nls} \circ (\Phi^{nls})^{- 1} : \ell^{1, 2}_+ \to \R$ is real analytic.  
\item[{\rm (BC4)}] The differential $d_0\Phi^{nls}$ of $\Phi^{nls}$ at $0$ is the
Fourier transform $F_{nls}$.
\end{description}
\noindent
$(ii)$ The nonlinear parts $A^{nls} := \Phi^{nls} - F_{nls}$ of $\Phi^{nls}$ and $B^{nls} := (\Phi^{nls})^{- 1} - F_{nls}^{- 1}$ of $(\Phi^{nls})^{- 1}$ are one smoothing in the sense that for any $\sigma \in \Z_{\geq 1}$
$$
A^{nls} : H^\sigma_r \to h^{\sigma + 1}_r \quad \text{and} \quad B^{nls} : h^\sigma_r \to H^{\sigma + 1}_r
$$
are real analytic and bounded, meaning  that  the image of any bounded subset is bounded. 

The map $\Phi^{nls}$ is referred to as Birkhoff map and the coordinates $(w_k)_{k \in \Z}$
are called (complex) Birkhoff coordinates for the dNLS equation.
\end{theorem}

\begin{pf}
Item $(i)$ of Theorem \ref{Theorem Birkhoff coordinates} is the reformulation of 
the corresponding theorem of \cite{GK} for the dNLS equation in complex coordinates 
\begin{equation}\label{real Birkhoff coordinates}
 w_k = (x_k - \ii y_k)/\sqrt{2} 
\,,  \qquad \forall k \in \Z\, ,
\end{equation}
where $x_k $, $y_k $ are the real coordinates of Theorem in \cite{GK}, page 5. 
For item $(ii)$, we refer to \cite{KST}.
\end{pf}

According to Theorem \ref{Theorem Birkhoff coordinates} $(i)$, the Hamiltonian equations of motion, when expressed in Birkhoff coordinates on $h^1_r$, take the form 
$$
\dot w_k = \{ w_k, H^{nls} \} = - \ii \partial_{\bar w_k} H^{nls} = - \ii \partial_{I_k} H^{nls} \cdot \partial_{\bar w_k} I_k \,.
$$
Since $I_k = w_k \bar w_k $, one then gets 
$$
\dot w_k = - \ii \omega_k^{nls} w_k \,, \qquad \omega_k^{nls} = \partial_{I_k} H^{nls}\,, \quad \forall k \in \Z\,.
$$
Note that by Theorem \ref{Theorem Birkhoff coordinates} $(i)$, $H^{nls} : \ell^{1, 2}_+ \to \R$ is real analytic and hence so are the frequencies  $\omega_k^{nls} = \partial_{I_k} H^{nls}$, 
$ k \in \Z$. In \cite{GK1}, asymptotic estimates for $\omega_k^{nls}$ as $|k| \to \infty$ were obtained  
$$
\omega_k^{nls} = 4 \pi^2 k^2 + O(1)\,.
$$
Actually, they can be refined on the space of actions $\ell^{1, 4}_+$, corresponding to potentials in $H^2_r$ (\cite{KST1}), 
$$
    \omega^{nls}_k = 4 \pi ^2 k^2 + 4 \sum _{j \in {\mathbb Z}}
      I_j +O(1/k)\,. 
$$
To state these results more precisely, let
  $\ell^\infty \equiv \ell^{\infty}(\Z, \C)$ denote the Banach space of complex valued, bounded sequences, endowed with the sup-norm $\| \cdot \|_{\ell^\infty} $. 
\begin{theorem}
{\bf (dNLS frequencies)}
\label{Corollary 2.2} 
There exists an open complex neighbhourhood $V$ of $\ell^{1, 2}_+$ in $\ell^{1, 2}$ so that the following holds: 

\noindent
$(i)$  The map
   \begin{equation}\label{asintotics frequencies A}
   V \rightarrow \ell ^\infty , \ (I_k)
      _{k \in {\mathbb Z}} \mapsto (\omega ^{ nls}_n(I) - 4 \pi ^2 n^2)
      _{n \in {\mathbb Z}}
   \end{equation}
is real analytic and bounded. Furthermore
 for any $I^{(0)} \in \ell^{1, 2}_+$
  there exist a complex neighbhourhood $V(I^{(0)}) \subseteq V$ and a constant $C > 0$ so that on $V(I^{(0)})$
\be\label{local-boundeness}
\sup_{n \in \Z}\Big\| \Big(\frac{1}{\langle k \rangle^2} \partial_{I_k} \omega^{nls}_n \Big)_{k \in \Z}  \Big\|_{\ell^\infty} \leq C \,.
\ee
As a consequence, for any $n \in \Z$, the map 
\be\label{mappa-analitica-Birk}
\ell_+^{1, 2} \to \ell^\infty\,, \quad I \mapsto \Big( \frac{1}{\langle k \rangle^2} \partial_{I_k} \omega^{nls}_n \Big)_{k \in \Z}
\ee
is real analytic and locally bounded uniformly in $n$.
More generally, for any $N \in \Z_{\ge1}$ and $I^{(0)} \in \ell^{1, 2}_+$, there exist 
a complex neighbhourhood $V_N(I^{(0)})  \subseteq V(I^{(0)})$ and a constant $C_N > 0$ so that on $V_N(I^{(0)})$
 \begin{equation}\label{tame composizione derivate ennesime frequenza-0}
\sup_{|\alpha|=N}\sup_{n \in \Z } \big| \big( \prod_{k \in \Z} \langle k \rangle^{-2\alpha_k}\big)
\partial_{I}^\alpha \omega_n^{nls}(I) \big| 
\leq C_N
\end{equation}
where the supremum is taken over all multi-indices $\alpha = (\alpha_k)_{k \in \Z}$ with $\alpha_k \in \Z_{\ge 0}$ and $|\alpha| := \sum_{k \in \Z} \alpha_k = N$.
\smallskip

\noindent
$(ii)$ The map 
\be\label{asymptotic expansion NLS}
 V \cap \ell^{1, 4} \to \ell^\infty\,, \quad I = (I_k)_{k \in \Z} \mapsto (r_n)_{n \in \Z}\,, 
\qquad  r_n := n \Big(\omega^{nls}_n - 4 \pi ^2 n^2 - 4 \sum _{k \in {\mathbb Z}}  I_k \Big) 
  \ee
 is real analytic and bounded.
\end{theorem}

\begin{proof} 
$(i)$ The analyticity 
and boundedness of the map $(I_k)_{k \in \Z} \mapsto (\omega_n^{nls} - 4 \pi^2 n^2)_{n \in \Z}$ 
(cf  \eqref{asintotics frequencies A}) 
is proved in \cite{KST1}, Corollary 2.1. Let $I^{(0)} \in \ell^{1, 2}_+$. Then there exist a closed complex ball $B_r(I^{(0)}) \subseteq \ell^{1, 2}$ of radius $r > 0$, centered at $I^{(0)}$, and $C > 0$ so that for any $n \in \Z$, the real analytic map 
$
\omega^{nls}_n - 4 \pi^2 n^2 : B_r(I^{(0)}) \to \C
$
satisfies 
$$
\sup_{I \in B_r(I^{(0)}) }|\omega^{nls}_n (I) - 4 \pi^2 n^2| \leq C/2\,.
$$
 By Cauchy's estimate, the differential $d \omega^{nls}_n : \ell^{1, 2} \to \C$ satisfies the estimate 
$$
\sup_{I \in B_{r/2}(I^{(0)})} \| d \omega_n^{nls} \|_{(\ell^{1, 2})^*} \leq C/r
$$
where $(\ell^{1, 2})^*$ is the dual of $\ell^{1, 2}$ and given by $\ell^{\infty, -2}$. Hence 
$\big( \frac{1}{\langle k \rangle^2} \partial_{I_k} \omega_n^{nls}(I) \big)_{k \in \Z} \in \ell^\infty$ and 
$$
\sup_{I \in B_{r/2}(I^{(0)})} 
\Big\| \Big(\frac{1}{\langle k \rangle^2} \partial_{I_k} \omega_n^{nls}(I)\Big)_{k \in \Z} \Big\|_{\ell^{\infty}} 
\leq C/r\,, \quad \forall n \in \Z\,,
$$
proving \eqref{local-boundeness} with $V(I^{(0)}) := B_{r/2}(I^{(0)})$. The analyticity of 
the map \eqref{mappa-analitica-Birk} 
then follows from the characterization of analytic maps with
values in $\ell ^\infty$, see e.g.
\cite[Theorem A.3]{KP}.  The estimates \eqref{tame composizione derivate ennesime frequenza-0} 
of the higher derivatives of the dNLS frequencies $\omega_n^{nls}$
are proved in a similar way. Since we need to apply again Cauchy's estimate we might have to choose
the neighborhood $V_N(I^{(0)})$ smaller than $V(I^{(0)})$.

\noindent
$(ii)$ The claimed statement is proved in \cite{KST1}, Theorem 2.3.
\end{proof}

Finally we recall from \cite{GK1} that the dNLS frequencies satisfy Kolmogorov and Melnikov conditions. In \cite{GK1} (cf also \cite{KP2}), the Birkhoff normal form of the Hamiltonian ${\cal H}^{nls}$ of \eqref{Poisson brackets} has been computed near $u = 0$ up to order four, yielding
$$
\omega_n^{nls}(I) = 4 \pi^2 n^2 + 4 \sum_{k \in \Z} I_k - 2 I_n + O(I^2)\,.
$$
In particular, it follows that for any $S \subseteq \Z$ with $|S| <  \infty$, 
$$
{\rm det}\big( (\partial_{I_k} \omega_n^{nls})_{k, n \in S}\big)|_{I = 0} = - (- 2)^{|S|} (2 |S| - 1) \neq 0\,.
$$
Hence by the analyticity of $\omega_n^{nls}$ we have the following result. 

\begin{proposition}[\cite{GK1}] \label{Proposition 2.3}
{\bf (Non-degeneracy of dNLS frequencies)} For any $S \subset {\mathbb Z}$ with $|S| < \infty ,
\ \Pi _S \rightarrow {\mathbb R}, \ I \mapsto {\rm det}\big( (\partial _{I_k}
\omega _n^{nls})_{k,n \in S} \big) $ is a real analytic map
satisfying
\be\label{Kolmogorov-c}
{\rm det} \big( (\partial _{I_k} \omega _n^{nls})
_{k,n \in S} \big) \not= 0 \qquad \text{a.e. on} \qquad 
\Pi _S =  \big\{ (I_k)_{k \in \Z}\, :\, I_k > 0\,\, \forall k \in S\,; \,\,\, I_k = 0\,\, \forall k \in S^\bot \big\} \,.
\ee
\noindent In addition, for any $\ell \in {\mathbb Z}^S, \ a, b \in S^\bot $, with $a \neq b$, the following functions are real
analytic and satisfy a.e. on $\Pi_S$
\be\label{Melnikov-1-2}
\sum _{n\in S} \ell _n \omega^{nls}_n \pm \omega
^{nls}_a \not= 0\,, \quad \sum _{n \in S} \ell _n \omega^{nls}_n \pm (\omega
^{nls}_a + \omega^{nls}_b) \not= 0\,, \quad \sum _{n \in S} \ell _n \omega^{nls}_n + \omega
^{nls}_a - \omega^{nls}_b \not= 0\,.
\ee
\end{proposition}

\subsection{Hamiltonian setup}\label{Hamiltonian setup}

Recall that in \eqref{def:M-sigma} we introduced as phase space 
$$
M^\sigma := \T^S \times U_0 \times h_\bot^\sigma\,, \quad h^\sigma_\bot = h^\sigma(S^\bot, \C)\,,
$$
with coordinates denoted by $(\theta, y, z)$. 
Note that the tangent space of $M^\sigma$ is independent of the base point $(\theta, y, z)$ of $M^\sigma$. It is denoted by $T M^\sigma$ and given by
$$
T M^\sigma = \R^S \times \R^S \times h^\sigma_\bot\,.
$$
Denote by ${\rm Id}_\bot$  the identity operator on $h^\sigma_\bot$ 
and by ${\rm Id}_S$ the one  on $\R^S$.
The Poisson bracket between functionals $F, G : M^\sigma \to \R$ with sufficiently regular gradient
 is given by 
\begin{equation}\label{Poisson action angle}
\{ F, G \} :=  \begin{pmatrix}
\nabla_\theta F \\
\nabla_y F
\end{pmatrix}\, \cdot\, \begin{pmatrix}
0 & {\rm Id}_S \\
- {\rm Id}_S & 0
\end{pmatrix}\begin{pmatrix}
\nabla_\theta G \\
\nabla_y G
\end{pmatrix}  +  \begin{pmatrix}
\nabla_z F \\
\nabla_{\bar z} F
\end{pmatrix}\,\cdot\,  \begin{pmatrix}
0 & - \ii \,\, {\rm Id}_\bot \\
\ii\,\, {\rm Id}_\bot & 0
\end{pmatrix}\begin{pmatrix}
\nabla_z G \\
\nabla_{\bar z} G
\end{pmatrix} \,,
\end{equation}
where in the latter expression, the dot denotes the bilinear form on 
$(h_\bot^\sigma)^2 \times (h_\bot^\sigma)^2$ given by  
\be\label{bilinear-pairing}
\big( (w, \tilde w)\,, \, (z, \tilde z) \big) \mapsto \begin{pmatrix}
w \\
\tilde w
\end{pmatrix} \cdot \begin{pmatrix}
z \\
\tilde z
\end{pmatrix} := w \cdot z + \tilde w \cdot \tilde z\, , \qquad 
w \cdot z   = \sum_{k \in S^\bot} w_{k} z_{ k} \in \C 
\ee
and $\nabla_z F = (\partial_{z_k} F)_{k \in S^\bot}$, 
$ \nabla_{\bar z} F = (\partial_{\bar z_k} F)_{k \in S^\bot} $ with 
$$
\partial_{z_k} F := \frac{1}{ \sqrt{2}} (\partial_{ x_k} F + \ii \partial_{y_k} F) \, , \quad
\partial_{\bar z_k} F := \frac{1}{ \sqrt{2}} (\partial_{x_k } F - \ii \partial_{y_k} F) 
$$
and   $ x_k =  \sqrt{2} {\rm Re} z_k$,  $y_k = -   \sqrt{2} {\rm Im} z_k$ defined as in  \eqref{real Birkhoff coordinates}.
For such a functional $F$, the corresponding Hamiltonian vector field is written as 
\begin{equation}\label{Hamiltonian vector field theta y z}
X_F := (\nabla_y F, - \nabla_\theta F, - \ii \nabla_{\bar z} F) \, .
\end{equation}
The Hamiltonian vector field $ X_F $ may be in $ T M^{\sigma } $ or
lose regularity as  the dNLS Hamiltonian vector field  which takes values in $T M^{\sigma - 2}$.
In complex notations, the differential $d X_F$ of the vector field $X_F$ is given by 
$$
\begin{pmatrix}
\widehat \theta \\
 \widehat y \\
  \widehat z
\end{pmatrix}
\mapsto
\begin{pmatrix}
\partial_\theta \nabla_y F [\widehat \theta] + \partial_y \nabla_y F[\widehat y] + \partial_z \nabla_y F[\widehat z] + \partial_{\bar z} \nabla_y F[ \, \widehat {\bar z} \, ] \\
- \partial_\theta \nabla_\theta F [\widehat \theta] - \partial_y \nabla_\theta F[\widehat y] - \partial_z \nabla_\theta F[\widehat z] - \partial_{\bar z} \nabla_\theta F[ \, \widehat {\bar z} \, ] \\
- \ii \partial_\theta \nabla_{\bar z} F [\widehat \theta] - \ii \partial_y \nabla_{\bar z} F[\widehat y] - \ii \partial_z \nabla_{\bar z} F[\widehat z] - \ii \partial_{\bar z} \nabla_{\bar z} F[ \, \widehat {\bar z} \, ]
\end{pmatrix}\,
$$
where $\partial_\theta$, $\partial_y$,  $\partial_z$, and $\partial_{\bar z}$ are defined in the standard way, i.e., for instance
$$
\partial_z \nabla_y F[\widehat z] = \sum_{k \in S^{\bot}} \widehat z_k \partial_{z_k} \nabla_y F\,.
$$
 It turns out to be convenient to add to the domain of $d X_F$ as fourth component the complex conjugate of the third one and to extend the resulting map to the following linear operator defined on $\R^S \times \R^S \times h^\sigma_\bot \times h^\sigma_\bot$, still denoted by $d X_F$,
\begin{equation}\label{differential XF (2)}
d X_F : \begin{pmatrix}
\widehat \theta \\
 \widehat y \\
  \widehat z_1 \\
  \widehat z_2
\end{pmatrix}
\mapsto
\begin{pmatrix}
\partial_\theta \nabla_y F [\widehat \theta] + \partial_y \nabla_y F[\widehat y] + \partial_z \nabla_y F[\widehat z_1] + \partial_{\bar z} \nabla_y F[\widehat z_2] \\
- \partial_\theta \nabla_\theta F [\widehat \theta] - \partial_y \nabla_\theta F[\widehat y] - \partial_z \nabla_\theta F[\widehat z_1] - \partial_{\bar z} \nabla_\theta F[\widehat z_2] \\
- \ii \partial_\theta \nabla_{\bar z} F [\widehat \theta] - \ii \partial_y \nabla_{\bar z} F[\widehat y] - \ii \partial_z \nabla_{\bar z} F[\widehat z_1] - \ii \partial_{\bar z} \nabla_{\bar z} F[\widehat z_2] \\
 \ii \partial_\theta \nabla_{ z} F [\widehat \theta]  + \ii \partial_y \nabla_{ z} F[\widehat y]  + \ii \partial_{ z} \nabla_{z} F[\widehat z_1] + \ii \partial_{ \bar z} \nabla_{ z} F[\widehat z_2]
\end{pmatrix}\,.
\end{equation} 
Here we use that by assumption $F$ is real valued and 
hence $\overline{\nabla_z F} = \nabla_{\bar z} F$.

The  symplectic form corresponding to the Poisson bracket \eqref{Poisson action angle} is the restriction to the real subspace $\{ (\theta, y, z, \bar z) : (\theta, y, z) \in T M^\sigma \}$ of $\R^S \times \R^S \times h^\sigma_\bot \times h^\sigma_\bot$ of the skew symmetric $\C$-bilinear form 
$$
\big( \R^S \times \R^S \times h^\sigma_\bot \times h^\sigma_\bot \big) \times \big( \R^S \times \R^S \times h^\sigma_\bot \times h^\sigma_\bot \big) \to \C\, , 
$$
associating to two elements $(\widehat \theta^{(i)}, \widehat y^{(i)}, \widehat z^{(i)}_1, \widehat z_2^{(i)} )$, $i = 1,2 $, the complex number 
\begin{equation}\label{symplectic form per esteso}
 \begin{pmatrix}
0 & {\rm Id}_S \\
- {\rm Id}_S & 0
\end{pmatrix}^{- 1} \begin{pmatrix}
\widehat \theta^{(1)} \\
\widehat y^{(1)}
\end{pmatrix}\, \cdot \, \begin{pmatrix}
\widehat \theta^{(2)} \\
\widehat y^{(2)}
\end{pmatrix}  +  \begin{pmatrix}
0 & - \ii \,\,{\rm Id}_\bot \\
\ii\,\, {\rm Id}_\bot & 0
\end{pmatrix}^{- 1} \begin{pmatrix}
\widehat z^{(1)}_1 \\
\widehat z_2^{(1)}
\end{pmatrix}\,\cdot \,\begin{pmatrix}
\widehat z_1^{(2)} \\
\widehat z_2^{(2)}
\end{pmatrix}\,.
\end{equation}
This symplectic form $ \Lambda $ can be expressed as in \eqref{symplectic-2-form}.

It immediately follows from the above definition that for any $Y \in T M^\sigma$ and any ${\cal C}^1$ functional $F : M^\sigma\to \C$ with sufficiently regular gradient, one has $d F(Y) = \Lambda (X_F, Y)$. We also introduce the Liouville
 1-form $\lambda : T M^\sigma \to \C$ defined by
\begin{equation}\label{contact 1 form}
\lambda = - \sum_{k \in S} y_k d \theta_k + \ii \sum_{k \in S^\bot} z_k d \bar z_k\,.
\end{equation}
At any given point $(\theta, y, z)$, $\lambda$ is the bounded $\R$-linear functional 
$$
T M^\sigma \to \C\,, \quad (\widehat \theta, \widehat y , \widehat z) \to - \sum_{k \in S} y_k \widehat \theta_k + \ii \sum_{k \in S^\bot} z_k \overline{\widehat z}_k.
$$
A diffeomorphism $\Gamma : {\cal U} \to M^\sigma$, defined on an open subset ${\cal U}$ of $M^\sigma$, is said to be symplectic if $\Gamma^* \Lambda = \Lambda$ at any point $(\theta, y, z) \in {\cal U}$.
Note that $h^\sigma_\bot$ is a symplectic subspace of $h^\sigma$. Indeed the pull back $\Lambda_\bot$ of the symplectic form $\Lambda$ by the inclusion $h^\sigma_\bot \hookrightarrow M^\sigma$, is given by 
$$
\Lambda_\bot =  \ii \sum_{k \in S^\bot} d z_k \wedge d \bar z_k\,,
$$
which is clearly a non-degenerate bilinear form on $h^\sigma_\bot$. 
Now we consider $\vphi$-dependent canonical transformations on $h^\sigma_\bot$.

\begin{definition}\label{linear symplectic transformations}
{\bf (Symplectic operator)}
An operator valued map $\T^S \to {\cal L}(h^\sigma_\bot)$ 
of the form $h \mapsto \Phi_1(\vphi) h + \Phi_2(\vphi) \bar h$
is said to be symplectic if $\Phi(\vphi)^* \Lambda_\bot = \Lambda_\bot$ for any $\varphi \in \T^S$. 
The map $\Phi(\vphi)$, when extended as a $\C$-linear map to  
$h^\sigma_\bot  \times h^\sigma_\bot$,
\begin{equation}\label{mappa simpatetica doppie coordinate}
h^\sigma_\bot  \times h^\sigma_\bot \to h^\sigma_\bot \times h^\sigma_\bot\,, \quad   \begin{pmatrix}
h_1 \\
h_2
\end{pmatrix} \mapsto \begin{pmatrix}
\Phi_1(\vphi) & \Phi_2(\vphi) \\
\overline{\Phi_2(\vphi)} & \overline{\Phi_1(\vphi)}
\end{pmatrix}
\begin{pmatrix}
h_1 \\
h_2
\end{pmatrix}
\end{equation}
 is also denoted by $\Phi(\vphi)$. We denote by $ \overline{\Phi_i} $ the operators 
given by  $ \overline{\Phi_i} ( h) := \overline{\Phi_i ( \bar h) } $ where
 $ \bar h := ( {\bar h}_k )_{k \in S^\bot } $. 
\end{definition}
In view of \eqref{symplectic form per esteso}, the property of $\Phi(\vphi)$ being symplectic can be expressed in terms of the map \eqref{mappa simpatetica doppie coordinate} as follows  
\begin{equation}\label{condizione simpletticita matrice}
\Phi(\vphi)^t {\mathbb J}_2 \Phi(\vphi) = {\mathbb  J}_2\,,
\end{equation}
where 
\begin{equation}\label{notazione J bot}
\Phi(\vphi)^t = \begin{pmatrix}
\Phi_1(\vphi)^t & \overline{\Phi_2(\vphi)}^t \\
{\Phi_2(\vphi)^t}  & \overline{\Phi_1(\vphi)}^t
\end{pmatrix}\,, \qquad {\mathbb J}_2 := \ii \begin{pmatrix}
0 & {\rm Id}_{\bot} \\
- {\rm Id}_\bot & 0
\end{pmatrix} 
\end{equation}
where $ [\Phi_i (\vphi)]^t $ denotes the transpose  
with respect to the bilinear form defined in \eqref{bilinear-pairing}. 

Next, let us consider a family of quadratic Hamiltonians $F(\vphi, \cdot) : h^\sigma_\bot \to \R$, $\vphi \in \T^S$, of the form
\be\label{quadratic formH}
F(\vphi, z) = \bar z \cdot A_1(\vphi) z +  \frac12 \bar z \cdot A_2(\vphi) \bar z + \frac12 z \cdot A_3(\vphi) z \, , \quad z \in h^\sigma_\bot\,, 
\ee
where $A_i (\vphi) $, $1 \le i \le 3$,  $ \vphi \in \T^S $,  are (possibly unbounded) linear operators on $ h^\sigma_\bot $.
Without loss of generality we may require that for $i= 2, 3$, one has  $A_i^t = A_i$.
The assumption that $F$ is real valued implies that
$$
A_1^* = A_1 \, , \quad \bar A_2 = A_3 \, ,
$$
where for any $\vphi \in \T^S $, $ A_1^*(\vphi) $ is the adjoint operator of $ A_1(\vphi)$ with respect 
to the standard complex scalar product on $ h_\bot^0 $,  
\be\label{def:scalar-product}
(  z, w ) := \sum_{n \in S^\bot} z_n {\bar w}_n \, , \quad \forall z, w  \in h_\bot^0  \, . 
\ee
Note that $ A_1 = \pa_{ z } \nabla_{\bar z} F$, $ A_2 = \pa_{\bar z} \nabla_{\bar z} F$ 
and $ A_3 = \pa_{z } \nabla_z F $. 
The $\vphi$-dependent Hamiltonian vector field $X_F,$ associated to the Hamiltonian $F$, is the map $\vphi \mapsto X_F(\vphi)$ with $X_F(\vphi)$ given for any $\vphi \in {\mathbb T}^S$ by 
$$
h^\sigma_\bot \to h^\sigma_\bot\,, \quad h \mapsto  - \ii (A_1(\vphi) h + A_2(\vphi) \bar h)\,.
$$
In the case at hand, the formula analogous to \eqref{differential XF (2)} is then given by 
$$
- \left(
\begin{array}{cc}
 \ii {\rm Id}_\bot & 0    \\
0  & - \ii   {\rm Id}_\bot\\
\end{array} \right)
\left(
\begin{array}{cc}
A_1  & A_2    \\
\overline A_2  &  \overline A_1   \\
\end{array} \right) \, , \quad A_1^* = A_1 \, ,   \quad A_2^t = A_2 \,   . 
$$

\begin{definition}{\bf (Hamiltonian operator)} \label{def:Hamiltonian op}
The operator $ J A(\vphi) $ where
\begin{equation}\label{operatore Hamiltoniano lineare 2} 
 J := \begin{pmatrix}
\ii {\rm Id}_\bot & 0 \\
0 & - \ii {\rm Id}_\bot
\end{pmatrix}\,, \quad A(\vphi) := \begin{pmatrix}
A_1(\vphi) & A_2(\vphi) \\
\overline{A_2(\vphi)} & \overline{A_1(\vphi)}
\end{pmatrix} \,, \quad 
 A_1^* = A_1 \, ,   \quad A_2^t = A_2 \,   ,
\end{equation}
as well as 
 the operator  ${\frak L}(\vphi)$ 
defined, for $\vphi \in \T^S$,  by
\begin{equation}\label{operatore Hamiltoniano lineare}
{\frak L}(\vphi) = \omega \cdot \partial_\vphi {\mathbb I}_2 + J  A(\vphi)
\,, \qquad {\mathbb I}_2 = \begin{pmatrix}
{\rm Id}_\bot & 0 \\
0 & {\rm Id}_\bot
\end{pmatrix}
\end{equation}
are referred to as linear Hamiltonian operators associated to the Hamiltonian $ F $ in \eqref{quadratic formH}. 
\end{definition}

Equivalently  the Hamiltonian operator $J A(\vphi)$ can be written in the form 
\begin{equation}\label{forma diversa hamiltoniana cal Q (vphi)}
 J A(\vphi) = {\mathbb  J}_2 {\mathbb A} (\vphi)\,, \quad  {\mathbb A} (\vphi) := \begin{pmatrix}
\overline{A_2(\vphi)} & \overline{A_1(\vphi)} \\
A_1(\vphi) & A_2(\vphi)
\end{pmatrix}\, \quad {\mathbb A}^t (\vphi) = {\mathbb A}(\vphi)
\end{equation}
where ${\mathbb  J}_2$ is defined in \eqref{notazione J bot} and 
${\mathbb A}^t(\vphi) = {\mathbb A}(\vphi)$, since $A_1^t = \bar A_1$ and $A_2^t = A_2$. 

\begin{lemma}\label{transformation of Hamiltonian operators}
Assume that $\Phi \in {\cal C}^1(\T^S, {\cal L}(h^\sigma_\bot \times h^\sigma_\bot))$ is a map
with $\Phi(\varphi)$ a linear symplectic transformation for any $\vphi \in \T^S $ 
(cf Definition~\ref{linear symplectic transformations}) and $ {\frak L} (\vphi)  $ a Hamiltonian operator (cf Definition \ref{def:Hamiltonian op}). 
Then the transformed operator 
${\frak L}_+(\vphi) := \Phi^{- 1}(\vphi) {\frak L}(\vphi) \Phi(\vphi)$   is Hamiltonian 
and  of the form 
${\frak L}_+(\vphi) = \omega \cdot \partial_\vphi {\mathbb I}_2
+ {\mathbb J}_2 {\mathbb A}_+(\vphi)$,  
where 
\be\label{def:A+}
{\mathbb A}_+(\vphi) := \Phi^t(\vphi) {\mathbb A}(\vphi) \Phi(\vphi) 
+ \Phi^t(\vphi) {\mathbb J}_2 \,  (\omega \cdot \partial_\vphi)(\Phi(\vphi))\,,
\ee
and satisfies ${\mathbb A}_+(\vphi) = {\mathbb A}^t_+(\vphi)$. 
Here we denoted by $ \Phi^{- 1}(\vphi) $ the operator $\Phi^{- 1}(\vphi) := (\Phi(\vphi))^{-1} $ for any  $ \vphi \in \T^S $. 
\end{lemma}
\begin{proof}
Using the representation \eqref{forma diversa hamiltoniana cal Q (vphi)} for the Hamiltonian operator 
$ {\frak L}(\vphi) = \omega \cdot \partial_\vphi {\mathbb I}_2 + {\mathbb J}_2 {\mathbb A}(\vphi) $
we have 
\begin{equation}\label{campo trasformato generale}
{\frak L}_+(\vphi) = \Phi^{- 1}(\vphi) {\frak L}(\vphi) \Phi(\vphi) 
= \omega \cdot \partial_\vphi {\mathbb I}_2
+ \Phi^{- 1}(\vphi) {\mathbb J}_2 {\mathbb A}(\vphi) \Phi(\vphi)
+ \Phi^{- 1}(\vphi) (\omega \cdot \partial_\vphi) (\Phi(\vphi)) \,.
\end{equation}
By the condition \eqref{condizione simpletticita matrice} and using that 
${\mathbb J}_2^{- 1} = {\mathbb J}_2$, one has 
$ \Phi^{- 1}(\vphi) {\mathbb J}_2 = {\mathbb J}_2 \Phi^t(\vphi) $, 
yielding
\begin{equation}\label{campo trasformato generale 1}
\Phi^{- 1}(\vphi){\mathbb J}_2 {\mathbb A}(\vphi) \Phi(\vphi) 
=  {\mathbb J}_2 \Phi^t(\vphi) {\mathbb A}(\vphi) \Phi(\vphi)\,.
\end{equation}
Since ${\mathbb J}_2^2 = {\mathbb I}_2$, and using that 
by \eqref{condizione simpletticita matrice} ${\mathbb J}_2  \Phi^{- 1}(\vphi) = {\Phi^t(\vphi) {\mathbb J}_2},$
 we have 
\begin{equation}\label{campo trasformato generale 2}
\Phi^{- 1}(\vphi) (\omega \cdot \partial_\vphi)(\Phi(\vphi)) 
= {\mathbb J}_2 \big( {\mathbb J}_2  \Phi^{- 1}(\vphi) (\omega \cdot \partial_\vphi)(\Phi(\vphi))  \big) 
= {\mathbb J}_2 \big( \Phi^t(\vphi) {\mathbb J}_2   (\omega \cdot \partial_\vphi)(\Phi(\vphi))  \big)\,.
\end{equation}
Combining \eqref{campo trasformato generale}, \eqref{campo trasformato generale 1}, \eqref{campo trasformato generale 2} we get the claimed formula
$ {\frak L}_+(\vphi) = \omega \cdot \partial_\vphi {\mathbb I}_2 
+ {\mathbb J}_2 {\mathbb A}_+(\vphi) $ with 
$ {\mathbb A}_+(\vphi) $ given in \eqref{def:A+}.

It remains to verify that ${\mathbb A}_+(\vphi) = {\mathbb A}^t_+(\vphi)$. To see that $\Phi^t(\vphi) {\mathbb J}_2   (\omega \cdot \partial_\vphi)(\Phi(\vphi))$ is symmetric, 
note that by \eqref{condizione simpletticita matrice}, for any $\vphi \in \T^S$, 
$$
0 = (\omega \cdot \partial_\vphi) \big( \Phi^t(\vphi) {\mathbb J}_2 \Phi(\vphi) \big) 
= (\omega \cdot \partial_\vphi)(\Phi^t(\vphi) ) {\mathbb J}_2 \Phi(\vphi) 
+ \Phi^t(\vphi){\mathbb J}_2 (\omega \cdot \partial_\vphi)(\Phi(\vphi)) \,,
$$
implying that 
$$
\Phi^t(\vphi) {\mathbb J}_2   (\omega \cdot \partial_\vphi)(\Phi(\vphi)) 
= - (\omega \cdot \partial_\vphi)(\Phi^t(\vphi) ) {\mathbb J}_2 \Phi(\vphi) 
= (\omega \cdot \partial_\vphi)(\Phi^t(\vphi) ) {\mathbb J}_2^t \Phi(\vphi) 
= \big( \Phi^t(\vphi) {\mathbb J}_2   (\omega \cdot \partial_\vphi)(\Phi(\vphi)) \big)^t\,.
$$
Since by assumption ${\mathbb A}(\vphi)$ is symmetric, so is $\Phi^t(\vphi) {\mathbb A}(\vphi) \Phi(\vphi)$.
In view of the formula for ${\mathbb A}_+(\vphi)$, it then follows that 
${\mathbb A}_+(\vphi)$ is symmetric.
\end{proof}

In the sequel we use the shorthand notations $F_{nls}^\bot$ and 
$(F_{nls}^{- 1})_{\hookrightarrow}$, the latter being identified by a slight abuse of terminology with $ F_{nls}^{- 1}$, i.e.,
 \begin{equation}\label{proiettori trasformata di fourier}
 F_{nls}^\bot := {\mathbb I}_\bot F_{nls} \quad \mbox{and}  \quad F_{nls}^{- 1} \equiv (F_{nls}^{- 1})_{\hookrightarrow} := 
 F_{nls}^{- 1}{\mathbb I}_{\hookrightarrow} 
 \end{equation}
where, recalling that  $ \pi_\bot $ denotes  the $ L^2 $ projector \eqref{def:pi0-bot} onto $ H^\s_\bot $, 
 \be\label{def:op-bot}
\mathbb I_\bot := 
\begin{pmatrix}
 \pi_\bot & 0 \\
 0 & \pi_\bot
 \end{pmatrix} \qquad {\rm and } \qquad {\mathbb I}_{\hookrightarrow} :
h^\s_\bot \times h^\s_\bot \to h^\s \times h^\s 
\ee 
denotes the inclusion map. 
Note that 
\be\label{property-FF-1}
F_{nls}^{- 1} F_{nls}^\bot = \mathbb I_\bot \, .
\ee
 According to \eqref{Fourier transform F nf} 
 \begin{equation}\label{matrix representation F nls}
 F_{nls}^\bot = \begin{pmatrix}
 F_1 & 0 \\
 0 & F_2
 \end{pmatrix}\,, \qquad F_{nls}^{- 1} = \begin{pmatrix}
 G_1 & 0 \\
 0 & G_2
 \end{pmatrix}
 \end{equation}
 where for any $u \in H^\sigma$ 
 $$
 F_1 (u) = - (u_{- n})_{n \in S^\bot}\,, \quad F_2(u) = - (u_n)_{n \in S^\bot}
 $$
 and for any $z = (z_n)_{n \in S^\bot} \in h^\sigma_\bot$ 
 $$
 G_1(z) = - \sum_{n \in S^\bot} z_{- n} e^{2 \pi \ii n x}\,, \qquad G_2(z) = - \sum_{n \in S^\bot } z_n e^{2 \pi \ii n x}\,.
 $$
In view of the definitions \eqref{complex-real-scalar-pr}, \eqref{bilinear-pairing}, \eqref{def:scalar-product} one verifies that 
 \begin{align}\label{prop F nls 0}
 F_2 & = \overline F_1 \,,  \qquad \qquad\qquad\qquad   G_2  = \overline G_1\,, \\
 \label{prop F nls 1}
 z \cdot F_1(u) & = \langle G_2(z), u \rangle_r\,,  \qquad z \cdot F_2(u)  = \langle G_1(z), u \rangle_r\,, \\
\label{prop F nls 2}
  (z, F_1(u)) & = \langle G_1(z), u \rangle\,,  \qquad  (z, F_2(u))  = \langle G_2(z), u \rangle\,.
  \end{align}

\begin{lemma}\label{lemma:HamiltonianVF}
Assume that $ A $ is a linear operator acting on $  H^\sigma \times H^\sigma $ of  the form 
\be\label{form-simm-op}
A = \begin{pmatrix} 
B  & C  \\
\overline{C} & \overline{B}
\end{pmatrix} \,, \quad 
B^* = B \, ,   \quad C^t = C \,   
\ee
where $ B^* $ is the adjoint of $ B $ with respect to the 
complex $ L^2 (\T_1) $ scalar product  $\langle \ ,  \ \rangle$ 
and $ C^t $ is the transposed  with respect to the real bilinear form $ \langle \ ,  \ \rangle_r $,
 where $\langle \ ,  \ \rangle$ and $ \langle \ ,  \  \rangle_r $  are defined in  \eqref{complex-real-scalar-pr}. Then the operator
$ J F_{nls}^{\bot} A F_{nls}^{-1} $ is Hamiltonian. 
\end{lemma}

\begin{proof}
By \eqref{matrix representation F nls} one has 
$$
F_{nls}^\bot A F_{nls}^{- 1} = \begin{pmatrix}
F_1 B G_1 & F_1 C G_2 \\
F_2 \overline C G_1 & F_2 \overline B G_2
\end{pmatrix}\,.
$$
Using the identities \eqref{prop F nls 0}-\eqref{prop F nls 2} one verifies that all the conditions listed in the Definition \ref{def:Hamiltonian op} of a Hamiltonian operator are satisfied. 
\end{proof}

\subsection{Tame estimates for the Hamiltonian vector fields
 $X_{H^{nls}}\circ \breve \io$ and $X_{P} \circ \breve \io$} 
\label{composition of Hamiltonian with io}

In this subsection we derive tame estimates for the compositions
of torus embeddings $\breve \io: \T^S \to M^\s$ with the dNLS Hamiltonian $H^{nls}$
and with the perturbation $P$
where 
$M^\sigma$ is the phase space introduced in \eqref{def:M-sigma}.   

Recall that the dNLS Hamiltonian
$H^{nls}$ is a function of the actions $I_n, n \in \Z$, alone and that $I_n = \xi_n + y_n$, $n \in S,$
and $I_n = z_n \bar z_n,$ $n \in S^\bot$. To simplify notation, given a
map $\breve \io: \T^S \to M^\s$, we will frequently suppress
the variable $\vphi$ in $\breve \io(\vphi) = (\theta(\vphi), y(\vphi), z(\vphi))$.
 The main results are the following ones.
\begin{proposition}\label{stime derivate H nls} Given an integer $s \ge s_0$,
there exists $0 < \rho_1 \le 1$ so that for any map
$\breve \io(\vphi) = (\vphi, 0, 0) + \io(\vphi)$  with 
$\io \in H^{s+2s_0}(\T^S, \R^S \times \R^S \times h_\bot^\s)$ 
and $\| \io\|_{ 3s_0} \leq \rho_1$, one has $\breve \io(\T^S) \subset  M^\s$ and the following holds:
 
\smallskip

\noindent
$(i)$ The dNLS frequencies $\omega_n^{nls}$ satisfy the tame estimate 
  \begin{equation}\label{tame estimates for omega}
 \sup_{n \in \Z} \| \omega_n^{nls}(\xi + y, z \bar z) - \omega_n^{nls}(\xi, 0) \|_s 
\leq_s \| \io\|_{s + 2 s_0} \,.
 \end{equation}
Moreover, for any $N \in \Z_{\ge1}$, there exists $0 < \rho_N \le \rho_1$
so that in case  $\| \io\|_{3s_0} \leq \rho_N,$ 
 \begin{equation}\label{tame composizione derivate ennesime frequenza}
\sup_{ 1 \le |\alpha| \leq N}\sup_{n \in \Z} \big\| \big( \prod_{j \in \Z} \langle j \rangle^{-2\alpha_j}\big)
\partial_{I}^\alpha \omega_n^{nls}(\xi + y, z \bar z) \big\|_s
\leq_s 1 + \| \io\|_{s + 2 s_0}
\end{equation}
 where the supremum is taken over all multi-indices $\alpha = (\alpha_j)_{j \in \Z}$ 
with $\alpha_j \in \Z_{\ge 0}$ and $1 \le |\alpha| =\sum_{j \in \Z} \alpha_j \leq N$.
 
 \noindent
 $(ii)$ The derivatives of $\nabla_y H^{nls}(\xi + y, z \bar z)$ and $\nabla_z H^{nls}(\xi + y, z \bar z)$ with respect to $y$ satisfy the tame estimates
 $$
 \| \partial_y \nabla_y H^{nls}(\xi + y, z \bar z) - \partial_y \nabla_y H^{nls}(\xi , 0) \|_s 
\leq_s \| \io\|_{s + 2s_0}\, , \qquad
 \| \partial_y \nabla_z H^{nls}(\xi + y, z \bar z) \|_s\, \leq_s \| \io\|_{s + 2 s_0} \, .
 $$
 Since $\nabla_{\bar z} H^{nls} =\overline{\nabla_{z} H^{nls}}$, the derivative
$\partial_y\nabla_{\bar z} H^{nls}(\xi + y, z \bar z)$ satisfies the same tame estimate.

\smallskip

\noindent
$(iii)$ For any map $\widehat z$ in $H^s(\T^S, h_\bot^\s)$,
the derivatives of $\nabla_y H^{nls}$, $\nabla_z H^{nls}$, and  $\nabla_{\bar z} H^{nls}$
with respect to $z$ in direction $\widehat z$ satisfy the tame estimates
\begin{align*}
&  \| \partial_z \nabla_y H^{nls}(\xi + y, z \bar z)[\widehat z]\|_s \leq_s \| \io\|_{3 s_0} \| \widehat z\|_{s} + \| \io\|_{s + 2s_0} \| \widehat z\|_{s_0}\,, \\
& 
 \|\partial_z \nabla_z H^{nls}(\xi + y, z \bar z)[\widehat z] \|_s \leq_s \|\io \|_{3 s_0} \| \widehat{z}\|_s + \|\io \|_{s + 2 s_0} \| \widehat{z}\|_{s_0}\,,
\end{align*}
and
 \begin{equation}\label{estimate z bar z}
 \| \big(\partial_{ z} \nabla_{\bar z} H^{nls}(\xi + y, z \bar z) 
- \partial_{ z} \nabla_{\bar z} H^{nls}(\xi , 0) \big)[\widehat{z}]  \|_s 
\leq_s \|\io \|_{3 s_0} \| \widehat{z}\|_s + \|\io \|_{s + 2 s_0} \| \widehat{z}\|_{s_0}\,.
\end{equation}
Since $\partial_{\bar z} = \overline{\partial_{z} }$, the derivatives of
$\nabla_y H^{nls}(\xi + y, z \bar z)$, $ \nabla_z H^{nls}(\xi + y, z \bar z)$,
and $\nabla_{\bar z} H^{nls}(\xi + y, z \bar z)$ with respect to $\bar z$ in direction $\overline{\widehat z}$ 
satisfy corresponding tame estimates.

\smallskip
 
 \noindent
 $(iv)$ If in addition $\io \equiv \io_\omega$ is
 Lipschitz continuous in $\omega \in \Omega$ and satisfies $\| \io\|_{ 3s_0}^{\Lipg} \leq \rho_1$ 
it follows that for any map $\widehat z \equiv \widehat z_\omega$ 
in $H^s(\T^S, h_\bot^\s)$, which is also Lipschitz continuous in $\omega \in \Omega$,
 all the previous estimates hold with  $\| \cdot\|_s$ replaced by $\| \cdot\|_s^\Lipg$. 
\end{proposition} 
\begin{remark}
The estimate \eqref{tame composizione derivate ennesime frequenza} is only used in this paper for $N \leq 3$. See for instance Lemma \ref{quadraticPart X H nls} and Lemmata \ref{I bot Lip gamma}, \ref{lemma variazione i Omega nls}.  
\end{remark}

\begin{proof}
$(i)$ To obtain the claimed tame estimates, we want to apply 
Lemma~\ref{composition in Sobolev} ($ii$). First we need to make some preliminary considerations.
By \eqref{norm-l12}, for any $(z_n)_{n \in S^\bot} \in h^\s_\bot$,
$(z_n \bar z_n)_{n \in S^\bot}$ is in 
$\ell^{1,2\s}_{+,\bot} := \ell_+^{1,2\s}(S^\bot, \R)$ and 
$$
h^\s_\bot \to \ell^{1, 2\s}_{+,\bot}, (z_n)_{n \in S^\bot} \mapsto (z_n \bar z_n)_{n \in S^\bot} \, , \quad
\| (z_n \bar z_n)_{n \in S^\bot} \|_{\ell^{1,2\sigma}} = \| (z_n)_{n \in S^\bot} \|_\s^2 \, , 
$$
is a bounded quadratic map. 
In particular, this map is in 
${\cal C}^{\infty}(h^\s_\bot, \ell^{1, 2\s}_{+, \bot} ) $. By Theorem~\ref{Corollary 2.2}, 
for any $\xi \in \R^S_{>0},$ there exists an open neighborhood $V'$ of $(\xi, 0)$ in 
$\ell^{1, 2\s}_+$ so that the map 
$$
(\omega_n^{nls}- 4n^2\pi^2)_{n \in \Z} : V' \to \ell^\infty
$$
is in ${\cal C}^{\infty}(V', \ell^\infty)$. Altogether it then follows that there is an open convex  neighborhood $V$ of $(0, 0)$ in $U_0 \times h^\s_\bot$ so that the composition 
$f : V \to \ell^\infty$, defined by 
$f(y, z) : = (\omega_n^{nls}(\xi + y, z \bar z)- 4n^2\pi^2)_{n \in \Z}$,
is in ${\cal C}^{s + s_0}(V', \ell^\infty)$. Choose $0 < \rho_1 \le 1$ so that the closed ball
in $U_0 \times h^\s_\bot$ of radius $\rho_1$, centered at $(0,0)$, is contained in $V$.
By Lemma~\ref{lemma D omega}($iii$) (Sobolev embedding), it then follows that for any map
$\breve \io(\vphi) = (\vphi, 0, 0) +\io(\vphi)$ with $\| \io \|_{s_{0}} \le \rho_1,$
one has $(y(\vphi), z(\vphi)) \in V$ and hence by Lemma~\ref{composition in Sobolev}($ii$)
with $\breve \io^{(1)} :=\breve \io$, $\breve \io^{(2)} $ given by $ \breve \io^{(2)}(\varphi) = (\vphi, 0, 0)$,  
and 
$\breve \io^{(1)} - \breve \io^{(2)} = \io$
$$
 \sup_{n \in \Z} \| \omega_n^{nls}(\xi + y, z \bar z) - \omega_n^{nls}(\xi, 0) \|_s 
\leq_s \| \io\|_{s + 2 s_0} \,.
 $$ 
The tame estimates \eqref{tame composizione derivate ennesime frequenza} 
can be derived in a similar way, using this time item ($i$) of Lemma~\ref{composition in Sobolev}
as well as Theorem~\ref{Corollary 2.2}.

\noindent
$(ii)$ Note that 
$ \nabla_y H^{nls}(\xi + y, z \bar z) = \big( \omega_n(\xi + y, z \bar z) \big)_{n \in S}$
and hence
$$
\partial_y \nabla_y H^{nls}(\xi + y, z)= \big(  \partial_{I_k}\omega^{nls}_n(\xi + y, z \bar z) \big)_{n, k \in S}\,.
$$
Arguing similarly as in the proof of item ($i$), the claimed estimates for
$\partial_y \nabla_y H^{nls}(\xi + y, z \bar z) - \partial_y \nabla_y H^{nls}(\xi , 0)$
 follow from Lemma~\ref{composition in Sobolev}($ii$). Since
$ \nabla_z H^{nls}(\xi + y, z \bar z) = 
\big( \omega^{nls}_n(\xi + y, z \bar z) \bar z_n \big)_{n \in S^\bot}$
vanishes at $z=0$, one concludes that $\partial_y \nabla_z H^{nls}(\xi, 0)= 0$
and that in turn -- again in view of Lemma~\ref{composition in Sobolev}($ii$) --  
the tame estimates
$\| \partial_y \nabla_z H^{nls}(\xi + y, z \bar z) \|_s\, \leq_s \| \io\|_{s + 2 s_0}$ hold.

\smallskip
\noindent
$(iii)$ We only prove estimate \eqref{estimate z bar z} since the other ones can be derived by similar arguments. Taking the derivative of
$ \nabla_{\bar z} H^{nls}(\xi + y, z \bar z) = 
\big( \omega^{nls}_n(\xi + y, z \bar z)  z_n \big)_{n \in S^\bot}$
with respect to $z$ yields
$$
\partial_{ z} \nabla_{\bar z} H^{nls}(\xi + y, z \bar z)[\widehat z] = 
T_1  + T_2\,,
$$
where 
$$
T_1  := \Big(\omega_n^{nls}(\xi + y, z \bar z) \widehat{ z}_n \Big)_{n \in S^\bot}\, \quad 
{\mbox{and}} \quad 
T_2 := \Big( z_n \sum_{k \in S^\bot} \partial_{I_k} \omega_n^{nls}(\xi + y, z \bar z) \bar z_k  \widehat{ z}_k\Big)_{n \in S^\bot}\,.
$$
Concerning the term $T_1$, note that 
$$
\partial_{ z} \nabla_{\bar z} H^{nls}(\xi, 0)[\widehat{ z}] = 
\big( \omega^{nls}_n(\xi, 0) \widehat{ z}_n \big)_{n \in S^\bot}\,.
$$
 By Lemma \ref{interpolation product} (tame estimates for products of functions)
it follows that for any $n \in S^\bot$, the expression
$\| \big(\omega_n^{nls}(\xi + y, z \bar z) - \omega^{nls}_n(\xi, 0) \big) \cdot \widehat{ z}_n\|_s$
can be $\leq_s$-bounded by
$$
\|  \omega_n^{nls}(\xi + y, z \bar z) - \omega^{nls}_n(\xi, 0) \|_{s_{0}} 
\| \widehat{ z}_n \|_s
+ \| \omega_n^{nls}(\xi + y, z \bar z) - \omega^{nls}_n(\xi, 0) \|_s 
\| \widehat{ z}_n \|_{s_{0}} \, .
$$
Together with the estimates \eqref{tame estimates for omega} for 
$\omega_n^{nls}(\xi + y, z \bar z) - \omega^{nls}_n(\xi, 0)$,
this  yields
$$
\| \big(\omega_n^{nls}(\xi + y, z \bar z) - \omega^{nls}_n(\xi, 0) \big) \cdot \widehat{ z}_n\|_s
\leq_s  \| \io\|_{3 s_0} \| \widehat{ z}_n \|_s + \| \io\|_{s + 2 s_0} \| \widehat{ z}_n \|_{s_0}\,,
$$
implying,  by \eqref{norma other},  that 
\begin{equation}\label{che palle cacca A}
\big\| T_1 - \partial_{ z} \nabla_{\bar z} H^{nls}(\xi, 0) [\widehat{ z}] \big\|_s
\leq_s \| \io\|_{3 s_0} \| \widehat{z}\|_s + \| \io\|_{s + 2 s_0} \| \widehat{ z}\|_{s_0}\,.
\end{equation}
Towards the term $T_2$, note that for any $n, k \in S^\bot$, Lemma \ref{interpolation product} implies that
$
\| \partial_{I_k} \omega_n^{nls}(\xi + y, z \bar z) \bar z_k   \widehat{ z}_k \|_s$
is $\leq_s$- bounded by
$$
\| \partial_{I_k} \omega_n^{nls}(\xi + y, z \bar z) \|_s
\| z_k\|_{s_0}  \| \widehat{ z}_k\|_{s_0}
+ \| \partial_{I_k} \omega_n^{nls}(\xi + y, z \bar z) \|_{s_0} 
\big( \| z_k \|_{s} \|  \widehat{ z}_k\|_{s_0} + \| z_k \|_{s_0} \|  \widehat{ z}_k\|_{s} \big) \, . 
$$
By \eqref{norma other} we have $ \langle k \rangle^\s \| z_k \|_{s} \leq \| z \|_{s, \s} $. 
By assumption, $\langle k \rangle^2 \| z_k \|_{s_0} \le 1$ (recall that $\s \ge 4$)
whereas by \eqref{tame composizione derivate ennesime frequenza},
$$
\| \partial_{I_k} \omega_n^{nls}(\xi + y, z \bar z) \|_s 
\leq_s \langle k \rangle^2 \big( 1 + \|\io \|_{s+2s_0} \big)\,.
$$
Hence
$ 
\sum_{k \in S^\bot} \|  \partial_{I_k} \omega_n^{nls}(\xi + y, z \bar z) {\bar z}_k   \widehat{ z}_k \|_s 
$
is $\leq_s$-bounded by
$$
\big( 1 + \|\io \|_{s+2s_0} \big) \sum_{k \in S^\bot} \| \widehat{ z}_k\|_{s_0} +
\big( 1 + \|\io \|_{3s_0} \big) 
\Big( \| \io \|_{s} \sum_{k \in S^\bot}  \| \widehat{ z}_k\|_{s_0} +
 \sum_{k \in S^\bot} \| \widehat{ z}_k\|_{s} \Big)
$$
implying that (recall that $\s \ge 4$ and $\|\io\|_{3s_0} \le 1$)
\be\label{bound-interm1}
\Big\| \sum_{k \in S^\bot} \partial_{I_k} \omega_n^{nls}(\xi + y, z \bar z) \bar z_k \widehat z_k \Big\|_s
 \leq_s  \| \io\|_{s + 2 s_0} \| \widehat{ z}\|_{s_0} + \| \widehat{ z}\|_s \, .
\ee
Using again Lemma~\ref{interpolation product}, the term
$
\| z_n \sum_{k \in S^\bot} 
\partial_{I_k} \omega_n^{nls}(\xi + y, z \bar z) \bar z_k \widehat{ z}_k  \|_s 
$
can be $\leq_s$-bounded by
$$
 \| z_n\|_s  \cdot \Big\|  \sum_{k \in S^\bot} \partial_{I_k} \omega_n^{nls}(\xi + y, z \bar z) 
\bar z_k   \widehat{ z}_k \Big\|_{s_0}  +
 \| z_n\|_{s_0}  \cdot \Big\|  \sum_{k \in S^\bot} \partial_{I_k} \omega_n^{nls}(\xi + y, z \bar z) 
\bar z_k   \widehat{z}_k \Big\|_{s}\, ,
$$
yielding, by \eqref{bound-interm1}, the estimate
$$
\Big\| z_n \sum_{k \in S^\bot} 
\partial_{I_k} \omega_n^{nls}(\xi + y, z \bar z) \bar z_k \widehat{ z}_k  \Big\|_s  
\leq_s  \| z_n\|_s \cdot  \| \widehat{ z}  \|_{s_0}
+ \| z_n\|_{s_0} \cdot \big( \| \io\|_{s + 2 s_0} \| \widehat{ z}\|_{s_0} + \| \widehat{ z}\|_s
\big)\, .
$$
Therefore
$$
\| T_2\|_s^2  = \sum_{n \in S^\bot} \langle n \rangle^{2 \sigma}  
\big\| z_n \sum_{k \in S^\bot} \partial_{I_k} \omega_n^{nls}(\xi + y, z \bar z) 
\bar z_k \widehat{ z}_k  \big\|_s^2 
$$
is $\leq_s$-bounded by
$$
\sum_{n \in S^\bot} \langle n \rangle^{2 \sigma} \| z_n\|_s^2 \cdot  \| \widehat{ z}  \|_{s_0}^2 +
\sum_{n \in S^\bot} \langle n \rangle^{2 \sigma} \| z_n\|_{s_0}^2 \cdot
\big( \| \io\|_{s + 2 s_0} \| \widehat{ z}\|_{s_0} + \| \widehat{ z}\|_s)^2
$$
leading to the estimate (recall that $\|\io\|_{3s_{0}} \le 1$)
\be\label{estim-T2}
\| T_2\|_s \leq_s  \| \io\|_{s + 2 s_0} \| \widehat{ z}\|_{s_0} + \|\io\|_{s_0} \| \widehat{ z}\|_s \, .
\ee
The estimate \eqref{estimate z bar z} now follows from the bounds 
\eqref{che palle cacca A}, \eqref{estim-T2} derived for $T_1$ and $T_2$. 

\noindent
$(iv)$ The Lipschitz estimates are obtained by using similar arguments.
\end{proof}

Proposition~\ref{stime derivate H nls} can be applied to obtain tame estimates for the composition
of the differential $dX_{H^{nls}}$ of the Hamiltonian vector field $X_{H^{nls}}$ with a map
$\breve \io :\T^S \to M^\s, \vphi \mapsto \big(\theta(\vphi), y(\vphi), z(\vphi)\big).$ 
We denote by $d X_F$ the linear operator in \eqref{differential XF (2)}.

\begin{corollary}\label{differenziale X H nls imperturbato}
 Given an integer $s \ge s_0$,
there exists $0 < \rho \le 1$ so that for any map
$\breve \io(\vphi) = (\vphi, 0, 0) + \io(\vphi)$ with
$\io \in H^{s+2s_0}(\T^S, \R^S \times \R^S \times h_\bot^\s)$ 
and $\| \io\|_{ 3s_0} \leq \rho$, one has $\breve \io(\T^S) \subset  M^\s$ and the following holds:

\noindent
$(i)$ For any map 
$\widehat \io = (\widehat \theta, \widehat y, \widehat z_1, \widehat z_2) $ 
in $H^s(\T^S, \R^S \times \R^S \times h^\s_\bot \times h^\s_\bot)$,
$$
\big\| d X_{H^{nls}}(\xi + y, z\bar z)[\widehat \imath] - dX_{H^{nls}}(\xi, 0)[\widehat \imath] \, \big\|_s 
\leq_s  \| \io\|_{3 s_0} \| \widehat \imath\|_{s} + \| \io\|_{s + 2 s_0} \| \widehat \imath\|_{s_0}  
$$
where
$$
d X_{H^{nls}}(\xi, 0) [\widehat \imath] = \Big( \partial_y\nabla_y H^{nls} (\xi,0) [\widehat y]  , \,\, 0 , \,\,
- \ii \partial_z \nabla_{\bar z} H^{nls}(\xi, 0)[\widehat z_1], \,\,
 \ii \partial_{\bar z} \nabla_{ z} H^{nls}(\xi, 0)[\widehat z_2 ]
\Big)
$$
with
$\partial_y\nabla_y H^{nls} (\xi,0) [\widehat y]= 
\big( \sum_{k\in S}\partial_{I_k} \omega_n^{nls}(\xi, 0) \widehat y_k \big)_{ n \in S}$ and
$
\partial_z \nabla_{\bar z} H^{nls}(\xi, 0)[\widehat z_1] = 
\big( \omega_n^{nls}(\xi, 0) \widehat{ z_1}_n \big)_{n \in S^\bot}$.

\noindent
$(ii)$ If in addition $\breve \io \equiv \breve \io_\omega$ is Lipschitz continuous in $\omega \in \Omega$ and satisfies $\| \io\|_{3 s_0}^{\Lipg} \leq \rho$, then for any map 
$\widehat \io \equiv \widehat \io_\omega$ 
in $H^s(\T^S, \R^S \times \R^S \times h^\s_\bot \times h^\s_\bot)$
which are Lipschitz continuous in $\omega \in \Omega$,
the estimates of item ($i$) hold with  $\| \cdot\|_s$ replaced by $\| \cdot\|_s^\Lipg$. 
\end{corollary}

\begin{proof}
Since the Hamiltonian vector field $X_{H^{nls}}$ is given by
$$ 
X_{H^{nls}} = \big( \nabla_y H^{nls}, 0, - \ii \nabla_{\bar z} H^{nls} ) = 
\big(  (\omega_n^{nls})_{ n \in S}, \, 0 , \, -\ii \big( \omega_n^{nls} z_n \big)_{n \in S^\bot} \big),
$$
the first component of $d X_{H^{nls}}[\widehat \imath]$ is given by
$$
\partial_y\nabla_y H^{nls}[\widehat y] + \partial_z\nabla_y H^{nls}[\widehat z_1]
+ \partial_{\bar z}\nabla_y H^{nls}[\widehat z_2],
$$ the second component is $0,$ 
whereas the third and fourth components are 
$$
-\ii \big(\partial_y \nabla_{\bar z} H^{nls} [\widehat y] 
+ \partial_z \nabla_{\bar z} H^{nls} [\widehat z_1]
+ \partial_{\bar z} \nabla_{\bar z} H^{nls} [\widehat z_2] \big)
\quad {\mbox{and}} \quad
\ii \big(\partial_y \nabla_{ z} H^{nls} [\widehat y] 
+ \partial_{ z} \nabla_{z} H^{nls} [\widehat z_1] 
+ \partial_{\bar z} \nabla_{z} H^{nls} [\widehat z_2]\big).
$$
In particular, one obtains the claimed formula for 
$d X_{H^{nls}}(\xi, 0) [\widehat \imath]$ and
items $(i)$ and $(ii)$ follow from items $(ii)$ - $(iii)$, respectively item $(iv)$ 
of Proposition \ref{stime derivate H nls}.
\end{proof}

By Proposition \ref{stime derivate H nls} and the arguments used in its proof, one can also derive the following
\begin{lemma}\label{quadraticPart X H nls}
 Given an integer $s \ge s_0$,
there exists $0 < \rho \le 1$ so that for any map
$\breve \io(\vphi) = (\vphi, 0, 0) + \io(\vphi)$ with
$\io \equiv \io_\omega$  in $H^{s+2s_0}(\T^S, \R^S \times \R^S \times h_\bot^\s)$,
which is Lipschitz continuous in $\omega \in \Omega \subset \R^S$ and satisfies
$\| \io\|_{ 3s_0}^{\Lipg} \leq \rho$, one has $\breve \io(\T^S) \subset  M^\s$ 
and for any maps $\widehat \io^{(a)} \equiv \widehat \io^{(a)}_\omega$ 
in $H^s(\T^S, \R^S \times \R^S \times h^\s_\bot \times h^\s_\bot)$, $a= 1,2$, 
which are Lipschitz continuous in $\omega \in \Omega$, 
$$
\| d^2 X_{H^{nls}}(\xi + y, z\bar z)[\widehat \imath^{(1)}, \widehat \imath^{(2)}] \|_s^\Lipg \leq_s 
\|\widehat \imath^{(1)}\|_s^\Lipg \| \widehat \imath^{(2)}\|_{s_0}^\Lipg + 
\|\widehat \imath^{(1)}\|_{s_0}^\Lipg \| \widehat \imath^{(2)}\|_{s}^\Lipg +
\| \io\|_{s + 2 s_0}^\Lipg \| \widehat \imath^{(1)}\|_{s_0}^\Lipg \| \widehat \imath^{(2)}\|_{s_0}^\Lipg \,.
$$
\end{lemma}

We now state tame estimates for the Hamiltonian vector field 
of the perturbation $ P $. 
Recall that $P$ is the Hamiltonian ${\mathcal P}$, expressed in
Birkhoff coordinates on $M^\s$, where 
${\mathcal P} (u) = \int^1_0 \mbox{\em p}(x,u_1(x) , u_2(x))dx$ (cf  \eqref{1.4}) and
$\partial_{\bar \zeta} p$ is assumed to be of class 
$\mathcal C^{\sigma, s_*}$ with
$s_{*} > \mbox{max} (\s , s_0)$ sufficiently large. In the following proposition,
we restrict the range of $s$ so that
Lemma~\ref{composition in Sobolev} applies.

\begin{proposition}\label{teorema stime perturbazione}
 Given an integer $s$ with $s_0 \le s \le  s_{*}  - s_0 - 3$ ,
there exists $0 < \rho \le 1$ so that for any map
$\breve \io(\vphi) = (\vphi, 0, 0) + \io(\vphi)$ with
$\io \equiv \io_\omega$  in $H^{s+2s_0}(\T^S, \R^S \times \R^S \times h_\bot^\s)$,
which is Lipschitz continuous in $\omega \in \Omega$ and satisfies
$\| \io\|_{ 3s_0}^{\Lipg} \leq \rho$, one has $\breve \io(\T^S) \subset  M^\s$ 
and the following holds:

\noindent
$(i)$ $\nabla_\theta P,\nabla_y P$, and $\nabla_z P$ satisfy the tame estimates 
$$
\| \nabla_\theta P \|_s^\Lipg\,,\| \nabla_y P \|_s^\Lipg\,, \| \nabla_z P \|_s^\Lipg \leq_s 1 + \| \io\|_{s + 2 s_0}^\Lipg\,.
$$
The derivatives of $\nabla_\theta P,\nabla_y P$, and $\nabla_z P$ with respect to $\theta$ and $y$ satisfy the tame estimates
$$
\| \partial_\theta \nabla_{\theta} P \circ \breve \io \|_s^\Lipg,\,\, 
\| \partial_y \nabla_\theta P  \circ \breve \io \|_s^\Lipg,\,\,
\| \partial_\theta \nabla_y P  \circ \breve \io \|_s^\Lipg,\,\, 
\| \partial_y \nabla_y P  \circ \breve \io \|_s^\Lipg\,
\leq_s 1 + \| \io\|_{s + 2 s_0}^\Lipg
$$
and 
$$
\| \partial_\theta \nabla_z P  \circ \breve \io \|_s^\Lipg,\,\, 
\| \partial_y \nabla_z P  \circ \breve\io \|_s^\Lipg\,
\leq_s 1 + \| \io\|_{s + 2 s_0}^\Lipg\, .
$$
Since $\nabla_{\bar z} P =\overline{\nabla_{z} P}$, the derivatives of 
$\nabla_{\bar z} P$ with respect to $\theta$ and $y$ also satisfy the same tame estimates.

\smallskip

\noindent
$(ii)$ For any map $\widehat z_1 \equiv \widehat z_{1, \omega}$ 
in $H^s(\T^S, h_\bot^\s)$, which is Lipschitz continuous in $\omega \in \Omega$,
the derivatives of 
$\nabla_\theta P, \nabla_y P, \nabla_z P,$ and  $\nabla_{\bar z} P$
with respect to $z$ in direction $\widehat z_1$ satisfy the tame estimates
\begin{align}
& \| \partial_z \nabla_\theta P  \circ \breve \io \, [\widehat z_1] \|_s^\Lipg,\, \,
\|\partial_z \nabla_y P  \circ \breve \io \, [\widehat z_1] \|_s^\Lipg,\, \,
\| \partial_z \nabla_z P  \circ \breve \io \, [\widehat z_1] \|_s^\Lipg, \,\,
\| \partial_z \nabla_{\bar z} P  \circ \breve \io \, [\widehat z_1] \|_s^\Lipg \nonumber\\
& \leq_s \| \widehat z_1\|_s^\Lipg + \| \io\|_{s + 2 s_0}^\Lipg \| \widehat z_1\|_{s_0}^\Lipg\,. \nonumber
\end{align}
Since $\partial_{\bar z} = \overline{\partial_{z}}$, the derivatives of $\nabla_\theta P,\, \nabla_y P,\, \nabla_z P,$ and  $\nabla_{\bar z} P$
with respect to $\bar z$ in direction $\widehat z_2 \equiv \widehat z_{2, \omega}$ 
 admit the same bounds for any $\widehat z_2$
in $H^s(\T^S, h_\bot^\s)$, which is Lipschitz continuous in $\omega \in \Omega$.

\end{proposition}
\begin{proof}
The stated estimates can be shown in a similar way as the ones for 
the dNLS Hamiltonian.
\end{proof}

Finally, one can also derive tame estimates for the second derivative of the Hamiltonian vector field $X_P$.
 Again we restrict the range of $s$ so that Lemma~\ref{composition in Sobolev} applies.

\begin{lemma}\label{quadraticPart X P}
 Given an integer $s$ with $s_0 \le s \le  s_{*}  - s_0 -4$ ,
there exists $0 < \rho \le 1$ so that for any map
$\breve \io(\vphi) = (\vphi, 0, 0) + \io(\vphi)$ with
$\io \equiv \io_\omega$  in $H^{s+2s_0}(\T^S, \R^S \times \R^S \times h_\bot^\s)$,
which is Lipschitz continuous in $\omega \in \Omega$ and satisfies
$\| \io\|_{ 3s_0}^{\Lipg} \leq \rho$, one has $\breve \io(\T^S) \subset  M^\s$ 
and for any maps $\widehat \io^{(a)} \equiv \widehat \io^{(a)}_\omega$ 
in $H^s(\T^S, \R^S \times \R^S \times h^\s_\bot \times h^\s_\bot)$, $a= 1,2$, 
which are Lipschitz continuous in $\omega \in \Omega$, one has
$$
\| d^2 X_P  \circ \breve \io \, [\widehat \io^{(1)}, \widehat \io^{(2)}]\|_s^\Lipg 
\leq_s 
\|\widehat \imath^{(1)}\|_s^\Lipg \| \widehat \imath^{(2)}\|_{s_0}^\Lipg + 
\|\widehat \imath^{(1)}\|_{s_0}^\Lipg \| \widehat \imath^{(2)}\|_{s}^\Lipg +
\| \io\|_{s + 2 s_0}^\Lipg \| \widehat \imath^{(1)}\|_{s_0}^\Lipg \| \widehat \imath^{(2)}\|_{s_0}^\Lipg \,.
$$
\end{lemma}
\begin{proof}
The stated tame estimates correspond to the ones of Lemma~\ref{quadraticPart X H nls}
for the Hamiltonian vector field $X_{H^{nls}}$ and can be derived by the arguments used in the 
proof of Proposition~\ref{stime derivate H nls}.
\end{proof}

\section{Nash-Moser theorem} \label{Main result restated}

The purpose of this short section is to  reformulate Theorem~\ref{Theorem 1.1}  in the functional setup, described in the previous sections, and outline the organisation of its proof.

We consider torus embeddings
$$
 \breve \io : {\mathbb T}^S \rightarrow M^\sigma : \vphi \mapsto ( \theta(\vphi), y(\vphi), z(\vphi))
$$
whose lifts are assumed to be of the form 
$ (\vphi, 0, 0) + \io(\vphi)$ where
$$ 
\io (\vphi) = ( \Theta (\varphi ), y(\varphi ), z(\varphi )) \,
$$
with $\Theta : \R^S \to \R^S$ being $2 \pi$-periodic in each component of  
$\vphi = (\vphi_n)_{n \in S}$. 
We look   for zeros $   \io $ of the nonlinear operator $ F_\om $ defined in \eqref{definitionFep} by a Nash - Moser theorem.

In the sequel, we will identify such embeddings with their lifts.
Furthermore recall that the Sobolev norm $\| \io \|_{s, \sigma'}$, $\sigma' \leq \sigma$, of the periodic part  $ \io $ of the map $\breve \io$,  is given by
$$
\|  \io   \|_{s, \s'} := \| \Theta \|_s +  \| y  \|_s +  \| z \|_{s, \s'}
$$
where $ \| \Theta \|_s :=  \| \Theta  \|_{H^s (\T^S, \R^S)} $,  $ \| y \|_s :=  \| y  \|_{H^s (\T^S, \R^S)} $,
and $ \| z \|_{s, \s'} := \| z \|_{H^s(\T^S, h^{\s'}_\bot) } $ (cf \eqref{norma other}). In case
$ \s' = \s $ we also write $ \|  \io   \|_s $, $ \| z \|_s $, instead of $ \|  \io   \|_{s, \s},  \| z \|_{s, \s} $.

\begin{theorem}\label{main theorem}
Assume the assumptions of Theorem \ref{Theorem 1.1} hold. 
Then there is
$s_* > \max \big(\sigma, s_0 \big)$,  $s_0 = [|S|/2] +1,$ so that for any $ f  \in {\cal C}^{\sigma, s_*} $ in the perturbed equation 
\eqref{1.3}, 
there exists $0 < \varepsilon _0 < 1$  such that the following holds: for any $0 <\varepsilon  \leq \varepsilon _0$, 
there is a closed subset  $\Omega _\varepsilon \subseteq \Omega$ satisfying 
\be\label{measure estimate Omega in Theorem 4.1}
 \lim_{\e\to 0} \, \frac{{\rm meas}(\Omega_\e)}{{\rm meas}(\Omega )} = 1\, ,
\ee
so that for any $ \omega \in \Omega_\e $, there exists a torus embedding
$\breve \io_\omega : \T^S \to M^\s,$ satisfying 
$ \om \cdot \partial_\vphi \breve \io_\omega(\vphi) - X_{H_\e}( \breve \io_\omega(\vphi)) = 0 $.
This means  that the embedded torus 
$\breve \io_\omega (\T^S) $ is invariant for the Hamiltonian vector field $ X_{H_\e (\cdot, \xi)} $  
with $ \xi = (\om^{nls})^{-1}( \om ) $,  
and is filled by quasi-periodic solutions with the frequency $ \om $.
The map $\breve \io_\omega(\vphi)$ admits a lift of the form $ (\vphi,0,0) + \io_\omega(\vphi)$
where $\io_\omega$ is in $H^{s_0 + \mu_1}(\T^S, \R^S \times \R^S \times  h^\sigma_\bot)$ 
for some $\mu_2 > 0$ (depending only on $|S|$) with $s_0 + \mu_2 < s_* $, is Lipschitz continuous 
in $\omega \in \Omega_\e$,  and satisfies 
$$
\| \io_\omega\|_{s_0 + \mu_2}^\Lipg = O( \e \gamma^{-2}) 
\quad  \mbox{ with} \quad  \gamma \equiv \gamma_\e := \e^{\frak a} (< 1)\,, \quad  0 <  {\frak a} < 1/4\, .
$$
 Furthermore 
 the linearized equation at the quasi-periodic solution 
$\breve \io_\omega(\omega t) = \omega t + \io_\omega(\omega t)$ is 
 stable -- see 
Corollary~\ref{corollario finale nash moser} for a precise statement.  
\end{theorem}
\begin{remark}\label{remark size of torus}
In the estimates of the embedded tori we do not distinguish between the different components $\Theta$, $y$, $z$ of $\io$. Actually, the estimates for $y$ and $z$ can be sharpened for most $\o$ in $\Omega_\e$. It turns out that an effective way for proving the improved ones is to do so a posteriori, using that $F_\omega(\io_\omega, 0) = 0$ and that $\| \io_\omega\|_{s_0 + \mu_2}^\Lipg = O( \e \gamma^{-2})$. See Corollary \ref{stima optimal size torus} and its proof for details.
\end{remark}

\smallskip

\noindent
{\em Comments:}
\begin{enumerate} 
\item Up to the end of Section~\ref{sec:NM}, $\g \in (0,1) $ is assumed to be a
constant independent of $\e$ with $\e \g^{-4} $ small.
Only in  Section~\ref{sec:measure} (Theorem~\ref{measure estimate}), $\g$ and $\e$ are 
assumed to be related by requiring that $\gamma_\e = \e^{\frak a}$ for some $ 0 < {\frak a} < 1/4$.
The set $ \Om_\e $ is defined in \eqref{definizione cal G infty}.
\item
Let $\Pi \subseteq \Pi _S$ be a compact subset with measure $ |\Pi| >  0 $.
By Proposition~\ref{Proposition  2.3}, for any $\delta > 0$ there
exists an open subset $\Pi _\delta $ of
$\Pi _S$ so that ${\rm meas } (\Pi \cap \Pi _\delta ) \leq \delta $ and
on $\Pi \backslash \Pi _\delta , \ {\rm det }\big( (\partial _{I_j}
\omega _n^{nls}) _{i,j \in S} \big)$ is bounded and uniformly
bounded away from $0$. Hence on $\Pi \backslash \Pi _\delta $, the
action to frequency map $I \mapsto (\omega^{nls}_n)_{n \in S}$
is a local diffeomorphism. As $\Pi \backslash \Pi _\delta $ is compact
there exists a finite cover $(\Pi^{(i)})_{i \in {\cal I}}$ of $\Pi \backslash
\Pi _\delta $ with $\Pi ^{(i)}$ compact so that $\Pi^{(i)} \to \R^S, I \mapsto
(\omega^{nls}_n)_{n \in S}$ is a bi-Lipschitz homeomorphism onto
its image. By first choosing $\delta > 0 $ and then  applying 
Theorem~\ref{main theorem} for the finitely many parameter sets
$\Pi^{(i)}, i \in {\cal I}$, for $ 0 < \e \leq \e_0 (\delta) $, 
one sees that Theorem~\ref{main theorem} holds for any compact subset
 $\Pi \subseteq \Pi _S$ with ${\rm meas } (\Pi ) > 0$ as set of parameters.
\end{enumerate}

Theorem \ref{main theorem} -- which implies Theorem \ref{Theorem 1.1} --  is shown in 
Section~\ref{3. Set up} - \ref{sec:measure} by means of a Nash-Moser iteration scheme. 
Let us give a brief outline of its proof.
It is convenient to introduce an auxiliary variable $\zeta \in \R^S$ and consider the modified Hamiltonian vector field $ X_{H_{\e, \zeta}} = X_{H_\e} + (0, \zeta, 0 )$
with Hamiltonian 
\be\label{def:Hep}
H_{\e, \zeta} (\teta, y, z) \equiv H_{\e, \zeta} (\teta, y, z; \om )
:=  H_\e (\teta, y, z) + \zeta \cdot \theta\,,\quad \zeta \in \R^S \, , 
\ee
where $ H_\e $ is defined  in \eqref{HamiltonianHep} and considered as a function of the parameter $ \om \in \Om $  
 by setting $ \xi = (\om^{nls})^{-1} (\om) $.  
Lemma  \ref{zeta = 0} shows that any invariant torus 
for $ X_{H_{\e, \zeta}} $ is actually invariant  for $ X_{H_\e} $. 
The variable $ \zeta $ will allow us to control the average of the 
$ y$-component of approximations of the linearized Hamiltonian vector fields, 
adding in this way flexibility for choosing such approximations.

\smallskip

 We look for zeros of the map
\begin{align}\label{definitionFepagain}
  F_\omega (  \io ,  \zeta )  & := \omega \cdot \partial_\vphi  \breve \io (\vphi) - X_{H_{\e, \zeta}}  ( \breve \io (\vphi)) =  
  \omega \cdot \partial_\vphi \breve \io (\vphi) - X_{H_\e}  ( \breve \io (\vphi)) + (0, \zeta, 0)  
 \end{align}
 which when written componentwise reads
 \begin{align}
  F_\omega (  \io ,  \zeta )&  = \big(  \omega \cdot \partial_\vphi  \theta - \nabla _y H
      _\varepsilon , \  \omega \cdot \partial_\vphi  y + \nabla _\theta 
      H_\varepsilon + \zeta , \  \omega \cdot \partial_\vphi  z + \ii \nabla_{\bar z} H_\varepsilon    \big)\,. \label{operatorF}
 \end{align}
  In order to implement a convergent Nash-Moser scheme that leads to a solution of  
$ F_\omega (  \io , \zeta) = 0 $, the main task is to construct an \emph{approximate right inverse} 
 of the differential $d_{ \io, \zeta} F_\omega$, satisfying tame estimates -- 
see Theorem \ref{thm:stima inverso approssimato} in the subsequent section. Note that
the derivative of $ F_\omega (  \io ,  \zeta )$ in direction $(\widehat \imath \,, \widehat \zeta)$ is given by
\begin{equation}\label{operatore linearizzato}
d_{ \io, \zeta} F_\omega [\widehat \imath \,, \widehat \zeta ] =
\om \cdot \partial_\vphi \widehat \imath - \pa_\io X_{H_\e} (  \breve \io  (\vphi) ) [\widehat \imath ] + (0, \widehat \zeta, 0,0 )
 \,,  
\end{equation}
 which is independent of $ \zeta $. According to \cite{Z1}, an 
approximate right inverse of $d_{\io, \zeta} F_\omega$ is a map 
with the property that, when composed with $d_{\io, \zeta} F_\omega$,  it is equal to the identity up to an error of the size of $ F_\omega ( \io, \zeta )$. 
In particular, at a solution $ (  \io , \zeta, \om ) $  of $ F_\omega ( \io , \zeta) = 0 $, an approximate right inverse is an \emph{exact} one. 
For constructing an approximate right inverse, we implement the strategy developed in \cite{BB1}, \cite{BBM3}
which reduces the search of such an operator
to the one of an approximate right inverse of the part of $d_{\io, \zeta} F_\omega$, acting on the normal directions only -- see Theorem~\ref{invertibility of frak L omega}, which is proved in Section~\ref{sec:5}
and Section~\ref{sec:redu}. In these sections we also provide estimates for 
the variation of the quantities considered  with respect to the torus embedding $ \breve \io $.
This information is needed 
for the proof of the measure estimates of Section \ref{sec:measure} (Theorem~\ref{measure estimate}). 
The construction of solutions of $ F_\omega (  \io , \zeta) = 0 $
via a Nash-Moser iteration scheme and the proof of their linear stability is presented in Section~\ref{sec:NM} (Theorem~\ref{iterazione-non-lineare} and Corollary~\ref{corollario finale nash moser}).


\section{Approximate right inverse}\label{3. Set up}

The main result of this section is Theorem~\ref{thm:stima inverso approssimato}. 
Throughout the remainder of the paper, we always assume that 
$\breve \io \equiv \breve \io_\omega : \T^S \to M^\s\,, 
\vphi \mapsto \breve \io(\vphi)$ 
is a $ {\cal C}^\infty $ torus embedding of the form $(\vphi, 0, 0) + \io(\vphi)$ 
Lipschitz continuous in $\omega$  on 
a closed subset 
\be\label{definition Omega-o}
\Omega_o(\io) \subset  \Omega_{\gamma, \tau} \subset \Omega \, , 
\ee
where  $\Omega_{\gamma, \tau}$ is the set of diophantine frequencies introduced in \eqref{Omega o Omega gamma tau}.
Furthermore, we assume that $\io$ is small in the sense that
\be\label{size of torus}
\| \io \|_{s_0 + \mu_1}^\Lipg
 \lessdot  \,\e \g^{-2} \, , \quad \| E \|_{s_0 + \mu_1, \s-2}^\Lipg \lessdot  \,\e  \quad 
\mbox{with} \quad \e \gamma^{- 4} \ll 1 \, \mbox{ and } \, 0 < \g < 1 \, 
\ee
where $  E: \T^S \to \R^S \times \R^S \times h^{\s-2}_\bot $ is the 'error function' of $(\io, \zeta)$,
\be\label{error function}
 E (\vphi ) := (E_\theta(\vphi), E_y(\vphi), E_z(\vphi)) = F_\omega ( \io , \zeta ) (\vphi)\, .
\ee
 It will be verified in Section~\ref{sec:NM} that the smallness assumptions \eqref{size of torus} 
hold along the Nash-Moser iteration scheme. 
In all of Section~\ref{3. Set up}, if not stated otherwise, the Lipschitz estimates are computed on $\Omega_o(\io)$.
Furthermore, in the estimates in the subsequent subsections,  the Sobolev exponent $s$ will 
be an arbitrary integer satisfying
$$
s_0 \le s \le s_* - \mu_1  \, , \qquad  s_0 = [S/2] + 1\,.
$$
Here,  
$ \mu_1 \equiv \mu_1(|S|, \tau) \in \Z_{\ge 1} $
is assumed to be sufficiently large so that it is bigger than various integers 
$\mu \equiv \mu(|S|, \tau)$, coming up in the lemmas below,
and so that the tame estimates of Subsection~\ref{sec:2} such as the ones of 
Lemma~\ref{composition in Sobolev} apply in the situations considered. 


\subsection{Formula for $\zeta$}

For any given torus embedding the vector  $\zeta $ 
and the error function $E  $ defined in \eqref{error function} are related: 
\begin{lemma}  \label{zeta = 0}
For any torus embedding  $\breve \io  \equiv \breve \io_\omega$, 
we have 
\be\label{formula:zeta}
\zeta = 
\frac{1}{(2 \pi)^S} \int_{\T^S}  
\Big(- (\pa_\vphi \theta (\vphi))^t \cdot E_y  + (\partial_\vphi y)^t \cdot E_\theta - \ii (\partial_\vphi z)^t \cdot \overline E_{ z} + \ii (\partial_\vphi \bar z)^t \cdot E_z \, \Big) d \vphi\,.
\ee
Hence $ \zeta $ is Lipschitz continuous in $\omega \in \Omega_o (\io) $ and satisfies the estimate 
$$
 |\zeta |^{\Lipg} \lessdot  \| E \|_{s_0, \s-2}^{\Lipg} \, .
 $$ 
As a consequence, for any $(\io, \zeta)$ with $ F_\omega ( \io , \zeta ) = 0 $ one has $  \zeta = 0  $, 
and the torus $ \breve \io(\T^S)  $ is invariant  for the Hamiltonian vector field $ X_{H_\e}$.
 \end{lemma}

\begin{proof}
We follow the arguments in \cite{BB1}. Since $H_\e$ is an autonomous Hamiltonian 
one verifies by a straightforward change of variables that the function 
$$
{G} : \T^S \to \C\,, \quad \psi \mapsto {G}(\psi) := 
\int_{\T^S} \Big( - \lambda_{\breve \io^{(\psi)}}(\omega \cdot \partial_\vphi \breve \io^{(\psi)})  - H_{\e}(\breve \io^{(\psi)})\Big)\, d \vphi
$$
is constant, where $\breve \io^{(\psi)} (\vphi) := \breve \io(\psi + \vphi)$ and 
$\lambda_{\breve \io(\psi +\vphi)}$ is the canonical one form $\lambda$ defined in \eqref{contact 1 form} evaluated at
$\breve \io(\psi +\vphi)$. 
Note that  
$- \lambda_{\breve \io}(\omega \cdot \partial_\vphi \breve \io)  - H_{\e}(\breve \io)$
is the Lagrangian associated to $H_{\e}$.
Using that $\partial_\psi { G}(0) = 0$, a direct calculation proves \eqref{formula:zeta}.
By Lemma \ref{interpolation product} (tame estimates for products of maps), the fact that $E \in H^s(\T^S, \R^S \times \R^S \times h_\bot^{\sigma - 2})$ and the smallness assumption \eqref{size of torus},
 the claimed estimate follows.  
\end{proof}


\subsection{Isotropic torus embeddings}\label{Isotropic torus embeddings}

An invariant torus  $\breve \io(\T^S)$, densely filled by a quasi-periodic solution, 
is isotropic (cf e.g. Lemma 1 in \cite{BB1}). It means that the pullback of the symplectic form $\Lambda$ by $\breve \io$ vanishes, $\breve \io^* \Lambda= 0$. In our symplectic setup it is useful to work with isotropic torus embeddings. In Lemma \ref{toro isotropico modificato} below we provide a canonical construction for approximating a torus embedding $\breve \io$ by an isotropic one. By a straightforward computation one verifies that in our infinite dimensional setup
\begin{equation}\label{d Lambda lambda}
\breve \io^* \Lambda = d(\breve \io^* \lambda) 
\end{equation}
where $\breve \io^* \lambda$ is the pullback of the  {\it one-form} $\lambda$ defined by \eqref{contact 1 form}. Here $d$ denotes the exterior differential of the one-form $\breve \io^* \lambda$ on the torus $\T^S$. Our task is therefore to provide a canonical construction of approximating $\breve \io$ by an embedding $\breve \io_{\rm iso}$ so that $\breve \io_{\rm iso}^* \lambda$ is a closed one form. Any ${\cal C}^2$-smooth one-form 
$\alpha = \sum_{j \in S} a_j d \vphi_j$ on the torus $\T^S$ admits a Hodge decomposition 
$$
\alpha = \sum_{j \in S} [[a_j]] d \vphi_j  + d f + \rho\,, 
$$
where the constant one-form $\sum_{j \in S} [[a_j]] d \vphi_j$ is the harmonic part of $\alpha$ with 
$$
[[a_j]] :=\frac{1}{(2 \pi)^{|S|}} \int_{\T^S} a_j(\vphi)\, d \vphi\,,
$$
$d f$ is the exact one-form with $f : \T^S \to \C$ having average $0$ and $\rho := \sum_{j \in S} r_j d \vphi_j$ is a co-closed one-form, meaning that $r = (r_j)_{j \in S}$ satisfies ${\rm div}( r )= 0$. In the language of differential forms it means that $d^* \rho = 0$, where $d^*$ denotes the adjoint of $d$ with respect to the standard inner product. Using integration by parts, a standard computation yields $d^* \alpha = - {\rm div }(a)$ where $a = (a_j)_{j \in S}$. 
Since $d^* d f = d^* \alpha$ it then follows that 
$$
f = \Delta^{- 1}({\rm div}(a))\,, \qquad \Delta = \sum_{j \in S} \partial_{\vphi_j }^2\,.
$$
The expression $\Delta^{- 1}({\rm div} (a))$ is well defined as the average of ${\rm div} (a)$ vanishes. Similarly, since $d \rho = d \alpha = \sum_{k < j} A_{kj} d \vphi_k \wedge d \vphi_j$ with 
$A_{k j} := \partial_{\vphi_k} a_j - \partial_{\vphi_j} a_k$, one computes 
$d^* d \rho = \sum_{k \in S} \big( \sum_{j \in S} \partial_{\vphi_j} A_{kj} \big) d \vphi_k$, yielding 
\begin{equation}\label{formula divergence free part}
r_k = - \Delta^{- 1} \Big( \sum_{j \in S} \partial_{\vphi_j} A_{kj}\Big)\,, \quad \forall k \in S\,.
\end{equation}
In the situation at hand, the one-form $\sum_{j \in S} a_j d \vphi_j$ is given 
by the pullback $\breve \io^* \lambda$ of $\lambda$, 
\begin{equation}\label{coefficienti pull back di Lambda}
a = (a_j)_{j \in S} = - (\partial_\vphi \theta)^t y + \ii (\partial_\vphi \bar z)^t z
\end{equation}
and one has
\begin{equation}\label{closed form rho lambda}
d(\breve \io^* \lambda - \rho) = 0\,, \qquad \breve \io^* \lambda - \rho = \sum_{k \in S} (a_k - r_k)\, d \vphi_k
\end{equation}
where $r = (r_k)_{k \in S}$ is of the form \eqref{formula divergence free part}. In view of \eqref{formula divergence free part}, \eqref{coefficienti pull back di Lambda} define 
$\breve \io_{\rm iso}(\vphi) := (\vphi, 0, 0) + \io_{\rm iso}(\vphi)$ where
\begin{equation}\label{toro isotropico modificato A}
\io_{\rm iso}(\vphi) := (\theta(\vphi) - \vphi, y_{\rm iso}(\vphi), z(\vphi))\,, \qquad y_{\rm iso}(\vphi) := y(\vphi) + (\partial_\vphi \theta(\vphi))^{- t} r(\vphi)\,.
\end{equation}
We prove in Lemma \ref{toro isotropico modificato} that $\breve \io_{\rm iso}(\T^S) \subseteq M^\sigma$ is an isotropic torus. 
First we  estimate the coefficients $A_{kj}$, $k, j \in S$, in terms of the error function $E$. Denoting by $(\underline{e}_j)_{j \in S}$ the standard basis of $\R^S$,  one has
$$
A_{kj} \stackrel{\eqref{d Lambda lambda}}{=}
 \breve \io^* \Lambda [\underline{e}_k, \underline{e}_j] = \Lambda[\partial_{\vphi_k} \breve \io, \partial_{\vphi_j} \breve \io]  
$$
and hence 
$$
\omega \cdot \partial_\vphi A_{kj} = \Lambda[\partial_{\vphi_k} (\omega \cdot \partial_\vphi\breve \io), \partial_{\vphi_j} \breve \io]  + \Lambda[\partial_{\vphi_k} \breve \io, \partial_{\vphi_j}(\omega \cdot \partial_\vphi\breve \io)] \,.
$$
Recall that $\omega \cdot \partial_\vphi \breve \io = E + X_{H_\e} - (0, \zeta, 0) $ and hence
$\partial_{\vphi_k}\omega \cdot \partial_\vphi \breve \io 
=\partial_{\vphi_k} E + \partial_{\vphi_k} X_{H_\e} $. 
In view of the formula \eqref{symplectic form per esteso} for $\Lambda$ 
and since  the Hessian $d^2H_\e$ is symmetric one has
$$
\Lambda\big[ \pa_{\vphi_k} X_{H_\e} ,  \pa_{\ph_j} \breve \io\big] 
+ \Lambda \big[ \pa_{\ph_k} {\breve \io} , \pa_{\ph_j} X_{H_\e}  \big]
= d^2H_\e [\pa_{\ph_k} {\breve \io}, \pa_{\ph_j} {\breve \io}]
- d^2H_\e [\pa_{\ph_j} {\breve \io}, \pa_{\ph_k} {\breve \io}] =0
$$
implying  that
\begin{equation}\label{om d vphi A kj}
\omega \cdot \partial_\vphi A_{kj} = \Lambda\big[ \pa_{\vphi_k} E ,  \pa_{\ph_j} \breve \io\big] 
+ \Lambda \big[ \pa_{\ph_k} {\breve \io} , \pa_{\ph_j} E  \big].
\end{equation}
This formula allows to prove the following lemma. 
\begin{lemma} \label{tame estiamte Akj}
There exists $\mu \equiv \mu(|S|, \tau) \in \Z_{\ge 1}$ so that for any
 integer $s_0 \le s \le s_* - \mu$, the following tame estimate holds: 
$$
\sup_{k, j \in S} \| A_{k j} \|_s^{\Lipg} \leq_s \gamma^{-1} \big(\| E   \|_{s+2\t+2, \s-2}^{\Lipg} 
+ \| E  \|_{s_0+1, \s-2}^{\Lipg} \|  \io \|_{s+ 2 \tau + 2}^{\Lipg} \big)\,.
$$
\end{lemma}

\begin{proof}
In view of the formula \eqref{symplectic form per esteso} for $\Lambda$, the identity 
\eqref{om d vphi A kj} for $A_{kj}$, the estimate of Lemma \ref{om vphi - 1 lip gamma}
for the solution $A_{kj}$ of \eqref{om d vphi A kj},
the tame estimates for products of functions in $H^s(\T^S, \C)$ of Lemma~\ref{interpolation product},
the assumptions $\s \ge 4$, and the smallness condition \eqref{size of torus}, the claimed estimate follows. 
\end{proof}

\noindent
The main result of this section is the following lemma. 

\begin{lemma}\label{toro isotropico modificato} {\bf (Isotropic torus)} 
The torus embedding ${\breve \io}_{\rm iso}(\vphi) := ( \theta (\vphi), y_{\rm iso}(\vphi), z (\vphi) ) $, defined by \eqref{toro isotropico modificato A}, 
is isotropic, $\breve \io^* \Lambda = 0$. Expressed in coordinates, it means that 
\begin{equation}\label{identity isotropy}
(\partial_\vphi \theta )^t \partial_\vphi y_{\rm iso} - (\partial_\vphi y_{\rm iso})^t \partial_\vphi \theta + \ii (\partial_\vphi z)^t \partial_\vphi \bar z - \ii (\partial_\vphi \bar z)^t \partial_\vphi z = 0\,.
\end{equation}
Moreover there exist $\mu = \mu(|S|, \tau)\in \Z_{\geq 1}$ 
so that for any integer $s_0 \le s \le s_* - \mu$
\begin{align} \label{stima y - y delta}
\| y_{\rm iso} - y \|_s^{\Lipg} 
& \leq_s  \gamma^{-1} \big(\| E \|_{s + \mu, \s-2}^{\Lipg} + 
\| E \|_{s_0 + \mu, \s-2}^{\Lipg} \|  \io  \|_{s + \mu}^{\Lipg} \big) 
\\
\| \io_{\rm iso} \|_s^\Lipg & \leq_s \| \io \|_{s+\mu}^\Lipg \label{stima toro modificato 2}
\\
\label{stima toro modificato}
\| F_\omega (  \io_{\rm iso}, \zeta )  \|_{s, \s-2}^{\Lipg} 
& \leq_s  \gamma^{-1} \big( \| E  \|_{s + \mu, \s-2}^{\Lipg}  +  
\| E \|_{s_0 + \mu, \s-2}^{\Lipg} \|  \io \|_{s + \mu}^{\Lipg} \big)
\\
\label{derivata i delta}
\| d_\io (   \io_{\rm iso} ) [ \hat \imath ] \|_s & \leq_s \| \hat \imath \|_{s+\mu} +  \| \io \|_{s + \mu} \| \hat \imath  \|_{s_0+\mu} \,.
\end{align}
\end{lemma}

\begin{proof}
By \eqref{toro isotropico modificato} one sees that 
$
\breve \io^*_{\rm iso} \lambda = \sum_{j \in S} a_j^{\rm iso}(\vphi)\, d \vphi_j
$
is given by 
$$
a_{\rm iso} = (a_j^{\rm iso})_{j \in S} = - (\partial_\vphi \theta)^t y_{\rm iso} + \ii (\partial_\vphi \bar z)^t z  = - (\partial_\vphi \theta)^t y - r + \ii (\partial_\vphi \bar z)^t z = a - r\,.
$$
Hence $\breve \io^*_{\rm iso} \Lambda \stackrel{\eqref{d Lambda lambda}}{=} d (\breve \io^*_{\rm iso} \lambda) \stackrel{\eqref{closed form rho lambda}}{=} 0$. As a consequence $\Lambda[\partial_{\vphi_k} \breve \io_{\rm iso}, \partial_{\vphi_j} \breve \io_{\rm iso}] = 0$ for any $k, j \in S$. By the formula \eqref{symplectic form per esteso} for $\Lambda$, the claimed identity \eqref{identity isotropy} follows.
The estimate \eqref{stima y - y delta} follows from the definition of 
$y_{\rm iso}$ (cf \eqref{toro isotropico modificato A}), the one of 
$r$ (cf \eqref{formula divergence free part}), and Lemma~\ref{tame estiamte Akj}.
To obtain  \eqref{stima toro modificato 2}, one expresses $r$ in terms of $a$ 
(cf formula \eqref{coefficienti pull back di Lambda}) and uses the tame estimates of products
of Lemma~\ref{interpolation product}.
The estimate \eqref{stima toro modificato} is obtained   
by the mean value theorem, using the estimate of $y_{\rm iso} - y$ of \eqref{stima y - y delta}
and the estimates for $\partial_y X_{H_\e}$
(cf Proposition~\ref{stime derivate H nls} and Proposition~\ref{teorema stime perturbazione}), 
and \eqref{stima toro modificato 2}.
The remaining estimate \eqref{derivata i delta} is derived in a similar fashion.
\end{proof}

\subsection{Canonical coordinates near an isotropic torus}

In order to facilitate the search of an approximate inverse of the differential 
$d_{\io, \zeta} F_\omega( \io_{\rm iso}, \zeta)$ we introduce suitable coordinates 
$(\psi, \upsilon, w)$ near the isotropic torus $\breve \io_{\rm iso}(\T^S) \subseteq M^\sigma$,

\begin{equation} \label{2.23} 
 \Gamma :  \begin{pmatrix}
  \psi \\
  \upsilon \\
  w 
  \end{pmatrix}  \mapsto      
                \begin{pmatrix}  \theta (\psi ) \\ 
                y_{\rm iso}(\psi ) + Y(\psi , \upsilon , w
             ) \\ z(\psi ) + w 
                 \end{pmatrix}
   \end{equation}
where 
   \begin{equation}
   \label{2.23bis} 
   Y(\psi , \upsilon, w) := 
   (\partial _\psi  \theta)^{-t} (\psi ) \upsilon + Y_w(\psi )w + Y_{\bar w} (\psi ) \bar w 
   \end{equation}
and for any $\psi \in \T^S$, $Y_w(\psi)$ is the linear operator 
\begin{equation}\label{operator Yw}
Y_w (\psi) : h^\sigma_\bot \to \C^S \,, \quad w \mapsto \ii (\partial_\psi \theta)^{- t} (\partial_\psi \bar z)^t w\, ,
\quad Y_{\bar w} = \overline Y_w  \,  .
\end{equation}
By the definition \eqref{2.23} 
of the transformation $\Gamma$ one has  
\begin{equation}\label{def:trivial-torus}
\breve \io_{\rm iso} = \Gamma \circ \breve \io_0  \qquad {\rm where} \qquad 
\breve \io_0 : \T^S \to M^\sigma\,, \quad \vphi \mapsto (\vphi, 0, 0)\,,
\end{equation}
 i.e., in the new coordinates, $\breve \io_{\rm iso}$ is given by $\breve \io_0$. 
Furthermore, using \eqref{identity isotropy} (since $\breve \io_{\rm iso}(\T^S)$ is an isotropic torus)
 one verifies  that $\Gamma^* \Lambda = \Lambda$, i.e., $\Gamma$ is canonical, see also \cite{BB1}. 
For our purposes, it suffices to consider 
$d_\io(\Gamma \circ \breve \io)$ at $ \io = 0$, 
which we denote by $d\Gamma\circ \breve\io_0$.
Following the procedure described in Subsection~\ref{Hamiltonian setup},
we extend the bilinear map $d^2_\io (\Gamma \circ {\breve \io})$ to be defined for elements
$(\widehat \io^{(1)}, \widehat \io^{(2)})$ with
$\widehat \io^{(a)} := (\widehat \psi^{(a)}, \widehat \upsilon^{(a)}, \widehat{ w}_1^{(a)}, \widehat{ w}_2^{(a)} )$ in 
$H^s(\T^S, \R^S \times \R^S \times h^\sigma_\bot \times h^\sigma_\bot)$, $a = 1, 2$, and denote it by $d^2\Gamma\circ \breve\io_0$, when evaluated
at $ \io = 0.$

\begin{lemma} \label{lemma:DG}  
There exist $\mu = \mu(|S|, \tau) \in \Z_{\ge 1},$ so that for 
any $ \widehat \io := (\widehat \psi, \widehat \upsilon, \widehat w)$ in 
$H^s(\R^S \times \R^S \times h^{\sigma'}_\bot)$ with
$s_0 \le s \le s_* - \mu$ and $\s -2 \le \s' \le \s$, 
\begin{align} \label{DG delta1}
& \| \big( d \Gamma(\breve\io_0 (\vphi)) - {\rm Id} \big) [ \widehat \io] \|_{s, \s'}   
\leq_s \| \io \|_{s_0 + \mu} \| \widehat \io \|_{s, \s'} + \| \io \|_{s + \mu}  \| \widehat \io \|_{s_0, \s'}\,, \\
& \label{DG delta2}
\| \big( d \Gamma(\breve\io_0 (\vphi)) \big)^{-1} [\widehat \io] \|_{s, \s'} 
\leq_s \| \widehat \io \|_{s, \s'} + \| \io \|_{s + \mu}  \| \widehat \io \|_{s_0, \s'}\,.
\end{align}
Moreover, for any $\widehat \io^{(a)} := (\widehat \psi^{(a)}, \widehat \upsilon^{(a)}, \widehat w^{(a)}_1, \widehat{ w}^{(a)}_2) \in 
H^s(\T^S, \R^S \times \R^S \times h^\sigma_\bot \times h^\sigma_\bot)$, $a = 1, 2$, 
$$
\| d^2 \Gamma(\breve\io_0 (\vphi))[\widehat \imath^{(1)}, \widehat \imath^{(2)}] \|_s 
\leq_s  \| \widehat \io^{(1)}\|_s \| \widehat \io^{(2)} \|_{s_0} 
+ \|  \widehat \io^{(1)}\|_{s_0} \| \widehat \io^{(2)} \|_{s} 
+  \| \io  \|_{s + \mu} \|\widehat \io^{(1)}  \|_{s_0} \| \widehat \io^{(2)} \|_{s_0}\,.
$$
The same estimates hold if the norm $\| \ \|_s$ is replaced by $\| \ \|_s^{\Lipg}$.
\end{lemma}

\begin{proof}
The estimate \eqref{DG delta1} 
is obtained  from the formula of the differential of 
$\Gamma \circ \breve \io$  with respect to $ \io$ at $ \io = 0$ and 
the tame estimates for products of maps of Lemma \ref{interpolation product}. 
As mentioned at the beginning of this section, we choose $\mu_0$ larger than $\mu$. Hence
by the smallness condition \eqref{size of torus}, 
the estimate of $\big(d \Gamma(\vphi,0,0) - {\rm Id} \big)[\widehat \io]$  for $s = s_0$  yields
$$
 \| \big( d \Gamma(\breve\io_0 (\vphi)) - {\rm Id} \big)[ \widehat \io]\|_{s_0} 
\lessdot   \e\gamma^{-2} \| \widehat \io \|_{s_0}\,.
$$
Since $\e \gamma^{- 2}$ is assumed to be sufficiently small, it follows that for any 
$\vphi \in \T^S,$ the operator $d \Gamma(\breve\io_0 (\vphi))$ on 
$\R^S \times \R^S \times h^\sigma_\bot$  is invertible by Neumann series.
One then verifies in a straightforward way that 
$\|\big( d \Gamma (\breve\io_0 (\vphi)) \big)^{-1} [\widehat \io] \|_s $
satisfies the bound, stated in \eqref{DG delta2}. 
The claimed bound for 
$\| d^2 \Gamma(\breve\io_0 (\vphi))[\widehat \imath^{(1)}, \widehat \imath^{(2)}] \|_s $ 
is obtained from the formula of the second derivative of $\Gamma \circ \breve \io$
and the tame estimates for products of maps, stated in Lemma \ref{interpolation product}. 
The stated estimates of the $\Lipg$-norms of the expressions considered can be derived 
by similar arguments. 
\end{proof}

Denote by $K_{\e, \zeta}$ the Hamiltonian $H_{\e, \zeta}$, expressed in the new coordinates, 
\be\label{def:Kep}
K_{\e, \zeta} := H_{\e, \zeta} \circ \Gamma =  
H_\e \circ \Gamma + \zeta \cdot \theta(\psi) \,, \quad K_\e := H_\e \circ \Gamma \, .
\ee
The corresponding Hamiltonian vector field is then given by 
\begin{equation}\label{Hamiltonian vector field K}
X_{K_{\e, \zeta}} := 
(\nabla_\upsilon K_\e, \,\,  
- \nabla_\psi K_\e - (\partial_\psi \theta)^t \zeta, \,\,
- \ii \nabla_{\bar w} K_\e)\,.
\end{equation}
Furthermore, since $\breve \io_{\rm iso}(\vphi) = \Gamma( \breve \io_0(\vphi))$, the directional derivative $\omega \cdot \partial_\vphi \breve \io_{\rm iso}(\vphi)$ equals  
$d \Gamma(\breve \io_0(\vphi))[(\omega, 0, 0)]$. 
Using the transformation law of vector fields one concludes that 
$$
F_\omega( \io_{\rm iso}, \zeta)(\vphi) = \omega \cdot \partial_\vphi \breve \io_{\rm iso}(\vphi) 
- X_{H_{\e, \zeta}} (\breve \io_{\rm iso}(\vphi)) = d \Gamma(\breve \io_0(\vphi)) [(\omega, 0, 0)] - d \Gamma(\breve \io_0(\vphi))X_{K_{\e, \zeta}} ( \breve \io_0(\vphi))\,,
$$
or 
\begin{equation}\label{X K Gamma su toro triviale}
X_{K_{\e, \zeta}}(\breve \io_0(\vphi)) = (\omega, 0, 0) - (d \Gamma(\breve \io_0(\vphi)))^{- 1} F_\omega (\breve \io_{\rm iso}, \zeta)(\vphi)\,.
\end{equation}
Note that if $\breve \io_{\rm iso}$ is a solution, i.e., $F_\omega(\breve \io_{\rm iso}, \zeta) = 0$, then by Lemma \ref{zeta = 0}, $\zeta = 0$ and hence by the formula above, 
$X_{K_{\e, 0}}(\breve \io_0(\vphi)) = (\omega, 0, 0)$. 
Comparing this with this formula \eqref{Hamiltonian vector field K} one gets in this case
$$
\nabla_\upsilon K_\e \circ \breve \io_0 (\vphi) = \omega\,, \quad \nabla_\psi K_\e \circ \breve \io_0(\vphi) = 0\,, \quad \nabla_w K_\e \circ \breve \io_0(\vphi) = 0\,.
$$
In the general case one has the following estimates: 
\begin{lemma} \label{coefficienti nuovi}
There exist $\mu = \mu(|S|, \tau)\in \Z_{\ge 1},$ 
so that for any integer $s_0 \le s \le s_* - \mu$
\begin{align*}
& \|  \nabla_\psi K_\e \circ \breve \io_0 \|_s^{\Lipg}\,,\,\,
\| \nabla_\upsilon K_\e \circ \breve \io_0 - \om  \|_s^{\Lipg}
\leq_s  \g^{-1} \big(\| E \|_{s + \mu, \s-2}^{\Lipg}  
+  \| E \|_{s_0 + \mu, \s-2}^{\Lipg} \| \io \|_{s + \mu}^{\Lipg}\big)\, , \\
&
\| \nabla_w K_\e \circ \breve \io_0 \|_{s, \s-2}^{\Lipg} \, ,\,\,
\| \nabla_{\bar w} K_\e \circ \breve \io_0 \|_{s, \s-2}^{\Lipg}
\leq_s  \g^{-1} \big(\| E \|_{s + \mu, \s-2}^{\Lipg}  
+  \| E \|_{s_0 + \mu, \s-2}^{\Lipg} \| \io \|_{s + \mu}^{\Lipg}\big) \,.
\end{align*}
\end{lemma}
\begin{proof}
The claimed estimates follow from the formula \eqref{X K Gamma su toro triviale} and the estimates \eqref{stima toro modificato}, \eqref{DG delta2}. 
\end{proof}


\subsection{Approximate right inverse of the differential of $F_\omega$}

By formula \eqref{definitionFepagain}, the differential $d_{\io, \zeta} F_\omega$ is independent of $\zeta$ and hence we write $d_{\io, \zeta} F_\omega(\io)$ for its value at $\io$. To get an approximate right inverse for the differential $d_{\io, \zeta} F_\omega$ at $( \io, \zeta)$, it suffices to construct an approximate inverse of the differential at $(\breve \io_{\rm iso}, \zeta)$. Indeed 
\begin{equation}\label{error term approximate inverse 1}
G_1[\widehat \io, \widehat \zeta]  :=  
d_{\io, \zeta} F_\omega(\io) [\widehat \io, \widehat \zeta] -
d_{\io, \zeta} F_\omega( \io_{\rm iso})[\widehat \io, \widehat \zeta] \stackrel{\eqref{operatore linearizzato}} =
- d_\io X_{H_\e} (  \breve \io  (\vphi) ) [\widehat \io ]
+ d_\io X_{H_\e} (  \breve \io_{\rm iso}  (\vphi) ) [\widehat \io ] 
\end{equation}
satisfies the following estimates:

\begin{lemma}\label{estimate G1}
There exist $\mu = \mu(|S|, \tau)\in \Z_{\ge 1},$ 
so that for any $ \widehat \io  := (\widehat \vphi, \widehat y, \widehat z_1, \widehat z_2)$ in 
$H^{s + \mu}(\T^S, \R^S \times \R^S \times h^{\sigma}_\bot \times h^{\sigma}_\bot)$ with $s_0 \le s \le s_* - \mu$
and any $\widehat \zeta \in \R^S$, which are both Lipschitz continuous in $\omega$
$$
\| G_1 [\widehat \io, \widehat \zeta ]\|_{s, \sigma - 2}^\Lipg \leq_s 
\gamma^{- 1} \Big( \| E\|_{s + \mu, \sigma - 2}^\Lipg  \| \widehat \io \|_{s_0 + \mu}^\Lipg + 
\| E\|_{s_0 + \mu, \sigma - 2}^\Lipg  \| \widehat \io \|_{s + \mu}^\Lipg  + 
\|  \io \|_{s + \mu}^\Lipg \| E \|_{s_0 + \mu, \sigma - 2}^\Lipg \| \widehat \io \|_{s_0 + \mu}^\Lipg \Big) \,.
$$
\end{lemma}
\begin{proof}

By the mean value theorem and the definition \eqref{toro isotropico modificato} of $\breve \io_{\rm iso}$, one has 
$$
G_1 = \int_0^1 ( y_{\rm iso} - y) \cdot \partial_y  \big(
d_\io X_{H_\e} (\breve \io + t ( \io_{\rm iso} - \io ) ) [\widehat \io ] \big)
\, d t\
=  \int_0^1 d^2_\io X_{H_\e} (\breve \io + t ( \io_{\rm iso} - \io) ) 
[\widehat \io, \widehat \io^{(1)} ]  \, d t
$$
where $\widehat \io^{(1)} = (0, y_{\rm iso} - y, 0, 0).$
The claimed estimate then follows from the tame estimate of $y_{\rm iso} - y$ of \eqref{stima y - y delta}
and the tame estimate for
$d^2 X_{H_\e}  \circ \breve \io \, [\widehat \io, \widehat \io^{(1)}]$,
obtained from Lemma~\ref {quadraticPart X H nls} and Lemma~\ref{quadraticPart X P}).
\end{proof}

\noindent
We consider torus embeddings of the form $\Gamma(\breve \io)$, where $\breve \io (\vphi ) : = (\psi(\vphi), y(\vphi), z (\vphi))$ and $\Gamma$ is the coordinate transformation, introduced in \eqref{2.23}. 
Since $\Gamma$ is symplectic 
$$
X_{H_{\e, \zeta}} \circ \Gamma = d \Gamma \circ X_{K_{\e, \zeta}}
$$
and one has 
$$
F_\omega (\Gamma (\breve \io) - \breve \io_0, \zeta) = d \Gamma(\breve \io) \big( \omega \cdot \partial_\vphi \breve \io - X_{K_{\e, \zeta}}(\breve \io, \zeta)  \big)\,.
$$
Denoting the differential of $F_\omega$ with respect to the two arguments temporarily by $d F_\omega$ one then gets by the chain and product rule for any 
$\widehat \io(\vphi) = 
(\widehat \psi (\vphi), \widehat \upsilon(\vphi), \widehat w(\vphi) , \overline{\widehat w} (\vphi))$ 
and $\widehat \zeta \in \R^S$
\begin{align}
d F_\omega (\Gamma (\breve \io) - \breve \io_0, \zeta) [d \Gamma (\breve \io) \widehat \io, \widehat \zeta] & = d_{ \io, \zeta} \big( F_\omega(\Gamma(\breve \io) - \breve \io_0, \zeta) \big)[\widehat \io, \widehat \zeta] \nonumber \\
& = d \Gamma(\breve \io) \big( \omega \cdot \partial_\vphi \widehat \io - d_{\io, \zeta} X_{K_{\e, \zeta}}(\breve \io) [\widehat \imath, \widehat \zeta]\big)   + d^2  \Gamma (\breve \io) \big[d \Gamma(\breve \io)^{- 1} \big(F_\omega (\Gamma(\breve \io) - \breve \io_0, \zeta) \big), \widehat \imath \big] \, . \nonumber
\end{align}
Now we evaluate the above expression at $\breve \io = \breve \io_0$
and $\widehat \io$ given by $d \Gamma (\breve \io)^{-1}\widehat \io$. 
Recalling that $\Gamma(\breve \io_0) = \breve \io_{\rm iso}$ we get 
\begin{align}
d_{\io, \zeta} F_\omega ( \io_{\rm iso})[\widehat \imath, \widehat \zeta] & = d \Gamma(\breve \io_0) 
\big( \omega \cdot \partial_\vphi  - d_{\io, \zeta} X_{K_{\e, \zeta}}(\breve \io_0)\big) [d \Gamma(\breve \io_0)^{- 1}[\widehat \imath], \widehat \zeta] + G_2 [\widehat \io, \widehat \zeta]\,, 
\label{splitting d F d X K}
\end{align}
where 
\begin{equation}\label{error term approximate inverse 2}
G_2 [\widehat \io, \widehat \zeta] :=  d^2  \Gamma (\breve \io_0) [d \Gamma(\breve \io_0)^{- 1}[\,\, F_\omega ( \io_{\rm iso})], \, \, d \Gamma(\breve \io_0)^{- 1}[\widehat \imath] \, ]\,.
\end{equation}
Note that $G_2 [\widehat \io, \widehat \zeta]$ is independent of $\widehat \zeta$.
It can be estimated as follows:
\begin{lemma}\label{stima G2}
There exists $\mu = \mu(|S|, \tau) \in \Z_{\ge 1},$ 
so that for any $ \widehat \io  := (\widehat \vphi, \widehat y, \widehat z_1, \widehat z_2)$ in 
$H^{s + \mu}(\T^S, \R^S \times \R^S \times h^{\sigma}_\bot \times h^{\sigma}_\bot)$ with $s_0 \le s \le s_* - \mu$ and any $\widehat \zeta \in \R^S$,
 which are both Lipschitz continuous in $\omega$,
$$
\| G_2 [\widehat \io, \widehat \zeta] \|_{s, \sigma - 2}^\Lipg \leq_s 
\g^{-1} \Big( \| E\|_{s + \mu, \sigma - 2}^\Lipg  \| \widehat \io \|_{s_0 + \mu}^\Lipg 
+ \| E\|_{s_0 + \mu, \sigma - 2}^\Lipg  \| \widehat \io \|_{s + \mu}^\Lipg
+ \|  \io \|_{s + \mu}^\Lipg \| E \|_{s_0 + \mu, \sigma - 2}^\Lipg \| \widehat \io\|_{s_0 + \mu}^\Lipg \Big) \, .
$$
\end{lemma}
\begin{proof}
The claimed estimate follows by the estimates of Lemma \ref{lemma:DG} and \eqref{stima toro modificato}.
\end{proof}

In view of the formula \eqref{splitting d F d X K} and Lemma \ref{stima G2}, the problem of finding an approximate right inverse of $d F_\omega(\breve \io_{\rm iso}, \zeta)$ is reduced to find an approximate right inverse of the operator 
$\omega \cdot \partial_\vphi - d_{\io, \zeta} X_{K_{\e, \zeta}}(\breve \io_0, \zeta)$ where 
 $X_{K_{\e, \zeta}} $  is given in \eqref{Hamiltonian vector field K}.
In order to compute the differential of $X_{K_{\e, \zeta}}$ at $\breve \io_0(\vphi)= (\vphi, 0, 0)$, we compute the Taylor expansion of $K_{\e, \zeta}$ in $\upsilon,$ $w,$ $\bar w$  
at $(\upsilon, w) = ( 0, 0)$ up to order $2$. 
Denoting $(w, \bar w) \in h^\sigma_\bot \times h^\sigma_\bot$ by $W$, 
the expansion is given by
$$
\theta(\psi ) \cdot \zeta  + K_{0,0}(\psi) + K_{1,0}(\psi) \cdot \upsilon + K_{0,1}(\psi) \cdot W + \frac12 \upsilon \cdot K_{2, 0}(\psi) \upsilon + \upsilon \cdot K_{1,1}(\psi) W+ \frac{1}{2}\,  W \cdot K_{0 ,2}(\psi) W  
$$
where 
\begin{align}\label{taylor coefficients A}
& K_{0,0}(\psi) : = K_\e(\psi, 0, 0)\,, \quad K_{1,0}(\psi) := \nabla_\upsilon K_\e (\psi, 0, 0)\,,
\quad K_{2, 0}(\psi) : = \partial_{\upsilon} \nabla_\upsilon K_\e(\psi, 0, 0)\,, \\
& \label{taylor coefficients B}
K_{0, 1}(\psi) := \nabla_W K_\e(\psi, 0, 0) =  
\big(\nabla_w K_\e(\psi, 0, 0), \nabla_{\bar w} K_\e(\psi, 0, 0) \big)\,, \quad K_{1, 1}(\psi) := \partial_W \nabla_\upsilon K_\e(\psi, 0, 0)\,, \quad 
\end{align}
and 
$$
K_{0, 2}(\psi) : = \partial_W \nabla_W K_\e(\psi, 0, 0) = \begin{pmatrix}
\partial_w\nabla_w K_\e(\psi, 0, 0) & \partial_{\bar w}\nabla_{w} K_\e(\psi, 0, 0) \\
\partial_{ w}\nabla_{\bar w} K_\e(\psi, 0, 0) & \partial_{\bar w}\nabla_{\bar w} K_\e(\psi, 0, 0)
\end{pmatrix}\,.
$$
With $ {\mathbb J}_2 $ given by  \eqref{notazione J bot}, the differential of the map 
$(\breve \io, \zeta) \mapsto \omega \cdot \partial_\vphi \breve \io -  X_{K_{\e, \zeta}} (\breve \io)$ 
at $\breve \io_0$ in direction  $(\widehat \io, \widehat \zeta)$ reads as 
\begin{align*}
\begin{pmatrix}
\omega \cdot \partial_\vphi \widehat \psi  -  \partial_\vphi K_{1,0}(\vphi)[\widehat \psi] - K_{2, 0}(\vphi) [\widehat \upsilon] - K_{1,1}(\vphi) [\widehat W]  \\
\omega \cdot \partial_\vphi \widehat \upsilon + (\partial_\vphi \theta(\vphi))^t[\widehat \zeta] + \partial_\vphi ((\partial_\vphi \theta(\vphi))^t\zeta)[\widehat \psi] + \partial_\vphi \nabla_\vphi K_{0, 0}(\vphi) [\widehat \psi]  + 
\nabla_\vphi \big( K_{1, 0}(\vphi) \cdot \widehat \upsilon + K_{0, 1}(\vphi) \cdot \widehat W \big) \\ 
\omega \cdot \partial_\vphi \widehat W + {\mathbb J}_2 \big( \partial_\vphi K_{0, 1}(\vphi)[\widehat \psi]  + K_{1, 1}(\vphi)^t [ \widehat \upsilon] + K_{0 ,2}(\vphi) [\widehat W] \big)
\end{pmatrix}
\end{align*}
where $\widehat \io(\vphi) = 
(\widehat \psi (\vphi), \widehat \upsilon(\vphi), \widehat W(\vphi))$ with
$\widehat W(\vphi) = (\widehat w_1 (\vphi), \widehat w_2 (\vphi))$ in 
$h^{\sigma}_\bot \times h^{\sigma}_\bot$. 
In the above expression, various terms can be estimated in terms of the error function $ E $ introduced in \eqref{error function}. Indeed, since  
\begin{equation}\label{Taylor coefficients K Errors}
\nabla_\vphi K_{0, 0}(\vphi) = \nabla_\psi K_\e(\breve \io_0(\vphi))\,, \quad K_{1, 0}(\vphi) 
= \nabla_\upsilon K_\e(\breve \io_0(\vphi))\,, \quad K_{0, 1}(\vphi) 
= (\nabla_w K_\e(\breve \io_0(\vphi)), \nabla_{\bar w} K_\e (\breve \io_0(\vphi)))\,,
\end{equation}
it follows from Lemma \ref{coefficienti nuovi} and \ref{zeta = 0} that  the operator $\omega \cdot \partial_\vphi - d_{\io, \zeta} X_{K_{\e, \zeta}}(\breve \io_0)$ is of the form
\begin{equation}\label{error term approximate inverse 3}
\omega \cdot \partial_\vphi - d_{\io, \zeta} X_{K_{\e, \zeta}}(\breve \io_0) = {\frak T}_\omega + G_3\,,
\end{equation}
where 
$$
\frak T_\om[\widehat \io, \widehat \zeta] := \begin{pmatrix}
\omega \cdot \partial_\vphi \widehat \psi  - K_{2, 0}(\vphi) [\widehat \upsilon] - K_{1,1}(\vphi) [\widehat W]  \\
\omega \cdot \partial_\vphi \widehat \upsilon + (\partial_\vphi \theta(\vphi))^t[\widehat \zeta]  \\ 
\omega \cdot \partial_\vphi \widehat W + {\mathbb J}_2 \big(  K_{1, 1}(\vphi)^t [ \widehat \upsilon] + K_{0 ,2}(\vphi) [\widehat W] \big)
\end{pmatrix}
$$
and
$$
G_3 [\widehat \io, \widehat \zeta] := \begin{pmatrix}
  -  \partial_\vphi K_{1,0}(\vphi)[\widehat \psi]  \\
  \partial_\vphi \big( (\partial_\vphi \theta(\vphi))^t\zeta \big)[\widehat \psi] + \partial_\vphi \nabla_\vphi K_{0, 0}(\vphi) [\widehat \psi]  + \nabla_\vphi \big( K_{1, 0}(\vphi) \cdot \widehat \upsilon + K_{0, 1}(\vphi) \cdot \widehat W \big) \\ 
 {\mathbb J}_2  \partial_\vphi K_{0, 1}(\vphi)[\widehat \psi] 
\end{pmatrix} \, .
$$
Note that $G_3[\widehat \io, \widehat \zeta]$ is independent of $\widehat \zeta$ and 
can be estimated as follows.
\begin{lemma}\label{stima G3}
There exist $\mu = \mu(|S|, \tau) \in \Z_{\ge 1},$ 
so that for any $ \widehat \io  := (\widehat \psi, \widehat \upsilon, \widehat W)$ in 
$H^{s + \mu}(\T^S, \R^S \times \R^S \times h^{\sigma}_\bot \times h^{\sigma}_\bot)$ 
with $s_0 \le s \le s_* - \mu$ and any $\widehat \zeta \in \R^S$,
 which are both Lipschitz continuous in $\omega$,
$$
\| G_3 [\widehat \io, \widehat \zeta] \|_{s, \sigma - 2}^\Lipg 
\leq_s \g^{-1} \Big( \| E\|_{s + \mu, \sigma - 2}^\Lipg  \| \widehat \io \|_{s_0 + \mu}^\Lipg + 
\| E\|_{s_0 + \mu, \sigma - 2}^\Lipg  \| \widehat \io \|_{s + \mu}^\Lipg  + 
\|  \io \|_{s + \mu}^\Lipg \| E \|_{s_0 + \mu, \sigma - 2}^\Lipg \| \widehat \io \|_{s_0 + \mu}^\Lipg \Big)
\,.
$$
\end{lemma}
\begin{proof}
In view of the formula \eqref{Taylor coefficients K Errors}, the claimed estimates follow from Lemma \ref{zeta = 0} and Lemma \ref{coefficienti nuovi}.
\end{proof}

Our aim is to construct a right inverse of ${\frak T}_\omega$. It means that for given maps 
$\vphi \mapsto (g_1(\vphi), g_2(\vphi) , g_3(\vphi) ) \in
\R^S \times \R^S \times (h^{\sigma -2}_\bot \times h^{\sigma -2}_\bot)$ 
of appropriate regularity,
we have to solve the inhomogenous linear system 
\begin{align}\label{first equation right inverse}
& \omega \cdot \partial_\vphi \widehat \psi  - K_{2, 0}(\vphi) [\widehat \upsilon] - K_{1,1}(\vphi) [\widehat W] = g_1\,, \\\
& \label{second equation right inverse}
\omega \cdot \partial_\vphi \widehat \upsilon + (\partial_\vphi \theta(\vphi))^t[\widehat \zeta] = g_2 \,, \\
& \label{third equation right inverse}
 {\frak L}_\omega \widehat W + \mathbb J_2 K_{1, 1}(\vphi)^t [ \widehat \upsilon]   = g_3\,,  
\end{align}
where for any $\omega \in \Omega_o(\io)$, the operator
$\frak L_\omega : H^s(\T^S, h^{\s}_\bot \times h^\sigma_\bot) 
\to H^{s - 1}( \T^S, h^{\s-2}_\bot \times h^{\sigma - 2}_\bot)$ is defined by
\begin{equation}\label{definition frak L omega sec 4}
 {\frak L}_\omega(\vphi) := \omega \cdot \partial_\vphi {\mathbb I}_2 + \mathbb J_2 K_{0, 2}(\vphi)\, ,\qquad
K_{0, 2} = 
\begin{pmatrix}
\partial_w\nabla_w K_\e & \partial_{\bar w}\nabla_{w} K_\e \\
\partial_{ w}\nabla_{\bar w} K_\e & \partial_{\bar w}\nabla_{\bar w} K_\e
\end{pmatrix} \circ \breve \io_0 \,.
\end{equation}
The maps  $g_1, g_2$ are assumed to be in $ H^{s + 2\tau +1}(\T^S, \R^S)$ and 
$g_3 \in H^{s + 2\tau +1}(\T^S, h^{\sigma - 2}_\bot \times h^{\sigma - 2}_\bot)$ 
with $s_0 \le s \le s_* - \nu$ and $\nu = \nu(|S|, \tau)$ being an integer, which can be explicitly computed. 

Note that the above inhomogeneous linear system is in triangular form: We first solve the second equation \eqref{second equation right inverse}. It turns out to be convenient
to write $\widehat \upsilon =\widehat \upsilon_1 + \widehat \upsilon_0$
with $[[ \widehat \upsilon_1 ]] = 0$ and $\widehat \upsilon_0 = [[ \widehat \upsilon ]]$
where we recall that for any given continuous map $f: \T^S \to X$ with values in a Banach space $X$,
$[[ f ]]$ denotes its average $(2\pi)^{-|S|}\int_{\T^S} f(\vphi) d\vphi$.
The second equation \eqref{second equation right inverse} is the solved for for $\widehat \zeta$ and 
$\widehat \upsilon_1$.
Next we solve the third equation \eqref{third equation right inverse} for $\widehat W$ and then finally solve the first equation \eqref{first equation right inverse} for $\widehat \psi$ and $\widehat \upsilon_0$. Let us first consider in detail the second equation. Recall that $\theta(\vphi ) = \vphi + \Theta(\vphi)$, where $\Theta(\cdot )$ is $2 \pi$-periodic in each component. Hence 
$$
[[ (\partial_\vphi \theta)^t ]] = {\rm Id}_S + [[ (\partial_\vphi \Theta)^t ]] = {\rm Id}_S
$$
and the solution of the second equation is given by 
\begin{equation}\label{soluzione seconda equazione}
\widehat \zeta := [[ g_2 ]]\,,\quad 
\widehat \upsilon_1  := (\omega \cdot \partial_\vphi)^{-1}
\big(g_2 - [[g_2]] - (\partial_\vphi \Theta(\vphi))^t[\widehat \zeta] \big)\,.
\end{equation}

\begin{lemma}\label{lemma soluzioni seconda equazione}
For any $g_2$ in $ H^{s + 2\tau +1}(\T^S, \R^S)$ with $s \ge s_0,$
$\widehat \upsilon_1$ and $\widehat \zeta$ of  
\eqref{soluzione seconda equazione} satisfy
\begin{equation}\label{soluzioni seconda equazione}
\| \widehat \upsilon_1 \|_s^\Lipg \lessdot 
\gamma^{- 1} \big( \| g_2\|_{s + 2 \tau + 1}^\Lipg   + \| \io \|_{s + 2 \tau + 2}^\Lipg\| g_2\|_{s_0}^\Lipg \big) \,, 
\qquad |\widehat \zeta|^\Lipg \lessdot \| g_2\|_{s_0}^\Lipg\,.
\end{equation}
\end{lemma}
\begin{proof}
The claimed estimate for $|\widehat \zeta|^\Lipg$ is straightforward. To prove the one for
$\| \widehat \upsilon_1 \|_s^\Lipg$, we apply  Lemma \ref{om vphi - 1 lip gamma}
to get the bound  $\|g_2 -[[ g_2 ]]\|_{s + 2 \tau + 1}^\Lipg + 
\| (\partial_\vphi \Theta(\vphi))^t[\widehat \zeta])\|_{s + 2 \tau + 1}^\Lipg$.
Since $\|g_2 -[[ g_2 ]]\|_{s + 2 \tau + 1}^\Lipg \le  \|g_2\|_{s + 2 \tau + 1}^\Lipg$ and
$\| (\partial_\vphi \Theta(\vphi))^t[\widehat \zeta])\|_{s + 2 \tau + 1}^\Lipg
\le \| \io \|_{s + 2 \tau + 2}^\Lipg| \widehat \zeta|^\Lipg $ one has
 $\| \widehat \upsilon_1 \|_s^\Lipg
\lessdot  \gamma^{- 1} \big( \| g_2\|_{s + 2 \tau + 1}^\Lipg   
+ \| \io \|_{s + 2 \tau + 2}^\Lipg\| g_2\|_{s_0}^\Lipg \big)$.
\end{proof}

We point out that
the average $\widehat \upsilon_0$ of $\widehat \upsilon$ will be determined 
by equation \eqref{first equation right inverse}, 
but temporarily, we will consider it as a free parameter.
Now we have to solve the equation
\begin{equation}\label{terza equazione B}
 {\frak L}_\omega \widehat W    = g_3 - \mathbb J_2 K_{1, 1}(\vphi)^t [ \widehat \upsilon] \, . 
\end{equation}
We summarize our results on the invertibility of ${\frak L}_\omega$ with the following theorem.

 \begin{theorem}[\bf Invertibility of $ \frak L_\omega$]\label{invertibility of frak L omega}
For any constant $C > 0$, there exist $0 < \delta_0(|S|, \tau, s_*, C) < 1$ and $\mu_0 = \mu_0(|S|, \tau) \in \Z_{\ge 1}$ 
so that for any $\io$ with
$$
\| \io \|_{s_0 + \mu_0}^\Lipg
 \leq  \, C\e \g^{-2} \, , \quad \| E \|_{s_0 + \mu_0, \s-2}^\Lipg \leq  \, C \e \,, \qquad \e \gamma^{- 4} \leq \delta_0,
$$
there exists a subset of $\Omega_o(\io)$, denoted by 
$\Omega^{2 \gamma}_{\rm Mel}(\io) \equiv \Omega^{2 \gamma}_{\rm Mel}(\io ; \Omega_o(\io) )$,
with the following properties:  for any 
$g \in H^{s + 2 \tau + 1}(\T^S, h^{\sigma - 2}_\bot \times h^{\sigma - 2}_\bot)$ 
with $s_0 \le s \le s_* - \mu_0$ and any $\omega \in \Omega^{2 \gamma}_{\rm Mel}(\io)$, 
the linear equation ${\frak L}_\omega h = g$ has a unique solution 
$h = {\frak L}_\omega^{- 1} g\in H^s(\T^S, h^\sigma_\bot \times h^{\sigma}_\bot)$. 
In case $g$ is Lipschitz continuous on $\Omega^{2 \gamma}_{\rm Mel}(\io)$, 
the solution $h$ is Lipschitz continuous on $\Omega^{2 \gamma}_{\rm Mel}(\io)$
and satisfies the estimate
 \begin{equation}\label{tame inverse}
 \|{\frak L}_\omega^{- 1} g \|_{s, \sigma}^\Lipg \leq_s \gamma^{- 1} \Big(\| g \|_{s + 2 \tau + 1, \sigma - 2}^\Lipg +  \|\io \|_{s + \mu_0}^\Lipg \| g \|_{s_0 +2 \tau + 1, \sigma - 2}^\Lipg \Big)\,.
 \end{equation}
 \end{theorem}  

\noindent {\em Remark:} According to \eqref{value mu0}, a possible choice of $\mu_0$ 
in Theorem \ref{invertibility of frak L omega} is $\mu_0 = 4s_0 +10 \tau + 7.$

\smallskip
 
 Theorem~\ref{invertibility of frak L omega} is proved in Section~\ref{proof of invertibility of frak L omega}, using the results established in Sections~\ref{sec:5} and \ref{sec:redu}.  
In the sequel, the integers $\mu = \mu(|S|, \tau) \in \Z_{\geq 1}$
coming up in lemmas, where Theorem~\ref{invertibility of frak L omega} is applied, 
will be chosen larger than the corresponding integer $\mu_0,$ of Theorem~\ref{invertibility of frak L omega}.

In order to apply Theorem \ref{invertibility of frak L omega} to solve the equation \eqref{terza equazione B} 
we need the following estimate for the Taylor coefficients $K_{2, 0}$ and $K_{1, 1}$ defined in \eqref{taylor coefficients A}, \eqref{taylor coefficients B}:

\begin{lemma} \label{lemma:Kapponi vari}
There exist $\mu = \mu(|S|, \tau) \in \Z_{\geq 1}$
so that for any $ \widehat \upsilon \in H^s(\T^S, \R^S)$, 
$\widehat W = (\widehat w_1, \widehat w_2) \in H^s(\T^S, h^{\sigma}_\bot \times h^{\sigma}_\bot)$
with $s_0 \le s \le s_* - \mu$, which are both Lipschitz continuous in $\omega$,
\begin{align*}
& \| K_{2, 0} - (\partial_{I_j} \omega_k^{nls}(\xi, 0))_{k, j \in S}\|_s^\Lipg \leq_s \e + \| \io  \|_{s + \mu}^\Lipg  \,, \\ 
& \| (K_{1, 1})^t[\widehat \upsilon]  \|_s^\Lipg \leq_s \e \gamma^{- 2} \| \widehat \upsilon \|_s^\Lipg +   \| \io \|_{s+\mu}^\Lipg \| \widehat \upsilon \|_{s_0}^\Lipg\,, \\
& 
\| K_{1, 1} [\widehat W]  \|_s^\Lipg \leq_s \e \gamma^{- 2} \| \widehat W \|_s^\Lipg +   \| \io \|_{s+\mu}^\Lipg \| \widehat W \|_{s_0}^\Lipg\,.
\end{align*}
\end{lemma}

\begin{proof}
By \eqref{2.23} - \eqref{2.23bis},
$\partial_{\upsilon} K_\e
= \partial_y H_\e \circ \Gamma \cdot (\pa_\psi  \theta (\psi))^{-t}$ or
$
\nabla_\upsilon K_\e  = (\pa_\psi  \theta (\psi))^{-1} \nabla_y H_\e \circ \Gamma.
$ 
Hence
\begin{align*}
\partial_\upsilon \nabla_\upsilon K_\e(\breve \io(\vphi)) & = (\pa_\ph  \theta(\vphi))^{-1}  \partial_{y} \nabla_y H_\e( \breve \io_{\rm iso}(\vphi)) 
(\pa_\ph  \theta (\vphi))^{-t}  \\ 
& \stackrel{\eqref{HamiltonianHep}}{=} 
(\pa_\ph \theta(\vphi))^{-1} \partial_{y} \nabla_y H^{nls} ( \breve \io_{\rm iso}(\vphi)) 
(\pa_\ph  \theta (\vphi))^{-t} + 
\e (\pa_\ph  \theta(\vphi))^{-1} \partial_{y} \nabla_y P( \breve \io_{\rm iso}(\vphi)) 
(\pa_\ph  \theta (\vphi))^{-t} \,.
\end{align*}
We claim that the first term in the latter expression can be bounded
by $ C (s) \| \io \|_{s+ \mu}^\Lipg $  and the second one by 
$ \e  C(s) \big(1 + \| \io \|_{s+\mu} ^\Lipg \big) $.
Indeed, the estimate of the first term is derived from Proposition~\ref {stime derivate H nls} ($ii$),
$$
 \| \partial_y \nabla_y H^{nls}(\breve \io_{\rm iso}) - \partial_y \nabla_y H^{nls}(\xi , 0) \|_s^\Lipg 
\leq_s \| \io_{\rm iso}\|_{s + 2s_0}^\Lipg\, ,
$$
using that $\pa_\ph \theta(\vphi) = {\rm Id}_{\R^S} + \pa_\ph \Theta(\vphi)$ with $\|\pa_\ph \Theta(\vphi) \|_s^\Lipg \lessdot \|\io \|_{s+1}^\Lipg$,
$\partial_y \nabla_y H^{nls}(\xi , 0) =  (\partial_{I_j} \omega_k^{nls}(\xi, 0))_{k, j \in S}$, and
$\| \io_{\rm iso} \|_s^\Lipg  \leq_s \| \io \|_{s+\mu}^\Lipg$ by \eqref{stima toro modificato 2} .
To estimate the second term, one argues in a similar way, using this time that by
Proposition ~\ref{teorema stime perturbazione}, 
$\| \partial_y \nabla_y P ( \breve \io_{\rm iso}) \|_s^\Lipg\, 
\leq_s 1 + \| \io_{\rm iso}\|_{s + 2 s_0}^\Lipg$.
The  claimed estimates for $K_{1,1}[\widehat \upsilon]$ and $(K_{1, 1})^t[\widehat W]$ 
can be proved by similar arguments. 
\end{proof}

\noindent
Combining Theorem \ref{invertibility of frak L omega} and Lemma \ref{lemma:Kapponi vari},
we get the following estimate for the solution $\widehat W$ of equation \eqref{terza equazione B}.

\begin{corollary}\label{corollario expression widehat W}
There exist $\mu = \mu(|S|, \tau) \in \Z_{\ge 1}$
so that for any $g_3 \in H^{s+2\tau + 1}(\T^S, h^{\sigma -2}_\bot \times h^{\sigma-2}_\bot)$ and 
$ \widehat \upsilon \in H^{s+ 2 \tau +1}(\T^S, \R^S)$ 
with $s_0 \le s \le s_* - \mu$, which are both Lipschitz continuous in $\omega$
on $\Omega^{2 \gamma}_{\rm Mel}(\io)$,
 the solution 
\begin{equation}\label{expression widehat W}
\widehat W = {\frak L}_\omega^{- 1}(\vphi) \big( g_3 - \mathbb J_2 K_{1, 1}(\vphi)^t [ \widehat \upsilon]    \big) 
\end{equation}
of equation \eqref{terza equazione B} is Lipschitz continuous on $\Omega^{2 \gamma}_{\rm Mel}(\io)$ and
satisfies the estimate 
\begin{equation}\label{estimate widehat W}
\| \widehat W\|_s^\Lipg \leq_s   \gamma^{- 1} \Big(\| g_3 \|_{s + 2 \tau + 1, \sigma - 2}^\Lipg + \e \gamma^{- 2}\| \widehat \upsilon\|^\Lipg_{s + 2 \tau + 1}  +  \|\io \|_{s + \mu}^\Lipg \| g_3 \|_{s_0 + 2 \tau + 1, \sigma - 2}^\Lipg + \e \gamma^{- 2}\| \widehat \upsilon\|^\Lipg_{s_0 + 2 \tau + 1}  \Big) \, .
\end{equation}
\end{corollary}

Finally we solve the first equation \eqref{first equation right inverse} for 
$\omega \in \Omega^{2 \gamma}_{\rm Mel}(\io)$,
\begin{equation}\label{first equation right inverse B}
\omega \cdot \partial_\vphi \widehat \psi
= g_1 +   K_{1,1}(\vphi) [\widehat W] + K_{2, 0}(\vphi) [\widehat \upsilon] 
\end{equation}
where $\widehat W \in H^{s}(\T^S, h^{\sigma}_\bot \times h^{\sigma}_\bot)$ is given
by \eqref{expression widehat W} and
 $\widehat \upsilon$ is of the form $\widehat \upsilon_1 + \widehat \upsilon_0$
with $\widehat \upsilon_1 \in H^{s}(\T^S, \R^S)$ defined by  \eqref{soluzione seconda equazione}.
The first task for solving this equation is to prove that we can choose 
$\widehat \upsilon_0$ in such a way that the average of the right hand side of the above equation vanishes. By \eqref{expression widehat W}, the equation \eqref{first equation right inverse B} 
can be written as 
\begin{equation}\label{first equation right inverse C}
\omega \cdot \partial_\vphi \widehat \psi = g_1 + K_{1, 1}(\vphi) {\frak L}_\omega^{- 1}(\vphi) g_3 + M_\omega (\vphi)\widehat \upsilon
\end{equation}
where 
$$
M_\omega (\vphi) := K_{2, 0}(\vphi) - K_{1, 1}(\vphi) {\frak L}_\omega^{- 1}(\vphi) \mathbb J_2 K_{1, 1}(\vphi)^t \,.
$$
Taking the average in \eqref{first equation right inverse C} and using that 
$\widehat \upsilon = \widehat \upsilon_1 + \widehat \upsilon_0$, we get 
\begin{equation}\label{averaging first equation}
0 =  [[ g_1 ]] + [[ K_{1, 1} \mathbb J_2 {\frak L}_\omega^{- 1} g_3 ]] +
 [[ M_\o \widehat \upsilon_1 ]] +  [[ M_\o ]] \widehat \upsilon_0\,.
\end{equation}
In order to solve this latter equation for $\widehat \upsilon_0$, we need to show that
$[[M_\o]]: \R^S \to \R^S$ is invertible. To this end, first note that for any $ x  \in \R^S,$
$\| \big( [[M_\o]] - (\partial_{I_j} \omega_k^{nls}(\xi, 0))_{k, j \in S} \big) x \|$ is bounded by
$$
{\rm sup}_{\vphi \in \T^S} \| K_{1, 1}(\vphi) {\frak L}_\omega^{- 1}(\vphi) \mathbb J_2 K_{1, 1}(\vphi)^t x\| +
{\rm sup}_{\vphi \in \T^S} \| \big( K_{2, 0}(\vphi) - (\partial_{I_j} \omega_k^{nls}(\xi, 0))_{k, j \in S} \big) x\|\,,
$$
yielding
$$
\| \big( [[M_\o]] - (\partial_{I_j} \omega_k^{nls}(\xi, 0))_{k, j \in S} \big) x \|
\le \| K_{1, 1} {\frak L}_\omega^{- 1} \mathbb J_2 K_{1, 1}^t x\|_{s_0}
+ \| \big( K_{2, 0} - (\partial_{I_j} \omega_k^{nls}(\xi, 0))_{k, j \in S} \big) x\|_{s_0}\,.
$$
It then follows from Lemma \ref{lemma:Kapponi vari}, the tame estimate \eqref{tame inverse} 
for the inverse ${\frak L}_\omega^{- 1}$, and the smallness condition \eqref{size of torus} that
$\| [[M_\o]] - (\partial_{I_j} \omega_k^{nls}(\xi, 0))_{k, j \in S}  \| \lessdot \e \g^{-2}\,.$
En passant we mention that by the same arguments, one sees that
\begin{equation}\label{bound M}
\| M_\o - (\partial_{I_j} \omega_k^{nls}(\xi, 0))_{k, j \in S}  \|_{s_0}^\Lipg \lessdot \e \g^{-2}\,.
\end{equation}
Since by assumption, the inverse of $(\partial_{I_j} \omega_k^{nls}(\xi(\omega)))_{j, k \in S}$ is   bounded uniformly on $\Omega$ and $\Omega^{2 \gamma}_{\rm Mel}(\io) \subset \Omega$, it follows from Lemma \ref{lemma:Kapponi vari} and the smallness assumption \eqref{size of torus} that the operator $[[ M_\o ]]$ is invertible with the norm of $[[ M_\o ]]^{- 1}$  uniformly bounded. In fact,
\begin{equation}\label{bound [[M]]}
\| [[ M_\o ]]^{- 1}\|^\Lipg \lessdot 1 \, .
\end{equation}
The operator $[[ M_\o ]]$ being invertible implies that for any $\omega$ in 
$\Omega^{2 \gamma}_{\rm Mel}(\io)$, equation \eqref{averaging first equation} 
can be solved for $\widehat \upsilon_0$, 
\begin{equation}\label{def average upsilon}
\widehat \upsilon_0 = - 
[[ M_\o ]]^{- 1} \Big( [[ g_1 ]] + [[ K_{1,1} {\frak L}_\omega^{- 1} g_3]] +  [[ M_\o \widehat \upsilon_1 ]]\Big)\,.
\end{equation}
As a consequence, equation \eqref{first equation right inverse B} can be solved for $\widehat \psi$, 
\begin{equation}\label{definition widehat psi}
\widehat \psi = (\omega \cdot \partial_\vphi)^{- 1} \Big(  g_1 + K_{1, 1}(\vphi) {\frak L}_\omega^{- 1}(\vphi) g_3 + M_\o(\vphi)\widehat \upsilon \Big)\,.
\end{equation}

\begin{lemma}\label{lemma estimate of psi hat average upsilon}
There exist $\mu = \mu(|S|, \tau) \in \Z_{\ge 1}$
so that for any map
$g = ( g_1, g_2, g_3)$ in 
$H^{s + 4 \tau +2}(\T^S, \R^S\times \R^S \times  h^{\sigma -2}_\bot \times h^{\sigma -2}_\bot)
$
with $s_0 \le s \le s_* - \mu$, and any $\o \in \Omega^{2 \gamma}_{\rm Mel}(\io)$
with $\Omega^{2 \gamma}_{\rm Mel}(\io) \equiv \Omega^{2 \gamma}_{\rm Mel}(\io ; \Omega_o(\io))$ 
as in Theorem \ref{invertibility of frak L omega}, 
$ \widehat \upsilon_0$, defined in \eqref{def average upsilon}, and $\widehat \psi$, defined in \eqref{definition widehat psi}, satisfy the estimates
\begin{align}\label{estimate of average upsilon}
& |\widehat \upsilon_0 |^\Lipg  \lessdot \gamma^{- 1}\| g \|_{s_0 + 2 \tau + 1, \sigma - 2}^\Lipg \\
& \label{estimate of psi hat average upsilon}
\| \widehat \psi\|_s^\Lipg \, \le_s  \gamma^{- 2}\| g \|_{s + 4 \tau + 2, \sigma - 2}^\Lipg + \gamma^{- 3} \| \io\|_{s + \mu}^\Lipg \| g \|_{s_0 + 4 \tau + 2, \sigma - 2}^\Lipg \, .
\end{align}
\end{lemma}

\begin{proof}
By the formula \eqref{def average upsilon} and the estimate \eqref{bound [[M]]}, 
\begin{align}
| \widehat \upsilon_0 |^\Lipg & \lessdot \| [[ M_\o(\vphi)\widehat \upsilon_1]] \|^\Lipg  + 
\|  [[ K_{1,1}(\vphi)  {\frak L}_\omega^{- 1}(\vphi) g_3 ]]\|^\Lipg  + 
|[[ g_1 ]] |^\Lipg \nonumber \\
& \lessdot \|  M_\o(\vphi)\widehat \upsilon_1  \|_{s_0}^\Lipg + \| K_{1,1}(\vphi)  {\frak L}_\omega^{- 1}(\vphi) g_3 \|_{s_0}^\Lipg + \| g_1 \|_{s_0}^\Lipg\,.
\nonumber 
\end{align}
Since by \eqref{bound M}
$$
\| M_\o \|_{s_0}^\Lipg \lessdot \| (\partial_{I_j} \omega_k^{nls}(\xi(\omega)))_{j, k \in S} \|^\Lipg + \e \gamma^{- 2} \stackrel{Prop\,\ref{Proposition 2.3}}{\lessdot} 1
$$
one gets by the estimate \eqref{soluzioni seconda equazione}
$$
\| M_\o(\vphi)\widehat \upsilon_1 \|_{s_0}^\Lipg \lessdot 
\gamma^{- 1} \| g_2\|_{s_0 + 2 \tau + 1}^\Lipg\,.
$$
Furthermore by Lemma \ref{lemma:Kapponi vari}, Theorem \ref{invertibility of frak L omega}, and the smallness condition \eqref{size of torus} we get 
$$
\| K_{1,1}(\vphi)  {\frak L}_\omega^{- 1}(\vphi) g_3 \|_{s_0}^\Lipg \lessdot \e \gamma^{- 3} \| g_3\|_{s_0 + 2 \tau + 1, \sigma - 2}^\Lipg\,.
$$
Altogether, this then proves  \eqref{estimate of average upsilon}. The estimate for $\widehat \psi$, defined by formula \eqref{definition widehat psi} is derived from 
Lemma~\ref{om vphi - 1 lip gamma}, using arguments similar to the ones above. 
\end{proof}
Summarizing our results obtained so far, we have constructed the unique solution $(\widehat \psi, \widehat \upsilon, \widehat W, \widehat \zeta)$ of the linear system \eqref{first equation right inverse}-\eqref{third equation right inverse}. Combining Lemma \ref{lemma soluzioni seconda equazione}, Corollary \ref{corollario expression widehat W} and Lemma \ref{lemma estimate of psi hat average upsilon} we get the following corollary.

\begin{corollary}\label{tame estimate for inverse T omega}
There exists $\mu = \mu(|S|, \tau) \in \Z_{\ge 1}$
so that for any map
$
g = ( g_1, g_2, g_3)$ in 
$H^{s + \mu}(\T^S, \R^S\times \R^S \times  h^{\sigma -2}_\bot \times h^{\sigma -2}_\bot)
$
with $s_0 \le s \le s_* - \mu$, and any $\o \in \Omega^{2 \gamma}_{\rm Mel}(\io)$
with $\Omega^{2 \gamma}_{\rm Mel}(\io) \equiv \Omega^{2 \gamma}_{\rm Mel}(\io ; \Omega_o(\io))$ 
as in Theorem \ref{invertibility of frak L omega}, the linear system  \eqref{first equation right inverse}-\eqref{third equation right inverse} admits  a unique solution ${\frak T}_\omega^{- 1} g = (\widehat \io, \widehat \zeta)$. It satisfies the tame estimate 
$$
\| {\frak T}_\omega^{- 1} g \|_s^\Lipg 
\leq_s \gamma^{- 2} \big(\| g \|_{s + \mu, \sigma - 2}^\Lipg +  \| \io\|_{s + \mu}^\Lipg \| g \|_{s_0 + \mu, \sigma - 2}^\Lipg \big)\,.
$$
\end{corollary}
\begin{proof}
Combining Lemmas \ref{lemma soluzioni seconda equazione} and \ref{lemma estimate of psi hat average upsilon} yields 
$$
\| \widehat \upsilon \|_s^\Lipg \leq_s \| \widehat \upsilon_1 \|_s^\Lipg 
+ \|\widehat \upsilon_0 \|_s^\Lipg \leq_s 
\gamma^{- 1}  \| g\|_{s + 2 \tau + 1 , \sigma - 2}^\Lipg + 
\gamma^{- 1} \| \io \|_{s + 2 \tau + 2}^\Lipg\| g\|_{s_0, \s-2}^\Lipg .
$$
From this and the estimate \eqref{estimate widehat W} we conclude the claimed estimate for $\widehat W$. Finally the claimed estimate for $\widehat \psi$ is given in \eqref{estimate of psi hat average upsilon} and the one for $\widehat \zeta$ in \eqref{soluzioni seconda equazione}.
\end{proof}

With these preparations we now prove that the operator 
\begin{equation}\label{definizione T} 
{\bf T}_\omega := d \wtilde{\Gamma} (\breve \io_0) \circ {\frak T}_\om^{-1} 
\circ d \Gamma (\breve \io_0 )^{- 1}\,, \qquad \widetilde \Gamma (\psi, \upsilon, w, \zeta) :=    \big(\Gamma (\psi, \upsilon, w ), \zeta \big)
\end{equation}
is an approximate right  inverse for 
\begin{align}\label{formula for d F omega}
 d_{\io,\zeta} F_\omega (  \io ) & \stackrel{\eqref{error term approximate inverse 1}}{=} d_{\io,\zeta} F_\omega (  \io_{\rm iso}) + G_1 \nonumber\\ 
 & \stackrel{\eqref{error term approximate inverse 2}}{=} d \Gamma(\breve \io_0) 
 \big( \omega \cdot \partial_\vphi  - d_{\io, \zeta} X_{K_{\e, \zeta}}(\breve \io_0)\big) d \widetilde \Gamma(\breve \io_0)^{- 1} + G_2 + G_1 \nonumber\\
 & \stackrel{\eqref{error term approximate inverse 3}}{=} 
d \Gamma(\breve \io_0) {\frak T}_\omega d \widetilde \Gamma(\breve \io_0)^{- 1} + 
d \Gamma(\breve \io_0) G_3 d \widetilde \Gamma(\breve \io_0)^{- 1} + G_2 + G_1 \,.
\end{align}
It is convenient to introduce the norm $ \| (\psi, \upsilon , W, \zeta) \|_{s, \s}^\Lipg := $ $ \max \{  \| (\psi, \upsilon, W)  \|_{s, \s}^\Lipg, $ $ | \zeta |^\Lipg  \} $.

\begin{theorem} {\bf (Approximate right inverse)} \label{thm:stima inverso approssimato}
For any constant $C > 0$, there exist $\delta_1 = \delta_1(|S|, \tau, s_*, C)$ with  
$0 < \delta_1 < 1$ and a positive integer $\mu_1 = \mu_1(|S|, \tau) \in \Z_{\ge 1}$ with $\delta_1 < \delta_0$, $\mu_1 > \mu_0$
 and $\delta_0$, $\mu_0$ given as in Theorem \ref{invertibility of frak L omega}, such that whenever 
\begin{equation}\label{final smallness condition for approximate inverse}
 \| \io \|_{s_0 + \mu_1}^\Lipg
 \leq \, C \e \g^{-2} \, , \quad \| F_\omega(\io, \zeta) \|_{s_0 + \mu_1, \s-2}^\Lipg \leq \,C  \e\,, \qquad \e \gamma^{- 4} \leq \delta_1,
\end{equation}
then the family of operators ${\bf T } = ({\bf T}_\omega)_{\omega \in \Om_{\rm Mel}^{2 \gamma}(\io)}$ with $\Omega^{2 \gamma}_{\rm Mel}(\io) \equiv \Omega^{2 \gamma}_{\rm Mel}(\io ; \Omega_0(\io))$ 
as in Theorem \ref{invertibility of frak L omega} has the following properties:  for any $ g := (g_1, g_2, g_3) \in H^{s + \mu_1}(\T^S, \R^S \times \R^S \times h^{\sigma - 2}_\bot \times h^{\sigma - 2}_\bot) $
with $s_0 \le s \le s_* - \mu_1,$ 
the operator $ {\bf T} $ defined in \eqref{definizione T} satisfies  
\begin{equation}\label{stima inverso approssimato 1}
\| {\bf T} g \|_{s, \sigma}^{\Lipg}  
\leq_s  \gamma^{- 2} \big( \| g \|_{s + \mu_1, \sigma - 2}^\Lipg +  \| \io\|_{s + \mu_1}^\Lipg \| g \|_{s_0 + \mu_1, \sigma - 2}^\Lipg \big)\,.
\end{equation}
Furthermore ${\bf T}_\omega$ is an approximate right inverse of $ d_{\io, \zeta} F_\omega( \io )$, namely 
\begin{align}
& \| ( d_{\io, \zeta} F_\omega (  \io ) \circ {\bf T}_\omega - {\rm Id} ) g \|_{s, \s -2}^{\gamma {\rm lip}}  
\label{stima inverso approssimato 2} 
\\ 
& \leq_s \gamma^{-3}  \Big(\|  F_\omega (  \io , \zeta) \|_{s_0 + \mu_1, \s-2}^\Lipg 
\| g \|_{s + \mu_1, \s -2}^\Lipg +  
\| F_\omega (  \io , \zeta) \|_{s + \mu_1, \s -2}^\Lipg \| g \|_{s_0 + \mu_1, \s -2}^\Lipg \nonumber\\
& +  \| F_\omega ( \io, \zeta) \|_{s_0 + \mu_1, \s -2}^\Lipg  
\| \io  \|_{s + \mu_1}^\Lipg  \| g \|_{s_0 + \mu_1, \s -2}^\Lipg \Big)\,.
\nonumber
\end{align}
\end{theorem}

\begin{proof}
The tame estimate \eqref{stima inverso approssimato 1} follows from the definition
  \eqref{definizione T} of ${\bf T}_\omega,$ the estimate of
${\frak T}_\omega^{-1}$ of Corollary~\ref{tame estimate for inverse T omega},
and the estimates of $d \Gamma (\breve \io_0 )$, $d \Gamma (\breve \io_0 )^{- 1}$
of Lemma~\ref{lemma:DG} .

The estimate \eqref{stima inverso approssimato 2} can be obtained as follows:
using the formula \eqref{formula for d F omega} for $d_{\io,\zeta} F_\omega (  \io )$ and 
the definition \eqref{definizione T} of ${\bf T}_\omega,$
one sees that $ d_{\io, \zeta} F_\omega (  \io ) \circ {\bf T}_\omega - {\rm Id}$ is the sum
of the three terms
$d \Gamma(\breve \io_0) G_3 {\frak T}_\omega^{-1}d 
\widetilde \Gamma(\breve \io_0)^{- 1} $,
$G_2 d \wtilde{\Gamma} (\breve \io_0) {\frak T}_\om^{-1} 
 d \Gamma (\breve \io_0 )^{- 1}$, and
$G_1 d \wtilde{\Gamma} (\breve \io_0)  {\frak T}_\om^{-1} 
 d \Gamma (\breve \io_0 )^{- 1}$, which are estimated separately,
combining the estimates of $G_1$, $G_2$, and $G_3$ of Lemma~\ref{estimate G1}, Lemma~\ref{stima G2}, and, respectively, Lemma~\ref{stima G3}
with the estimate of ${\frak T}_\omega^{-1}$ of 
Corollary~\ref{tame estimate for inverse T omega},
and the estimates of $d \Gamma (\breve \io_0 )$, $d \Gamma (\breve \io_0 )^{- 1}$
of Lemma~\ref{lemma:DG} .

The integer $\mu_1 > \mu_0,$ and the constant $0 < \delta_1 < \delta_0$
are chosen in such way that the lemmas used to derive the estimates
\eqref{stima inverso approssimato 1}, \eqref{stima inverso approssimato 2} apply.
\end{proof}

\section{Reduction of $ {\frak L}_\om $. Part 1}\label{sec:5}

For proving Theorem~\ref{invertibility of frak L omega}
it is  useful to express the Hamiltonian operator ${\frak L}_\omega$, introduced
in \eqref{definition frak L omega sec 4},  in terms 
of the Hamiltonian $H_\e$ rather than $K_\e = H_\e \circ \Gamma$ defined in \eqref{def:Kep}. 
By 
\eqref{definition frak L omega sec 4}, 
\eqref{forma diversa hamiltoniana cal Q (vphi)} and \eqref{operatore Hamiltoniano lineare 2} we have 
\be\label{forma-L-omega1}
{\frak L}_\omega = \omega \cdot \partial_\vphi {\mathbb I}_2 + J \begin{pmatrix}
\partial_w \nabla_{\bar w} K_\e & \partial_{\bar w} \nabla_{\bar w} K_\e \\
\overline{\partial_{\bar w} \nabla_{\bar w} K_\e}  & \overline{\partial_w \nabla_{\bar w} K_\e}
\end{pmatrix} \circ \breve \io_0\,, \qquad 
J = \left(
\begin{array}{cc}
\ii \,{\rm Id}_\bot & 0 \\
0 & - \ii \,{\rm Id}_\bot 
\end{array} \right) \, .
\ee
Taking into account the definition of $ \Gamma $ in \eqref{2.23}, \eqref{2.23bis}, \eqref{operator Yw}
one computes
\be\label{der-comp1}
\nabla_{\bar w} K_\e = \nabla_{\bar z} H_\e \circ \Gamma + Y_{\bar w}^t \nabla_y H_\e \circ \Gamma\,, \quad 
\partial_w \nabla_{\bar w} K_\e = \partial_z \nabla_{\bar z} H_\e \circ \Gamma +  R_1^\e \circ \Gamma\,,
\ee
where, by \eqref{operator Yw},  
\begin{equation}\label{R 1 epsilon tilde}
 R_1^\e := \partial_y (\nabla_{\bar z} H_\e ) Y_w + Y_{\bar w}^t \partial_z \nabla_{ y} H_\e + Y_{\bar w}^t \partial_y (\nabla_y H_\e) Y_w \,.
\end{equation}
Similarly, one has 
\be\label{der-comp2}
\partial_{\bar w} \nabla_{\bar w} K_\e = \partial_{\bar z} \nabla_{\bar z} H_\e \circ \Gamma + 
R_2^\e \circ \Gamma\,,
\ee
where 
\begin{equation}\label{R 2 epsilon tilde}
  R_2^\e := \partial_y (\nabla_{\bar z} H_\e ) Y_{\bar w} + 
Y_{\bar w}^t \partial_{\bar z} \nabla_y H_\e + Y_{\bar w}^t \partial_y (\nabla_y H_\e) Y_{\bar w} \,.
\end{equation}
By \eqref{forma-L-omega1} \eqref{der-comp1}, \eqref{der-comp2}
and since by \eqref{def:trivial-torus},  $ \breve \io_{\rm iso} = \Gamma \circ \breve \io_0 $, we get  
\begin{equation}\label{Linearized op normal}
{\frak L}_\omega = \omega \cdot \partial_\vphi {\mathbb I}_2 +  J {\frak A} 
+  J {\frak R}^\e \qquad {\rm where} \qquad
{\frak A} := \begin{pmatrix}
\partial_z \nabla_{\bar z} H_\e  & \partial_{\bar z} \nabla_{\bar z} H_\e \\
\overline{ \partial_{\bar z} \nabla_{\bar z} H_\e} & \overline{\partial_z \nabla_{\bar z} H_\e} 
\end{pmatrix} \circ \breve \io_{\rm iso} 
\end{equation}
and 
\begin{equation}\label{formaK20}
{\frak R}^\e :=  \begin{pmatrix}
{\frak R}_1^\e & {\frak R}_2^\e \\
\overline{\frak R}_2^\e & \overline{\frak R}_1^\e
\end{pmatrix}\,, \quad {\frak R}_1^\e := R_1^\e \circ \breve \io_{\rm iso}\,, \quad {\frak R}_2^\e := R_2^\e \circ \breve \io_{\rm iso}\,.
\end{equation}
According to Definition \ref{def:Hamiltonian op}   $ J {\frak A} $ is Hamiltonian and since
$ {\frak L}_\om $ is also Hamiltonian  so is  $J {\frak R}^\e$. 
We will show in Lemma~\ref{resto finito dimensionale stima lipschitz} in Subsection \ref{analysis of cal A} below 
that ${\frak R}^\e$ can be regarded as a remainder term in the reduction scheme for  
${\frak L}_\omega$. 

 To reduce $ {\mathfrak L}_\om $ to a $ 2 \times 2 $ block diagonal operator with 
$\vphi$-independent
 coefficients, we will use a KAM iteration scheme which requires to impose 
pertinent nonresonance conditions along the iteration. 
In view of the near resonance of the dNLS frequencies $ \om^{nls}_k $ and  $ \om^{nls}_{-k} $, 
this requires an  asymptotic expansion of the eigenvalues of $ {\frak L}_\om $ with a remainder term which decays in $ k $. 
To this end, we perform in Subsections \ref{sec1} - \ref{sec-gauge} three preliminary symplectic transformations 
which put $ {\frak L}_\om $ into diagonal form with $\vphi$-independent coefficients 
up to a remainder, which is one smoothing and satisfies tame estimates.
From a technical point of view, for proving the reduction scheme for the operator ${\frak L}_\om$,
stated in Theorem~\ref{thm:abstract linear reducibility} in Section~\ref{sec:redu} below,
it is convenient to use for operator valued maps 
$\vphi \mapsto \frak R (\vphi) \in \mathcal L (h_\bot^{\s'}  \times h_\bot^{\s'})$
the norm $| \frak R |_{s, \sigma'}$ introduced in \eqref{def norm1}. 
We say that  an operator of this type is one smoothing 
if $| \frak R \mathfrak D |_{s, \s'} < \infty $.
Here $\frak D$ is the operator introduced in \eqref{def:D-doubled}.    

By a slight abuse of terminology, we consider in the entire  section operators
such as ${\frak A}$ or $\frak R^\e$ with $\breve \io_{\rm iso}$ in their definition
replaced by an arbitrary torus embedding $\breve \io \equiv \breve \io_\omega,$ 
of the type described at the beginning of Section~\ref{3. Set up}.
The estimates for $\frak L_\om$ are then
obtained by applying the estimates, derived in this section, for $\breve \io$
given by $\breve \io_{\rm iso}$ and using the estimates 
$\| \io_{\rm iso} \|_s^\Lipg  \leq_s \| \io \|_{s+\mu}^\Lipg$
and
$\| d_\io (   \io_{\rm iso} ) [ \hat \imath ] \|_s  
\leq_s \| \hat \imath \|_{s+\mu} +  \| \io \|_{s + \mu} \| \hat \imath  \|_{s_0+\mu} \,$
of Lemma~\ref{toro isotropico modificato}.
In the sequel, we always make the following smallness assumption, stated in \eqref{size of torus},
\begin{equation}\label{ansatz 1}
\| \io \|_{s_0 + \mu_1}^\Lipg
\lessdot  \e \g^{-2}  \qquad \mbox{with} \quad  
  \e \gamma^{- 4} \ll 1 \,\,\, \mbox{ and } \,\,\,0 < \gamma < 1\,.
\end{equation}


\subsection{Preliminary analysis of the operators $ {\frak A}$ and 
${\frak R}^\e$}\label{analysis of cal A}

The aim of this subsection is to identify the main part of the operator $ {\frak A} $ defined in \eqref{Linearized op normal} and 
to show that the remainder as well as the operator ${\frak R}^\e$ in \eqref{Linearized op normal} are one smoothing and satisfy tame estimates.

First note that since $H_\e = H^{nls} + \e P $ (cf \eqref{HamiltonianHep}), the operator $ {\frak A} $ can be written as 
$ {\frak A} = {\frak S}^{nl s} + \e {\frak S}^P$ where 
\be\label{NLS normal direction}
 {\frak S}^{nls} :=  
 \begin{pmatrix} \partial _z \nabla_{\bar z} H^{nls}  & \partial_{\bar z} \nabla_{\bar z} H^{nls} \\ 
   \overline{\partial_{\bar z} \nabla_{\bar z} H^{nls}}  &  \overline{\partial _{z} \nabla_{\bar z} H^{nls}} 
                  \end{pmatrix} \circ \breve \io\, \qquad 
                  {\frak S}^P  =  
\begin{pmatrix} \partial _z    \nabla_{\bar z} P  & \partial_{\bar z} \nabla_{\bar z} P \\ 
                \overline{\partial_{\bar z} \nabla_{\bar z} P}  &  \overline{\partial _{ z} \nabla_{\bar z} P }
                   \end{pmatrix} \circ \breve \io\,. \ee
The operators ${\frak S}^{nls}$, ${\frak S}^P$, and ${\frak R}^\e$ are analyzed separately. 

\medskip

\noindent
{\it Analysis of ${\frak S}^{nls}$.} Recall that $H^{nls} = H^{nls} (\xi + y, z \bar z)$ with $z \bar z := (z_n \bar z_n)_{n \in S^\bot}$, yielding 
$$
\nabla_{\bar z} H^{nls} \circ \breve \io 
= \big( (\omega^{nls}_k z_k )\circ \breve \io \big)_{k \in S^\bot}
$$
with $\omega_{k}^{nls} = \partial_{I_k} H^{nls}$.
To simplify notation, we will drop $ \breve \io$ whenever the context permits.
In particular, we will often write $I $ for $ I \circ \breve \io$ and $\omega_{k}^{nls}$
for $\omega_{k}^{nls}(I \circ \breve \io)$. 
Then we have 
\be\label{hessian-rem12}
\partial_z \nabla_{\bar z} H^{nls} = 
{\rm diag}_{k \in S^\bot} \big( \omega_k^{nls} ) + R_1^{nls}\,, \qquad 
\partial_{\bar z} \nabla_{\bar z} H^{nls} = R_2^{nls}
\ee
where  $R^{nls}_1$, $R^{nls}_2$  are the operators of $ h^\s_\bot $ with matrix coefficients  (cf \eqref{matrix:coefficients})
\be\label{def:R1R2}
(R_1^{nls})_k^j  :=  (\pa_{I_j} \om_k^{nls})z_k \bar z_j \,, \quad 
(R^{nls}_2)_{k}^{j} :=  ( \pa_{I_j} \om_k^{nls} ) z_k z_j \, ,  \qquad \forall j, k \in S^\bot\,.
\ee
By \eqref{NLS normal direction}, \eqref{hessian-rem12}, and in view of 
the asymptotics $\omega^{nls}_k = 4 \pi^2 k^2 + O(1)$ of Theorem \ref{Corollary 2.2} we write
\be\label{def:Rnls}
 {\frak S}^{nls} = D^2 \, {\mathbb I}_2 + 
  \Omega^{nls}  \, {\mathbb I}_2 + {\frak R}^{nls} \, , \quad  \quad {\frak R}^{nls} := 
\left(
\begin{array}{cc}
{\frak R}^{nls}_1 & {\frak R}^{nls}_2   \\
\overline {{\frak R}^{nls}_2}& \overline{ {\frak R}^{nls}_1}  
\end{array} 
\right), \quad\,\,  {\frak R}^{nls}_a =  R^{nls}_a \circ \breve \io, \,\, \,  a= 1, 2\,,
\ee
where $D$ is the diagonal operator defined in \eqref{def:D} and 
\begin{equation}\label{definition Omega nls}
 \Omega^{nls}  := {\rm diag}_{k \in S^\bot} \big(  \omega_k^{nls}  - 4 \pi^2 k^2\big) \, .
\end{equation}
We claim that $D^2 \, {\mathbb I}_2 +  \Omega^{nls}  \, {\mathbb I}_2$ is the main part  of $ {\frak S}^{nls},$
meaning that ${\frak R}^{nls}$ is a (small) one smoothing operator. More precisely the following estimates hold.
We recall that throughout the paper, we assume that $\s \ge 4,$ if not stated otherwise.
  
\begin{lemma}\label{I bot Lip gamma} 
{\bf (Estimates for $ \Omega^{nls} $ and $ {\frak R}^{nls} $)}
Let $s \ge s_0$. Then the following estimates hold:  

\noindent
$(i)$ For any $\sigma' \in \{ \sigma, \sigma - 1, \sigma - 2 \}$, 
\be\label{est:Omega1} 
|\Omega^{nls}|_{s, \sigma'} \leq_s 1 + \| \io \|_{s + 2 s_0} \, , \quad 
| \Omega^{nls} |_{s, \sigma'}^\Lipg \leq_s 1 + \| \io \|_{s + 2 s_0}^\Lipg \, . 
\ee 
\noindent
$(ii)$ The remainder ${\frak R}^{nls}$ defined in \eqref{def:Rnls} satisfies the estimates 
\be \label{est:Rnls-prima}
|{\frak R}^{nls} {\mathfrak D} |_{s, \sigma - 1 } \leq_s  \e \gamma^{- 2} \| \io \|_{s + 2 s_0} \, , \quad 
 |{\frak R}^{nls} {\mathfrak D} |_{s, \sigma - 1 }^\Lipg
\leq_s  \e \gamma^{- 2} \| \io \|_{s + 2 s_0}^\Lipg \, , 
\ee
where $ {\mathfrak D} $ is defined in \eqref{def:D-doubled}.
\end{lemma}

\begin{proof} 
($i$) We now prove the first estimate in \eqref{est:Omega1}.
As $\Omega^{nls}$ is a diagonal operator it suffices to prove the claimed estimate for $\sigma' = \sigma$. 
By Theorem \ref{Corollary 2.2}, the dNLS frequencies admit the asymptotics
$$
\om_k^{nls} (I) = 4 \pi^2 k^2 +  4 \sum_{j \in \Z} I_j + \frac{r_k (I)}{k}   
$$
where $ (r_k)_{k \in \Z} : \ell^{1,4}_+(\Z, \R)  \to \ell^\infty (\Z, \R) $ is real analytic. 
Accordingly we decompose $\Omega^{nls}$, 
defined in \eqref{definition Omega nls}, as
\begin{equation}\label{asymptotics Omega nls}
\Omega^{nls} = \Big( 4 \sum_{j \in \Z} I_j \Big) {\rm Id}_\bot + {\rm diag}_{k \in S^\bot} \frac{r_k(I)}{k} 
\end{equation}
and estimate the norms of the latter two operators separately. 
To estimate $ | \big( \sum_{j \in \Z} I_j(\vphi) \big) {\rm Id}_\bot  |_{s, \sigma} $ we write
\be\label{deff}
\sum_{j \in \Z} I_j(\vphi)  =  \Big( \sum_{j \in S}  \xi_j \Big) {\rm Id}_\bot +  g(\vphi) {\rm Id}_\bot   \qquad {\rm where} \qquad 
 g(\vphi) :=  \sum_{j \in S}  y_j(\vphi) + \sum_{j \in S^\bot} z_j(\vphi) {\bar z}_j (\vphi) \, . 
\ee
By the definition \eqref{def norm1} of the operator norm $| \cdot |_{s, \s}$,
\begin{equation}\label{f (vphi) pi bot}
\big| g \, {\rm Id}_\bot \big|_{s, \sigma} = \Big( \sum_{\ell \in \Z^S} 
\langle \ell \rangle^{2s} \| \hat g(\ell )  \, {\rm Id}_\bot \|_{{\cal L}(h^\s_\bot)}^2 \Big)^{1/2}
= \Big( \sum_{\ell \in \Z^S} 
\langle \ell \rangle^{2s} | \hat g(\ell )|^2  \Big)^{1/2}
= \| g \|_s   
\end{equation}
where, for brevity, we set $ \| g\|_s :=  \| g \|_{H^s(\T^S, \C)} $. 
By \eqref{deff}, using Lemma~\ref{interpolation product} and the Cauchy-Schwartz inequality, we estimate 
$$
\| g  \|_s   \leq_s  \| y \|_s + \sum_{j \in S^\bot} \| z_j  \bar z_j \|_s \leq_s
 \| y \|_s + \sum_{j \in S^\bot} \| z_j  \|_{s_0} \| \bar z_j \|_s \leq_s
  \| y \|_s +  \| z  \|_{s_0, \s} \| z  \|_{s,\s} \leq_s  \| \io \|_s \,.
$$
In conclusion 
\be\label{f (vphi) y z bar z}
\Big| \Big(\sum_{j \in \Z} I_j \Big) \, {\rm Id}_\bot \Big|_{s, \sigma} \leq_s | \xi | + \| g \|_s  \leq_s
| \xi | + \| \io \|_s \leq_s 1 + \| \io \|_s  \, .
\ee
Towards the second operator on the right hand side of \eqref{asymptotics Omega nls}, note that
 the operator norm of the Fourier coefficient $\hat A(\ell)$, $\ell \in \Z^S,$ of the map
$ \vphi \to A(\vphi) := {\rm diag}_{k \in S^\bot} \frac1k  (r_k \circ I) (\vphi)  $
is
$$
\| \hat A(\ell) \|_{{\cal L}(h^\s_\bot)} = \sup_{k \in S^\bot} \frac{1}{|k|}  | (\widehat{r_k \circ I}) (\ell)|  
$$
and hence, recalling the definition \eqref{def norm1} of the operator norm $| \cdot |_{s, \s}$,
\begin{equation}\label{A s sigma r k I}
| A |_{s, \sigma}^2 = \sum_{\ell \in \Z^S}  \langle \ell \rangle^{2s} 
\sup_{k \in S^\bot} \frac{1}{k^2}  |(\widehat{r_k \circ I})(\ell)|^2 \leq  
\sum_{k \in S^\bot} \frac{1}{k^2}  \sum_{\ell \in \Z^S}  \langle \ell \rangle^{2s} 
 |(\widehat{r_k \circ I})(\ell)|^2 =
 \sum_{k \in S^\bot} \frac{1}{k^2} \| r_k \circ I \|_s^2 \, . 
\end{equation}
By Theorem \ref{Corollary 2.2}, the map $ (r_k)_{k \in S^\bot} : \ell^{1,4} \to \ell^{\infty}_\bot $ is real analytic and 
 there exists a neighborhood $ V \subset \ell^{1,4} $ of $(\Pi + U_0)\times \{0\}$ and $ C >  0 $ such that
$ \sup_{I \in V} | r_k(I) | \leq C $, $ \forall k \in S^\bot $. Since for any $\xi \in \Pi$, the map 
$$
B_\sigma(0,0) \subseteq \R^S \times h^\sigma_\bot \ \to \ V\,, \quad (y, z) \mapsto (\xi + y, z \bar z) \in V
$$
is real analytic in a sufficiently small neighborhood of $(0, 0)$, $B_\sigma(0, 0) \subseteq \R^S \times h^\sigma_\bot$ 
(see  the proof of Proposition \ref{stime derivate H nls}), 
Lemma \ref{Lemma 2.4bis}, applied to $f$ given by the sequence $(r_k (\xi + y, z \bar z))_{k \in \Z}$ and $Y = \ell^\infty$ then yields 
\be\label{estimateFrequ}
\| (r_k (\xi + y, z \bar z) )_{k \in \Z}\|_{{\cal C}^s(\T^S, \ell^\infty)} \leq_s  1 + \| \io \|_{{\cal C}^s(\T^S, M^\sigma)} \, .
\ee
 As a consequence of \eqref{composition formula}, we get 
\be\label{estimateFrequ 2}
\|  r_k \circ I  \|_s   \leq_s 1 + \| \io \|_{s + 2 s_0} \, , \quad  \forall k \in S^\bot ,
\ee 
and, by \eqref{A s sigma r k I}, we conclude
\begin{equation}\label{fina2}
| A |_{s, \sigma} \leq_s  1 + \| \io \|_{s + 2 s_0}  \,. 
\end{equation}
Combining \eqref{asymptotics Omega nls} with 
\eqref{f (vphi) y z bar z} and \eqref{fina2},  the first estimate of \eqref{est:Omega1} follows.  
The second estimate of  \eqref{est:Omega1}  is proved in a similar way.

\noindent
 $(ii)$  
Let us begin by proving  the first estimate  of \eqref{est:Rnls-prima}. 
We only consider ${\frak R}^{nls}_1 \lla D \rra $ since the estimate for ${\frak R}^{nls}_2 \lla D \rra $ is done in the same way.
We recall that  $ \lla D \rra $ is the diagonal operator introduced in \eqref{def:DD}.  

We write $\frak R^{nls}_1 \lla D \rra $ as the sum of its columns, namely
\be\label{RNLS1D}
{\frak R}^{nls}_1 \lla D \rra  = \sum_{j \in S^\bot} A_{(j)} \pi_j \, , \quad 
A_{(j)}(\vphi) := \big(  z_k (\vphi) \langle j \rangle^2 f_{kj}(I(\vphi))) \bar z_j (\vphi)   \lla j \rra \big)_{k \in S^\bot} \, , 
\ee
where $ \pi_j $ denotes the projector 
\be\label{def:pi-bot}
 \pi_j : h^\s_\bot \to \C \, , \quad (w_n)_{n \in S^\bot} \to w_j \, , 
\ee
and 
\be \label{def:fkj}
f_{kj} (I) : = \langle j \rangle^{-2} \pa_{I_j} \om_k^{nls} (I) \, , \quad I(\vphi) := (\xi + y(\vphi), I_\bot (\vphi) ) \, , 
\quad  I_\bot  := ( z_k \bar z_k )_{k \in S^\bot} \, . 
\ee
Then we have 
$|{\frak R}^{nls}_1 \lla D \rra |_{s, \sigma - 1 }  \leq \sum_{j \in S^\bot} |A_{(j)} \pi_j|_{s, \sigma - 1}$.
Since by the definition \eqref{def norm1} of the operator norm $|\cdot |_{s, \s -1}$
$$
|A_{(j)} \pi_j|_{s, \sigma - 1  }  = \Big( \sum_{\ell \in \Z^S} \langle \ell\rangle^{2 s} 
\|\hat{ A}_{(j)}(\ell) \pi_j\|_{{\cal L}(h^{\sigma - 1}_\bot )}^2 \Big)^{\frac12} \,, 
\qquad \|\hat{A}_{(j)}(\ell) \pi_j \|_{{\cal L}(h^{\sigma - 1}_\bot)} 
= \| \hat{ A}_{(j)} (\ell) \|_{\sigma - 1} \langle j \rangle^{- {(\sigma - 1)}}\,,
$$
we have, by the property   \eqref{norma other} of the $ \| \cdot \|_s $-norm 
\begin{equation}\label{RNLS1Ds}
|A_{(j)} \pi_j|_{s, \sigma - 1 } 
=   \langle j \rangle^{- (\sigma-1)}  \| A_{(j)} \|_{s, \s-1} \leq  \langle j \rangle^{- (\sigma-1)}  \| A_{(j)} \|_{s, \s} \,.
\end{equation}
We claim that 
\be\label{estimate Aj}
\| A_{(j)}\|_{s, \s} \leq_s \langle j \rangle^3 \big(  \| \io \|_{s + 2 s_0} \| z_j \|_{s_0} +  \| \io \|_{s_0} \| z_j \|_s \big) \, .
\ee
Before proving \eqref{estimate Aj} we complete the proof of the first estimate of \eqref{est:Rnls-prima}. 
By \eqref{RNLS1Ds} and \eqref{estimate Aj}, we get
\begin{align*}
|\frak R^{nls}_1 \lla D \rra |_{s, \sigma - 1 }  & \leq_s  \sum_{j \in S^\bot}  \langle j \rangle^{4 -\s}
\big(  \| \io \|_{s + 2 s_0} \| z_j \|_{s_0} +  \| \io \|_{s_0} \| z_j \|_s  \big)  \\
& \leq_s \| \io \|_{s + 2 s_0} 
\Big( \sum_{j \in S^\bot}  \langle j \rangle^{4 - 2 \s}  \| z_j \|_{s_0}\langle j \rangle^\s \Big)  
+ 
 \| \io \|_{s_0} 
\Big( \sum_{j \in S^\bot}  \langle j \rangle^{4 - 2 \s}  \| z_j \|_s \langle j \rangle^\s \Big)  \\
&  \leq_s 
\| \io \|_{s +2 s_0}  \| z \|_{s_0,\s}   +  \| \io \|_{s_0} \| z \|_{s,\s} 
\end{align*}
by applying the Cauchy-Schwartz inequality, using that $ 4 ( \s - 2  ) > 1  $. 
By  the smallness assumption \eqref{ansatz 1}, the  first estimate of \eqref{est:Rnls-prima} then follows.  
It remains to prove the estimate \eqref{estimate Aj}. By the definition \eqref{def:fkj} of $f_{k j}$ and the estimates \eqref{tame composizione derivate ennesime frequenza} one gets 
\be\label{come voluta}
\| {f_{kj} }(\xi + y, z \bar z) \|_{s}  \leq_s   1 + \| \io \|_{s + 2 s_0}  \,, \quad \forall j, k \in S^\bot, \,\,\, \forall \xi \in \Pi .
\ee
We now can prove the estimate \eqref{estimate Aj}: recalling \eqref{norma other} and \eqref{RNLS1D} we have 
\begin{align}
\| A_{(j)} \|^2_{s, \s} & \leq_s \langle j \rangle^6 \sum_{k} \langle k \rangle^{2 \s} \| z_k (f_{kj} \circ I) \bar z_j  \|^2_{s} \nonumber \\
&\stackrel{\eqref{tame for functions}} 
{\leq_s} \langle j \rangle^6 \sum_{k} \langle k \rangle^{2 \s} 
\Big( \| z_k \|_s \|  f_{kj} \circ I \|_{s_0} \| z_j  \|_{s_0} + \| z_k \|_{s_0} \|  f_{kj} \circ I \|_{s} \| z_j  \|_{s_0} +
\| z_k \|_{s_0} \|  f_{kj} \circ I \|_{s_0} \| z_j  \|_s \Big)^2 \nonumber \\
& \stackrel{\eqref{come voluta}, \eqref{ansatz 1}}{\leq_s} \langle j \rangle^6 \sum_{k} \langle k \rangle^{2 \s} 
\Big( \| z_k \|_s  \| z_j  \|_{s_0} + \| z_k \|_{s_0} \| \io \|_{s + 2s_0} \| z_j  \|_{s_0} +
\| z_k \|_{s_0} \| z_j  \|_s \Big)^2 \nonumber \\
&  \stackrel{\eqref{norma other}} \leq_s \langle j \rangle^6 
\Big( \| z \|_{s,\s}^2 \| z_j  \|_{s_0}^2 + \| z \|_{s_0, \s}^2 \| \io \|_{s + 2s_0}^2  \| z_j  \|_{s_0}^2 
 + \| z \|_{s_0, \s}^2 \| z_j  \|_s^2 \Big) \nonumber .
\end{align}
 Using again the smallness assumptions \eqref{ansatz 1}, the claimed estimate \eqref{estimate Aj} then follows.  
The second estimate in \eqref{est:Rnls-prima} can be proved in a similar way. 
 \end{proof}
 
 The next result is only needed in Section \ref{sec:measure} for the proof of the measure estimates. 
Given two torus embeddings 
  $$  
  \breve \io^{(a)} (\vphi) := (\vphi, 0, 0) + \io^{(a)}(\vphi)\,,\qquad  
     \io^{(a)}(\vphi) = (\Theta^{(a)}(\vphi), y^{(a)}(\vphi), z^{(a)}(\vphi))\,, \quad a = 1,2\,,
  $$  
  we write 
  \begin{equation}\label{differenza tori}
  \Delta_{1 2} \breve \io := \breve \io^{(1)} - \breve \io^{(2)}, \quad
 \Delta_{1 2} \io := \io^{(1)} - \io^{(2)}\,, \quad
  \Delta_{12} z : = z^{(1)} - z^{(2)}\,,\,\, \ldots \,\, .
  \end{equation}
  Note that  $\Delta_{1 2} \breve \io =  \Delta_{1 2} \io$.  Furthermore, introduce for $s \geq s_0$ 
  \begin{equation}\label{norma s io 1 io 2}
  {\rm max}_s(\io) := {\rm max}\{ \|\io^{(1)} \|_s\,,\, \| \io^{(2)}\|_s \}\,, \quad {\rm max}_s(z) := {\rm max}\{ \|z^{(1)} \|_s\,,\, \| z^{(2)}\|_s \}\,, \,\, \ldots \, \, .
  \end{equation}
Define  $  \Omega^{nls}(\breve \io^{(a)}) := \Omega^{nls}( I \circ \breve \io^{(a)}) $, $ a = 1, 2 $, and use 
a similar notation for other operators. 

 \begin{lemma}\label{lemma variazione i Omega nls} Let $s \ge s_0.$ Then for any
torus embeddings $\breve \io^{(a)} (\vphi) := (\vphi, 0, 0) + \io^{(a)}(\vphi) $, $a = 1,2 $, 
satisfying \eqref{ansatz 1}, 
the following estimates hold:

\noindent
$(i)$ For any $\sigma' \in \{ \sigma, \sigma - 1, \sigma - 2 \}$, 
$\Delta_{12} \Omega^{nls} := \Omega^{nls}(\breve \io^{(1)}) - \Omega^{nls}(\breve \io^{(2)})$ satisfies the estimate 
$$|\Delta_{12} \Omega^{nls}|_{s, \sigma'} \leq_s \| \Delta_{12} \io \|_{s } +  
 {\rm max}_{s + 2 s_0}(\io) \| \Delta_{12 } \io \|_{s_0}.$$

\noindent
$(ii)$ The operator 
$\Delta_{12} {\frak R}^{nls} := {\frak R}^{nls}(\breve \io^{(1)}) - {\frak R}^{nls}(\breve \io^{(2)})$ 
satisfies the estimate 
$$
|\Delta_{12} {\frak R}^{nls} {\frak D}|_{s, \sigma - 1} \leq_s \e \gamma^{- 2} \|\Delta_{12} \io \|_s 
+ {\rm max}_{s + 2 s_0}(\io) \| \Delta_{12} \io\|_{s_0}\,.
$$
 \end{lemma}
 \begin{proof}
%

 \noindent
 $(i)$ 
As $\Omega^{nls}$ is a diagonal operator it suffices to prove the claimed estimate for $\sigma' = \sigma$. 
Writing $I^{(a)} := (\xi + y^{(a)}, I_\bot^{(a)})$, $a = 1, 2$ and $\Delta_{12} I_j := I^{(1)}_j - I^{(2)}_j$, $j \in \Z$, one has, 
by \eqref{asymptotics Omega nls}, 
\begin{equation}\label{def Omega io 1 - Omega io 2}
\Omega^{nls}(\breve \io^{(1)}) - \Omega^{nls}(\breve \io^{(2)}) = 
\Big( 4 \sum_{j \in \Z} \Delta_{12} I_j \Big) {\rm Id}_\bot 
+ {\rm diag}_{k \in S^\bot} \frac{\Delta_{12} r_k(I)}{k} \, . 
\end{equation}
Since
$
\sum_{j \in \Z} \Delta_{12} I_j = \sum_{j \in S} \Delta_{12} y_j + \sum_{j \in S^\bot} \Delta_{12} I_j\,,
$
one gets, arguing as in \eqref{f (vphi) pi bot}, \eqref{f (vphi) y z bar z}, 
\begin{align}
\Big|  \Big(\sum_{j \in \Z} \Delta_{12} I_j  \Big){\rm Id}_\bot\Big|_{s, \sigma} 
& \leq \Big\| \sum_{j \in \Z} \Delta_{12} I_j \Big\|_{s}  \nonumber \\
&\leq_s  \sum_{j \in S} \| y^{(1)}_j - y^{(2)}_j\|_{s} + 
\sum_{j \in S^\bot} \| (z_j^{(1)} - z_j^{(2)}) \bar z_j^{(1)} \|_{s} + 
\sum_{j \in S^\bot} \| z_j^{(2)}(\bar z_j^{(1)} - \bar z_j^{(2)})  \|_{s} \nonumber\\
 & \stackrel{ \eqref{ansatz 1}}{\leq_s} \| \Delta_{12} \io \|_{s} + {\rm max}_{s }(\io) \| \Delta_{12} \io \|_{s_0}\,. \label{stima primo pezzo Delta 12 Omega}
\end{align}
Now we estimate the second term on the right hand side of \eqref{def Omega io 1 - Omega io 2}. The operator norm of the Fourier coefficient $\hat A(\ell)$, $\ell \in \Z^S,$ of the map
$ \vphi \to A(\vphi) := {\rm diag}_{k \in S^\bot} \frac1k  \Delta_{12} r_k(\vphi)$ 
where $ \Delta_{12} r_k :=  r_k(I^{(1)}) - r_k(I^{(2)})  $
is
$$
\| \hat A(\ell) \|_{{\cal L}(h^\s_\bot)} = 
\sup_{k \in S^\bot} \frac{1}{|k|}  | \widehat {\Delta_{12} r_k} (\ell)|  
$$
and hence, arguing as in \eqref{A s sigma r k I}
\begin{equation}\label{A Delta 12 rk I}
| A |_{s, \sigma}^2 \leq 
 \sum_{k \in S^\bot} \frac{1}{k^2} \| \Delta_{12} r_k ( I) \|_s^2 \, . 
\end{equation}
By the mean value theorem one has
\begin{equation}\label{definizione Delta 12 rk}
\Delta_{12} r_k = \int_0^1 \partial_I r_k(I_t)\, d t \cdot \Delta_{12} I\,, \qquad I_t := t I^{(1)} + (1 - t) I^{(2)}\,
\end{equation}
where
\begin{equation}\label{definizione Delta 12 rk 2}
\partial_I r_k(I_t) \cdot \Delta_{12} I = \sum_{n \in \Z} \partial_{I_n} r_k(I_t) \Delta_{12} I_n\,.
\end{equation}
Since by Theorem \ref{Corollary 2.2} item $(ii)$, the map $ (r_k)_{k \in S^\bot} : \ell^{1,4} \to \ell^{\infty} $ is real analytic 
 there exists a neighborhood $ V \subset \ell^{1,4}$ of $(\Pi +U_0)\times \{0\}$ such that  
 \be\label{sup-pnk}
\sup_{k \in \Z} \sup_{I \in V} \| \partial_I r_k(I) \|_{(\ell^{1, 4})^*} = \sup_{k \in \Z} \sup_{I \in V} \sup_{n \in \Z}\frac{|\partial_{I_n} r_k(I)|}{\langle n \rangle^4} \leq C\,.
 \ee
(Here we used that  the dual space of $\ell^{1, 4}$ is $\ell^{\infty , - 4}$.)
 Defining $p_{nk} : = \langle n \rangle^{- 4} \partial_{I_n} r_k$ we have, by Lemma \ref{interpolation product},
\be\label{stim-int-pnk}
 \| \partial_I r_k(I_t)\cdot \Delta_{12} I\|_s \leq_s \sum_{n \in \Z} \| p_{n k} \circ I_t \|_s \langle n \rangle^4 \| \Delta_{12} I_n \|_{s_0} + \| p_{n k} \circ I_t \|_{s_0} \langle n \rangle^4 \| \Delta_{1 2} I_n \|_{s}\,.
\ee
 Moreover, by \eqref{sup-pnk}, 
arguing as in the proof of the estimate \eqref{estimateFrequ 2}, we get 
 \begin{equation}\label{p nk circ It}
 \| p_{n k} \circ I_t  \|_s   \stackrel{}{\leq_s} 1 + {\rm max}_{s + 2 s_0}(\io)\,.
 \end{equation}
 Combining the estimates \eqref{definizione Delta 12 rk} - \eqref{p nk circ It} 
with the smallness assumption \eqref{ansatz 1}
then yields 
 \begin{align}
 \| \Delta_{12} r_k\|_s & \leq_s {\rm max}_{s + 2 s_0}(\io) \sum_{n \in \Z}  \langle n \rangle^4 \|\Delta_{1 2} I_n \|_{s_0} +  \sum_{n \in \Z} \langle n \rangle^4 \|\Delta_{1 2} I_n \|_{s} \nonumber\\
 & \leq_s  \| \Delta_{12} y \|_s + {\rm max}_{s + 2 s_0}(\io) \| \Delta_{12} y \|_{s_0} + \sum_{n \in S^\bot} \langle n \rangle^4 \| \Delta_{12}( z_n \bar z_n) \|_s  +  {\rm max}_{s + 2 s_0}(\io)\sum_{n \in S^\bot} \langle n \rangle^4 \| \Delta_{12}( z_n \bar z_n) \|_{s_0}\,. \nonumber
 \end{align}
 Since 
 \begin{align}
& \sum_{n \in S^\bot} \langle n \rangle^4 \| \Delta_{12} (z_n \bar z_n) \|_s  \leq_s \sum_{n \in S^\bot } \langle n \rangle^4 \big( \| \bar z_n^{(1)} \Delta_{1 2} z_n  \|_s + \| z_n^{(2)} \Delta_{12} \bar z_n \|_s \big) \nonumber\\
 & \leq_s \sum_{n \in S^\bot} \langle n \rangle^4 \big( \| \Delta_{12} z_n \|_s \|  z_n^{(1)} \|_{s_0} + \| \Delta_{12} z_n \|_{s_0} \|  z_n^{(1)} \|_s + \| \Delta_{12} z_n \|_s \|  z_n^{(2)} \|_{s_0} + \| \Delta_{12} z_n \|_{s_0} \|  z_n^{(2)} \|_s  \big) \nonumber
\end{align}
one then gets by Cauchy-Schwartz, the smallness assumption \eqref{ansatz 1}, and  the assumption 
$\s \ge 4$
\begin{align*}
& \sum_{n \in S^\bot} \langle n \rangle^4 \| \Delta_{12} (z_n \bar z_n) \|_s
\leq_s\e \gamma^{- 2}\| \Delta_{12} z \|_s  
+ {\rm max}_s(z) \| \Delta_{12} z \|_{s_0}\,. 
 \end{align*}
 Altogether we proved that for any $k \in S^\bot$,
 \begin{equation}\label{stima Delta 12 rk I pippa}
 \| \Delta_{12} r_k \|_s \leq_s \| \Delta_{12} \io \|_s + {\rm max}_{s + 2 s_0}(\io) \| \Delta_{12} \io \|_{s_0}\,,
 \end{equation}
 implying, together with \eqref{A Delta 12 rk I}, that 
 $$
 |A|_{s, \sigma} \leq_s \| \Delta_{12} \io \|_{s }  + {\rm max}_{s + 2 s_0}(\io) \| \Delta_{12} \io \|_{s_0}\,.
 $$
Item $(i)$ then follows in combination with \eqref{def Omega io 1 - Omega io 2}, \eqref{stima primo pezzo Delta 12 Omega}.  
 
 \medskip
 
 \noindent
 $(ii)$   Since the claimed estimates for $ \Delta_{12} \frak R^{nls}_1 \lla D \rra $ and $ \Delta_{12} \frak R^{nls}_2 \lla D \rra $
 are obtained in the same way, we only consider $ \Delta_{12} \frak R^{nls}_1 \lla D \rra $. Recall that by \eqref{RNLS1D}, the operator $\frak R^{nls}_1 \lla D \rra $ can be written as
$$
\frak R^{nls}_1 \lla D \rra  = \sum_{j \in S^\bot} A_{(j)} \pi_j \, , \qquad 
A_{(j)}(\vphi) := \big(  z_k (\vphi) \langle j \rangle^2 f_{kj}(I(\vphi))) \bar z_j (\vphi)  \lla j \rra \big)_{k \in S^\bot} 
$$
where 
$\pi_j$ denotes the projector 
introduced in \eqref{def:pi-bot}  and 
$ f_{kj} (I) $  is defined in \eqref{def:fkj}. 

Then we have $| \Delta_{12} \frak R^{nls}_1 \lla D \rra |_{s, \sigma - 1 }  \leq $ $
\sum_{j \in S^\bot} |\Delta_{12} A_{(j)} \pi_j|_{s, \sigma - 1 }$.
Since
$$
|\Delta_{12} A_{(j)} \pi_j|_{s, \sigma - 1 }  = 
\big( \sum_{\ell } \langle \ell\rangle^{2 s} 
\| \Delta_{12} \hat A_{(j)}(\ell) \pi_j\|_{{\cal L}(h^{\sigma - 1}_\bot)}^2 \big)^{\frac12}, \
\|\Delta_{12} \hat A_{(j)}(\ell) \pi_j \|_{{\cal L}(h^{\sigma - 1}_\bot)} = 
\| \Delta_{12} \hat A_{(j)} (\ell) \|_{\sigma - 1} \langle j \rangle^{- (\sigma - 1)}
$$ 
one concludes in view of the property   \eqref{norma other} of the $ \| \ \|_s $-norm that
\begin{equation}\label{stima Delta 12 A (j)}
|\Delta_{12} A_{(j)} \pi_j|_{s, \sigma - 1 } = 
\langle j \rangle^{- (\sigma - 1)} 
\Big( \sum_{\ell, k}  \langle \ell \rangle^{2 s} \langle k \rangle^{2( \sigma - 1)} 
|\Delta_{12} \hat A_{(j),k}(\ell)|^2 \Big)^{\frac12} =  
 \langle j \rangle^{- (\sigma - 1)}  \| \Delta_{12} A_{(j)} \|_{s, \sigma - 1}\,.
\end{equation}

\smallskip
\noindent
To estimate $\|\Delta_{12} A_{(j)}\|_{s, \sigma - 1}$, let 
$\Delta_{12} f_{k j} := f_{kj}(I^{(1)}) - f_{kj }(I^{(2)})$ and write  
$\Delta_{12} A_{(j)}$ as a telescoping sum,
\begin{equation}\label{splitting Delta 12 A (j)}
\Delta_{12} A_{(j)} = B_{(j)} + C_{(j)} + D_{(j)}   
\end{equation}
where
\begin{align*}
& B_{(j)} := \big( \langle j \rangle^2  z_k^{(1)} \bar z_j^{(1)} \lla j \rra \Delta_{12} f_{kj} \big)_{k \in S^\bot}\,, \quad
C_{(j)} := \big( \langle j \rangle^2  f_{kj}(I^{(2)})  \bar z_j^{(1)} \lla j \rra \Delta_{12} z_k \big)_{k \in S^\bot}\,,  \\
&
\qquad \qquad\qquad\qquad
 D_{(j)} := \big( \langle j \rangle^2  f_{kj}(I^{(2)}) z_k^{(2)}  \lla j \rra  \Delta_{12} \bar z_j\big)_{k \in S^\bot}\,.
\end{align*}
We estimate the $\| \cdot \|_{s, \sigma - 1}$ norm of the above three terms separately. 
Actually, we estimate the larger  norm $\| \cdot \|_{s, \sigma}$ of these terms. One has 
\begin{align}
\| B_{(j)}\|_{s, \s}^2 & \leq_s \langle j \rangle^6 \sum_{k \in S^\bot} \langle k \rangle^{2 \sigma} \| z_k^{(1)} \bar z_j^{(1)}  \Delta_{12} f_{kj} \|_s^2 \nonumber\\
& \leq_s \langle j \rangle^6 \sum_{k \in S^\bot} \langle k \rangle^{2 \sigma } \Big( \| \Delta_{12} f_{k j} \|_s^2 \|z_j^{(1)} \|_{s_0}^2 \| z_k^{(1)}\|_{s_0}^2 + \| \Delta_{12} f_{kj} \|_{s_0}^2 \| z_j^{(1)}\|_s^2 \| z_k^{(1)}\|_{s_0}^2 \nonumber\\
& \qquad + \| \Delta_{12} f_{kj}\|_{s_0}^2 \| z_j^{(1)} \|_{s_0}^2 \| z_k^{(1)} \|_s^2 \Big)\,. \nonumber
\end{align}
The term $\Delta_{12} f_{kj}$ can be estimated in the same way as $\Delta_{12} r_k$ of item $(i)$, together 
with \eqref{tame composizione derivate ennesime frequenza} of Proposition \ref{stime derivate H nls}, obtaining  
$$
\| \Delta_{12} f_{kj}\|_s \leq_s \| \Delta_{12} \io \|_s + {\rm max}_{s + 2 s_0}(\io) \| \Delta_{12} \io \|_{s_0}\,.
$$
Hence by the smallness condition \eqref{ansatz 1},
\begin{align}
\| B_{(j)}\|_{s, \s}^2 & \leq_s \Big( \| \Delta_{12} \io \|_s^2  + {\rm max}_{s + 2 s_0}(\io)^2 \| \Delta_{12} \io\|_{s_0}^2 \Big)  \langle j \rangle^6 \|z_j^{(1)} \|_{s_0}^2 \sum_{k \in S^\bot} \langle k \rangle^{2 \sigma} \| z_k^{(1)}\|_{s_0}^2 \nonumber\\
& \quad + \langle j \rangle^6 \| z_j^{(1)}\|_s^2  \|\Delta_{12} \io \|_{s_0}^2 \sum_{k \in S^\bot} \langle k \rangle^{2 \sigma} \| z_k^{(1)}\|_{s_0}^2 + \langle j \rangle^6 \| z_j^{(1)}\|_{s_0}^2  \|\Delta_{12} \io \|_{s_0}^2 \sum_{k \in S^\bot} \langle k \rangle^{2 \sigma} \| z_k^{(1)}\|_{s}^2 \nonumber\\
& \leq_s \Big( \| \Delta_{12} \io \|_s^2  + {\rm max}_{s + 2 s_0}(\io)^2 \| \Delta_{12} \io\|_{s_0}^2 \Big)  
\langle j \rangle^6 \|z_j^{(1)} \|_{s_0}^2 \| z^{(1)} \|_{s_0, \s}^2 + 
\langle j\rangle^6 \|z_j^{(1)} \|_s^2 \| \Delta_{12} \io\|_{s_0}^2 \| z^{(1)} \|_{s_0, \s}^2 \nonumber\\
& \quad + \langle j\rangle^6 \|z_j^{(1)} \|_{s_0}^2 \| \Delta_{12} \io\|_{s_0}^2 \| z^{(1)} \|_{s, \s}^2\,, \nonumber
\end{align}
implying together with \eqref{ansatz 1} that
\begin{equation}\label{stima B (j)}
\| B_{(j)}\|_{s, \s} \leq_s \langle j \rangle^3 \Big( \e \gamma^{- 2} \| z_j^{(1)}\|_{s_0} \|\Delta_{12} \io \|_s  + 
 \big( \| z_j^{(1)}\|_{s} + {\rm max}_{s + 2 s_0}(\io) \, \| z_j^{(1)}\|_{s_0}  \big) 
\| \Delta_{12} \io\|_{s_0}  \Big)\,.
\end{equation}
Since by \eqref{come voluta}, $\| f_{kj } \circ I \|_s \leq_s 1 + \| \io\|_{s + 2 s_0}$, one can prove in a similar way that 
\begin{align}\label{stima C (j)}
& \| C_{(j)}\|_{s, \s} \leq_s \langle j \rangle^3 \big( \| z_j^{(1)} \|_{s_0} \| \Delta_{12} \io \|_s + 
 \big( \| z_j^{(1)}\|_{s} + {\rm max}_{s + 2 s_0}(\io) \, \| z_j^{(1)}\|_{s_0} \big) 
\| \Delta_{12} \io\|_{s_0} \big)\,, \\
& \label{stima D (j)}
\| D_{(j)}\|_{s, \s} \leq_s \langle j \rangle^3 \big( \e \gamma^{- 2} \| \Delta_{12} z_j \|_s + 
 {\rm max}_{s + 2 s_0}(\io) \, \| \Delta_{12} z_j\|_{s_0} \big)\,.
\end{align}
When combined,  the above three estimates yield  
\begin{align}
|\Delta_{12} \frak R_1^{nls} & \lla D \rra |_{s, \sigma - 1}  
\leq \sum_{j \in S^\bot} |\Delta_{12} A_{(j)} \pi_j|_{s, \sigma - 1 } 
\stackrel{\eqref{stima Delta 12 A (j)}}{\leq}  \sum_{j \in S^\bot} \langle j \rangle^{- (\sigma - 1)}  \| \Delta_{12} A_{(j)} \|_{s, \sigma - 1} \nonumber\\
& \stackrel{\eqref{splitting Delta 12 A (j)}}{\leq} \sum_{j \in S^\bot} \langle j \rangle^{- (\sigma - 1)} 
(\| B_{(j)}\|_{s, \s} + \| C_{(j)}\|_{s, \s} + \| D_{(j)}\|_{s, \s} ) \nonumber\\
& \stackrel{\eqref{stima B (j)}, \eqref{stima C (j)}, \eqref{stima D (j)}}{\leq_s} \sum_{j \in S^\bot} \langle j \rangle^{4 - \sigma} \Big( \| z_j^{(1)}\|_{s_0} \|\Delta_{12} \io \|_s  + 
\big( \| z_j^{(1)}\|_{s} + {\rm max}_{s + 2 s_0}(\io) \, \| z_j^{(1)}\|_{s_0}  \big) 
\| \Delta_{12} \io\|_{s_0}  \Big) \nonumber\\
& \qquad \qquad  + \sum_{j \in S^\bot} \langle j \rangle^{4 - \sigma} \big( \e \gamma^{- 2} \| \Delta_{12} z_j \|_s +{\rm max}_{s +2  s_0}(\io) \, \|\Delta_{12} z_j\|_{s_0} \big) \, . \nonumber
\end{align}
By the assumption $\sigma \geq 4$ and the smallness condition \eqref{ansatz 1} 
 the claimed estimate then follow.
\end{proof}
\begin{remark}
Arguing as in the proof of  Lemma \ref{lemma variazione i Omega nls} (i), one can also obtain
an estimate for $r_k(\xi + y, z \bar z) - r_k(\xi, 0)$, which we record for later reference:
by  the mean value theorem, one has
$$
 r_k(\xi + y, z \bar z) - r_k(\xi, 0) = \int_0^1 \partial_I r_k(I_t) dt \cdot (y, z\bar z)\quad \mbox{with}
\quad I_t = (\xi, 0) + t(y, z \bar z ), \quad z \bar z = (z_j \bar z_j)_{j \in S^\bot}\,.
$$
By Theorem \ref{Corollary 2.2} (dNLS frequencies), and using \eqref{ansatz 1}, one has 
$\langle n \rangle^{-4} | \partial_{I_n} r_k(I_t) | 
\lessdot 1$. Then, from Lemma \ref{Lemma 2.4bis}
(tame estimates for composition), it follows that
$\| r_k(\xi + y, z \bar z) - r_k(\xi, 0)\|_s \leq_s  \| \io \|_{s + 2 s_0}$, using also \eqref{ansatz 1}.
By similar arguments one can verify a corresponding bound for 
$\| r_k(\xi + y, z \bar z) - r_k(\xi, 0)\|_s^{\rm lip}$. 
Under the same assumptions as in Lemma \ref{I bot Lip gamma} one obtains in this way the estimate
\be\label{Taylor estimate for rk}
\| r_k(\xi + y, z \bar z) - r_k(\xi, 0)\|_s^\Lipg \leq_s \| \io \|_{s + 2 s_0}^\Lipg\,.
\ee
\end{remark}

\noindent
{\it Analysis of ${\frak S}^P$.}
In this paragraph it is convenient to denote by
$\widetilde X_{{\cal P}}$ the vector field obtained from the Hamiltonian vector field 
$- \ii \nabla_{\bar u} {\cal P}$ by adding 
its complex conjugate as a second component, 
$\widetilde X_{{\cal P}} := (- \ii \nabla_{\bar u} {\cal P}, \ii \nabla_{u} {\cal P})$. 
We denote by $\widetilde X_P$ the Hamiltonian vector field  $\widetilde X_{{\cal P}}$,
when expressed in Birkhoff coordinates,
\be\label{signXP}
\widetilde X_P : = (d \Phi   \widetilde X_{\cal P} )_{| \Phi^{-1}} \, , \quad P = {\cal P} \circ \Phi^{- 1}\,,
\ee
where $\Phi = \Phi^{nls}$ is the Birkhoff map of Theorem \ref{Theorem Birkhoff coordinates}. 
Recall that
 $F_{nls}$ denotes the version of the Fourier transform, introduced in 
\eqref{Fourier transform F nf}. Denote its inverse by $F_{nls}^{-1}$.
Using that by Theorem  \ref{Theorem Birkhoff coordinates}, $\Phi= F_{nls} + A^{nls} $ and 
$ \Phi^{-1}  = F_{nls}^{-1} + B^{nls} $, 
the differential of $\widetilde X_{P}$ can be computed as
\be\label{dXP}
d \widetilde X_P  = 
 F_{nls} \, (d \widetilde X_{\cal P})_{|\Phi^{-1}} \,  F_{nls}^{-1}  - J \big( T_1 + T_2 + T_3  \big)
\ee
with  
$$
T_1 :=  J  F_{nls} \big(d \widetilde X_{\cal P}\big)_{|\Phi^{-1}}   d B^{nls}, \quad 
T_2 :=  J (d A^{nls} \, d \widetilde X_{\cal P})_{|\Phi^{-1}}  d \Phi^{-1}, \quad 
T_3 :=  J  (d^2A^{nls})_{|\Phi^{-1}} \, \big(  d \Phi^{-1}( \cdot ), (\widetilde X_{\cal P})_{|\Phi^{-1}} \big)  \, .  
$$
By \eqref{nonlin:f}, one has 
$\widetilde X_{{\cal P}} = (- \ii f(x, u), \ii \, \overline f(x, u))$
with $f(x, u(x)) = \partial_{\bar \zeta} p_{| \zeta = u(x)}$ and hence
the differential $d \widetilde X_{\cal P}$ of $\widetilde X_{\cal P}$ is given by 
\be\label{def:Q-new}
d \widetilde X_{\cal P} = - J {\cal Q} \, , 
\qquad {\cal Q} := \left(
\begin{array}{cc}
 \partial_{u} f &  \partial_{\bar u} f   \\  
 \  \overline{\partial_{\bar u} f} &  \overline{\partial_{u} f}
\end{array}
\right) = 
\left(
\begin{array}{cc}
 \partial_{\zeta} \pa_{\bar \zeta} p  &  \partial_{\bar \zeta} \pa_{\bar \zeta} p   \\  
 \  \overline{ \partial_{\bar \zeta} \pa_{\bar \zeta} p} &  \overline{ \partial_{\zeta} \pa_{\bar \zeta} p} 
\end{array}
\right)_{| \zeta = u(x) } \, .
\ee
Since $\partial_{\zeta} \pa_{\bar \zeta} = \frac{1}{2}(\partial_{\zeta_1}^2 +  \pa_{\zeta_2}^2),$
 the function  $\partial_{\zeta} \pa_{\bar \zeta} p$ is real valued
whereas  by a similar computation,
$ \partial_{\zeta}   \pa_{\zeta} p $ is the complex conjugate of $ \partial_{\bar \zeta} \pa_{\bar \zeta} p $. 
Thus, by \eqref{dXP} and since $ F_{nls} $ and $ J $ commute,
\be\label{dXP-new}
 d \widetilde X_{P} =   - J \,\big( F_{nls} \,  {\cal Q}_{|\Phi^{-1}} \,  F_{nls}^{-1}  +  T_1 + T_2 + T_3 \big)  \, . 
\ee
We now evaluate $ d \widetilde X_P $ at the embedding $ {\breve \io} (\vphi ) $.
In view of the definition \eqref{NLS normal direction} of ${\frak S}^P$,  \eqref{dXP-new} and \eqref{def:Q-new} we get 
\be\label{lin:pert}
{\frak S}^P = {\frak Q}_\bot  + {\frak R}^P\,, \qquad {\frak Q}_\bot :=   F_{nls}^\bot 
\left(
\begin{array}{cc}
q_1 & q_2   \\
{\bar q}_2 & q_1  \\
\end{array}
\right) F_{nls}^{-1}\,, 
\ee
 where $ F_{nls}^\bot  ,  F_{nls}^{-1} $ were introduced in  \eqref{proiettori trasformata di fourier} and  
\begin{equation}\label{definition q1 q2 R}
q_1 :=  (\partial_{\zeta} \pa_{\bar \zeta} p)_{|\zeta =\Phi^{-1}( \breve \io)} \, , \qquad 
q_2 := (\partial_{\bar \zeta} \pa_{\bar \zeta} p)_{|\zeta =\Phi^{-1}( \breve \io)}\, , \qquad 
{\frak R}^P :=  {\mathbb I}_\bot \big( (T_1 +  T_2 +  T_3) \circ \breve \io  \big)  
\mathbb I_{\hookrightarrow} \, , 
\end{equation}
with ${\mathbb I}_\bot$ denoting the projector  and 
$\mathbb I_{\hookrightarrow} $  the standard inclusion  introduced in \eqref{def:op-bot}. 
Above, in defining $ \Phi^{-1} ( \breve \io ) $
we have identified, by a slight abuse
of terminology,  the two components $ \big( \theta(\vphi), y(\vphi) \big)$ 
of ${\breve \io} (\vphi ) $ with the Birkhoff coordinates 
 $(z_j(\vphi))_{j \in S} := (\sqrt{ \xi_j + y_j} e^{-\ii \theta_j})_{j \in S} \in \C^S$.

\begin{lemma}\label{stima lip q1 q2 RP}
{\bf (Estimates for $ q_1, q_2, $ and $ {\frak R}^P $) }
For any $s_0 \leq s \leq s_* - 2 s_0$ 
the following statements hold:  

\noindent
$(i)$ The functions $ q_1, q_2$ are in $H^s (\T^S, H^\sigma(\T_1) ) $, with  
$ q_1$ real-  and $q_2$ complex-valued. They satisfy 
\begin{equation}\label{estimates q1 q2}
\| q_1\|_s\,, \| q_2\|_s \leq_s 1 + \|  \io \|_{s + s_0}\, , \qquad 
\| q_1\|_s^\Lipg\,, \| q_2\|_s^\Lipg \leq_s 1 + \|  \io \|_{s + s_0}^\Lipg\,.
\end{equation}
\noindent
$(ii)$ The remainder $ {\frak R}^P $ defined in \eqref{definition q1 q2 R} satisfies
\be\label{lemma4.2}
 | {\frak R}^P {\mathfrak D} |_{s, \sigma - 1} \leq_s
 1+ \| \io \|_{s+  2 s_0}   \, , \quad 
 | {\frak R}^P {\mathfrak D} |_{s, \sigma - 1}^\Lipg \leq_s
 1+ \| \io \|_{s+  2 s_0}^\Lipg  \, . 
\ee
\end{lemma}

\begin{proof} {$(i)$} The bounds \eqref{estimates q1 q2} 
follow by 
the definition \eqref{definition q1 q2 R} of $q_1$ and $q_2$, 
the regularity assumption \eqref {regolarita di p}  of $\pa_{\bar \zeta} p$, 
 and the tame estimates for the composition of maps of 
Lemma \ref{Lemma 2.4bis} in the case where $Y = \C$.

\noindent
$(ii)$
We now prove the first estimate in \eqref{lemma4.2}.
According to Theorem~\ref{Theorem Birkhoff coordinates}, 
the maps $ A^{nls} $, $ B^{nls} $ are real analytic and one smoothing: for any $\s' \in \mathbb Z_{\ge 2},$
\begin{align*}
& A^{nls}  : H^{\s' -1}_r \to h^{\s'}_r \, , \quad 
B^{nls} : h^{\s' -1}_r \to H^{\s'}_r \, .
\end{align*}
By Cauchy's theorem it then follows that
\begin{align*}
& d A^{nls}   : H^{\s'-1}_r \to {\cal L}(H^{\s'-1}_r, \, h^{\s'}_r) \, , 
\quad d B^{nls}  : h^{\s' -1}_r \to {\cal L}(h^{\s' -1}_r , \, H^{\s'}_r ) \, ,
\end{align*} 
and $d^2 A^{nls}   : H^{\s'-1}_r \to {\cal L}(H^{\s'-1}_r \times H^{\s'-1}_r, \,h^{\s'}_r)$
are ${\cal C}^\infty$-smooth maps. 
It follows that $ T_1 {\frak D} $, $ T_2 {\frak D} $, $ T_3 {\frak D} $ 
are maps from the phase space $ M^\s $ into $ {\cal L}(h^{\s'})$ for $\s' \in \{\s,  \s -1, \s-2 \}$ 
which are as smooth as the second derivatives of $ p $. 
We now apply the estimate \eqref{composition formula} for the composite map 
$ \vphi \mapsto \breve \io (\vphi) \mapsto T_j ( {\breve \io} (\vphi) ) $, $ j =1, 2, 3 $, which yields
$$
| T_j {\frak D} \circ {\breve \io} |_{s, \sigma - 1}  \leq_s 1 + \| { \io} \|_{s+ 2 s_0}\, ,
$$
and hence \eqref{lemma4.2} is proved. 
The second  estimate in \eqref{lemma4.2} is proved in a similar way. 
\end{proof}

\begin{lemma}\label{estimates from the perturbation lip}
For any $s_0 \leq s \leq s_* - 2 s_0$ and any torus embeddings 
$\breve \io^{(a)}(\vphi) := (\vphi, 0, 0) + \io^{(a)} (\vphi)$, $a = 1, 2$, satisfying \eqref{ansatz 1}, the following holds: 

\noindent
$(i)$ The functions $\Delta_{12} q_1 := q_1(\breve \io^{(1)}) - q_1 (\breve \io^{(2)}) $ and $\Delta_{12} q_2 := q_2 (\breve \io^{(1)}) - q_2 (\breve \io^{(2)})$ satisfy the estimate 
\begin{equation}\label{estimates Delta 12 q1 q2}
\| \Delta_{12} q_1\|_s\,, \,\,\| \Delta_{12} q_2\|_s \leq_s \|\Delta_{12} \io \|_{s + s_0} + {\rm max}_{s + s_0}(\io) \| \Delta_{12} \io\|_{s_0}\,.
\end{equation}

\noindent
$(ii)$ The difference of the remainders, 
$\Delta_{12} {\frak R}^P := {\frak R}^P(\breve \io^{(1)}) - {\frak R}^P(\breve \io^{(2)})$, 
satisfies the estimate 
$$
|\Delta_{12} {\frak R}^P {\frak D}|_{s, \sigma - 1} \leq_s \|\Delta_{12} \io \|_{s + 2 s_0} + {\rm max}_{s + 2 s_0}(\io) \| \Delta_{12} \io\|_{s_0}\,.
$$
\end{lemma}
\begin{proof}
Items $(i)$ and $(ii)$ follow 
from the definition \eqref{definition q1 q2 R}, Lemma \ref{Lemma 2.4bis}$(ii)$ and 
Lemma \ref{composition in Sobolev}$(ii)$. 
\end{proof}

\smallskip

\noindent
{\it Analysis of ${\frak R}^\e$.}
The operator ${\frak R}^\e$, introduced in \eqref{formaK20}, is defined in terms of
the operators $\frak R^\varepsilon_1 = R^\varepsilon_1 \circ \breve \io$ and 
$\frak R^\varepsilon_2 = R^\varepsilon_2 \circ \breve \io$, where according to  
\eqref{R 1 epsilon tilde}, \eqref{R 2 epsilon tilde}
$$
 R^\varepsilon_1  =  \partial_y (\nabla_{\bar z} H_\e ) Y_w + 
  Y_{ \bar w}^t \partial_z \nabla_{ y} H_\e +  Y_{ \bar w}^t \partial_y (\nabla_y H_\e) Y_w 
  \, , \,
   R^\varepsilon_2  = \partial_y (\nabla_{ \bar z } H_\e) Y_{\bar w} + Y_{\bar w}^t 
   \pa_{ \bar z } \nabla_y H_\e   +  Y_{\bar w}^t \partial_y (\nabla_y  H_\e) Y_{\bar w} 
$$
and $ Y_w $ is defined in \eqref{operator Yw}.

\begin{lemma}\label{resto finito dimensionale stima lipschitz}
{\bf (Estimate of $ {\frak R}^\varepsilon $) } For any $s_0 \leq s \leq s_* - 2 s_0$
one has
\be\label{stima-frak-R-ep}
| \frak R^\varepsilon \frak D |_{s, \s-1} \leq_s \e \gamma^{- 2} \| \io \|_{s+2 s_0} \, , \quad 
|{\frak R}^\varepsilon {\frak D}|_{s, \sigma - 1}^\Lipg \leq_s \e \gamma^{- 2}\| \io \|_{s + 2 s_0}^\Lipg\,.
\ee
\end{lemma}
\begin{proof}
We now prove the first bound in \eqref{stima-frak-R-ep}.
The various terms in ${\frak R}^\e_1$ and ${\frak R}^\e_2$ are estimated individually. 
Since these terms can be estimated in a similar way, let us concentrate on  
$(\partial_y \nabla_{\bar z} H_\e Y_w) \circ \breve \io$ only. Recall that by  \eqref{operator Yw}, 
$$
Y_w ( \breve  \io(\vphi))   := \ii  B(\vphi) (\partial_\vphi {\bar z})^t(\vphi)  : h^\sigma_\bot \to \C^S\,, \quad B(\vphi) := (\partial_\vphi \theta(\vphi))^{- t} \, , 
$$
and, since  $ (\pa_\vphi \bar z)^t = \sum_{m \in S^\bot} \pa_\vphi \bar z_m \pi_m  $ where $\pi_m $ is the projector defined in \eqref{def:pi-bot}, we have 
$$
\partial_y (\nabla_{\bar z} H_\e) Y_w  = \ii \sum_{m \in S^\bot} \sum_{j, k \in S} 
\partial_{y_j} \nabla_{\bar z} H_\e B_j^k \partial_{\vphi_k} \bar z_m \pi_m \, . 
$$
Clearly, recalling \eqref{def:DD}, one gets 
\begin{equation}\label{ufpea - 1}
| \partial_y (\nabla_{\bar z} H_\e) Y_w \lla D \rra  {\rm Id}_\bot |_{s, \sigma - 1} \leq 
\sum_{m \in S^\bot} \sum_{j, k \in S} 
\big| \partial_{y_j} \nabla_{\bar z} H_\e B_j^k \partial_{\vphi_k} \bar z_m \lla m \rra  \pi_m \big|_{s, \sigma - 1}\,.
\end{equation}
Arguing as in \eqref{RNLS1Ds} one concludes that 
\begin{align}
|\partial_{y_j} \nabla_{\bar z} H_\e B_j^k \partial_{\vphi_k} \bar z_m \lla m \rra  \pi_m |_{s, \sigma - 1} & 
\leq_s \langle m \rangle^{-(\sigma - 1)}  \| \partial_{y_j} \nabla_{\bar z} H_\e B_j^k \partial_{\vphi_k} \bar z_m 
\lla m \rra \|_{s, \sigma - 1} \nonumber\\
& \leq_s \langle m \rangle^{-(2 \sigma - 2)}  \| \partial_{y_j} \nabla_{\bar z} H_\e B_j^k \partial_{\vphi_k}( \langle m \rangle^\sigma \bar z_m) \|_{s, \sigma - 1} \label{ufpea 0}\,.
\end{align}
 Since $B(\vphi) = (\partial_\vphi \theta(\vphi))^{- t}$ one has $\| B^k_j\|_s \leq_s 1 + \| \io\|_{s + 1}$.
Furthermore, for any $m \in S^\bot$ and $k \in S$,
 $\| \partial_{\vphi_k}(\langle m \rangle^\sigma z_m) \|_s \leq_s \| \io\|_{s + 1}$. Finally we analyze 
 $$
 \partial_y \nabla_{\bar z} H_\e =  \partial_y \nabla_{\bar z} H^{nls} + \e  \partial_y \nabla_{\bar z} P\,.
 $$  
 Note that 
 $ \partial_{y_j} \nabla_{\bar z} H^{nls} = (\partial_{y_j} \omega^{nls}_n  z_n)_{n \in S^\bot} $. 
By \eqref{tame composizione derivate ennesime frequenza}, one has that  
 $$
\sup_n  \| \partial_{y_j} \omega^{nls}_n \|_s \leq_s 1 + \| \io \|_{s + 2 s_0}\, , \quad \forall j \in S \, .
 $$
By the tame estimates for products of maps and the smallness assumption \eqref{ansatz 1} one then concludes that
 \begin{equation}\label{ufpea 1}
 \big\| \big( \partial_{y_j} \omega_n^{nls}  z_n  \big)_{n \in S^\bot}  
 B_j^k \partial_{\vphi_k} (\langle m \rangle^\sigma \bar z_m) 
  \big\|_{s, \s} \leq_s  \e \gamma^{- 2}  \| \io\|_{s + 2 s_0}\, , \quad \quad \forall j, k \in S, \ m \in S^\bot \, .
 \end{equation}
 Next we consider $\partial_{y_j} \nabla_{\bar z} P$. 
By Proposition  \ref{teorema stime perturbazione}, 
 $$
 \| \partial_{y_j} \nabla_{\bar z} P \circ \breve \io \|_{s, \sigma} \leq_s  1 + \| \io\|_{s + 2 s_0}\, ,
 $$
that, together with the smallness assumption \eqref{ansatz 1},  yields the estimate
 \begin{equation}\label{ufpea 2}
 \| \big( \partial_{y_j} \nabla_{\bar z} \e P \circ \breve \io \big) 
 B_j^k \partial_{\vphi_k}(\langle m \rangle^\sigma \bar z_m)  \|_{s, \sigma} 
\leq_s \e \gamma^{- 2} \| \io\|_{s  + 2 s_0}\,, \quad \forall j, k \in S, \ m \in S^\bot  \,.
 \end{equation}
Combining \eqref{ufpea - 1}, \eqref{ufpea 0}, \eqref{ufpea 1}, \eqref{ufpea 2} we get the claimed estimate for the term $\partial_y \nabla_{\bar z} H_\e Y_w$. 
The second estimate in \eqref{stima-frak-R-ep} follows in a similar way.
\end{proof}

\begin{lemma}\label{resto finito dimensionale variazione io}
For any $s_0 \leq s \leq s_* - 2 s_0$ 
and any torus embeddings $\breve \io^{(a)}(\vphi) = (\vphi, 0, 0) + \io^{(a)} (\vphi)$, $a = 1, 2$,
satisfying \eqref{ansatz 1}, the operator 
$\Delta_{12} {\frak R}^\varepsilon :=
 {\frak R}^\varepsilon(\breve \io^{(1)}) - {\frak R}^\varepsilon(\breve \io^{(2)})$ 
satisfies the estimate 
$$
|\Delta_{12} {\frak R}^\varepsilon \frak D |_{s, \sigma - 1} 
\leq_s  \e \gamma^{- 2}\| \Delta_{12} \io\|_{s + 2s_0}+ {\rm max}_{s + 2 s_0}(\io) \| \Delta_{12} \io \|_{ s_0}\,.
$$
\end{lemma}
\begin{proof}
The claimed estimate can be deduced by arguing as in the proofs of Lemma \ref{lemma variazione i Omega nls} and Lemma \ref{estimates from the perturbation lip}.   
\end{proof}

We summarize the results obtained in this subsection as follows. 

\begin{proposition}
The Hamiltonian operator $ {\mathfrak L}_\om $ (cf \eqref{Linearized op normal}) can be decomposed  
as  
\be\label{Operator before trans}
{\frak L}_\om = \om \cdot \partial_\vphi \, {\mathbb I}_2 + J \big( D^2 \, {\mathbb I}_2 + 
\Omega^{nls} \, {\mathbb I}_2 + \e 
 {\frak Q}_\bot \big) + {\frak R}_0 \,,  \quad {\mathbb I}_2 = {\rm diag}({\rm Id}_\bot, {\rm Id}_\bot)\,,
\ee 
 where $ \Omega^{nls}$ is defined in \eqref{definition Omega nls},  ${\frak Q}_\bot$ in \eqref{lin:pert}, and 
$$
{\frak R}_0 :=  
J {\frak R}^\e +  J {\frak R}^{nls} + \e J {\frak R}^P  
$$
with 
 $ {\frak R}^\e $ introduced in \eqref{formaK20},  ${\frak R}^{nls} $ in  \eqref{def:Rnls} and $ {\frak R}^P  $ in 
\eqref{definition q1 q2 R}.
The remainder $ {\frak R}_0$ is a linear Hamiltonian operator which 
 is one smoothing and satisfies, for any $  s_0 \leq s \leq s_* - 2 s_0 $, 
\be\label{small-remainder}
 | {\frak R}_0 \frak D |_{s, \sigma - 1} \leq_s \e + \e \g^{-2} \| \io \|_{s +  2 s_0} \, , \qquad 
| {\frak R}_0 {\frak D}|_{s, \sigma - 1}^\Lipg \leq_s \e + \e \gamma^{- 2}  
\| \io \|_{s + 2 s_0}^\Lipg\, . 
\ee
Moreover if $\breve \io^{(a)}(\vphi) = (\vphi, 0, 0) +\io^{(a)}(\vphi)$, $a = 1, 2$, are two torus embeddings
satisfying \eqref{ansatz 1}, 
then, 
$\Delta_{12} {\frak R}_0 := {\frak R}_0(\breve \io^{(1)}) - {\frak R}_0(\breve \io^{(2)})$ 
satisfies the estimate 
\begin{equation}\label{small-remainder Delta 12}
|\Delta_{12} {\frak R}_0  {\frak D}|_{s, \sigma - 1} \leq_s \e \gamma^{- 2}\| \Delta_{12} \io\|_{s + 2 s_0} + {\rm max}_{s + 2 s_0}(\io) \| \Delta_{12} \io\|_{s_0}\,, \quad \forall s_0 \leq s \leq s_* - 2 s_0\, .
\end{equation}
\end{proposition}

\begin{proof}
Lemmata \ref{I bot Lip gamma}, \ref{stima lip q1 q2 RP}, and
\ref{resto finito dimensionale stima lipschitz} yield the estimate \eqref{small-remainder}.
Lemmata \ref{lemma variazione i Omega nls}, 
\ref{estimates from the perturbation lip}, and \ref{resto finito dimensionale variazione io} imply 
\eqref{small-remainder Delta 12}.
\end{proof}

Note that   
the operator $\Omega^{nls} {\mathbb I}_2 : H^s(\T^S, h^{\sigma-1}_\bot \times h^{\sigma-1}_\bot ) \to H^s(\T^S, h^{\sigma-1}_\bot \times h^{\sigma-1}_\bot)$ is neither one smoothing nor small, whereas $\e {\frak Q}_\bot$, which acts between the same spaces, is small but not one smoothing. In the subsequent sections we will introduce three linear symplectic transformations so that,  when conjugated with these transformations, the operator 
$J (\Omega^{nls} {\mathbb I}_2 + \e {\frak Q}_\bot)$ becomes a diagonal one with constant coefficients 
up to a one smoothing remainder.  
Note also that the leading part $ J D^2 {\mathbb I}_2 $ in $ {\frak L}_\om $ is already 
a diagonal operator with constant coefficients.

\subsection{First transformation}\label{sec1}

The purpose of the first transformation is to eliminate the off diagonal terms of ${\frak Q}_\bot$in \eqref{Operator before trans} up to a one smoothing remainder. 
The transformation is chosen to be the time $ 1 $-flow 
${\mathtt \Phi}_1 : H^s(\T^S, h^{\sigma'}_\bot \times h^{\sigma'}_\bot) \to 
H^s(\T^S, h^{\sigma'}_\bot \times h^{\sigma'}_\bot)$, $\sigma' \in \{\sigma, \sigma - 1, \sigma- 2 \}$, 
$$
{\mathtt \Phi}_1 := \exp( - \e  J \, F_{nls}^\bot A_1 F_{nls}^{- 1}) = 
{\mathbb I}_2  - \e J \, F_{nls}^\bot A_1 F_{nls}^{- 1} + \ldots 
$$
of the linear vector field $ - \e J F_{nls}^\bot A_1 F_{nls}^{- 1}$ 
with $A_1$ of the form
\be\label{def: A1}
A_1 =  
\left(
\begin{array}{cc}
0 &\lla D \rra^{-1} a_1 \lla D \rra^{-1}   \\
\lla D \rra^{-1} \bar a_1 \lla D \rra^{-1}  & 0 
\end{array}
\right)\,, \quad \lla D \rra = (1 + D^2)^{\frac12} \, , \quad D = \frac{1}{\ii}\partial_x \, .
\ee
By Lemma \ref{lemma:HamiltonianVF}  
the operator $  J F_{nls}^\bot A_1 F_{nls}^{- 1}  $ is Hamiltonian   and hence the flow $ {\mathtt \Phi}_1 $ symplectic
(cf Definition~\ref{linear symplectic transformations}). 
Note that for any $ \vphi \in \T^S $, the operator 
$ A_1 (\vphi) $ is one smoothing (actually, it is even two smoothing)
and  the linear map $ {\mathtt \Phi}_1(\varphi) $ is invertible 
with inverse $ {\mathtt \Phi}_1^{-1} (\vphi) \equiv ({\mathtt \Phi}_1(\vphi))^{-1}$ given by
$ \exp (  \e J F_{nls}^\bot A_1(\varphi) F_{nls}^{- 1} )$. 
The form of the operator $ A_1 $ is chosen in such a way that  the coefficients of the
remainder $ R $ in 
\eqref{remainder-1-tra} below involve only $ \pa_{x} a_1 $, and hence, by  \eqref{choice p2}, $ \pa_x q_2 $.   

The complex valued function  $ a_1 \equiv a_1 (\vphi, x )$ will be chosen in such a way that 
the off-diagonal part in  ${\frak L}_1 :=   {\mathtt \Phi}_1^{-1} {\frak L}_\om {\mathtt \Phi}_1$
vanishes up to a one smoothing remainder.  
Note that the operators \,$\om \cdot \partial_\vphi \, {\mathbb I}_2$,\,
$JD^2 \, {\mathbb I}_2$, and  $J\Omega^{nls} \, {\mathbb I}_2$ in ${\frak L}_\om = \om \cdot \partial_\vphi \, {\mathbb I}_2 + J D^2 \, {\mathbb I}_2 + 
J \Omega^{nls} \, {\mathbb I}_2 + \e J {\frak Q}_\bot + {\frak R}_0$
 are diagonal whereas (cf \eqref{lin:pert}) 
\be\label{formula for Qbot}
J{\frak Q}_\bot =  J F_{nls}^\bot  
\left(
\begin{array}{cc}
 q_1 & q_2   \\
 {\bar q}_2 & q_1  \\
\end{array}
\right) \, F_{nls}^{- 1}  
\ee
is not and  ${\frak R}_0$ is one smoothing.
We then write ${\frak L}_\om {\mathtt \Phi}_1$ in the form
\be\label{First composition}
{\frak L}_\om {\mathtt \Phi}_1 =  {\mathtt \Phi}_1 \big( 
\om \cdot \partial_\vphi \, {\mathbb I}_2  
+ J  D^2  \, {\mathbb I}_2  + J \Omega^{nls} \, {\mathbb I}_2 \big) + \e J {\frak Q}_\bot   
- \e [J D^2  \, {\mathbb I}_2, \, J F_{nls}^\bot A_1 F_{nls}^{- 1} ] + {\frak R}^I 
\ee
where $[\, \cdot , \, \cdot]$ denotes the commutator of operators and 
$$
{\frak R}^I  := (\om \cdot \partial_\vphi)  \big( {\mathtt \Phi}_1 - {\mathbb I}_2 \big) +  
[J \Omega^{nls}{\mathbb I}_2, \, {\mathtt \Phi}_1 - {\mathbb I}_2] + 
\e J {\frak Q}_\bot   ({\mathtt \Phi}_1 - {\mathbb I}_2) +  {\frak R}_0 {\mathtt \Phi}_1 +
[ J D^2\, {\mathbb I}_2 , \, {\mathtt \Phi}_1 - {\mathbb I}_2 + 
\e J \, F_{nls}^\bot A_1 F_{nls}^{- 1} ]  
$$
collects operators which are one smoothing. 
We claim that the commutator
$ [J D^2  \, {\mathbb I}_2, \,  J F_{nls}^\bot A_1 F_{nls}^{- 1} ] $
 is  a Hamiltonian operator of order zero. Indeed, since $ J D^2 $ commutes 
 with $ J $, $  F_{nls}^\bot $ and $  F_{nls}^{-1} $, one has 
$$
[ J D^2  \, {\mathbb I}_2, J F_{nls}^\bot A_1 F_{nls}^{- 1} ] 
= J F_{nls}^\bot  [JD^2 , A_1] F_{nls}^{- 1} 
$$
and, recalling  \eqref{def: A1}, 
\begin{align}
[JD^2 , A_1] =  \ii \begin{pmatrix} 
0  & D^2 \lla D \rra^{- 1} a_1 \lla D \rra^{- 1} + \lla D \rra^{- 1} a_1 \lla D \rra^{- 1} D^2 \\
- D^2 \lla D \rra^{- 1} \bar a_1 \lla D \rra^{- 1} - \lla D \rra^{- 1} \bar a_1 \lla D \rra^{- 1} D^2 & 0
\end{pmatrix}\,. \nonumber
\end{align}
 Then, since 
 $ D^2 = \lla D \rra^2 - 1$, 
one has
\be\label{zero order}
 [J D^2  \, {\mathbb I}_2, J F_{nls}^\bot A_1 F_{nls}^{- 1} ]  
= J F_{nls}^\bot
\left(
\begin{array}{cc}
0 & 2 \ii a_1 \\
- 2 \ii \bar a_1 & 0 
\end{array} 
\right) F_{nls}^{- 1} - {\frak R}^{II}  \, , \quad 
{\frak R}^{II} :=  F_{nls}^\bot
\left(
\begin{array}{cc}
0 & R \\
\overline R & 0 
\end{array} 
\right) F_{nls}^{- 1} 
\ee
where 
\be\label{remainder-1-tra}
R =  2 \lla D \rra^{-1} a_1 \lla D \rra^{-1} 
- [a_1, \, \lla D \rra] \, \lla D \rra^{-1}  -
\lla D \rra^{-1}  [\lla D \rra, \, a_1 ] = R^t  \, .
\ee
Note that $ {\frak R}^{II} $ is one smoothing, but its coefficients involve 
$\pa_{x} a_1  \in H^{\s-1} $. 
In view of  \eqref{formula for Qbot}, we choose 
\be\label{choice p2}
a_1 := - \frac{\ii }{2 } q_2
\ee
so that by  \eqref{First composition}, \eqref{zero order} 
\be\label{expansion zero term}
 J {\frak Q}_\bot   
-  [J D^2  \, {\mathbb I}_2, J F_{nls}^\bot A_1 F_{nls}^{- 1}] 
=  J F_{nls}^\bot \begin{pmatrix}
q_1 & 0 \\
0 & q_1
\end{pmatrix} F_{nls}^{- 1}  + {\frak R}^{II} \, .
\ee
Applying $ {\mathtt \Phi}_1^{-1}$ to the identity \eqref{First composition} and using \eqref{expansion zero term} one gets
\be\label{op L1}
{\frak L}_1 =   {\mathtt \Phi}_1^{-1} {\frak L}_\om {\mathtt \Phi}_1 =  
\om \cdot \partial_\vphi {\mathbb I}_2 
+ J  \big(D^2  \, {\mathbb I}_2  +  \Omega^{nls} \, {\mathbb I}_2 
+ \e  F_{nls}^\bot q_1 F_{nls}^{- 1} \big)   + {\frak R}_1\,, 
\ee
where $ {\frak R}_1 $ is the one smoothing operator  
\be\label{formulaR1}
{\frak R}_1 :=  \e ({\mathtt \Phi}_1^{-1} - {\mathbb I}_2 ) 
 J F_{nls}^\bot q_1 F_{nls}^{- 1} 
+ {\mathtt \Phi}_1^{-1} \big( {\frak R}^I + \e {\frak R}^{II} \big) \, . 
\ee
Since  $ {\mathtt \Phi}_1 $ is symplectic and ${\frak L}_\om$ is a linear Hamiltonian operator,
Lemma~\ref{transformation of Hamiltonian operators} implies that also 
$ {\frak L}_1 $ is  Hamiltonian. 
Furthermore, the 0th order term of $ {\frak L}_1 $ is given by
$ J \big( \Omega^{nls} \, {\mathbb I}_2 + \e F_{nls}^\bot q_1 F_{nls}^{- 1} \big)   $ where
$ \Omega^{nls}$ is the $ \vphi $-
dependent diagonal operator defined in \eqref{definition Omega nls}. 
As pointed out above, the operator $ {\frak R}_1 $ is one smoothing, but its coefficients involve $\partial_x a_1$, i.e., they are maps
with values in $ h^{\s-1} $.

\begin{lemma}\label{lem:44 Lip}
{\bf (Estimates of $ A_1 $, $ {\mathtt \Phi}_1 $ and ${\frak R}_1$)}
For any $s_0 \leq s \leq s_* - 2 s_0$ 
the following statements hold:  

\noindent
$(i)$ 
For any $ \vphi \in \T^S $ and $  \s' \in \{  \s, \s - 1, \sigma - 2, \sigma - 3\} $, 
$A_1(\vphi) \in {\cal L}(H^{\sigma' - 1}, H^{\sigma'})$ and 
\begin{align}
& |J F_{nls}^\bot A_1 F_{nls}^{- 1} |_{s , \sigma'}\,,\,
| J F_{nls}^\bot A_1 F_{nls}^{- 1} {\frak D}|_{s , \sigma'} \leq_s 1 + \| \io\|_{s + s_0 } \label{pr-A1-es} \\
& |J F_{nls}^\bot A_1 F_{nls}^{- 1} |_{s , \sigma'}^\Lipg\,,\,
| J F_{nls}^\bot A_1 F_{nls}^{- 1} {\frak D}|_{s , \sigma'}^\Lipg 
\leq_s 1 + \| \io\|_{s + s_0 }^\Lipg\,. \label{sec-A1-es}
\end{align} 
\noindent
$(ii)$ 
For any $ \vphi \in \T^S $ and $  \s' \in \{  \s, \s - 1, \s - 2 \} $, $ {\mathtt \Phi}_1 (\vphi ) \in {\cal L} ( h^{\s'}_\bot) $ and 
\begin{align*} 
& | {\mathtt \Phi}_1^{\pm 1} - {\mathbb I}_2 |_{s, \s'}\,,\,| ({\mathtt \Phi}_1^{\pm 1} - {\mathbb I}_2) {\frak D} |_{s, \s'} \leq_s  \e  (1 + \| \io \|_{s +  s_0}) \\
&
 | {\mathtt \Phi}_1^{\pm 1} - {\mathbb I}_2 |_{s, \s'}^\Lipg\,,\,| ({\mathtt \Phi}_1^{\pm 1} - {\mathbb I}_2) {\frak D} |_{s, \s'}^\Lipg \leq_s  \e  (1 + \| \io \|_{s + s_0}^\Lipg)\,.
\end{align*}
\noindent
$(iii)$ ${\frak R}_1$ is a linear Hamiltonian operator with $ {\frak R}_1 (\vphi ) \in {\cal L} ( h^{\s-2}_\bot \times h^{\s-2}_\bot   , h^{\s-1}_\bot \times h^{\s-1}_\bot) $ for any $\vphi \in \T^S$,   and 
\be\label{lem:esR1f}
 | {\frak R}_1 \frak D |_{s, \s-1} \leq_s \e + \e \g^{-2} \| \io \|_{s+ 2 s_0} \, , 
 \quad 
  | {\frak R}_1 \frak D |_{s, \s-1}^\Lipg \leq_s \e +  \e \g^{-2} \| \io \|_{s+ 2 s_0}^\Lipg \, . 
\ee 
\end{lemma}

\begin{proof}
Since the proofs of the stated inequalities are similar for the range of values of $\sigma'$ considered, 
we only treat the case $\sigma' = \sigma$. 
\\[1mm]
$(i)$ We begin by proving the  estimate  \eqref{pr-A1-es}. 
In view of \eqref{def:D-doubled} and \eqref{def: A1} we can write 
$$
J F_{nls}^\bot A_1 F_{nls}^{- 1} 
= J  {\frak D}^{- 1}  F_{nls}^\bot  \begin{pmatrix}
0 & a_1 \\
\bar a_1 & 0
\end{pmatrix} F_{nls}^{- 1}{\frak D}^{- 1}\,, 
$$
Since $|{\frak D}^{- 1}|_{s, \sigma} = \|{\frak D}^{- 1} \|_{{\cal L}(h^{\sigma})}  \leq 1$ one has 
$|J F_{nls}^\bot A_1 F_{nls}^{- 1} |_{s, \sigma} 
\leq |J F_{nls}^\bot A_1 F_{nls}^{- 1}   {\frak D}|_{s, \sigma} $ and 
$$
|J F_{nls}^\bot A_1 F_{nls}^{- 1}   {\frak D}|_{s, \sigma} \stackrel{Lemma \,\ref{lemma:mult}}{\leq_s} 
\| a_1\|_s \stackrel{\eqref{choice p2}}{\leq_s} \| q_2 \|_s \stackrel{\eqref{estimates q1 q2}}{\leq_s} 1 + \| \io \|_{s + s_0}\,.
$$
The estimate \eqref{sec-A1-es} is proved in a similar way. 

\noindent
$(ii)$
By the smallness condition \eqref{ansatz 1}, the assumption of  Lemma \ref{lem:inverti} is satisfied for the operator $\e  J F_{nls}^\bot A_1 F_{nls}^{- 1}$ with $ \varepsilon $ sufficiently small, 
hence the claimed statement follows from this lemma and item $(i)$.

\noindent
$(iii)$ We begin proving the first estimate in \eqref{lem:esR1f}. 
The terms in ${\frak R}_1 {\frak D}$, with ${\frak R}_1$ defined in \eqref{formulaR1} 
are estimated individually. 
The statement concerning ${\frak R}_1 (\vphi )$  can be verified in a straightforward way.
Furthermore, the following estimates hold: 
\begin{align*}
& |{\mathtt \Phi}_1^{\pm 1}|_{s, \sigma - 1} \stackrel{(ii)}{\leq_s} 1 + \e \| \io\|_{s + s_0}\,, \quad | ({\mathtt \Phi}_1^{\pm 1} - {\mathbb I}_2) {\frak D} |_{s, \s -1} \stackrel{(ii)}{\leq_s}  \e  (1 + \| \io \|_{s + s_0})\,, \\
&
|{\frak D}^{- 1} J F_{nls}^\bot q_1 F_{nls}^{- 1} {\frak D}|_{s, \sigma - 1} 
\leq_s |F_{nls}^\bot q_1 F_{nls}^{- 1} |_{s, \sigma -2} 
\stackrel{Lemma \,\ref{lemma:mult}}{\leq_s} \| q_1\|_s
\stackrel{\eqref{estimates q1 q2}}{\leq_s} 1 + \| \io \|_{s + s_0}\,, \\
&
|(\omega \cdot \partial_\vphi)({\mathtt \Phi}_1 - {\mathbb I}_2) {\frak D}|_{s, \sigma - 1} 
\stackrel{\rm{ Def \, of\, } |\cdot |_{s, \sigma - 1}}{\leq_s} |({\mathtt \Phi}_1 - {\mathbb I}_2) {\frak D}|_{s + 1, \sigma - 1} 
\stackrel{(ii)}{\leq_s}  \e  (1 + \| \io \|_{s + s_0 + 1})\,,\\
&
|J \Omega^{nls} {\mathbb I}_2|_{s, \sigma - 1} \stackrel{\eqref{est:Omega1}}{\leq_s} 1 + \| \io \|_{s + 2 s_0}\,, \quad |{\frak Q}_\bot |_{s, \sigma - 1} \stackrel{Lemma \,\ref{lemma:mult}}{\leq_s} \| q_1\|_s + \| q_2 \|_s \stackrel{\eqref{estimates q1 q2}}{\leq_s} 1 + \| \io \|_{s + s_0}\,, \\
&
|{\frak R}_0 {\frak D}|_{s, \sigma - 1} \stackrel{\eqref{small-remainder}}{\leq_s} \e + \e \gamma^{- 2} \|\io \|_{s + 2 s_0}\,, \quad |[q_2, \langle D \rangle]|_{s, \sigma - 1} \stackrel{Lemma\, \ref{lemma:mult}}{\leq_s} \| q_2\|_{s} \stackrel{\eqref{estimates q1 q2}}{\leq_s} 1 + \| \io\|_{s + s_0}\,, \\
&
\Big| J {\frak D}^2 \sum_{n \geq 2} \frac{1}{n !} 
(- \e J F_{nls}^\bot A_1 F_{nls}^{- 1})^n {\frak D}  \Big|_{s, \sigma - 1} 
\stackrel{ Lemma \, \ref{lem:inverti} }{\leq_s} 
\e^2 | J F_{nls}^\bot A_1 F_{nls}^{- 1}|_{s, \sigma} 
| J F_{nls}^\bot A_1 F_{nls}^{- 1}|_{s_0, \sigma}
\stackrel{(i)}{\leq_s} \e^2 (1 + \| \io \|_{s + s_0 })\,, \\
&
\Big| \sum_{n \geq 2} \frac{1}{n !} 
(- \e J F_{nls}^\bot A_1 F_{nls}^{- 1})^n J {\frak D}^3  \Big|_{s, \sigma - 1} 
\stackrel{ Lemma \, \ref{lem:inverti} }{\leq_s} 
\e^2 | J F_{nls}^\bot A_1 F_{nls}^{- 1}|_{s, \sigma} | J F_{nls}^\bot A_1 F_{nls}^{- 1}|_{s_0, \sigma}
\stackrel{(i)}{\leq_s} \e^2 (1 + \| \io \|_{s + s_0 })\,.
\end{align*}
These estimates together with the tame estimate \eqref{interpm} for the composition of operator valued
 maps,  allow to bound each term in ${\frak R}_1 {\frak D}$ by $\e + \e \gamma^{- 2} \| \io \|_{s + 2 s_0}$. 
The second estimate in \eqref{lem:esR1f} is
proved in a similar way. 
\end{proof}

\begin{lemma}\label{lem:44 Delta 12} 
For any $s_0 \leq s \leq s_* - 2 s_0$ and any torus embeddings $\breve \io^{(a)}(\vphi) = (\vphi, 0, 0) + \io^{(a)}(\vphi)$, $a = 1, 2$, the following holds: 

\noindent
$(i)$ For any $  \s' \in \{  \s, \s - 1, \sigma - 2, \sigma - 3\} $, the operator $\Delta_{12} A_1 := A_1(\breve \io^{(1)}) - A_1(\breve \io^{(2)})$ satisfies 
$$
|J F_{nls}^\bot \Delta_{12} A_1 F_{nls}^{- 1} |_{s , \sigma'}\,,\,\,
| J F_{nls}^\bot \Delta_{12} A_1 F_{nls}^{- 1} {\frak D}|_{s , \sigma'} 
\leq_s \| \Delta_{12} \io \|_{s + s_0} + {\rm max}_{s + s_0}(\io) \| \Delta_{12} \io\|_{s_0}\, .
$$
 
\noindent
$(ii)$For any $  \s' \in \{  \s, \s - 1, \s - 2 \} $, the operators
$\Delta_{12} {\mathtt \Phi}_1:=  {\mathtt \Phi}_1(\breve \io^{(1)})  -  {\mathtt \Phi}_1(\breve \io^{(2)})$ and 
$\Delta_{12} {\mathtt \Phi}_1^{-1} :=  
{\mathtt \Phi}_1^{-1}(\breve \io^{(1)})  -  {\mathtt \Phi}_1^{- 1}(\breve \io^{(2)})$
satisfy the estimate
$$ 
 | \Delta_{12}{\mathtt \Phi}_1^{\pm 1} |_{s, \s'}\,,\,\,| \Delta_{12}{\mathtt \Phi}_1^{\pm 1} {\frak D} |_{s, \s'} \leq_s  \e \, (\| \Delta_{12} \io\|_{s + s_0} + {\rm max}_{s + s_0}(\io) \| \Delta_{12} \io\|_{2s_0})\,.
$$

\noindent
$(iii)$ The operator $\Delta_{12}{\frak R}_1 := {\frak R}_1(\breve \io^{(1)}) - {\frak R}_1(\breve \io^{(2)})$ satisfies the estimate 
$$
 | \Delta_{12}{\frak R}_1 \frak D |_{s, \s-1} \leq_s \e \gamma^{- 2}\| \Delta_{12} \io\|_{s + 2 s_0 } 
+  {\rm max}_{s + 2s_0}(\io) \|\Delta_{12}\io\|_{3 s_0} \,.
 $$ 
\end{lemma}
\begin{proof}
$(i)$  Since the proofs of the stated inequalities are similar for the range of the values of 
$\sigma'$ considered, 
we only treat the case $\sigma' = \sigma$. By the definition \eqref{def: A1} of $A_1$ one has
$$
J F_{nls}^\bot \Delta_{12} A_1 F_{nls}^{- 1} 
= J {\frak D}^{- 1}  F_{nls}^\bot \begin{pmatrix}
0 & \Delta_{12} a_1 \\
 \Delta_{12} \bar a_1 & 0
\end{pmatrix} F_{nls}^{- 1} {\frak D}^{- 1}\,. 
$$
Since $|{\frak D}^{- 1}|_{s, \sigma} = \|{\frak D}^{- 1} \|_{{\cal L}(h^{\sigma})}  \leq 1$ it then follows that
$$
|J F_{nls}^\bot \Delta_{12} A_1 F_{nls}^{- 1}   {\frak D}|_{s, \sigma} \stackrel{Lemma \,\ref{lemma:mult}}{\leq_s} \| \Delta_{12} a_1\|_s \stackrel{\eqref{choice p2}}{\leq_s} \| \Delta_{12} q_2 \|_s \stackrel{\eqref{estimates Delta 12 q1 q2}}{\leq_s} \|\Delta_{12} \io \|_{s + s_0} + {\rm max}_{s + s_0}(\io) \| \Delta_{12} \io\|_{s_0}
$$
and $|J F_{nls}^\bot \Delta_{12} A_1 F_{nls}^{- 1}|_{s, \sigma} 
\le |J F_{nls}^\bot \Delta_{12} A_1 F_{nls}^{- 1}   {\frak D}|_{s, \sigma} $,
establishing the claimed estimates in the case $\s' = \s$. 

\noindent
$(ii)$ The claimed estimate follows by Lemma  \ref{lem:inverti} $(v)$ and item $(i)$.

\noindent
$(iii)$ The terms in $\Delta_{12} {\frak R}_1 {\frak D}$, with ${\frak R}_1$ defined in \eqref{formulaR1}, are estimated individually. The following estimates hold: 
\begin{align*}
& |\Delta_{12} {\mathtt \Phi}_1^{\pm 1}|_{s, \sigma - 1}\,,  | \Delta_{12} {\mathtt \Phi}_1^{\pm 1} {\frak D}  |_{s, \s'} \stackrel{(ii)}{\leq_s} \e \big( \| \Delta_{12} \io \|_{s +s_0} +  {\rm max}_{s + s_0}(\io) \|\Delta_{12} \io \|_{ 2s_0} \big)\,, \\
& 
|{\frak D}^{- 1} J F_{nls}^\bot \Delta_{12} q_1 F_{nls}^{- 1} {\frak D}|_{s, \sigma - 1}
 \leq_s |F_{nls}^\bot \Delta_{12} q_1 F_{nls}^{- 1} |_{s, \sigma -2} 
\stackrel{Lemma \,\ref{lemma:mult}}{\leq_s} \| \Delta_{12} q_1\|_s \stackrel{\eqref{estimates Delta 12 q1 q2}}{\leq_s} \|\Delta_{12} \io \|_{s + s_0} + {\rm max}_{s + s_0}(\io) \| \Delta_{12} \io\|_{s_0}\,, \\
&
|(\omega \cdot \partial_\vphi)(\Delta_{12} {\mathtt \Phi}_1 ) {\frak D}|_{s, \sigma - 1} 
\stackrel{\rm{ Def \, of\, } |\cdot |_{s, \sigma - 1}}{\leq_s} 
| \Delta_{12} {\mathtt \Phi}_1  {\frak D}|_{s + 1, \sigma - 1} 
\stackrel{(ii)}{\leq_s}  \e  ( \| \Delta_{12} \io\|_{s + s_0 + 1} +  
{\rm max}_{s + s_0 + 1}(\io)\|\Delta_{12} \io \|_{2s_0})\,, \\
&
|J \Delta_{12} \Omega^{nls} {\mathbb I}_2|_{s, \sigma - 1} \stackrel{Lemma\,\ref{lemma variazione i Omega nls}\,(ii)}{\leq_s} \| \Delta_{12} \io\|_{s } + {\rm max}_{s + 2 s_0}(\io)  \|\Delta_{12} \io \|_{ s_0}\,, \\
&
|\Delta_{12} {\frak Q}_\bot |_{s, \sigma - 1} \stackrel{Lemma \,\ref{lemma:mult}}{\leq_s} \|\Delta_{12} q_1\|_s + \| \Delta_{12} q_2 \|_s \stackrel{\eqref{estimates Delta 12 q1 q2}}{\leq_s} \|\Delta_{12} \io \|_{s + s_0} + {\rm max}_{s + s_0}(\io) \| \Delta_{12} \io\|_{s_0}\,, \\
&
|\Delta_{12} {\frak R}_0 {\frak D}|_{s, \sigma - 1} \stackrel{\eqref{small-remainder Delta 12}}{\leq_s} \e \gamma^{- 2}\| \Delta_{12} \io\|_{s + 2 s_0}+ {\rm max}_{s + 2s_0}(\io) \| \Delta_{12} \io\|_{s_0}\,, \\
&
|[\Delta_{12} q_2, \langle D \rangle]|_{s, \sigma - 1} \stackrel{Lemma\, \ref{lemma:mult}}{\leq_s} \| \Delta_{12} q_2\|_{s} \stackrel{\eqref{estimates Delta 12 q1 q2}}{\leq_s} \|\Delta_{12} \io \|_{s + s_0} + {\rm max}_{s + s_0}(\io) \| \Delta_{12} \io\|_{s_0}\,.
\end{align*}
Next we prove that  
\begin{equation}\label{scocciatura Delta 12 A}
S_1, \, S_2
\stackrel{}{\leq_s}\e^2 \,(\| \Delta_{12} \io\|_{s + s_0} + \| \io \|_{s + s_0 } \| \Delta_{12} \io\|_{s_0})
\end{equation}
where $S_1$ and $S_2$ are defined as follows
$$
S_1 := \Big| J {\frak D}^2 \sum_{n \geq 2} \frac{1}{n !} 
\Delta_{12}(- \e J F_{nls}^\bot A_1 F_{nls}^{- 1})^n {\frak D}  \Big|_{s, \sigma - 1}\,, \quad 
S_2 := \Big| \sum_{n \geq 2} \frac{1}{n !} 
\Delta_{12} (- \e J F_{nls}^\bot A_1 F_{nls}^{- 1})^n J {\frak D}^3  \Big|_{s, \sigma - 1}\, .
$$
Since the estimates for $S_1$ and $S_2$ can be proved in a similar fashion,
we consider $S_1$ only. Let 
$$
B(\breve \io^{(a)}) := J F_{nls}^\bot A_1(\breve \io^{(a)}) F_{nls}^{- 1}\,,\,\,\,  a = 1, 2, 
\qquad
\Delta_{12} B^n := B(\breve \io^{(1)})^n - B(\breve \io^{(2)})^n\,.
$$
 We then write $\Delta_{12} B^n$ with $n \ge 2$ as a telescoping sum,
\begin{align}
\Delta_{12} B^n =  
(\Delta_{12} B)  B(\breve \io^{(1)})^{n - 1} + 
B(\breve \io^{(2)}) (\Delta_{12} B) B(\breve \io^{(1)})^{n - 2} + \cdots + 
B(\breve \io^{(2)})^{n - 1} (\Delta_{12} B) \,. \label{Delta 12 B1 n}
\end{align}
Each term $J {\frak D}^2 B(\breve \io^{(2)})^k (\Delta_{12} B) 
B(\breve \io^{(1)})^{n - k - 1}{\frak D}$,
$0 \leq k \leq n-1$, is estimated individually. 
It turns out to be convenient to write the operator $B(\breve \io^{(a)})$ in the form 
$$
B(\breve \io^{(a)}) = {\frak D}^{- 1} E(\breve \io^{(a)}) {\frak D}^{- 1}, 
\qquad E (\breve \io^{(a)}) := 
J  F_{nls}^\bot \begin{pmatrix}
0 & a_1(\breve \io^{(a)}) \\
\bar a_1(\breve \io^{(a)}) & 0
\end{pmatrix}  F_{nls}^{- 1}\,,
$$
so that
$
\Delta_{12} B = {\frak D}^{- 1} \, \Delta_{12} E \, {\frak D}^{- 1}\,.
$
Thus
$$
J {\frak D} (\Delta_{12} B) B(\breve \io^{(1)})^{n - 1}{\frak D} 
= J ({\frak D} (\Delta_{12} E) {\frak D}^{- 1}) ({\frak D}^{- 1} E(\breve \io^{(1)}) {\frak D}^{- 1} )^{n - 2}
({\frak D}^{- 1} E(\breve \io^{(1)}))
$$
and for any $1 \le k \le n -2$,\,
$J {\frak D}^2 B(\breve \io^{(2)})^k (\Delta_{12} B) B(\breve \io^{(1)})^{n - k - 1}{\frak D}$ equals
$$
 J ({\frak D}E(\breve \io^{(2)}){\frak D}^{- 1}) ( {\frak D}^{- 1} E(\breve \io^{(2)}) {\frak D}^{- 1})^{k - 1}  
({\frak D}^{- 1} \Delta_{12} E {\frak D}^{- 1}) 
({\frak D}^{- 1} E(\breve \io^{(1)}) {\frak D}^{- 1} )^{n - k - 2} ({\frak D}^{- 1} E(\breve \io^{(1)}))
$$
whereas for $k=n-1$ one has
$$
J {\frak D}^2 B(\breve \io^{(2)})^{n-1} (\Delta_{12} B) {\frak D} = 
 J ({\frak D}E(\breve \io^{(2)}){\frak D}^{- 1}) ( {\frak D}^{- 1} E(\breve \io^{(2)}) {\frak D}^{- 1})^{n - 2}  
({\frak D}^{- 1} \Delta_{12} E ) \,.
$$
Note that 
\begin{align*}
& |{\frak D}E(\breve \io^{(2)}){\frak D}^{- 1} |_{s, \sigma - 1} 
\stackrel{Lemma \,\ref{lemma:mult}}{\leq_s} 
\| {\frak D}\|_{{\cal L}(h^\sigma_\bot, h^{\sigma - 1}_\bot)} \| a_1(\breve \io^{(2)})\|_s 
\| {\frak D}^{- 1}\|_{{\cal L}(h^{\sigma - 1}_\bot, h^\sigma_\bot)} 
\stackrel{\eqref{choice p2}, \eqref{estimates q1 q2}}{\leq_s} 1 + {\rm max}_{s + s_0}(\io)\,, \\
& 
 |{\frak D}^{- 1} E(\breve \io^{(1)})|_{s, \sigma - 1} 
\stackrel{Lemma \,\ref{lemma:mult}}{\leq_s} 
\| {\frak D}^{- 1}\|_{{\cal L}(h^{\sigma - 1}_\bot , h^{\sigma -1}_\bot)}   \| a_1(\breve \io^{(1)}) \|_s 
\stackrel{\eqref{choice p2}, \eqref{estimates q1 q2}}{\leq_s} 1 + {\rm max}_{s + s_0}(\io)\,,
\end{align*}
and that by the same arguments, 
$ | {\frak D}^{- 1} E(\breve \io^{(a)}) {\frak D}^{- 1} |_{s, \sigma - 1}$, $a = 1, \, 2,$
is also bounded by $1 + {\rm max}_{s + s_0}(\io)$.
Furthermore, again by Lemma \ref{lemma:mult},
$|{\frak D} \Delta_{12} E{\frak D}^{- 1} |_{s, \sigma - 1} $ can be estimated by 
$$
\| {\frak D}\|_{{\cal L}(h^\sigma_\bot, h^{\sigma - 1}_\bot)}
\| \Delta_{12} a_1\|_s \| {\frak D}^{- 1}\|_{{\cal L}(h^{\sigma - 1}_\bot, h^\sigma_\bot)} \stackrel{\eqref{choice p2}, \eqref{estimates Delta 12 q1 q2}}{\leq_s} \| \Delta_{12} \io\|_{s + s_0} + {\rm max}_{s + s_0}(\io) \| \Delta_{12} \io\|_{s_0}
$$
and the same estimates hold for $|{\frak D} ^{- 1}\Delta_{12} E{\frak D}^{- 1} |_{s, \sigma - 1} $
and $|{\frak D} ^{- 1}\Delta_{12} E |_{s, \sigma - 1} $.
By the tame estimate for the composition of operator valued maps \eqref{interpm} and 
the smallness condition \eqref{ansatz 1} it then follows that for any $0 \leq k \leq n - 1$, 
$$
|J {\frak D}^2 B(\breve \io^{(2)})^k\,  (\Delta_{12} B) \, 
B(\breve \io^{(1)})^{n - k - 1}{\frak D}|_{s, \sigma - 1} \leq C(s)^{n - 1} \big( \| \Delta_{12} \io\|_{s +s_0} + 
{\rm max}_{s + s_0}(\io) \| \Delta_{12} \io \|_{s_0} \big)\,.
$$ 
In view of \eqref{Delta 12 B1 n} this yields
$$
|J {\frak D}^2  \Delta_{12}( J F_{nls}^\bot A_1 F_{nls}^{- 1})^n {\frak D} |_{s, \sigma - 1} \leq n C(s)^{n - 1} \big( \| \Delta_{12} \io\|_{s +s_0} + {\rm max}_{s + s_0}(\io) \| \Delta_{12} \io \|_{s_0} \big)\,
$$
and leads to the claimed estimate \eqref{scocciatura Delta 12 A},
\begin{align}
S_1 = \Big| J {\frak D}^2 \sum_{n \geq 2} \frac{1}{n !} 
\Delta_{12}(- \e J F_{nls}^\bot A_1 F_{nls}^{- 1})^n {\frak D} \Big|_{s, \sigma - 1} & \leq 
\sum_{n \geq 2} \frac{n C(s)^{n - 1} \e^n}{n !} \Big( \| \Delta_{12} \io\|_{s +s_0} + 
{\rm max}_{s + s_0}(\io) \| \Delta_{12} \io \|_{s_0} \Big) \nonumber\\
& \leq_s \e^2 \big( \| \Delta_{12} \io\|_{s +s_0} + {\rm max}_{s + s_0}(\io) \| \Delta_{12} \io \|_{s_0} \big)\,. \nonumber
\end{align}
The above estimates together with the estimates given in Lemma \ref{lem:44 Lip}, 
the tame estimate \eqref{interpm} for the composition of operator valued maps, 
and the smallness assumption \eqref{ansatz 1} allow to bound the $| \cdot |_{s, \s -1} $ norm of
each term in $\Delta_{12} ({\frak R}_1 {\frak D})$ by 
$\e \gamma^{- 2}\| \Delta_{12} \io\|_{s + 2 s_0 } +  {\rm max}_{s + 2s_0}(\io) \|\Delta_{12}\io\|_{3 s_0}$.
Let us indicate how this bound is obtained by considering one specific term. Note that by the definition of ${\frak R}^I$ and the one of ${\frak R}_1$,
${\frak R}^I{\frak D}$ contains the operator 
${\mathtt \Phi}_1^{-1} {\frak R}_0 {\mathtt \Phi}_1 {\frak D}$, which we write as
${\mathtt \Phi}_1^{-1} ({\frak R}_0{\frak D}) ({\frak D}^{-1} {\mathtt \Phi}_1{\frak D})$.
We then develop
$
\Delta_{12} \big( {\mathtt \Phi}_1^{-1} ({\frak R}_0 {\frak D}) ({\frak D}^{-1} {\mathtt \Phi}_1 {\frak D}) \big)
$
in a telescoping sum, which among others contains the term 
${\mathtt \Phi}_1^{-1}(\breve \io^{(2)}) \Delta_{12} ({\frak R}_0{\frak D})
 ({\frak D}^{-1}{\mathtt \Phi}_1(\breve \io^{(1)}) {\frak D})$.
By the tame estimate \eqref{interpm} for the composition of operator valued maps,
one then obtains
a bound, given by a sum, which contains among other terms the following one
$$
|{\mathtt \Phi}_1^{-1}(\breve \io^{(2)})|_{s, \s-1} |\Delta_{12} ({\frak R}_0{\frak D})|_{s_0, \s-1} 
|{\frak D}^{-1}{\mathtt \Phi}_1(\breve \io^{(1)}) {\frak D}|_{s_0, \s-1}.
$$ 
Then the estimate \eqref{small-remainder Delta 12} for 
$|\Delta_{12} {\frak R}_0  {\frak D}|_{s, \sigma - 1}$, 
applied for $s$ given by $s_0$, yields
$$
|\Delta_{12} {\frak R}_0  {\frak D}|_{s_0, \sigma - 1} 
\leq_s \e \gamma^{- 2}\| \Delta_{12} \io\|_{3 s_0} + {\rm max}_{3 s_0}(\io) \| \Delta_{12} \io\|_{s_0}\,.
$$
Furthermore, by Lemma \ref{lem:44 Lip}, 
$$
 | {\mathtt \Phi}_1^{- 1}(\breve \io^{(2)}) - {\mathbb I}_2 |_{s, \s -1}\,\leq_s  
\e  (1 + \| \io^{(2)} \|_{s +  s_0})\,  \quad
\mbox{and} \quad
\,| {\frak D}^{-1} {\mathtt \Phi}_1(\breve \io^{(1)}) {\frak D} |_{s_0, \s -1} \leq_s  1\,.
$$
Combining the above estimates, one concludes that 
$$
|{\mathtt \Phi}_1^{-1}(\breve \io^{(2)})|_{s, \s-1} |\Delta_{12} ({\frak R}_0{\frak D})|_{s_0, \s-1} 
|{\frak D}^{-1}{\mathtt \Phi}_1(\breve \io^{(1)}) {\frak D}|_{s_0, \s-1} \leq_s
\e \gamma^{- 2}\| \Delta_{12} \io\|_{s + 2 s_0 } + {\rm max}_{s + 2s_0}(\io) \|\Delta_{12}\io\|_{3 s_0}\,.
$$
All other terms are estimated in a similar fashion.
\end{proof}

\noindent


\subsection{Second transformation}\label{sec2}

The purpose of the second transformation is to eliminate the space dependence of $ q_1 $, 
appearing in the expression \eqref{op L1} for the operator ${\frak L}_1$, up to a one smoothing remainder.
The transformation is chosen to be the time $ 1 $-flow 
${\mathtt \Phi}_2 : H^s(\T^S, h^{\sigma'}_\bot \times h^{\sigma'}_\bot) \to 
H^s(\T^S, h^{\sigma'}_\bot \times h^{\sigma'}_\bot)$, $\sigma' \in \{\sigma, \sigma - 1, \sigma- 2 \}$, 
$$
{\mathtt \Phi}_2 := \exp( - \e J F_{nls}^\bot A_2 F_{nls}^{- 1} ) 
= {\mathbb I}_2 - \e J F_{nls}^\bot A_2F_{nls}^{- 1} + \ldots 
$$
of the linear vector field $ - \e J F_{nls}^\bot A_2 F_{nls}^{- 1}  $ where
\be\label{def: A2}
A_2 := \begin{pmatrix}
 D \lla D \rra^{- 2} a_2 +  a_2 D  \lla D \rra^{- 2}  & 0 \\
 0 &  \ov{D} \lla D \rra^{- 2} a_2 +  a_2 \ov{D}  \lla D \rra^{- 2} 
\end{pmatrix} \, .
\ee
Since we will  chose $ a_2(\vphi, x)$  to be real valued the 
operator $  J F_{nls}^\bot A_2 F_{nls}^{- 1}  $ is Hamiltonian  (cf Lemma \ref{lemma:HamiltonianVF}) and 
hence the flow $ {\mathtt \Phi}_2 $ symplectic. 
Furthermore we record that $ A_2 $ is one smoothing.
We will choose 
$ a_2 \equiv a_2 (\vphi, x )$ in such a way that 
 ${\frak L}_2 :=   {\mathtt \Phi}_2^{-1} {\frak L}_1 {\mathtt \Phi}_2$
is $x$-independent up to a one smoothing remainder.  
To this end we write
\be\label{Second composition}
{\frak L}_1 {\mathtt \Phi}_2 =  {\mathtt \Phi}_2 \big( 
\om \cdot \partial_\vphi \, {\mathbb I}_2  
+ J  D^2  \, {\mathbb I}_2  + J \Omega^{nls} \, {\mathbb I}_2 \big) +
\e J  F_{nls}^\bot q_1   \, F_{nls}^{- 1}   - 
\e [J D^2  \, {\mathbb I}_2, J F_{nls}^\bot A_2 F_{nls}^{- 1} ] + {\frak R}^I 
\ee
where 
$$
{\frak R}^I  :=  (\om \cdot \partial_\vphi) \big( {\mathtt \Phi}_2 - {\mathbb I}_2 \big) + 
[J \Omega^{nls} {\mathbb I}_2, {\mathtt \Phi}_2 - {\mathbb I}_2] +  
\e  J F_{nls}^\bot q_1   F_{nls}^{- 1} ({\mathtt \Phi}_2 -{\mathbb I}_2) +  {\frak R}_1 {\mathtt \Phi}_2 +
[ J D^2\, {\mathbb I}_2 ,  {\mathtt \Phi}_2 - {\mathbb I}_2 
+ \e J \, F_{nls}^\bot A_2 F_{nls}^{- 1}] \, 
$$
collects terms which are one smoothing. We now compute the commutator
$ [J D^2  \, {\mathbb I}_2, \, J F_{nls}^\bot  A_2 F_{nls}^{- 1} ] $.

\begin{lemma} The Hamiltonian operator $ [J D^2  \, {\mathbb I}_2, \, J F_{nls}^\bot  A_2 F_{nls}^{- 1} ] $ can be 
expanded as
\begin{align}
[J D^2  \, {\mathbb I}_2, \, J F_{nls}^\bot  A_2 F_{nls}^{- 1} ] & 
= 4  J F_{nls}^\bot (\partial_x a_2)   F_{nls}^{- 1}    - {\frak R}^{II}    \label{Comm2}
\end{align}
where ${\frak R}^{II}$ is the one smoothing operator given by
\begin{align}
& \qquad \qquad \qquad \qquad {\frak R}^{II}  := F_{nls}^\bot {\rm diag}(R^{II}, {\overline R}^{II}) \, F_{nls}^{- 1} \, , \label{remainder Rfr-1} \\
& \label{remainder R II}
R^{II} := \big( D  \lla D \rra^{- 2} (\partial^2_{x} a_2) -
(\partial^2_{x} a_2)  D \lla D \rra^{- 2}  +
2 \ii  \lla D \rra^{- 2} (\partial_x a_2) + 
2 \ii (\partial_x a_2)\lla D \rra^{- 2}  \big)\,.
\end{align}
\end{lemma}

\begin{proof}
Since $ J D^2 $ commutes  with $ J $, $  F_{nls}^\bot $ and $  F_{nls}^{-1} $, 
we have 
$$
[J D^2  \, {\mathbb I}_2, \, J F_{nls}^\bot  A_2 F_{nls}^{- 1} ]  =  J F_{nls}^\bot [ JD^2,  A_2]   F_{nls}^{- 1}  \, .
$$
By the definition of $ J $ in \eqref{forma-L-omega1} and of 
$ A_2 $ in \eqref{def: A2} the operator $ [ JD^2,  A_2] $ is diagonal and with first component given by 
$$
[\ii D^2, (\lla D \rra^{- 2} D a_2 + a_2 D \lla D \rra^{- 2})] = T_1 + T_2
$$
where 
$$
T_1 = \ii D^2 \lla D \rra^{- 2} D a_2 - \ii \lla D \rra^{- 2} D a_2 D^2\qquad 
\text{and } \qquad
T_2 = \ii D^2 a_2 D \lla D \rra^{- 2} - \ii a_2 D  \lla D \rra^{- 2}D^2\,.
$$
Use that  $\ii D = \partial_x$ and $D^2 \lla D\rra^{- 2} = 1 - \lla D \rra^{- 2}$ to conclude that 
\begin{align*}
T_1 & = \ii D^2 \lla D \rra^{- 2} D a_2 - \ii \lla D \rra^{- 2} D^2 a_2 D + \lla D \rra^{- 2} D (\partial_x a_2) D  \\
& = 2 \lla D \rra^{- 2} D^2 (\partial_x a_2) + \ii \lla D \rra^{- 2} D (\partial_x^2 a_2) \\
& = 2 (\partial_x a_2) - 2 \lla D \rra^{- 2} (\partial_x a_2) + \ii \lla D \rra^{- 2} D (\partial_x^2 a_2)\,.
\end{align*}
Similarly one has
$ T_2 = 2 (\partial_x a_2) - 2 (\partial_x a_2) \lla D \rra^{- 2} - \ii (\partial_x^2 a_2)  \lla D \rra^{- 2} D $. 
Thus 
\begin{align*}
\ii ( T_1 + T_2 ) & = 4 \ii (\partial_x a_2) - \big( 2 \ii \lla D \rra^{- 2}(\partial_x a_2) +  \lla D \rra^{- 2}D (\partial_x^2 a_2)  + 2 \ii (\partial_x a_2) \lla D \rra^{- 2} - (\partial_x^2 a_2) D \lla D \rra^{- 2} \big) 
\end{align*}
proving the lemma. 
\end{proof}
 We choose $ a_2 $ so that   
$ q_1  - 4  \partial_x a_2 $ is independent of $ x $, i.e., 
$ 4 \partial_x a_2 =  q_1 - {\rm av}(q_1) $ or
\be\label{def p}
 \quad a_2 :=  \frac14 \pa_x^{-1} (   q_1 - {\rm av}(q_1)) \, , 
\quad {\rm av}(q_1) := \int_0^1 q_1 \, dx \, ,
\ee
where the operator $\partial_x^{- 1}: H^{\s'} \to H^{\s' + 1}$ is defined by setting
$$
\partial_x^{- 1}(1) = 0\,, \qquad \partial_x^{- 1}(e^{\ii 2 \pi  j x}) = 
\frac{1}{\ii 2 \pi j} e^{\ii 2 \pi  j x} \quad \forall j \in \Z \setminus \{ 0 \}\,.
$$
Note that by  \eqref{def p} and Lemma \ref{stima lip q1 q2 RP},  
$ a_2 (\vphi, \cdot ) \in  H^{\s+1} $ for any $\vphi \in \T^S$. 
The remainder ${R}^{II}$, defined in \eqref{remainder R II}, is  given by
\begin{equation}\label{frak R II nuovo}
\frac14  \Big(  D \lla D \rra^{- 2} (\partial_{x} q_1) -
(\partial_{x} q_1)  D \lla D \rra^{- 2} + 
2 \ii  \lla D \rra^{- 2} (q_1 - {\rm av}(q_1)) + 
2 \ii (q_1 - {\rm av}(q_1)) \lla D \rra^{- 2}  \Big) \, 
\end{equation}  
and combining 
\eqref{Comm2}, \eqref{def p} one has
$$
 JF_{nls}^\bot q_1   \, F_{nls}^{- 1}   - 
 [J D^2  \, {\mathbb I}_2, \, J F_{nls}^\bot  A_2 F_{nls}^{- 1} ] = 
 J F_{nls}^\bot  {\rm av}(q_1) F_{nls}^{- 1} + {\mathfrak R}^{II} \, . 
$$
By applying the inverse  
$ {\mathtt \Phi}_2^{-1} = \exp (  \e J F_{nls}^\bot A_2 F_{nls}^{- 1}) $  
to \eqref{Second composition}, 
we get
\be\label{op L2}
{\frak L}_2 =   {\mathtt \Phi}_2^{-1} {\frak L}_1 {\mathtt \Phi}_2 =  
\om \cdot \partial_\vphi {\mathbb I}_2 + J \big( D^2  \, {\mathbb I}_2  + 
 \Omega^{nls} {\mathbb I}_2 + \e  \,{\rm av}(q_1) \, {\mathbb I}_2 \big)  + {\frak R}_2 
\ee
where $ {\frak R}_2  $ is the one smoothing operator
\begin{equation}\label{frak R2}
{\frak R}_2  :=  \e ({\mathtt \Phi}_2^{-1} - {\mathbb I}_2) J \,{\rm av}(q_1) {\mathbb I}_{2} +
 {\mathtt \Phi}_2^{-1} \big( {\frak R}^I + \e {\frak R}^{II} \big) 
\end{equation}
with ${\frak R}^I$ defined in \eqref{Second composition} and ${\frak R}^{II}$
	in \eqref{remainder Rfr-1}.
Since  $ {\mathtt \Phi}_2 $ is symplectic and ${\frak L}_1$ is a linear Hamiltonian operator,
Lemma~\ref{transformation of Hamiltonian operators} implies that
also $ {\frak L}_2 $ is Hamiltonian. 
We point out that the 0th order term $\big( \Omega^{nls}  + \e  {\rm av}(q_1) \big) {\mathbb I}_2$ in \eqref{op L2} is diagonal and $x$-independent, but  still depends on $\vphi$. 
Note that the coefficients of the operator $ {\frak R}_2 $ 
involve $\partial^2_{x} a_2(\vphi, \cdot) \in H^{\sigma - 1}$.

\noindent
Using Lemma \ref{lem:44 Lip} to estimate the term $ \frak R_1 {\mathtt \Phi}_2 $ in $ \frak R^I  $
and arguing as in the proof of Lemma \ref{lem:44 Lip}, we get

\begin{lemma} \label{stima R2 lip}
{\bf (Estimates of $ A_2 $, $\mathtt \Phi_2 $ and $ {\frak R}_2 $)}
For any $s_0 \leq s \leq s_* - 2 s_0$ 
the following statements hold:

\noindent
$(i)$ For any $ \vphi \in \T^S $ and 
$  \s' \in \{  \sigma + 1, \s, \s - 1, \sigma - 2, \sigma - 3\} $, $A_2(\vphi) \in {\cal L}(H^{\sigma' - 1}, H^{\sigma'})$ and 
\begin{align}
& |J F_{nls}^\bot A_2 F_{nls}^{- 1} |_{s , \sigma'}\,,\,\,
| J F_{nls}^\bot A_2 F_{nls}^{- 1} {\frak D}|_{s , \sigma'} \leq_s 
1 + \| \io\|_{s + s_0 } \label{lem:A2-s1} \\
& 
|J F_{nls}^\bot A_2 F_{nls}^{- 1} |_{s , \sigma'}^\Lipg\,,\,\,
| J F_{nls}^\bot A_2 F_{nls}^{- 1} {\frak D}|_{s , \sigma'}^\Lipg \leq_s 1 + \| \io\|_{s + s_0 }^\Lipg\,. \label{lem:A2-lip1}
\end{align}

\noindent
$(ii)$ For any $ \vphi \in \T^S $, $  \s' \in \{ \sigma, \sigma - 1, \s- 2 \} $,
$ {\mathtt \Phi}_2 (\vphi ) \in {\cal L} ( h^{\s'}_\bot \times h^{\s'}_\bot) $ and 
\begin{align*}
& | {\mathtt \Phi}_2^{\pm 1} - {\mathbb I}_2 |_{s, \s'}\,,\,  | ({\mathtt \Phi}_2^{\pm 1} - {\mathbb I}_2) {\frak D} |_{s, \s'}  \leq_s \e (1 + \| \io \|_{s + s_0}) \\ 
& 
 | {\mathtt \Phi}_2^{\pm 1} - {\mathbb I}_2 |_{s, \s'}^\Lipg\,,\,\,
  | ({\mathtt \Phi}_2^{\pm 1} - {\mathbb I}_2) {\frak D} |_{s, \s'}^\Lipg  
\leq_s \e \, (1 + \| \io \|_{s + s_0}^\Lipg)\,.
\end{align*} 

\noindent
$(iii)$
${\frak R}_2$ is a linear Hamiltonian operator with 
$ {\frak R}_2 (\vphi ) \in {\cal L} ( h^{\s-2}_\bot \times h^{\s-2}_\bot , h^{\s-1}_\bot \times h^{\s-1}_\bot) $ 
for any $\vphi \in \T^S$ and 
\be\label{lem:esR2f}
| {\frak R}_2 \frak D |_{s, \s-1} \leq_s \e + \e \g^{-2} \| \io \|_{s+ 2 s_0} \, , \quad
 |{\frak R}_2 \frak D |_{s, \s-1}^\Lipg \leq_s \e +  \e \g^{-2}\| \io \|_{s+ 2 s_0}^\Lipg \,. 
\ee
\end{lemma}

\begin{proof}
$(i)$ 
We begin proving \eqref{lem:A2-s1}. 
We consider the case $\sigma' = \sigma + 1$ only, since the other cases can be treated in a similar way. According to 
\eqref{def: A2} we can write 
$$
J F_{nls}^\bot A_2 F_{nls}^{- 1} = 
J  {\frak D}^{- 2} F_{nls}^\bot
\begin{pmatrix}
 D a_2 & 0 \\
0 & - D a_2
\end{pmatrix} F_{nls}^{- 1} + 
 J F_{nls}^\bot \begin{pmatrix}
a_2 D & 0 \\
0 & - a_2 D
\end{pmatrix} F_{nls}^{- 1}  {\frak D}^{- 2}
$$
Since $|D  \lla D \rra^{- 2}|_{s, \sigma + 1} \lessdot  
\|\lla D \rra^{- 1} \|_{{\cal L}(h^{\sigma + 1})}  \lessdot 1$ one has 
$|J F_{nls}^\bot A_2F_{nls}^{- 1} |_{s, \sigma + 1} 
\leq |J F_{nls}^\bot A_2 F_{nls}^{- 1}   {\frak D}|_{s, \sigma+ 1} $ and 
$$
|J F_{nls}^\bot A_2 F_{nls}^{- 1}   {\frak D}|_{s, \sigma + 1} 
\stackrel{Lemma \,\ref{lemma:mult}}{\leq_s} \| a_2\|_{s, \sigma + 1} 
\stackrel{\eqref{def p}}{\leq_s} \| q_1 \|_{s, \sigma } 
\stackrel{\eqref{estimates q1 q2}}{\leq_s} 1 + \| \io \|_{s + s_0}\,.
$$
The estimates \eqref{lem:A2-lip1} are proved in a similar way. 

\noindent
$(ii)$ is proved in a similar way as item $(ii)$ of Lemma \ref{lem:44 Lip}. 

\noindent
$(iii)$ We begin by proving  the first estimate in \eqref{lem:esR2f}. 
Note that the remainder $ {\frak R}_2 $ introduced in \eqref{frak R2}, 
$$
{\frak R}_2 = \e ({\mathtt \Phi}_2^{-1} - {\mathbb I}_2 ) 
 {\rm av}(q_1) J + {\mathtt \Phi}_2^{-1} \big( {\frak R}^I + \e {\frak R}^{II} \big) \, , 
$$
 is of the same form as the remainder ${\frak R}_1$ in Lemma \ref{lem:44 Lip}. 
Due to the definition 
 \eqref{remainder Rfr-1} - \eqref{remainder R II} 
of  ${\frak R}^{II}$, the term
$\e |{\frak R}^{II}{\frak D}|_{s, \sigma - 1}$ can be estimated in the same way as the corresponding term of ${\frak R}_1$. Since, in contrast to $A_1$, the operator $A_2$ is only one
smoothing, the main difference for estimating $|{\frak R}^{I} {\frak D}|_{s, \sigma - 1}$ concerns the term
 $$
 [ J D^2\, {\mathbb I}_2 , \, {\mathtt \Phi}_2 - {\mathbb I}_2 + 
\e J \, F_{nls}^\bot A_2 F_{nls}^{- 1}]\,.
 $$
 Using that $J$ and $ F_{nls}^\bot A_2 F_{nls}^{- 1}$ commute one has
 $$
 {\mathtt \Phi}_2 - {\mathbb I}_2 + \e J \,  F_{nls}^\bot A_2 F_{nls}^{- 1} = 
- \frac12 \e^2 (  F_{nls}^\bot A_2 F_{nls}^{- 1} )^2 + 
\sum_{n \geq 3} \frac{ (- \e J  F_{nls}^\bot A_2 F_{nls}^{- 1} )^n}{n !}\,.
 $$
Using item ($i$) together with Lemma \ref{lem:inverti} $(iv)$ we get 
 \begin{align*}
 \Big| J D^2 {\mathbb I}_2\, 
\sum_{n \geq 3} \frac{ (- \e J  F_{nls}^\bot A_2 F_{nls}^{- 1} )^n}{n !}  \, 
{\frak D} \Big|_{s, \sigma - 1}\,, \,\,\,
\Big| \sum_{n \geq 3} \frac{ (- \e J  F_{nls}^\bot A_2 F_{nls}^{- 1} )^n}{n !} 
J D^2 {\mathbb I}_2 \, {\frak D} \Big|_{s, \sigma - 1}\, {\leq_s} \,\,
\e^3(1 + \| \io \|_{s + s_0})\,.
 \end{align*}
 The estimate of the norm of the commutator 
$[J D^2{\mathbb I}_2 , \, ( \,  F_{nls}^\bot A_2 F_{nls}^{- 1})^2 ]{\frak D}$ 
requires more attention. 
Recalling  \eqref{proiettori trasformata di fourier} 
one has  
$$
 [J D^2{\mathbb I}_2 , \, (  F_{nls}^\bot A_2 F_{nls}^{- 1})^2 ]  = 
J [D^2{\mathbb I}_2 , \, (  F_{nls}^\bot A_2 F_{nls}^{- 1})^2 ]
 = J F_{nls}^\bot \big( D^2  A_2 \mathbb I_\bot  A_2 - 
A_2 \mathbb I_\bot A_2 D^2  \big) F_{nls}^{- 1}\,.
$$
The operator
$A_2 \mathbb I_\bot A_2 $ is of the form ${\rm diag}( B , \overline B )$ where,
with the short hand notation $\Lambda := D \lla D \rra^{- 2}$, 
 \begin{equation}\label{alvaro 1}
 B :=(\Lambda a_2 + a_2 \Lambda) \pi_\bot (\Lambda a_2 + a_2 \Lambda)  = 
\Lambda a_2 \pi_\bot \Lambda a_2 + \Lambda a_2 \pi_\bot a_2  \Lambda + 
a_2 \Lambda^2 \pi_\bot  a_2 + a_2 \Lambda \pi_\bot a_2 \Lambda\, .
 \end{equation}
Hence 
 \begin{equation}\label{alvaro la tartaruga}
 - [J D^2{\mathbb I}_2 , \, (  F_{nls}^\bot A_2 F_{nls}^{- 1})^2 ]
= J \, [(  F_{nls}^\bot A_2 F_{nls}^{- 1})^2 ,\,  D^2{\mathbb I}_2 ]
=  J F_{nls}^\bot {\rm diag}( [ B, D^2] , [ \overline B, D^2] ) F_{nls}^{- 1}
 \end{equation}
 and the commutator $[B, D^2]$ is given by the sum $T_1 + T_2 + T_3 + T_4$ with
 \begin{align}
T_1 := [ \Lambda a_2 \pi_\bot \Lambda a_2, D^2]\,, \quad 
T_2 := [ \Lambda a_2 \pi_\bot a_2  \Lambda, D^2]\,, \quad 
T_3 := [ a_2 \Lambda^2 \pi_\bot a_2, D^2]\,, \quad 
T_4 := [ a_2 \Lambda \pi_\bot a_2 \Lambda, D^2]\,. \label{alvaro 2}
 \end{align}
 The four operators are treated in the same way, so we consider $T_1$ only. 
Since $D^2 = - \partial_x^2$ one has 
$$
 T_1 =  \Lambda (\partial_x^2 a_2) \pi_\bot \Lambda a_2 + \Lambda a_2 \pi_\bot \Lambda (\partial_x^2 a_2) + 2  \Lambda (\partial_x a_2)\pi_\bot \Lambda (\partial_x a_2) + 2 \ii \Lambda(\partial_x a_2) \pi_\bot \Lambda a_2 D + 2 \ii \Lambda a_2 \pi_\bot \Lambda (\partial_x a_2) D\,. 
$$
 Since by \eqref{def p}
 $$
 \| a_2\|_{s, \sigma - 1}\,, \,\, \| \partial_x a_2\|_{s, \sigma - 1}\,, \,\, \| \partial_x^2 a_2\|_{s, \sigma - 1} \leq_s \| q_1 \|_{s}
 $$
it  follows from Lemma \ref{lemma:mult} and the estimate 
$\| \Lambda\|_{{\cal L}(h^{\sigma' - 1}, h^{\sigma'})} \lessdot 1$, valid for arbitrary $\sigma'$, that
 $$
 |T_1 \lla D \rra |_{s, \sigma - 1} \leq_s 
 \| q_1\|_s \|q_1 \|_{s_0} \stackrel{\eqref{estimates q1 q2}}{\leq_s} 1 + \| \io \|_{s + s_0}\,.
$$
Since the operators $T_2,$ $T_3$, and $T_4$  can be estimated in the same way, one concludes that
$$
| [ J D^2 \mathbb I_2 , \,\, \e^2 (  F_{nls}^\bot A_2 F_{nls}^{- 1} )^2] \frak D |_{s, \s-1}
\le \e^2 (1 + \| \io \|_{s + s_0})\,.
$$ 
Altogether, this proves the first estimate in \eqref{lem:esR2f}. 
The second estimate in \eqref{lem:esR2f} follows in a similar way. 
 \end{proof}

 \begin{lemma}\label{stima R2 Delta 12}
 For any $s_0 \leq s \leq s_* - 2 s_0$ and any torus embeddings $\breve \io^{(a)}(\vphi) = (\vphi, 0, 0) + \io^{(a)}(\vphi)$, $a = 1, 2$, satisfying \eqref{ansatz 1}, the following estimates hold:  
 
 \noindent
 $(i)$ For any $  \s' \in \{  \sigma + 1, \s, \s - 1, \sigma - 2, \sigma - 3\} $, the operator 
$\Delta_{12} A_2 := A_2(\breve \io^{(1)}) - A_2(\breve \io^{(2)})$ satisfies the estimates 
 $$
 |J F_{nls}^\bot \Delta_{12} A_2 F_{nls}^{- 1} |_{s, \sigma'}\,,\, \,
|J F_{nls}^\bot \Delta_{12} A_2 F_{nls}^{- 1} {\frak D}|_{s, \sigma'} \leq_s 
\| \Delta_{12} \io\|_{s + s_0} + {\rm max}_{s + s_0}(\io) \| \Delta_{12} \io\|_{s_0}\,,
 $$
 
 \noindent
 $(ii)$ For any $  \s' \in \{ \sigma, \sigma - 1, \s- 2 \} $,  the operators
$\Delta_{12} \mathtt \Phi_2 := 
\mathtt \Phi_2(\breve \io^{(1)}) - \mathtt \Phi_2 (\breve \io^{(2)})$
and $\Delta_{12} \mathtt \Phi_2^{-1} := 
\mathtt \Phi_2^{-1}(\breve \io^{(1)}) - \mathtt \Phi_2^{-1} (\breve \io^{(2)})$
 satisfy the etimate
 $$
 | \Delta_{12} \mathtt \Phi_2^{\pm 1} |_{s, \sigma'}\,, |(\Delta_{12} \mathtt \Phi_2^{\pm 1} ) {\frak D}|_{s, \sigma'} \leq_s \e 
 \big( \| \Delta_{12} \io\|_{s + s_0} +
 {\rm max}_{s + s_0}(\io) \, \| \Delta_{12} \io \|_{2s_0} \big)\,,
 $$
 
 \noindent
 $(iii)$ The operator $\Delta_{12} {\frak R}_2 := 
{\frak R}_2(\breve \io^{(1)}) - {\frak R}_2 (\breve \io^{(2)})$
 satisfies the estimate 
 $$
 |\Delta_{12} {\frak R}_2 {\frak D}|_{s, \sigma - 1} \leq_s \e \gamma^{- 2}\| \Delta_{12} \io\|_{s + 2 s_0 } + 
{\rm max}_{s + 2 s_0 }(\io) \, \| \Delta_{12} \io\|_{3 s_0}\,.
 $$
 \end{lemma}
\begin{proof}
$(i)$ We consider the case $\sigma' = \sigma + 1$ only, since the other cases can be treated in a similar way. According to the definition \eqref{def: A2} we can write 
$$
J F_{nls}^\bot \Delta_{12} A_2 F_{nls}^{- 1} =
 J  {\frak D}^{- 2} F_{nls}^\bot \begin{pmatrix}
D \Delta_{12} a_2 & 0 \\
0 & - D \Delta_{12} a_2
\end{pmatrix} F_{nls}^{- 1} +  
J F_{nls}^\bot \begin{pmatrix}
\Delta_{12} a_2 D & 0 \\
0 & - \Delta_{12} a_2 D
\end{pmatrix} F_{nls}^{- 1}  {\frak D}^{- 2}
$$
Since $|D  \lla D \rra^{- 2}|_{s, \sigma + 1} \lessdot  \|\lla D \rra^{- 1} \|_{{\cal L}(h^{\sigma + 1})}  \lessdot 1$ one has
$$
|J F_{nls}^\bot \Delta_{12} A_2 F_{nls}^{- 1}   {\frak D}|_{s, \sigma + 1}
 \stackrel{Lemma \,\ref{lemma:mult}}{\leq_s} \| \Delta_{12} a_2\|_{s, \sigma + 1} 
\stackrel{\eqref{def p}}{\leq_s} \| \Delta_{12} q_1 \|_{s, \sigma } 
\stackrel{\eqref{estimates Delta 12 q1 q2}}{\leq_s} 
\|\Delta_{12} \io \|_{s + s_0} + {\rm max}_{s + s_0}(\io) \, \| \Delta_{12} \io\|_{s_0}\,.
$$

\noindent
$(ii)$ Follows by Lemma  \ref{lem:inverti} $(v)$ and item $(i)$.

\noindent
$(iii)$ Note that the remainder $ {\frak R}_2 $ introduced in \eqref{frak R2}, 
$$
{\frak R}_2 = \e ({\mathtt \Phi}_2^{-1} - {\mathbb I}_2 ) 
  {\rm av}(q_1) \, J + {\mathtt \Phi}_2^{-1} \big( {\frak R}^I + \e {\frak R}^{II} \big) \, , 
$$
 is of the same form as the remainder ${\frak R}_1$ in Lemma \ref{lem:44 Lip}. 
Due to the definition 
 \eqref{remainder Rfr-1} - \eqref{remainder R II} 
of  ${\frak R}^{II}$,
the term $\e |\Delta_{12} {\frak R}^{II}{\frak D}|_{s, \sigma - 1}$ can be estimated in the same way as the corresponding term of $\Delta_{12} {\frak R}_1$. Since, in contrast to $A_1$, 
the operator $A_2$ is only one smoothing, the main difference for estimating $|\Delta_{12} {\frak R}^{I} {\frak D}|_{s, \sigma - 1}$ concerns the operator
 $$
 \Delta_{12}[ J D^2\, {\mathbb I}_2 , \, {\mathtt \Phi}_2 - {\mathbb I}_2 + 
\e J \, F_{nls}^\bot A_2 F_{nls}^{- 1}] {\frak D}\,.
 $$
 Using that $J$ and $A_2$ commute one has
 $$
 \Delta_{12} \Big({\mathtt \Phi}_2 - {\mathbb I}_2 + 
\e J \, F_{nls}^\bot A_2 F_{nls}^{- 1} \Big) = 
- \frac12 \e^2 \Delta_{12} \big( F_{nls}^\bot A_2 F_{nls}^{- 1} \big)^2 + 
\sum_{n \geq 3} \frac{ \Delta_{12} (- \e J F_{nls}^\bot A_2 F_{nls}^{- 1} )^n}{n !}\,.
 $$
By the same arguments used for obtaining the estimate \eqref{scocciatura Delta 12 A} in the proof of Lemma \ref{lem:44 Delta 12}, one concludes from item $(i)$ and  Lemma \ref{stima R2 lip}$(i)$,  
 $$
  S_1, \,  S_2 \,\,  {\leq_s} \, \, \e^3( \| \Delta_{12} \io\|_{s + s_0} + {\rm max}_{s + s_0}(\io) \| \Delta_{12} \io\|_{s_0})
 $$
where
$$
S_1 := \Big| J D^2 {\mathbb I}_2\, 
\sum_{n \geq 3} \frac{ \Delta_{12} (- \e J F_{nls}^\bot A_2 F_{nls}^{- 1} )^n}{n !} 
{\frak D}\Big|_{s, \sigma - 1}\, ,
\qquad 
S_2: =\Big| 
\sum_{n \geq 3} \frac{ \Delta_{12}  (- \e J F_{nls}^\bot A_2 F_{nls}^{- 1} )^n}{n !}
 J D^2 {\mathbb I}_2 {\frak D} \Big|_{s, \sigma - 1}\,.
$$
 The estimate of the norm of 
$- \frac{1}{2} \e^2 [ J D^2{\mathbb I}_2 , \, \Delta_{12} (  \, F_{nls}^\bot A_2 F_{nls}^{- 1})^2 ] {\frak D}$ 
requires more attention. By \eqref{alvaro la tartaruga}
$$
 - [J D^2{\mathbb I}_2 , \, \Delta_{12}(  F_{nls}^\bot A_2 F_{nls}^{- 1})^2 ]
= J F_{nls}^\bot {\rm diag}( \Delta_{12} [ B, D^2] , \Delta_{12} [ \overline B, D^2] ) F_{nls}^{- 1}
$$
 where $B$ is defined in \eqref{alvaro 1} and
 $[B, \, D^2] = T_1 + T_2 + T_3 + T_4$ with $T_1, T_2, T_3, T_4$ defined in \eqref{alvaro 2}. Hence 
 $$
 \Delta_{12}[B, D^2]  =  \Delta_{12} T_1 + \Delta_{12} T_2 + \Delta_{12} T_3 + \Delta_{12} T_4 \,.
 $$
 The four terms are treated in the same way, so we consider $\Delta_{12} T_1$ only. Recall that
 $$
 T_1 =  \Lambda (\partial_x^2 a_2) \pi_\bot \Lambda a_2 + \Lambda a_2 \pi_\bot \Lambda (\partial_x^2 a_2) + 2  \Lambda (\partial_x a_2)\pi_\bot \Lambda (\partial_x a_2) + 2 \ii \Lambda(\partial_x a_2) \pi_\bot \Lambda a_2 D + 2 \ii \Lambda a_2 \pi_\bot \Lambda (\partial_x a_2) D\,. 
 $$
  By \eqref{def p} one has 
 $
 \| a_2\|_{s, \sigma - 1}\,, \,\, \| \partial_x a_2\|_{s, \sigma - 1}\,, \,\,
 \| \partial_x^2 a_2\|_{s, \sigma - 1} \leq_s \| q_1 \|_{s}$, and
$$
 \| \Delta_{12} a_2\|_{s, \sigma - 1}\,, \, \, \| \partial_x \Delta_{12} a_2\|_{s, \sigma - 1}\,, 
\,\,\| \partial_x^2 \Delta_{12} a_2\|_{s, \sigma - 1} \,  \leq_s  \, \| \Delta_{12} q_1 \|_{s}  \,. 
 $$
It then follows from Lemma \ref{lemma:mult} and the estimate 
$\| \Lambda\|_{{\cal L}(h^{\sigma' - 1}, h^{\sigma'})} \lessdot 1$ for $\sigma'$ arbitrary, that
 \begin{align*}
 |\Delta_{12} T_1 \lla D \rra |_{s, \sigma - 1}  & \leq_s  
\| \Delta_{12} q_1\|_s \big( \| q_1(\breve \io^{(1)})\|_{s_0} + \| q_1(\breve \io^{(2)} ) \|_{s_0} \big) +  
\|\Delta_{12} q_1 \|_{s_0} \big( \| q_1(\breve \io^{(1)})\|_{s} + \| q_1(\breve \io^{(2)} )\|_{s} \big) \nonumber\\
 & \stackrel{\eqref{estimates q1 q2}, \eqref{estimates Delta 12 q1 q2}}{\leq_s} 
\| \Delta_{12} \io\|_{s + s_0} + {\rm max}_{s + s_0}(\io) \, \| \Delta_{12} \io\|_{s_0}\,.
 \end{align*}
Since the operators $\Delta_{12} T_2,$ $\Delta_{12} T_3$, and $\Delta_{12} T_4$  
can be estimated in the same way, one concludes that
$$
| [ J D^2 \mathbb I_2 , \,\, \e^2 \Delta_{12}  (  F_{nls}^\bot A_2 F_{nls}^{- 1} )^2] \frak D |_{s, \s-1}
\le \e^2 \big( \| \Delta_{12} \io\|_{s + s_0} + {\rm max}_{s + s_0}(\io) \, \| \Delta_{12} \io\|_{s_0} \big) \,.
$$ 
One then  concludes the proof of item ($iii$) by arguing in the same way as at the end 
of the proof of item ($iii$) of Lemma \ref{lem:44 Delta 12}.
\end{proof}

\subsection{Gauge transformation}\label{sec-gauge}

Finally we eliminate the $ \vphi $-dependence from 
$J \big( \Omega^{nls}   + \e \,{\rm av}(q_1) \, \big) {\mathbb I}_2$ 
in \eqref{op L2} by a gauge transformation. 
More precisely, we conjugate  $ {\frak L}_2 $ with the symplectic map, given by the time 1-flow map  
$$
{\mathtt \Phi}_3  := \exp \Big(- {\rm diag}(\b_k )_{k \in S^\bot} J \Big) = 
{\rm diag}  \Big(\big( e^{  - \ii \beta_k  } \big)_{k \in S^\bot}, \big( e^{ \ii \beta_k  } \big)_{k \in S^\bot} \Big)\,,
$$
corresponding to the Hamiltonian $ \sum_{k \in S^\bot} \b_k (\vphi) z_k \bar z_k  $ with 
$ \beta_k = \beta_k (\vphi) \in \R $. 
The conjugated operator ${\frak L}_3 := {\mathtt \Phi}_3^{-1} {\frak L}_2 {\mathtt \Phi}_3 $
is then given by
\be\label{def:OpL3}
{\frak L}_3 = 
\om \cdot \partial_\vphi {\mathbb I}_2  - 
J {\rm diag}_{k \in S^\bot} (\omega \cdot \partial_\vphi \beta_k) {\mathbb I}_2 + J \big( D^2  + 
 \Omega^{nls}  + \e \, {\rm av}(q_1) \, \big) {\mathbb I}_2 + {\frak R}_3  
\ee
where $ {\frak R}_3 :=  {\mathtt \Phi}_3^{-1}  {\frak R}_2 {\mathtt \Phi}_3 $. We choose the functions 
$ \b_k (\vphi)  $, $ k \in S^\bot $,  so  that 
\be\label{eqbetak}
 \om \cdot \pa_\vphi \b_k (\vphi)  =  \om_k^{nls} ( I( \vphi)) + \e\, {\rm av}(q_1) (\vphi)  - 
[[  \om_k^{nls}  \circ I + \e q_1 ]]  \, , \quad {\hat \b}_k (0) = 0 \, , 
\ee
where $[[ g]]$ denotes the average in space and time of a function $g : \T^S \times \T_1 \to \C$,   
$$ 
[[ g ]] := \frac{1}{(2 \pi)^{|S|}}
\int_{\T^S \times \T_1}  g (\vphi, x)\, d \vphi \, dx \, . 
$$
Since $ \om$ is assumed to be in $\Omega_0(\io) \subset \Omega_{\g , \t}$ 
it satisfies the  diophantine condition \eqref{Omega o Omega gamma tau} and by  
Lemma \ref{om vphi - 1 lip gamma},  the equations \eqref{eqbetak} have unique solutions. 
As a consequence  by \eqref{def:OpL3} and \eqref{definition Omega nls} we have 
\be\label{L4}
{\frak L}_3  = \om \cdot \partial_\vphi {\mathbb I}_2 + J   \big( D^2   +
[[ \Omega^{nls}  ]] + \e  [[ q_1 ]] \big) {\mathbb I}_2 + {\frak R}_3 \,, \qquad
 {\frak R}_3 =  {\mathtt \Phi}_3^{-1}  {\frak R}_2 {\mathtt \Phi}_3
\ee
where  ${\frak R}_2$ is defined in \eqref{frak R2}.
By \eqref{definition Omega nls} one has  $ D^2   + [[ \Omega^{nls}  ]] =
{\rm diag}_k (  [[ \omega^{nls}_k  ]] )_{k \in S^\bot} $.  

\begin{lemma} {\bf (Normal form of $ {\frak L}_3 $)} \label{NFL3}
The diagonal elements of 
$ D^2   + [[ \Omega^{nls}  ]] + \e  [[ q_1 ]] $  satisfy 
\begin{equation}\label{espansione asintotica autovalori blocco iniziale}
[[ \om_k^{nls}  ]] + \e [[ q_1  ]]  = \omega_k^{nls}(\xi, 0) + c_\e + 
\frac{1}{k}  r_{k,\xi} \, , \quad k \in S^\bot \, , 
 \end{equation}
 where 
\begin{equation}\label{stime asintotica autovalori iniziali}
 |c_\e|^\Lipg\,,\,|r_{k,\xi}|^\Lipg \lessdot \e \gamma^{- 2}\,.
 \end{equation}
 Furthermore
\begin{equation}\label{lip-order-1}
| [[ \om_k^{nls}  ]] + \e [[ q_1  ]]  |^{\rm lip} \lessdot 1 \, . 
 \end{equation}

\end{lemma}

\begin{proof}
Since by Theorem \ref{Corollary 2.2},
$$ 
\omega^{nls}_k = 4 \pi ^2 k^2 +4 \sum _{j \in {\mathbb Z}}  I_j + \frac{r_k}{k} \, ,  \quad (r_k)_{k \in \Z} \in \ell^\infty \, , 
$$ 
we get \eqref{espansione asintotica autovalori blocco iniziale}
with
$$
 c_\e :=  \big[ \big[4 \sum_{j \in S} y_j + 4 \sum_{j \in S^\bot} z_j \bar z_j + \e q_1 \big] \big]
\quad \mbox{and} \quad
r_{k,\xi} := \big[ \big[ r_k(\xi + y, z \bar z) - r_k(\xi, 0)  \big] \big]\,.
$$
Since
$| [[ q_1  ]] \,|^\Lipg \le \| q_1\|_{s_0}^\Lipg$ and 
$\| q_1\|_{s_0}^\Lipg \stackrel{\eqref{estimates q1 q2}}{\lessdot } 1 + \|  \io \|_{2 s_0}$
it follows that  
$|\,[[ q_1  ]]\, |^\Lipg \stackrel{\eqref{ansatz 1}}{\lessdot } 1$.
Furthermore, by \eqref{f (vphi) y z bar z} and  Lemma \,\ref{lemma variazione i Omega nls} (i),
$|\, [ [4 \sum_{j \in S} y_j + 4 \sum_{j \in S^\bot} z_j \bar z_j ] ]\, |^{\Lipg }
\stackrel{\eqref{ansatz 1}}{\lessdot } \e \gamma^{- 2}$. Similarly, 
$|r_{k,\xi}|^\Lipg \le \| r_k(\xi + y, z \bar z) - r_k(\xi, 0)\|_{s_0}^\Lipg$ and hence by
\eqref{Taylor estimate for rk}, $|r_{k,\xi}|^\Lipg \lessdot \| \io \|_{3 s_0}^\Lipg$.
Altogether we thus have proved 
\eqref{stime asintotica autovalori iniziali}.
The estimate \eqref{lip-order-1} follows from \eqref{espansione asintotica autovalori blocco iniziale},
\eqref{stime asintotica autovalori iniziali} since $ \e \g^{-3} \leq 1 $  and
$ \omega_k^{nls}(\xi (\om) , 0) $ is analytic and hence Lipschitz in $ \om $. 
\end{proof}

Using the smallness assumption \eqref{ansatz 1}, we prove the following

\begin{lemma} 
\label{stima R3 lip}
{\bf (Estimates of  $ {\mathtt \Phi}_3 $ and ${\frak R}_3$)}
For any $ s_0 \leq s \leq s_* - 4 s_0 - \tau $, 
 the following holds:

\noindent
$(i)$ For any $ \vphi \in \T^S $ and $  \s' \in \{ \sigma, \sigma - 1, \s- 2 \} $, 
 $ {\mathtt \Phi}_3 (\vphi ) \in {\cal L} ( h^{\s'}_\bot) $ and 
\begin{align}
& | {\mathtt \Phi}_3  - {\mathbb I}_2 |_{s, \s'} \,,\,   | {\mathtt \Phi}_3^{-1} - {\mathbb I}_2 |_{s, \s'}  \leq_s \g^{-1}(\e +   \| \io \|_{s+ 4 s_0 + \tau}) \label{lem:phi3} \\
& | {\mathtt \Phi}_3^{\pm 1}  - {\mathbb I}_2 |_{s, \s'}^\Lipg  \leq_s \g^{-1}(\e +   \| \io \|_{s+ 4 s_0 + 2 \tau + 1}^\Lipg)\,.
\label{lem:phi3-lip} 
\end{align}
\noindent
$(ii)$ ${\frak R}_3$ is a linear Hamiltonian operator with $ {\frak R}_3 (\vphi ) \in {\cal L} ( h^{\s-2}_\bot \times h^{\s-2}_\bot  , h^{\s-1}_\bot \times h^{\s-1}_\bot  ) $ for any $\vphi \in \T^S$ and
\begin{equation}  \label{estimate R3 lip} 
 | {\frak R}_3 \frak D |_{s, \s-1} \leq_s \e + \e \gamma^{- 2}  \| \io \|_{s+4 s_0 + \tau} \,, \quad
 | {\frak R}_3 \frak D |_{s, \s-1}^\Lipg \leq_s \e + 
\e\gamma^{- 2}  \| \io \|_{s+4 s_0 + 2 \tau + 1}^\Lipg \,. 
\end{equation}
\end{lemma}
\begin{proof}

$(i)$
We begin by proving  the estimate  \eqref{lem:phi3}. 
We first estimate the right hand side of \eqref{eqbetak} which we rewrite as
$$
 \om_k^{nls} ( I( \vphi) )  -  \om_k^{nls} ( \xi, 0 ) - 
[[  \om_k^{nls}  \circ I   -  \om_k^{nls} ( \xi, 0 ) ]] 
+ \e \big( {\rm av}( q_1)  (\vphi)  - 
[[  q_1 ]] \big) \, ,
$$ 
where $ I ( \vphi) =  ( \xi+y (\vphi) , z \bar z (\vphi) ) $. By \eqref{tame estimates for omega} 
$$
\sup_{k \in S^\bot } \|  \om_k^{nls} ( I )  -  \om_k^{nls} ( \xi, 0 )\|_s \leq_s \| \io \|_{s + 2 s_0} \, . 
$$ 
By Lemma \ref{om vphi - 1 lip gamma}, the solutions $ \b_k $ of \eqref{eqbetak} satisfy 
$$
\sup_{k \in S^\bot} \| \b_k \|_s \leq_s \g^{-1} \big( \| \io \|_{s + 2 s_0 + \tau} + 
\e \|{\rm av}( q_1)  - [[  q_1 ]] \|_{s + \tau} \big)
$$
and since $\|{\rm av}( q_1)  - [[  q_1 ]]\|_{s+\tau} \le \|q_1\|_{s+\tau}$ and 
by \eqref{estimates q1 q2}, $\|q_1\|_{s+\t} \le_s 1 + \| \io \|_{s + \t + s_0}$
it then follows that
$$
\sup_{k \in S^\bot} \| \beta_k\|_s  {\leq_s} 
\g^{-1} \big( \e + \| \io \|_{s + 2 s_0 +  \tau}   \big)\,.
$$
Due to the fact that $ \mathtt \Phi_3 $ is diagonal we have, for 
$ \s' \in \{ \sigma,  \s -1, \sigma - 2 \} $,  
$$
\| {\mathtt \Phi}_3 - {\mathbb I}_2 \|_{{\cal C}^{s + s_0} (\T^S, {\cal L}(h^{\s'}_\bot)) } = 
\sup_{k \in S^\bot} \| e^{\ii \b_k } - 1 \|_{{\cal C}^{s + s_0}(\T^S, \C)} 
 \leq_s  \sup_{k \in S^\bot}  \| \b_k \|_{{\cal C}^{s + s_0}(\T^S, \C)} 
$$
and since, by \eqref{upper-bound-norm-op},  $|\mathtt \Phi_3 - {\mathbb I}_2|_{s, \sigma'} \leq_s
\| {\mathtt \Phi}_3 - {\mathbb I}_2 \|_{{\cal C}^{s + s_0} (\T^S, {\cal L}(h^{\s'}_\bot)) }$
it then follows that
$$
|\mathtt \Phi_3 - {\mathbb I}_2|_{s, \sigma'} \leq_s
\sup_{k \in S^\bot}  \| \b_k \|_{{\cal C}^{s + s_0}(\T^S, \C)} 
\leq_s \sup_{k \in S^\bot}  \| \b_k \|_{s + 2 s_0}  \leq_s \g^{-1} \big(\e +  \| \io \|_{s + \t + 4s_0}   \big) \,.
$$
In the same way, one derives the claimed estimate for ${\mathtt \Phi}_3^{-1}$. 
The estimate \eqref{lem:phi3-lip} is proved in a similar way. 

\noindent
$(ii)$ Since $ {\mathtt \Phi}_3 $ is diagonal it commutes with  $ \frak D $ and hence
$  {\frak R}_3 \frak D  =  {\mathtt \Phi}_3^{-1}  ({\frak R}_2 \frak D) {\mathtt \Phi}_3 $. The 
first estimate in 
\eqref{estimate R3 lip} then follows from $(i)$, Lemma \ref{stima R2 lip} $(iii)$, and the tame estimate of Lemma \ref{prodest} 
for operator valued maps. 
The second estimate in \eqref{estimate R3 lip} is proved in a similar way. 
\end{proof}

\begin{lemma}\label{stimaR3 Delta 12} 
For any torus embeddings 
$\breve \io^{(a)}(\vphi) = (\vphi, 0, 0) + \io^{(a)}(\vphi)$, $a = 1, 2$, satisfying \eqref{ansatz 1}
and any $s_0 \leq s \leq s_* - 4 s_0 - \tau$, the following estimates hold:  

\noindent
$(i)$ For any $  \s' \in \{ \sigma, \sigma - 1, \s- 2 \} $, the operators
$\Delta_{12} {\mathtt \Phi}_3 := 
\mathtt \Phi_3(\breve \io^{(1)}) - \mathtt \Phi_3(\breve \io^{(2)})$
and  $\Delta_{12} {\mathtt \Phi}_3^{-1} := 
\mathtt \Phi_3^{-1}(\breve \io^{(1)}) - \mathtt \Phi_3^{-1}(\breve \io^{(2)})$ satisfy  
$$
| \Delta_{12}{\mathtt \Phi}_3^{\pm 1}|_{s, \s'}   \leq_s \g^{-1} \big( \| \Delta_{12} \io \|_{s + 4 s_0 + \tau} + {\rm max}_{s + 4  s_0 + \tau}(\io) \| \Delta_{12} \io\|_{s_0}  \big)\,.
$$
$(ii)$ The operator $\Delta_{12}{\frak R}_3 :=
 {\frak R}_3(\breve \io^{(1)}) - {\frak R}_3(\breve \io^{(2)})$ 
satisfies the estimate
\begin{equation}\label{estimate Delta 12 R3}  
| \Delta_{12} {\frak R}_3 \frak D |_{s, \s-1} \leq_s
\e \gamma^{- 2}\| \Delta_{12} \io\|_{s + 4  s_0 + \tau}  +   
{\rm max}_{s + 4 s_0 + \tau }(\io) \| \Delta_{12} \io \|_{5 s_0 + \tau} \,. 
\end{equation}
\end{lemma}

\begin{proof}
$(i)$
Note that $\Delta_{12} \beta_k := \beta_k^{(1)} - \beta_k^{(2)}$ 
with $\beta_k^{(a)}\equiv \beta_k (\io^{(a)})$, $a = 1,2$, satisfies the equation
\be\label{equation for delta12beta}
\omega \cdot \partial_\varphi \Delta_{12} \beta_k = 
\Delta_{12} \big(\om_k^{nls} ( I( \vphi) )   - 
[[  \om_k^{nls}  \circ I   ]] + \e \big( {\rm av}(q_1) (\varphi)  - [[  q_1 ]] \big)  \big)\,.
\ee
Using the same strategy developed  in the proof of  Lemma \ref{lemma variazione i Omega nls}
to obtain the estimate \eqref{stima Delta 12 rk I pippa}, we get 
with $ I^{(a)} ( \vphi) :=  ( \xi+y^{(a)} (\vphi) , z^{(a)} \bar z^{(a)} (\vphi) ) $, $a = 1, 2$, 
$$
\|  \Delta_{12} (\om_k^{nls} \circ  I ) \|_s = 
\| \omega_k^{nls}  \circ I^{(1)} - \omega_k^{nls}  \circ I^{(2)} \|_s \stackrel{\eqref{tame composizione derivate ennesime frequenza}}{\leq_s}  \| \Delta_{12} \io\|_{s } + {\rm max}_{s + 2 s_0}(\io) \| \Delta_{12} \io \|_{s_0} \, . 
$$
Since $\|\Delta_{12} \big( {\rm av}(q_1)  - 
[[  q_1 ]] \big) \|_s \le \| \Delta_{12} q_1 \|_s$, it then follows from
  \eqref{estimates Delta 12 q1 q2} that it can be bounded in the same way as
$\| \Delta_{12} (\om_k^{nls} \circ  I ) \|_s$.
Hence by \eqref{equation for delta12beta} and Lemma \ref{om vphi - 1 lip gamma}, $ \Delta_{12} \b_k $ satisfies 
\begin{equation}\label{stima Delta 12 beta k}
\| \Delta_{12} \b_k \|_s \leq_s \g^{-1} 
\big( \| \Delta_{12} \io \|_{s + \tau} + {\rm max}_{s + 2 s_0 + \tau}(\io) \| \Delta_{12} \io\|_{s_0} \big)\,.
\end{equation}
Since $\mathtt \Phi_3$ is diagonal, so is $ \Delta_{12} \mathtt \Phi_3 $ and we have 
for any $ \s' \in \{ \sigma,  \s -1, \sigma - 2 \} $,  
$$
\| \Delta_{12} {\mathtt \Phi}_3 \|_{{\cal C}^{s + s_0} (\T^S, {\cal L}(h^{\s'}_\bot)) } 
= \sup_{k} \| \Delta_{12} e^{\ii \b_k }  \|_{{\cal C}^{s + s_0}(\T^S, \C)} \,.
$$
Using that, by \eqref{upper-bound-norm-op}  $| \Delta_{12} \mathtt \Phi_3 |_{s, \sigma'} \leq_s 
\| \Delta_{12} {\mathtt \Phi}_3 \|_{{\cal C}^{s + s_0} (\T^S, {\cal L}(h^{\s'}_\bot)) }$
it then follows from 
\eqref{stima Delta 12 beta k} that
$$
| \Delta_{12} \mathtt \Phi_3 |_{s, \sigma'} \leq_s
\g^{-1} \big( \| \Delta_{12} \io \|_{s + 4 s_0 + \tau} + 
{\rm max}_{s + 4 s_0 + \tau}(\io) \| \Delta_{12} \io\|_{s_0}  
\big)\,.
$$
In the same way one derives the claimed estimate for $\Delta_{12}{\mathtt \Phi}_3^{-1}$. 
This proves item $(i)$.
Concerning item $(ii)$, the claimed estimate follows from Lemma \ref{stima R2 lip}$(iii)$, 
Lemma \ref{stima R2 Delta 12}$(iii)$, Lemma \ref{stima R3 lip}$(i)$, and item $(i)$ by using 
the tame estimate of Lemma \ref{prodest} and the smallness assumption $\e \gamma^{- 4} \ll 1$. 
\end{proof}

\begin{remark}
Taking into account the asymptotics of the dNLS frequencies \eqref{asymptotic expansion NLS}, 
as an alternative, one can 
choose a simpler gauge transformation by defining $ \b_k(\vphi) : = \b (\vphi)$, $ \forall k $, with $\beta(\vphi)$ the solution of 
$$
 \om \cdot \pa_\vphi \b (\vphi)  =  c_0 (\vphi)  - [[ c_0]]  \, , \quad 
c_0 (\vphi) := 4 \sum_{j \in S} y_j  (\vphi) + 4 \sum _{j \in S^\bot} z_j (\vphi) \bar z_j (\vphi) + \e {\rm av}(q_1)(\breve \io (\vphi) )\, . 
$$
In this case, there are additional $ \vphi$-dependent diagonal terms
of size $ O(\e \g^{-2} / k ) $.  
\end{remark}

\noindent
The operator $ {\frak L}_3 $ in \eqref{L4} is now in diagonal form up to a one smoothing 
remainder of small norm. More precisely, the $k$-th diagonal component of
$ {\frak L}_3 (\widehat z, \widehat w)$ is of the form 
$$
\om \cdot \partial_\vphi \widehat z_k  + \ii  \big( [[ \om_k^{nls}  ]] + 
\e  [[ q_1 ]] \big) \widehat z_k + \ldots 
$$
 In the subsequent section we will block diagonalize the remainder in ${\frak L}_3$
by a KAM-reduction scheme.

\section{Reduction of $ {\frak L}_\om $. Part 2}\label{sec:redu}

In this section we reduce the linear Hamiltonian operator ${\frak L}_3$, defined in  \eqref{L4}, by means of a KAM iteration scheme. Recall that ${\frak L}_3$ is an operator from $H^s(\T^S, h^\sigma_\bot \times h^\sigma_\bot)$ into 
$H^{s - 1}(\T^S, h_\bot^{\sigma - 2} \times  h_\bot^{\sigma - 2})$ for any $s_0 \leq s \leq s_* - \bar \mu$, where  
\begin{equation}\label{perdita mu dopo prime trasformazioni}
 \bar \mu:= 4 s_0 + 2 \tau + 1\,.
\end{equation} 
To describe the reduction scheme, it is convenient to denote ${\frak L}_3$ by ${\bf L}_0$ and write

\be\label{L0}
{\bf L}_{0} =  \om \cdot \partial_\vphi {\mathbb I}_2 +  {\bf N}_0 + {\bf R}_{0}
\ee
where
\begin{equation}\label{first diagonal normal form}
{\bf N}_0 := J \begin{pmatrix}
{\bf N}_0^{(1)} & 0 \\
0 & {\bf N}_0^{(1)}
\end{pmatrix}  \, , \quad  {\bf N}_0^{(1)} :=   
{\rm diag}_{k \in S^\bot} \big( [[ \om_k^{nls}  ]] + \e  [[ q_1 ]] \big) \,, \quad {\bf R}_0 := {\frak R}_3\,,  
\end{equation}
with the normal form 
${\bf N}_0 $  described in Lemma \ref{NFL3} and  ${\frak R}_3$ given by \eqref{L4}. 
We recall that $ {\bf R}_0 $ is one smoothing \big(meaning that
$ {\bf R}_0 {\mathfrak D} \in H^s (\T^S, {\cal L}(h_\bot^{\s-1} \times h_\bot^{\s-1}))$\big)
and satisfies the estimate (cf \eqref{estimate R3 lip})
\begin{equation}\label{stima R0 riducibilita}
|{\bf R}_0 {\frak D}|_{s, \sigma - 1}^{\gamma {\rm lip}} \leq_s \e + \e \gamma^{- 2}\| \io \|_{s+ \bar \mu}^{\Lipg}\,, \quad \forall s_0 \leq s \leq s_* - \bar \mu\,.
 \end{equation}
 The linear Hamiltonian operators ${\bf L}_0$, ${\bf N}_0$, ${\bf R}_0$  depend on
 the  torus embedding $\breve \io\equiv \breve \io_\omega : \T^S \to M^\sigma$,
satisfying the smallness assumption \eqref{ansatz 1},  with
$\omega \in \Omega_o(\io)$. Here
\be \label{definition Omega io} 
\Omega_o(\io) \subset  \Omega_{\gamma, \tau} \subset \Omega\,, \qquad 0 < \g < 1\,, 
\ee
and $\Omega_{\gamma, \tau}$ denotes the set of diophantine frequencies
\eqref{Omega o Omega gamma tau}.


\subsection{KAM reduction scheme for ${\bf L}_0$}

In view of the near resonances of the dNLS frequencies $ \om^{nls}_k $, $ \om^{nls}_{-k} $, we group the
coordinates $ z_{-k}$ and $ z_k $ together. Our aim is to reduce $ {\bf L}_{0} $ to a $ 2 \times 2 $ block diagonal operator
with $\vphi$-independent coefficients, referred to as its {\it normal form}.  
Accordingly, a complex linear operator $ A $ in ${\cal L}(h^{\sigma'}_\bot)$ with matrix representation 
$(A_j^k )_{j, k \in S^\bot}$, $ A_j^k \in \C$ for all $j, k \in S^\bot $, (cf \eqref{matrix:coefficients})
is written as a matrix of $ 2 \times 2 $ matrices 
$ ([A]_j^k)_{j, k \in S^+_\bot}   $ where
$$
[A]_j^k :=
\left(
\begin{array}{cc}
A_{-j}^{-k}  & A_{-j}^{k}    \\
A_{j}^{ -k}    &  A_{j}^{ k}    \\
\end{array}
\right) \, , \quad  j, k \in  S_+^\bot := S^\bot \cap \N \, .
$$
We denote by $ \| \ \| $ the operator norm of these $ 2 \times 2 $ matrices. Actually any other norm could be used as well.
We say that $A$ is a $2 \times 2$ block diagonal operator if $[A]_j^k = 0$ for any $j, k \in S_+^\bot$ 
with $j \neq k$. 
Let $N_0 > 0$ be given and define
\be\label{defN}
N_{-1} := 1 \, ,   \quad
N_{\nu} := N_{0}^{\chi^{\nu}} \quad \forall \, \nu \geq 1 \, ,  \quad
\chi := 3 /2\,.
\ee
Note that $ N_{\nu+1} = N_{\nu}^{\frac32} $ for any $  \nu \geq 0 $. 
Along the iteration scheme, 
we shall consider 
the following decreasing sequence  $\big(\Om_{\nu}^{\g}( \io  ) \big)_{\nu \ge 0}$
of subsets of frequencies %
\be\label{Omega-second-meln}
 \Om_{0}^{\g}( \io ) := \Om_o(\io) \subset \Omega_{\gamma, \tau} 
 \,, \quad \Om_{\nu}^{\g}( \io  ) := \big\{ \om \in \Om_{\nu-1}^{\g}( \io  ) \,  :  
\eqref{Hyp1}-\eqref{Hyp2} \,  \text{hold} \, \big\} \, , \quad \nu \geq 1\, .  
\ee
We point out that the conditions  \eqref{Hyp1}-\eqref{Hyp2} also involve
an  exponent  $ \t >  |S| $ 
and that set $\Omega_{\gamma, \tau}$ 
is defined in \eqref{Omega o Omega gamma tau}.
We introduce the following constants  $ \a, \b $,  which appear in the exponents of the Sobolev spaces
in the iterative scheme, 
\begin{equation}\label{alpha beta}
 \a  \equiv  \a (\t) := 6 \tau + 4 \, , \quad  \quad \b \equiv \b(\t)  :=  \a + 1 \, .
\end{equation}
 In addition we require that 
\begin{equation}\label{regularity s-1}
s_0 + \beta + \bar \mu \le  s_* 
\end{equation}
where $\bar \mu$ is given by \eqref{perdita mu dopo prime trasformazioni}.

\begin{theorem}{{\bf (Reduction scheme for ${\bf L}_0$)}} \label{thm:abstract linear reducibility}
There exists $ N_{0} = N_0( \tau, |S|, s_*)  \in \N $   such that, if
\begin{equation}\label{piccolezza1}
\g^{-1} N_{0}^{C_0} \,  |{\bf R}_{0} \frak D |_{ s_0 + \beta, \s-1}^{\Lipg}  \leq 1, \quad C_0 := 2 \tau + 2 + \alpha
\end{equation}
then for any 
$ \nu \geq 1 $, the following statements hold:
\begin{itemize}
\item[${\bf(S1)_{\nu}}$]
For any $\omega \in \Omega^\gamma_\nu(\io)$
there exists a symplectic transformation $ \Phi_{\nu-1} := \exp( -\Psi_{\nu-1}) $
such that for any $\vphi \in \T^S $, 
$ \Phi_{\nu-1}(\vphi)  \in {\cal L}( h_\bot^{\s'} \times h_\bot^{\s'}) $, $ \s' \in \{ \s - 2,\s -1, \s \} $,  
$ \Psi_{\nu-1} $ is a linear Hamiltonian vector field satisfying for any $s \in [ s_0, s_*  - \bar \mu - \b ] $ the estimates
\be\label{Psinus}
\left|\Psi_{\nu-1}  \right|_{s, \s}^{\Lipg} \, , \ 
\left|\Psi_{\nu-1} \frak D  \right|_{s, \s-1}^{\Lipg} \lessdot \gamma^{-1}
|{\bf R}_{0} \frak D  |_{s+\beta, \s-1}^{\Lipg}  N_{\nu-1}^{2 \t+1} N_{\nu - 2}^{- \a}  \, , 
\ee
and
\be	\label{Lnu+1}
{\bf L}_{\nu} := \Phi_{\nu-1}^{-1} {\bf L}_{\nu-1} \Phi_{\nu-1} = 
\om \cdot \partial_\vphi {\mathbb I}_2  +   {\bf N}_\nu + {\bf R}_\nu
\ee
where ${\bf N}_\nu$ and ${\bf R}_\nu$ have the following properties:
${\bf N}_\nu$ is in normal form, i.e., ${\bf N}_\nu$ is a $\vphi$-independent $2 \times 2$ block diagonal operator, 
\be\label{normal-form}
{\bf N}_\nu = J 
\left(
\begin{array}{cc}
{\bf N}_\nu^{(1)} &   0   \\
0  & \overline{\bf N}_\nu^{(1)}   \\
\end{array}
\right) \, ,
\quad  {\bf N}_\nu^{(1)} =   {\rm diag}_{k \in S^\bot_+} \big[ {\bf N}_\nu^{(1)} \big]^k_k \, ,
\ee
where for any $k \in S_+^\bot$, $ \big[ {\bf N}_\nu^{(1)} \big]^k_k \in \C^{2 \times 2}$  is self-adjoint
\be\label{normal-form-specific}
({\bf N}_\nu^{(1)} \big)_{-k}^{-k} \, ,  \ ({\bf N}_\nu^{(1)} )_k^k \in \R \, , \quad
( {\bf N}_\nu^{(1)} )_{-k}^{k} =  ( \overline{\bf N}_\nu^{(1)} )_{k}^{ -k} \in \C
\ee
and satisfies 
\begin{equation}\label{N nu - N 0 lip gamma}
\big\| [{\bf N}_\nu^{(1)} - {\bf N}_{\nu - 1}^{(1)}]_k^k 
 \big\|^{\gamma {\rm lip}} \lessdot | {\bf R}_{\nu - 1} {\frak D}|_{s_0,\s-1}^{\gamma {\rm lip}} k^{-1} 
 \, , \qquad  \| [{\bf N}^{(1)}_\nu ]_k^k  \|^{ {\rm lip}} \lessdot 1  \, .
\end{equation}
The remainder ${\bf R}_\nu$ in \eqref{Lnu+1} is a linear Hamiltonian operator
\be\label{Remainder nu}
{\bf R}_\nu = J 
\left(
\begin{array}{cc}
  {\bf R}_\nu^{(1)}  &   {\bf R}_\nu^{(2)}     \\
  \overline{\bf R}_\nu^{(2)}      &     \overline{\bf R}_\nu^{(1)}   \\
\end{array}
\right),  \quad {\bf R}_\nu^{(1)} =   
\big({\bf R}_\nu^{(1)}\big)^*, \quad {\bf R}_\nu^{(2)} = ({\bf R}_\nu^{(2)} )^t
\ee
satisfying for any $s \in [ s_0, s_* - \bar \mu - \b ] $ the following estimates
\be\label{Rsb}
\left|{\bf R}_{\nu} {\mathfrak D}
\right|_{s, \s-1}^{\Lipg} \leq  \left|{\bf R}_{0} {\mathfrak D}  \right|_{s+\beta, \s-1}^{\Lipg}N_{\nu - 1}^{-\alpha} \, , \quad
\left|{\bf R}_{\nu}{\mathfrak D}  \right|_{s + \beta, \s-1}^{\Lipg} \leq \left|{\bf R}_{0} 
{\mathfrak D}  \right|_{s+\beta, \s-1}^{\Lipg}\,N_{\nu - 1}\,.
\ee
In ${\bf(S1)}_\nu$, all the Lipschitz norms are computed on the set $\Omega_\nu^\g (\io)$.

\item[${\bf(S2)_{\nu}}$] For any $k \in S_+^\bot$, there exists a Lipschitz extension $[ \widetilde{\bf N}^{(1)}_\nu]_k^k$ of $[{\bf N}^{(1)}_\nu]_k^k$ to the set $\Omega_o(\io)$, which is self-adjoint and satisfies the estimate 
\begin{equation}\label{closeness extended blocks}
\|[\widetilde{\bf N}^{(1)}_\nu]_k^k - [\widetilde{\bf N}^{(1)}_{\nu - 1}]_k^k \|^{\gamma {\rm lip}} \lessdot |{\bf R}_{\nu - 1} {\frak D}|_{s_0, \sigma - 1}^{\gamma {\rm lip}} k^{- 1}\,,
\end{equation}
where we set $[\widetilde{\bf N}_0^{(1)}]_k^k = [{\bf N}_0^{(1)}]_k^k $. 
\end{itemize}
\end{theorem}

Theorem \ref{thm:abstract linear reducibility} is proved in Section~\ref{proof of reduction theorem}. In the subsequent two sections we establish some auxiliary results.

\subsection{$ 2 \times 2 $ block representation of operators }
Let us write an element  $ z  = (z_k)_{k \in S^\bot } $ in $ h^{\s'}_\bot $ as a sequence of vectors
$$
 z  = (  \vec z_k )_{k \in S^\bot_+} \,, \quad \vec z_k := (z_{-k}, z_k)\,, \quad S_+^\bot = S^\bot \cap \N \, .
 $$ Its Sobolev norm is thus
$$
\| z \|_{\sigma'}^2 = \sum_{k \in S^\bot} |z_k|^2 \langle k \rangle^{2\s'} 
= \sum_{k \in S_+^\bot} | \vec z_k|^2 \langle k \rangle^{2\s'} \,.
$$
For each complex linear operator $A \in {\cal L}(h^{\sigma'}_\bot)$ and $z  = (\vec z_k)_{k \in S_+^\bot} \in h_\bot^{\sigma'}$, $A z = (\vec{A z})_{j \in S^\bot_+}$ with
$$
(\vec{A z})_j = \sum_{m \in S_+^\bot} [A]_j^m \vec z_m\,. 
$$

\noindent
 Furthermore, we denote by $ A^{\rm diag}$ the linear operator obtained from 
$A$ by setting for any $j, k \in S_+^\bot$ 
\be\label{def:diag-NF}
[ A^{\rm diag}]^k_j = [A]^k_k \, \quad \text{if} \ j = k \, , 
\quad [A^{\rm diag} ]^k_j = 0  \, \quad \text{if} \ j \neq k \, .
\ee
\begin{lemma}\label{bound op2}
Let  $ A \in {\cal L}( h^{\s'}_\bot) $ with $ \s' \leq \s $. Then the following holds:

\noindent
$(i)$ $ A^{\rm diag} \in  {\cal L}( h^{\s}_\bot)  $ and  
$  \| A^{\rm diag} \|_{{\cal L}( h^{\s}_\bot)} \lessdot  \| A \|_{ {\cal L}( h^{\s'}_\bot) }$;

\noindent
$(ii)$ $ \sum_{j \in S_+^\bot} \| [A]^k_j \|^2 \langle j \rangle^{2\s'}  \lessdot 
\| A \|_{{\cal L}(h^{\s'}_\bot)}^2 \langle k \rangle^{2\s'} $, $ \forall k \in S_+^\bot $;

\noindent
$(iii)$ for any $(\vec h_k)_{k \in S_+^\bot} \in h^{\sigma'}_\bot$, 
$$
\sum_{j \in S_+^\bot} \Big( \sum_{k \neq j} \frac{\| [A]^k_j \| \| \vec h_k \| }{| j - k |} \Big)^2 \langle j \rangle^{2 \s'} \lessdot 
\| A \|_{{\cal L}(h^{\s'}_\bot)}^2 \| h \|_{\s'}^2 \,.
$$
\end{lemma}
\begin{proof}

\noindent
$(i)$ The estimate holds, since each  matrix element of $ [A]^j_j \in \C^{2 \times 2} $, $ j \in S^\bot_+ $, is bounded by 
$  \| A \|_{ { {\cal L}( h^{\s'}_\bot) }} $. 

\noindent
$(ii)$ By the definition of the operator norm, for any $h \in h^{\sigma'}_\bot $ one has
$$
\| A h \|_{\sigma'}^2 = \sum_{j \in S_+^\bot} \Big\|  \sum_{m\in S_+^\bot} [A]^m_j \vec h_m  \Big\|^2 \langle j \rangle^{2 \s'} \leq \| A \|_{{\cal L}(h^{\s'}_\bot)}^2 \| h \|_{\s'}^2 \,.
$$
For the sequence $ h = (\vec h_k \d_{k, m})_{m \in  S_+^\bot}$
(with  $\delta_{k, m } = 0$ for $m \neq k$ and $\delta_{k, k} = 1$),  we find
$$
\sum_{j \in S_+^\bot} \big\|  [A]^k_j \vec h_k  \big\|^2 \langle j \rangle^{2 \s'} 
\lessdot \| A \|_{{\cal L}(h^{\s'}_\bot)}^2 | \vec h_k |^2 \langle k \rangle^{2\s'} \, .
$$
By choosing $ \vec h_k = (1,0) $ and $ \vec h_k = (0, 1) $, respectively, one gets
$$
\sum_{j \in S_+^\bot} \Big\|
\begin{pmatrix}
A_{-j}^{-k}    \\
A_{j}^{-k} 
\end{pmatrix}
\Big\|^2 \langle j \rangle^{2 \s'}, \quad
\sum_{j \in S_+^\bot} \Big\| 
\begin{pmatrix}
A_{-j}^{ k}    \\
A_j^k  
\end{pmatrix}
\Big\|^2 \langle j \rangle^{2 \s'} 
\lessdot \| A \|_{{\cal L}(h^{\s'}_\bot)}^2  \langle k \rangle^{2\s'} \,. 
$$
Since $ \| [A]^k_j \| $ is bounded by
$ |  A_{-j}^{-k}|^2 +  |  A_{j}^{ -k}|^2 +  |  A_{-j}^{ k}|^2 +  | A_{j}^{ k}|^2  $, item $(ii)$ follows.

\noindent
$(iii)$ Using the Cauchy Schwartz inequality one has  
\begin{align}
\sum_{j \in S_+^\bot} \Big( \sum_{k \neq j} \frac{\| [A]^k_j \| \| \vec h_k \| }{| j - k |} \Big)^2 \langle j \rangle^{2 \s'}  & 
 \leq 
\sum_{j \in S_+^\bot}  \Big(\sum_{k \in S_+^\bot} \| [A]^k_j \|^2 \| \vec h_k \|^2 \langle j \rangle^{ 2\s'} \Big)\Big( \sum_{k \neq j} \frac{1}{|j-k|^2} \Big) 
\nonumber	\\
& \lessdot \sum_{j \in S_+^\bot}  \sum_{k \in S_+^\bot} \| [A]^k_j \|^2 \| \vec h_k \|^2 \langle j \rangle^{ 2\s'} 
\lessdot  \sum_{k \in S_+^\bot} \| \vec h_k \|^2 \sum_{j \in S_+^\bot} \| [A]^k_j \|^2 \langle j \rangle^{ 2\s'}  \nonumber \\
& \stackrel{(ii)}{\lessdot} \sum_{k \in S_+^\bot} \| \vec h_k \|^2\| A \|_{{\cal L}(h^{\s'}_\bot)}^2 \langle k \rangle^{2\s'} 
= \| A \|_{{\cal L}(h^{\s'}_\bot)}^2 \| h \|_{\s'}^2\,, \nonumber 
\end{align}
establishing the claimed estimate.
\end{proof}
Let us denote by $\C^{2 \times 2}$ the 4-dimensional Hilbert space of the complex $2 \times 2$ matrices equipped with the inner product given for any $X, Y \in \C^{2 \times 2}$ by
\begin{equation}\label{prodotto scalare traccia matrici}
\langle X, Y \rangle := {\rm Tr}(X Y^*)\,, \qquad Y^* = \overline Y^t \,.
\end{equation}
For any $A \in \C^{2 \times 2}$, denote by $M_L(A)$, $M_R(A)$ the linear operators on $\C^{2 \times 2}$, defined for any $X \in \C^{2 \times 2}$ as 
left respectively right multiplication by $A$,
$$
M_L(A) X := A X\,, \qquad M_R(A) X := X A\,.
$$ 
For what follows it is convenient to associate to arbitrary vectors $v, w \in \C^2$ the $2 \times 2$ matrix $(v \,\, w)$ defined as
$$
(v \,\, w ) := \begin{pmatrix}
v_1 & w_1 \\
v_2 & w_2
\end{pmatrix}\,, \quad \text{where} \quad 
v := \begin{pmatrix}
v_1 \\
v_2
\end{pmatrix}\,, \quad w := \begin{pmatrix}
w_1 \\
w_2
\end{pmatrix}\,. 
$$
Furthermore, for any $A \in \C^{2 \times 2}$ denote by ${\rm spec}(A)$ the spectrum of $A$ and recall that ${\rm spec}(A) = {\rm spec}(A^t)$. 
\begin{lemma}\label{properties operators matrices}
$(i)$ Let $A \in \C^{2 \times 2}$. Then any $\lambda \in {\rm spec}(A)$ is an eigenvalue of the operators $M_L(A)$ and $M_R(A)$. More precisely for any $v, w \in \C^2$, with $A v = \lambda v$ and  $A^t w = \lambda w$, one has for any $\alpha, \beta \in \C $, 
$$
M_L(A)(\alpha v\,\, \beta v) = \lambda (\alpha v \,\, \beta v)\,,  \quad  M_R(A)(\alpha w \,\, \beta w)^t = \lambda (\alpha w \,\, \beta w)^t\,.
$$

\noindent
$(ii)$ For any $A, B \in \C^{2 \times 2}$, $\lambda \in {\rm spec}(A)$, $\mu \in {\rm spec}(B)$ and for any $v = \begin{pmatrix}
v_1 \\
v_2
\end{pmatrix}$, $w = \begin{pmatrix}
w_1 \\
w_2
\end{pmatrix}$ in $\C^2$ with $A v = \lambda v$, $B^t w = \mu w$, $\lambda \pm \mu$ is an eigenvalue of $M_L(A) \pm M_R(B)$, namely 
$$
\big(M_L(A) \pm M_R(B)  \big)(w_1v \,\, w_2 v) = (\lambda \pm \mu)(w_1v \,\, w_2 v)\,.
$$

\noindent 
$(iii)$ Let $A\in \C^{2 \times 2}$ be self-adjoint. Then  $M_L(A)$ and $M_R(A)$ are self-adjoint operators on $\C^{2 \times 2}$ with respect to the scalar product defined in \eqref{prodotto scalare traccia matrici}.
\end{lemma}
\begin{proof}
$(i)$ 
One has 
$$
M_R(A)(\alpha w\,\, \beta w)^t =(\alpha w \,\, \beta w)^t A  = \big( A^t (\alpha w\,\,\beta w)\big)^t = \lambda (\alpha w\,\, \beta w)^t\,.
$$
Similarly one proves $ M_L(A)(\alpha v\,\, \beta v) = \lambda (\alpha v \,\, \beta v) $. 

\noindent
$(ii)$ By item $(i)$ one has  
$$
M_L(A)(w_1 v \,\, w_2 v)  = \lambda (w_1v \,\, w_2 v)\,
$$
and using that $(w_1 v\,\, w_2 v)^t = (v_1 w\,\, v_2 w)  $ 
$$
M_R(B)(w_1v\,\,w_2 v) = (w_1v\,\,w_2 v) B = \big(B^t (w_1 v\,\, w_2 v)^t \big)^t =  
\big(B^t (v_1 w \,\, v_2 w) \big)^t = \mu (v_1 w\,\, v_2 w)^t = \mu (w_1 v\,\, w_2 v)\,.
$$
Altogether this proves item $(ii)$. 

\noindent
$(iii)$ For any $X, Y \in \C^{2 \times 2}$ 
$$
\langle M_L(A) X , Y \rangle \stackrel{\eqref{prodotto scalare traccia matrici}}{=} {\rm Tr}(A X Y^*) = {\rm Tr}(X Y^* A) \stackrel{A = A^*}{=} {\rm Tr}(X (A Y)^*) = \langle X, M_L(A)Y \rangle\,. 
$$
The self-adjointness of $M_R(A)$ is verified similarly. 
\end{proof}


\subsection{Homological equation} \label{the-reducibility-step}

We now show how, at the $\nu$th step of the KAM iteration scheme, described in Theorem \ref{thm:abstract linear reducibility}, one constructs a symplectic transformation  
$$
\Phi_\nu := \exp( - \Psi_\nu )  = {\mathbb I}_2 - \Psi_\nu + \ldots 
$$
so that  $ {\bf L}_{\nu+1 } = \Phi_\nu^{- 1} {\bf L}_\nu \Phi_\nu $ has the desired properties.
Recall that for any $\nu \ge 0$, ${\bf L}_\nu$ is of the form \eqref{Lnu+1}, ${\bf L}_\nu = \omega \cdot \partial_\vphi {\mathbb I}_2 + {\bf N}_\nu + {\bf R}_\nu$, and $ \Psi_\nu $ is required to be a linear Hamiltonian vector field acting on  $ h_\bot^\s \times h_\bot^\s, $ 
\begin{equation}\label{structure Psi nu}
\Psi_\nu =
 J
\left(
\begin{array}{cc}
  \Psi_\nu^{(1)}  &   \Psi_\nu^{(2)}     \\
 \overline\Psi_\nu^{(2)}       &    \overline\Psi_\nu^{(1)} \\
\end{array}
\right),  \quad   \Psi_\nu^{(1)}  =   \big(  \Psi_\nu^{(1)} \big)^*, \
 \Psi_\nu^{(2)}   =  \big( \Psi_\nu^{(2)}  \big)^t \, .
\end{equation}
The map $ \Psi_\nu $ will be chosen to be a trigonometric polynomial in $ \vphi $,
\be\label{def-psi-step}
 \Psi_\nu (\vphi) = \sum_{\ell \in \Z^S, |\ell| \leq N_\nu}  \hat\Psi_\nu (\ell) e^{ \ii \ell\cdot \vphi } \, ,
 \quad   \hat \Psi_\nu (\ell) \in {\cal L}( h^{\s'}_\bot \times h^{\s'}_\bot ) \,  , \ \s' \in \{ \s-2, \sigma - 1, \s \} \, .  
\ee
 With $\Pi_{N_\nu}$ denoting the projector introduced in \eqref{SM}, 
and $\Pi_{N_\nu}^\bot = {\rm Id} - \Pi_{N_\nu}$ we write
\begin{align}
{\bf L}_\nu \Phi_\nu
& =  \Phi_\nu 
 \big( \om \cdot \partial_\vphi {\mathbb I}_2    +  {\bf N}_\nu  \big) + \big( - (\om \cdot \partial_\vphi )   \Psi_\nu 
- \left[ {\bf N}_\nu, \Psi_\nu \right] + \Pi_{N_\nu} {\bf R}_\nu \big) + \widetilde{\bf R}_\nu\,, \label{1stepNM}  
\end{align}
where
\begin{equation}\label{bf R tilde nu}
 \widetilde{\bf R}_\nu := ( \om \cdot \partial_\vphi ) ( \Phi_\nu -{\mathbb I}_2 + \Psi_\nu ) + 
 \left[ {\bf N}_\nu, \Phi_\nu - {\mathbb I}_2 + \Psi_\nu \right] + \Pi_{N_\nu}^\bot {\bf R}_\nu +  {\bf R}_\nu ( \Phi_\nu -{\mathbb I}_2 ) \,.
\end{equation}
We remark that in a non-analytic setup such as ours, it is necessary for the convergence of the KAM scheme, to consider in \eqref{1stepNM}, the truncation $\Pi_{N_\nu} {\bf R}_\nu$ of the Fourier expansion of $ {\bf R}_\nu  $.

We look for a solution of the {\it homological} equation
\be\label{eq:homo}
- ( \om \cdot \partial_\vphi )    \Psi_\nu  - \left[{\bf N}_\nu, \Psi_\nu \right] + \Pi_{N_\nu} {\bf R}_\nu = {\bf R}^{nf}_\nu 
\ee
where $  {\bf R}^{nf}_\nu  $ is given by 
\be\label{bf R nf A 1 nu}
{\bf R}^{nf}_\nu  :=  
J \left(
\begin{array}{cc}
{\bf A}_\nu^{(1)}  &   0   \\
0     &     \overline{\bf A}_\nu^{(1)}   \\
\end{array}
\right),  \quad {\bf A}_\nu^{(1)} := \hat{\bf R}^{(1)}_\nu(0)^{\rm diag}
  \, .
\end{equation}
We recall that 
$ \hat{\bf R}^{(1)}_\nu(0)^{\rm diag} $ is defined in  \eqref{def:diag-NF}
and $ \hat{\bf R}_\nu^{(1)}(0)$ denotes the $0$th Fourier coefficient of ${\bf R}_\nu$,
$$ 
\hat{\bf R}_\nu^{(1)}(0) = 
\frac{1}{(2 \pi)^{|S|}} \int_{\T^{S}} {\bf R}_\nu^{(1)}(\vphi)\, d \vphi \, . 
$$ 
By \eqref{Remainder nu},  $ {\bf A}^{(1)}_\nu=   \big({\bf A}^{(1)}_{\nu}\big)^* $. 
For any $\ell \in \Z^S$  and $j, k \in S_+^\bot$, let us introduce the following linear operators on the vector space $\C^{2 \times 2}$ 
of $2 \times 2$ matrices with complex coefficients,
\begin{align}\label{L + nu}
& L^+_\nu(\ell, j, k) \equiv L^+_\nu(\ell, j, k; \omega) := \omega \cdot \ell \,\, {\rm Id}_{\C^{2 \times 2}} + M_L([{\bf N}^{(1)}_\nu]_j^j) + M_R([\overline{\bf N}^{(1)}_\nu]_k^k) \\
& \label{L - nu}
L^-_\nu(\ell, j, k) \equiv L^-_\nu(\ell, j, k; \omega) := \omega \cdot \ell \,\, {\rm Id}_{\C^{2 \times 2}} + M_L([{\bf N}^{(1)}_\nu]_j^j) - M_R([{\bf N}^{(1)}_\nu]_k^k)\,,
\end{align}
where $ {\rm Id}_{\C^{2 \times 2}}$ denotes the identity operator on $\C^{2 \times 2}$. Note that apart from the sign, $L^-_\nu(\ell, j, k)$ differs from $L_\nu^+(\ell, j, k)$ since $L^-_\nu(\ell, j, k)$ involves
the operator  $M_R([{\bf N}_\nu^{(1)}]_k^k)$  rather than  $M_R([\overline{\bf N}_\nu^{(1)}]_k^k)$.

Furthermore, let $\Omega_{0}^\gamma(\io) :=  \Omega_{o}(\io)$ (cf \eqref{definition Omega-o}),  
and for any $\nu \ge 0,$ let
$\Omega_{\nu + 1}^\gamma(\io)$ be the subset of $\Omega_{\nu}^\gamma(\io),$ consisting of all $\omega \in \Omega_\nu^\gamma(\io)$ satisfying the so-called second order Melnikov conditions: 

\medskip

\noindent
${ ({\bf M}^{II}_{+, \gamma})}_{\nu + 1}$ $\forall \ell \in \Z^S$, $|\ell| \leq N_\nu$, $\forall j, k \in S_+^\bot$, the operator $L^+_\nu(\ell, j, k; \omega)$ is invertible and 
\be\label{Hyp1}
\Big\| L^{+}_\nu(\ell, j, k; \omega)^{- 1}  \Big\| \leq \frac{ \langle \ell \rangle^\tau }{\gamma \langle j^2 + k^2 \rangle} 
\ee

\noindent
${ ({\bf M}^{II}_{-, \gamma})}_{\nu + 1}$ $\forall \ell \in \Z^S$, $|\ell| \leq N_\nu$, $\forall j, k \in S_+^\bot$ with $(\ell, j, k) \neq (0, j, j)$, the operator $L^-_\nu(\ell, j, k; \omega)$ is invertible and  
\be\label{Hyp2}
\Big\| L^-_\nu(\ell, j, k; \omega)^{- 1} \Big\| \leq  \frac{   \langle \ell \rangle^\tau }{  \gamma\langle j^2 - k^2 \rangle}\,.
\ee
Since $[{\bf N}^{(1)}_\nu]_j^j$ is self-adjoint 
it follows from Lemma \ref{properties operators matrices} $(iii)$ that $L_\nu^{\pm}(\ell, j, k)$ are self-adjoint operators 
on $\C^{2 \times 2}$ for any $\ell \in \Z^S$ and $j, k \in S_+^\bot$. 
Therefore conditions \eqref{Hyp1}, \eqref{Hyp2} are lower bounds for the modulus 
of the eigenvalues of $L_\nu^{\pm}(\ell , j, k)$. 
Note that by Lemma \ref{properties operators matrices} $(ii)$, the operator $L_\nu^{-}(0, j, j)$ has a zero eigenvalue, hence condition \eqref{Hyp2} is violated for $(\ell, j, k) = (0, j, j)$. 

In the next lemma Condition \eqref{Hyp1} will be used to reduce ${\bf R}^{(2)}_\nu$, 
whereas \eqref{Hyp2} will be used for ${\bf R}^{(1)}_\nu$. 

\begin{lemma} \label{lemma:redu} {\bf (Homological equation)}
For any $ \om \in {\Om}^{\gamma}_{\nu + 1}(\io)  $ there exists a unique solution
$ \Psi_\nu $ of the form \eqref{structure Psi nu}
of the homological equation \eqref{eq:homo} with the normalization 
$ [\hat\Psi^{(1)}_\nu (0)]_j^j = 0 $, $ j \in S_+^\bot $. For any $s_0 \leq s \leq s_* - \bar \mu$,
the map $ \Psi_\nu $ satisfies the following estimates
\begin{align}\label{PsiR sup}
& \left|\Psi_\nu \right|_{s, \s} \,,\,\,\,
\left|\Psi_\nu {\mathfrak D} \right|_{s, \s-1}\ 
 \lessdot    
\gamma^{-1} \left|{\bf R}_\nu{\mathfrak D}\right|_{s, \s-1} N_\nu^{\tau} \\
& \label{PsiR}
\left|\Psi_\nu \right|_{s, \s}^{\Lipg} \,,\,\,\,
\left|\Psi_\nu {\mathfrak D} \right|_{s, \s-1}^{\Lipg}\ 
 \lessdot    
\gamma^{-1} \left|{\bf R}_\nu{\mathfrak D}\right|_{s, \s-1}^{\Lipg} N_\nu^{2\tau + 1} \, .
\end{align}
As a consequence $\Psi_\nu \in H^s(\T^S, {\cal L}( h_\bot^{\sigma - 2}))$ and
\begin{equation}\label{estimate Psi nu sigma - 2}
|\Psi_\nu |_{s, \sigma - 2}^{\Lipg} \lessdot  \gamma^{-1} \left|{\bf R}_\nu{\mathfrak D}\right|_{s, \s-1}^{\Lipg} N_\nu^{2\tau + 1}
\end{equation}
\end{lemma}

\begin{proof} 
To simplify notations in this proof, we frequently drop the index $ \nu $ in 
$N_\nu$, $\Psi_\nu$, ${\bf R}_\nu$ and simply write  $N$, $\Psi$, ${\bf R}$ instead.
For any $\omega$ in ${\Om}^{\gamma}_{\nu}(\io) $,
the homological equation \eqref{eq:homo}, when expressed in Fourier coefficients, reads 
$$
\ii  \om \cdot \ell \, \hat\Psi (\ell) +
 \big[{\bf N} , \hat\Psi (\ell) \big]    =  \hat{\bf R} (\ell) - \hat{\bf R}^{nf}(\ell) \, , \quad \forall \ell \in \Z^S \, , \ |\ell| \leq N \, .
$$
 In view of \eqref{def-psi-step} it suffices to consider the equations for the components 
$ \hat\Psi^{(1)}(\ell) $ and $ \hat\Psi^{(2)}(\ell)  $ with $|\ell| \leq N$,
\begin{align*}
& \om \cdot \ell \, \hat\Psi^{(2)} (\ell) +
{\bf N}^{(1)} \hat\Psi^{(2)} (\ell) +
\hat\Psi^{(2)} (\ell) \, \overline{\bf N}^{(1) \,  }  =    - \ii \hat{\bf R}^{(2)} (\ell)  \,, \\
&
\om \cdot \ell \, \hat\Psi^{(1)} (\ell) +
{\bf N}^{(1)} \hat\Psi^{(1)} (\ell)  -
\hat\Psi^{(1)} (\ell) \, {\bf N}^{(1)}  =   - \ii \hat{ {\bf R}}^{(1)} (\ell) + \ii \hat{\bf R}^{(1)}(0)^{\rm diag} \,\,\delta_{0, \ell}
\end{align*}
where  $\delta_{0, \ell } = 0$ for $\ell \neq 0$ and $\delta_{0, 0} = 1$. 
Taking into account that $[\hat\Psi^{(1)} (0)]_k^k = 0$ by the chosen normalization, the following
 equations then need to be solved ($|\ell| \leq N$, \, $j, k \in S_+^\bot$)
\begin{align*}
&  \om \cdot \ell \, \ [\hat\Psi^{(2)} (\ell)]_j^k +
\big[{\bf N}^{(1)}\big]^j_j [\hat\Psi^{(2)} (\ell)]_j^k +
[\widehat\Psi^{(2)} (\ell)]_j^k  \big[\overline{\bf N}^{(1) } \big]^k_k =    - \ii [\hat{\bf R}^{(2)} (\ell)]^k_j  \, , \quad \forall (\ell, j, k)\,, \\
& 
 \om \cdot \ell \, \ [\hat\Psi^{(1)} (\ell)]_j^k +
\big[{\bf N}^{(1)}\big]^j_j [\hat\Psi^{(1)} (\ell)]_j^k -
[\hat\Psi^{(1)} (\ell)]_j^k  \big[{\bf N}^{(1)}\big]^k_k  =   - \ii[\hat{ {\bf R}}^{(1)} (\ell)]^k_j\,, \quad \forall (\ell, j, k) \neq (0, j, j)\,.
\end{align*}
For any $ \om \in {\Om}^{\gamma}_{\nu +1}(\io) $, these equations admit unique solutions. 
We have
\begin{align}\label{Psi 2 j k ell}
& [\hat\Psi^{(2)} (\ell)]_j^k =
- \ii L^{+}(\ell, j, k)^{-1}  [{\hat{\bf R}}^{(2)} (\ell)]^k_j\,, 
\quad \forall \ell \in \Z^S, \quad |\ell| \leq N, \quad j, k \in S_+^\bot\,, \\
& \label{Psi 1 j k ell}
[\hat\Psi^{(1)} (\ell)]_j^k =
- \ii L^{-}(\ell, j, k)^{-1}  [{\hat{\bf R}}^{(1)} (\ell)]^k_j\,,\quad 
\forall \ell \in \Z^S, \quad |\ell| \leq N, \quad (\ell, j, k) \neq (0, j, j) \,.
\end{align}
The remaining Fourier coefficients of $\Psi^{(1)}$ and $\Psi^{(2)}$ are set equal to $0$.
By \eqref{Hyp1}, \eqref{Hyp2} we deduce
$$
 \| [\hat\Psi^{(2)} (\ell)]_j^k \| \lessdot \frac{ N^\tau}{\g \langle j^2 + k^2 \rangle} \|  [{\hat{\bf R}}^{(2)} (\ell)]^k_j \|\, , \quad 
\| [\hat\Psi^{(1)} (\ell)]_j^k \| \lessdot \frac{ N^\tau}{\g \langle j^2 - k^2 \rangle} \|  [{\hat{\bf R}}^{(1)} (\ell)]^k_j \| \,.
$$
\noindent
{\it Estimate for $|\Psi {\frak D}|_{s, \sigma - 1}$:} In view of the definition operator norm \eqref{def norm1}, 
we need to estimate $\|\hat{\Psi}^{(1)}(\ell) \lla D \rra \|_{\sigma - 1}$.
For any $ h \in h^\s_\bot$ we have 
\begin{align}
\| (\hat\Psi^{(1)}(\ell)  \lla D \rra ) h \|_{\s-1}^2   \lessdot
\sum_{j \in S_+^\bot} \Big(  \sum_{k \in S_+^\bot} & \| [\hat\Psi^{(1)} (\ell)]_j^k \| \,  \lla k \rra \, | (h_{-k}, h_k)  |
 \Big)^2 \langle j \rangle^{2 (\s-1)} \nonumber  \\
 \lessdot  N^{2 \t} \g^{-2} \sum_{j \in S_+^\bot} \Big(
  \| [ \hat {\bf R}^{(1)} (\ell)]_j^j \|  \,  j \, | (h_{-j}, h_j)  | & +
 \sum_{k \in S_+^\bot, k \neq j}
 \frac{\| [ \hat {\bf R}^{(1)} (\ell)]_j^k \| }{| j - k| } \,  \frac{k}{ j + k } \, | (h_{-k}, h_k)  |
  \Big)^2 \langle j \rangle^{2 (\s-1)}\,. \nonumber
  \end{align} 
 Since 
 $$
 \sum_{j \in S_+^\bot}
  \|  [ \hat{\bf R}^{(1)} (\ell) \lla D \rra ]_j^j \|^2  \, | (h_{-j}, h_j)  |^2 \langle j \rangle^{2 (\s-1)} \stackrel{Lemma\, \ref{bound op2}\,(i)}{\lessdot} \| \hat{\bf R}^{(1)} (\ell) \lla D \rra \|^2_{{\cal L}(h^{\s-1}_\bot)}  \, \| h \|_{\s-1}^2 
 $$
 and 
 $$
\sum_{j \in S_+^\bot} \Big(\sum_{k \in S_+^\bot, k \neq j}
\frac{\| [ \hat{\bf R}^{(1)} (\ell)]_j^k \| }{| j - k| }  \, | (h_{-k}, h_k)|  \Big)^2 \langle j \rangle^{2 (\sigma - 1)} \stackrel{Lemma\, \ref{bound op2}\,(iii)}{\lessdot} \| \hat{\bf R}^{(1)} (\ell) \|^2_{{\cal L}(h^{\s-1}_\bot)} \| h \|_{\sigma - 1}^2\,, 
 $$
 one sees that
$$ 
\| \hat\Psi^{(1)}(\ell ) \lla D \rra  \|_{{\cal L}(h^{\s-1}_\bot)}    \lessdot N^\t \g^{-1} \|  \hat{\bf R}^{(1)}(\ell) 
\lla D \rra  \|_{{\cal L}(h^{\s-1}_\bot )} \,.
$$
A similar bound holds for $ \hat\Psi^{(2)}(\ell) $, hence in view of the definition of the operator norm \eqref{def norm1} 
$$
|\Psi {\frak D}|_{s, \sigma - 1} \lessdot  N^{\tau }  
\gamma^{-1} \left|{\bf R}{\mathfrak D}\right|_{s, \s-1}\,.
$$

\noindent
{\it Estimate for $|\Psi|_{s, \sigma}$:} Since
\begin{align}
\sum_{j \in S_+^\bot} \Big(  \sum_{k \in S_+^\bot} & \| [\hat\Psi^{(1)} (\ell)]_j^k \| \,  \, | (h_{-k}, h_k)  |
 \Big)^2 \langle j \rangle^{2 \sigma}  \nonumber \\
 \lessdot  N^{2 \t} \g^{-2} \sum_{j \in S_+^\bot} \Big(
  \| [ \hat {\bf R}^{(1)} (\ell)]_j^j \|  \,  j \, | (h_{-j}, h_j)  | & +
 \sum_{k \in S_+^\bot, k \neq j}
 \frac{\| [ \hat{\bf R}^{(1)} (\ell)]_j^k \| }{| j - k| } \,  \frac{j}{ j + k } \, | (h_{-k}, h_k)  |
  \Big)^2 \langle j \rangle^{2 (\s-1)}\,, \nonumber
  \end{align} 
  the previous arguments yield 
  $$ 
\| \hat \Psi^{(1)} (\ell )  \|_{{\cal L}(h^{\s}_\bot)}   \lessdot N^\t \g^{-1} \| \hat{\bf R}^{(1)} (\ell) \lla D \rra
\|_{{\cal L}(h^{\s-1}_\bot )} \, .
$$
Similar estimates also hold for $  \hat\Psi^{(2)}(\ell)  $  and hence $|\Psi|_{s, \sigma} \lessdot N^\tau \gamma^{- 1} |{\bf R} {\frak D}|_{s, \sigma - 1}$. 

\noindent
{\it Estimate for $|\Psi {\frak D}|_{s, \sigma - 1}^{{\rm lip}}$:} Let us first estimate
$|\Psi^{(1)} \lla D \rra |_{s, \sigma - 1}^{{\rm lip}}$. 
For any $\omega_1, \omega_2 \in \Omega_{\nu + 1}^\gamma(\io)$ one has 
\begin{align*}
L^{-}(\ell, j, k;\omega_1)^{- 1} - L^{-}(\ell, j, k;\omega_2)^{- 1} =  L^{-}(\ell, j, k ; \omega_2)^{- 1} \big( L^{-}(\ell, j, k ; \omega_2) - L^{-}(\ell, j, k ; \omega_1) \big) L^{-}(\ell, j, k ; \omega_1)^{- 1}
\end{align*}
with $L^{-}(\ell, j, k ; \omega_2) -L^{-}(\ell, j, k ; \omega_1)$ given by
$$
(\omega_2 - \omega_1) \cdot \ell +
M_L\big([{\bf N}^{(1)}(\omega_1) - {\bf N}^{(1)}(\omega_2) ]_j^j  \big)  - 
M_R\big([{\bf N}^{(1)}(\omega_1) - {\bf N}^{(1)}(\omega_2) ]_k^k \big)\,.  \nonumber
$$
Since by  
\eqref{N nu - N 0 lip gamma},  $\| [{\bf N}^{(1)} ]_j^j  \|^{ {\rm lip}} \lessdot 1$ for any $j \in S_+^\bot$, 
we get 
\begin{align*}
 \| L^-(\ell, j, k ; \omega_2) - L^-(\ell, j, k ;\omega_1)  \| & \lessdot  \langle \ell \rangle  |\omega_1 - \omega_2| \lessdot N |\omega_1 - \omega_2|, \quad \forall \ell \in \Z^S\,\, \mbox{with} \,\, |\ell| \leq N\,.
\end{align*}
This together with \eqref{Hyp2} yields
$$
 \| L^-(\ell, j, k ; \omega_1)^{- 1} - L^-(\ell, j, k ; \omega_2)^{- 1} \| \lessdot  \frac{N^{2 \tau + 1}}{\gamma^2 \langle j^2 - k^2 \rangle^2} |\omega_1 - \omega_2| \, .
$$
Arguing as in the proof of the estimate for $ \| \hat{\Psi}^{(1)}(\ell) \lla D \rra \|_{{\cal L}(h^{\sigma - 1}_\bot)}$, we get that for any $\ell \in \Z^S$, $|\ell| \leq N$,
\begin{align}
 \| \big(\hat\Psi^{(1)}(\ell ; \omega_1) - 
\hat\Psi^{(1)}(\ell ; \omega_2) \big) \lla D \rra   \|_{{\cal L}(h^{\sigma - 1}_\bot)} & \lessdot 	
{N^\tau}\gamma^{- 1} \| \big(\hat{\bf R}^{(1)}(\ell ; \omega_1) - 
\hat{\bf R}^{(1)}(\ell ; \omega_2) \big) \lla D \rra \|_{{\cal L}(h^{\sigma - 1}_\bot)} \nonumber\\
& \quad + N^{2 \tau + 1} \gamma^{- 2} |\omega_1 - \omega_2|  \| \hat{\bf R}^{(1)}(\ell ; \omega_2  )\|_{{\cal L}(h^{\sigma - 1}_\bot)} 
\nonumber
\end{align}
which in view of the definition of the norm 
$| \cdot |_{s, \sigma'}^{\gamma {\rm lip}} =
 | \cdot |_{s, \sigma'}^{\sup} + \gamma | \cdot |_{s, \sigma'}^{ {\rm lip}}$ implies that 
$$
 |\hat \Psi^{(1)}(\ell) \lla D \rra |_{s, \sigma - 1}^{ {\rm lip}} 
\lessdot {N^\tau}\gamma^{- 2} \cdot
\g \|\hat{\bf R}^{(1)}(\ell) \lla D \rra \|_{{\cal L}(h^{\sigma - 1}_\bot)}^{ {\rm lip}} +
N^{2 \tau + 1} \gamma^{- 2} \| \hat{\bf R}^{(1)}(\ell)\|_{{\cal L}(h^{\sigma - 1}_\bot)}^{ {\rm sup}}
\lessdot N^{2 \tau + 1} \gamma^{- 2} |\hat {\bf R}^{(1)}(\ell) \lla D \rra |_{s, \sigma - 1}^{\Lipg}\,.
$$
In the same way one proves the corresponding estimate for 
$|\hat \Psi^{(2)} (\ell) \lla D \rra |_{s, \sigma - 1}^{ {\rm lip}}$,
yielding altogether 
$$
\g |\Psi {\frak D}|_{s, \sigma - 1}^{ { \rm lip}} \lessdot 
N^{2 \tau + 1} \gamma^{- 1} |{\bf R} {\frak D}|_{s, \sigma - 1}^{\Lipg}\,.
$$
\noindent
{\it Estimate for $ |\Psi |_{s, \sigma }^{{\rm lip}}$ :}
In the same way one shows that
$\g  |\Psi|_{s, \sigma }^{ {\rm lip}} \lessdot 
N^{2 \tau + 1} \gamma^{- 1} |{\bf R} {\frak D}|_{s, \sigma - 1}^{\Lipg}\,.$  

\smallskip

\noindent
Combining the four estimates above then proves \eqref{PsiR}. 

\noindent
{\it Estimate of $|\Psi |_{s, \sigma - 2}^{\Lipg}$}:  Since 
${\frak D} : h^{\sigma - 1}_\bot \times  h^{\sigma - 1}_\bot  
\to h^{\sigma - 2}_\bot \times  h^{\sigma - 2}_\bot$ is a linear isomorphism, it follows from \eqref{PsiR} that for any $\ell \in \Z^S$, $\hat \Psi(\ell) \in {\cal L}(h^{\sigma - 2}_\bot \times h^{\sigma - 2}_\bot)$ and that the claimed estimate \eqref{estimate Psi nu sigma - 2} holds.
\end{proof}


\subsection{Proof of Theorem  \ref{thm:abstract linear reducibility}}\label{proof of reduction theorem}

{\it Proof of $({\bf S 1})_\nu$:} We prove $({\bf S1})_\nu$ by induction  with respect to $\nu \geq 1$. In view of the smallness assumption \eqref{piccolezza1}, the proof of $({\bf S1})_1$ and the one of the inductive step are similar, hence we only consider the latter one:
Assuming that ${\bf({S}1)_{\nu}}$ is true for a given $\nu \geq 1$,
it is to prove that  $({\bf S1})_{\nu+1}$ holds.
To simplify notations we write $|\cdot|_{s, \s-1}$ instead of $|\cdot|_{s, \s-1}^{\Lipg}$.
By Lemma \ref{lemma:redu}, for any $ \omega \in \Omega_{\nu+1}^{\g}(\io) $, there exists a solution
$ \Psi_{\nu} $ of the homological equation \eqref{eq:homo} of the form \eqref{structure Psi nu}, which by  \eqref{PsiR} satisfies for any $s_0 \leq s \leq  s_* - \bar \mu $
\be\label{Psinu}
\left|\Psi_\nu \right|_{s, \s} \, , \ 
\left|\Psi_{\nu} {\frak D} \right|_{s, \s-1} \stackrel{\eqref{PsiR}}{\lessdot}
N_{\nu}^{2\tau + 1}\gamma^{-1} \left|{\bf R}_{\nu}  {\frak D} \right|_{s, \s-1} \,.
\ee
By the induction hyphothesis, \eqref{Rsb} holds for any $s_0 \leq s \leq s_* - \bar \mu - \beta$  and hence
\begin{equation}\label{Psinu2}
\left|\Psi_\nu \right|_{s, \s} \, , \ 
\left|\Psi_{\nu} {\frak D} \right|_{s, \s-1}
{\lessdot} N_{\nu}^{2\tau + 1}\ N_{\nu-1}^{- \a}  \gamma^{-1} \left|{\bf R}_{0}  {\frak D} \right|_{s  + \beta, \s-1} 
\end{equation}
which is the estimate \eqref{Psinus} at the inductive step $ \nu +1 $. 
It follows that for any $ \vphi \in \T^S $, $ \Phi_\nu (\vphi ) = \exp(-\Psi_\nu(\vphi) )$
is bounded and invertible when viewed as an operator on $h^{\s-2}_\bot \times h^{\s-2}_\bot$. 
Furthermore, in view of the definition \eqref{defN} of $N_\nu$ and \eqref{alpha beta} of 
$\alpha \equiv \alpha(\tau)$ and by the assumption $\tau \ge |S|+1,$ 
it also follows that for any $s_0 \leq s \leq s_* - \beta$,
$\Phi_\nu^{\pm 1} = {\rm exp}(\mp \Psi_\nu)$ are maps in $H^s(\T^S, {\cal L}(h^{\sigma - 2}_\bot \times h^{\sigma - 2}_\bot))$ and $H^s(\T^S, {\cal L}(h^{\sigma }_\bot \times h^{\sigma }_\bot ))$. 
By \eqref{1stepNM} and \eqref{eq:homo} one has
$$
{\bf L}_{\nu + 1} = \Phi^{-1}_\nu {\bf L}_\nu \Phi_\nu = \om \cdot \partial_\vphi {\mathbb I}_2   +
{\bf N}_{\nu + 1}  + {\bf R}_{\nu + 1} 
$$
where
\begin{equation}\label{D+}
{\bf N}_{\nu + 1} := {\bf N}_\nu + {\bf R}^{nf}_{\nu } \, , \quad
{\bf R}_{\nu + 1} := \Phi^{-1}_\nu \tilde {\bf R}_\nu +  (\Phi^{-1}_\nu - {\mathbb I}_2 )  {\bf R}^{nf}_\nu
\end{equation}
and $\widetilde{\bf R}_\nu$ is defined in \eqref{bf R tilde nu}. By construction,  ${\bf N}_{\nu + 1}$  is of the form \eqref{normal-form}-\eqref{normal-form-specific}. In particular by \eqref{bf R nf A 1 nu},
$[{\bf N}_{\nu + 1}^{(1)} - {\bf N}_\nu^{(1)}]_k^k =  [ \hat{\bf R}_\nu^{(1)}(0)]_k^k$ 
 for any $k \in S_+^\bot$ and hence 
\begin{equation}\label{nuovadiagonale}
\| [{\bf N}_{\nu + 1}^{(1)} - {\bf N}_\nu^{(1)}]_k^k  \|^{\gamma {\rm lip}} \lessdot |{\bf R}_\nu {\frak D}|_{s_0, \sigma - 1}^{\gamma {\rm lip}} k^{- 1}\,,
\end{equation}
establishing the first estimate of \eqref{N nu - N 0 lip gamma} at the inductive step $\nu + 1$.
To prove the second estimate 
write 
$
[{\bf N}_{\nu+1}^{(1)}]_j^j  = [{\bf N}^{(1)}_0]_j^j + 
\sum_{n = 1}^{\nu+1} [{\bf N}_n^{(1)} - {\bf N}_{n - 1}^{(1)}]_j^j $ as a telescoping sum, 
and use the estimates 
\begin{equation}\label{estimates of the initial normal form}
\| [{\bf N}^{(1)}_0]_j^j  
\|^{\rm lip} 
\stackrel{\eqref{first diagonal normal form}, \eqref{lip-order-1}}\lessdot 1\,, \quad \forall j \in S^\bot \, , 
\end{equation}
$\big\| [{\bf N}_n^{(1)} - {\bf N}_{n - 1}^{(1)}]_j^j
 \big\|^{\gamma {\rm lip}} \lessdot | {\bf R}_{n- 1} {\frak D}|_{s_0,\s-1}^{\gamma {\rm lip}} j ^{-1}
$ (by  \eqref{N nu - N 0 lip gamma}), and
$\left|{\bf R}_{n-1} {\mathfrak D}
\right|_{s_0, \s-1}^{\Lipg} \leq  
\left|{\bf R}_{0} {\mathfrak D}  \right|_{s_0+\beta, \s-1}^{\Lipg}N_{n - 2}^{-\alpha} 
$ (by \eqref{Rsb})
to conclude that
$
  \| [{\bf N}^{(1)}_{\nu+1}]_j^j  \|^{ {\rm lip}} \lessdot 1 + \g^{-1}
\left|{\bf R}_{0} {\mathfrak D}  \right|_{s_0+\beta, \s-1}^{\Lipg} \stackrel{\eqref{piccolezza1}} \lessdot 1 $.

Since by Lemma \ref{transformation of Hamiltonian operators}, ${\bf L}_{\nu + 1}$ is 
a linear Hamiltonian operator, so is $  {\bf R}_{\nu + 1} $ and hence has the form \eqref{Remainder nu}. 
It remains to verify the claimed estimate \eqref{Rsb} for ${\bf R}_{\nu + 1}$.
To this end, we first need to establish estimates for $\Phi_\nu^{\pm 1}$
which we derive from Lemma \ref{lem:inverti}. 
Indeed, one has
 \begin{align}\label{Phis0}
 & \left|(\Phi_{\nu}^{\pm 1} - {\mathbb I}_2 ) {\frak D}  \right|_{s, \s-1} \stackrel{Lemma\,\, \ref{lem:inverti}\,\,(ii)}{\leq_s}  | \Psi_\nu  {\frak D} |_{s, \s-1} \stackrel{\eqref{Psinu}}{\leq_s} N_{\nu}^{2\tau + 1} \gamma^{- 1}\left|{\bf R}_{\nu}  {\frak D} \right|_{s, \s-1}\, , \\
& 
  \left| \Phi_{\nu}^{\pm 1} - {\mathbb I}_2  \right|_{s, \s } \stackrel{Lemma\,\, \ref{lem:inverti}\,\,(i)}{\leq_s}  | \Psi_\nu |_{s, \s} \stackrel{\eqref{Psinu}}{\leq_s} N_{\nu}^{2\tau + 1} \gamma^{- 1}\left|{\bf R}_{\nu}  {\frak D} \right|_{s, \s-1}  \, . \nonumber
\end{align}
We now estimate ${\bf R}_{\nu+1} = \Phi_\nu^{-1} \widetilde{\bf R}_{\nu} + (\Phi_\nu^{- 1} - {\mathbb I}_2) {\bf R}_\nu^{nf}$ where we recall that
$$
 \widetilde{\bf R}_{\nu} := (\om \cdot \partial_\vphi) ( \Phi_\nu -{\mathbb I}_2 - \Psi_\nu ) + 
 \left[ {\bf N}_\nu, \Phi_\nu - {\mathbb I}_2 - \Psi_\nu \right] +  (\Pi_{N_\nu}{\bf R}_\nu) ( \Phi_\nu -{\mathbb I}_2 ) +  
(\Pi_{N_\nu}^{\bot}  {\bf R}_\nu) \Phi_\nu\,.
$$
The terms in ${\bf R}_{\nu + 1}$ are estimated individually. 
One has 
$$
(\omega \cdot \partial_\vphi)(\Phi_\nu - {\mathbb I}_2 - \Psi_\nu) = 
\sum_{n \geq 2} (- 1)^n \frac{(\omega \cdot \partial_\vphi)(\Psi_\nu^n)}{n !}\,, \quad 
(\omega \cdot \partial_\vphi)(\Psi_\nu^n) = 
\sum_{n_1 + n_2 +1= n} \Psi_\nu^{n_1} (\omega \cdot \partial_\vphi \Psi_\nu) \Psi_\nu^{n_2}\,, \  \forall n \geq 2\,.
$$
Furthermore writing
$$
[{\bf N}_\nu, \Phi_\nu - {\mathbb I}_2 - \Psi_\nu] = \sum_{n \geq 2} (- 1)^n \frac{[{\bf N}_\nu, \Psi_\nu^n]}{n !}\,,
$$
and using that by the homological equation \eqref{eq:homo}, \,
$[{\bf N}_\nu, \Psi_\nu^n] = \sum_{n_1 + n_2 +1= n} \Psi_\nu^{n_1} [{\bf N}_\nu, \Psi_\nu] \Psi_\nu^{n_2}$ 
equals 
$$
- \sum_{n_1 + n_2 +1 = n}\Psi_\nu^{n_1} (\omega \cdot \partial_\vphi \Psi_\nu) \Psi_\nu^{n_2} \, + \sum_{n_1 + n_2+1 = n} \Psi_\nu^{n_1} (\Pi_{N_\nu} {\bf R}_\nu - {\bf R}_\nu^{n f}) \Psi_\nu^{n_2}\,,
$$
one obtains altogether
\begin{equation}\label{pizza margherita}
(\omega \cdot \partial_\vphi) (\Psi^n_\nu) + [{\bf N}_\nu, \Psi_\nu^n] =  \sum_{n_1 + n_2 +1= n} \Psi_\nu^{n_1} (\Pi_{N_\nu} {\bf R}_\nu - {\bf R}_\nu^{n f}) \Psi_\nu^{n_2}\,.
\end{equation}
Choosing $C(s) > 2 C_{op}(s)$ large enough with $C_{op}(s)$ as
in Lemma \ref{lem:inverti} we get for any $n \geq 2$, 
\begin{align*}
\big| \big( (\omega \cdot \partial_\vphi) & (\Psi_\nu^n) + 
[{\bf N}_\nu, \Psi_\nu^n] \big) {\frak D}\big|_{s, \sigma - 1} 
  \stackrel{\eqref{Mnab}}{\leq} 
n \big( C( s) \left|\Psi_{\nu}  {\frak D} \right|_{ s_0, \s-1} \big)^{n-1}
|{\bf R}_\nu {\frak D}|_{s, \sigma - 1} \\
& + n(n-1) \big( C( s) \left|\Psi_{\nu}  {\frak D} \right|_{ s_0, \s-1} \big)^{n-2}
 C( s) \left|\Psi_{\nu}  {\frak D} \right|_{ s, \s-1}
|{\bf R}_\nu {\frak D}|_{s_0, \sigma - 1} \\
& \stackrel{\eqref{Psinu}}{\leq} 
n^2 C(s)^{n-1} 
(\left|\Psi_{\nu}  {\frak D} \right|_{ s_0, \s-1} )^{n - 2}
N_\nu^{2 \tau + 1} \gamma^{- 1}  |{\bf R}_\nu {\frak D}|_{s_0, \sigma - 1}  
|{\bf R}_\nu {\frak D}|_{s, \sigma - 1} \,.
\end{align*}
Choosing $ N_0 = N_0(s_*, \tau, |S|) > 0 $ in \eqref{defN} large enough so that 
\be\label{Psinu0}
 \left|\Psi_{\nu}  {\frak D} \right|_{ s_0, \s-1} \stackrel{\eqref{Psinu2}} {\lessdot}
 N_{\nu}^{2\tau + 1}\ N_{\nu-1}^{- \a}  \gamma^{-1} 
\left|{\bf R}_{0}  {\frak D} \right|_{ s_0  + \beta, \s-1} 
\stackrel{\eqref{alpha beta}, \, \eqref{piccolezza1}} \leq 1
\ee
one then obtains
$$
\big| \big( (\omega \cdot \partial_\vphi)  (\Psi_\nu^n) + 
[{\bf N}_\nu, \Psi_\nu^n] \big) {\frak D}\big|_{s, \sigma - 1} 
 \stackrel{\eqref{Psinu0}}{\leq}
 n^2 C(s)^{n-1} N_\nu^{2 \tau + 1} \gamma^{- 1} |{\bf R}_\nu {\frak D}|_{s_0, \sigma - 1}  |{\bf R}_\nu {\frak D}|_{s, \sigma - 1}
$$
which implies 
$$
\big| \big( (\omega \cdot \partial_\vphi)(\Phi_\nu - {\mathbb I}_2 - \Psi_\nu) +  [{\bf N}_\nu, \Phi_\nu - {\mathbb I}_2 - \Psi_\nu] \big) {\frak D}\big|_{s, \sigma - 1} \leq_s N_\nu^{2 \tau + 1} \gamma^{- 1}
|{\bf R}_\nu {\frak D}|_{s_0, \sigma - 1} |{\bf R}_\nu {\frak D}|_{s, \sigma - 1} \,.
$$
Furthermore, by \eqref{interpm} and\eqref{Phis0} one has 
$$
|(\Pi_{N_\nu}{\bf R}_\nu) (\Phi_\nu - {\mathbb I}_2) {\frak D}|_{s, \sigma - 1}\,, \quad |(\Phi_\nu^{- 1} - {\mathbb I}_2) {\bf R}_\nu^{nf}{\frak D}|_{s, \sigma - 1} \leq_s N_\nu^{2 \tau + 1} \gamma^{- 1} |{\bf R}_\nu {\frak D}|_{s, \sigma - 1}  |{\bf R}_\nu {\frak D}|_{s_0, \sigma - 1}\,,
$$
yielding, with $\Phi_\nu = {\mathbb I}_2 + (\Phi_\nu - {\mathbb I}_2)$,
$$
|(\Pi_{N_\nu}^{\bot}  {\bf R}_\nu) \Phi_\nu {\frak D}|_{s, \sigma - 1} \leq_s 
|(\Pi_{N_\nu}^\bot {\bf R}_\nu) {\frak D}|_{s, \sigma - 1} + 
N_\nu^{2 \tau + 1} \gamma^{- 1} |{\bf R}_\nu {\frak D}|_{s, \sigma - 1}  
|{\bf R}_\nu {\frak D}|_{s_0, \sigma - 1}\,.
$$
Combining the estimates above with the estimate
$\left|\Psi_{\nu} \right|_{s, \s-1} \stackrel{\eqref{Psinu}}{\lessdot}
N_{\nu}^{2\tau + 1}\gamma^{-1} \left|{\bf R}_{\nu}  {\frak D} \right|_{s, \s-1} $
and using again \eqref{interpm} and the smallness assumption \eqref{piccolezza1} one then gets
\begin{equation}\label{Rsgen}
|{\bf R}_{\nu + 1}{\frak D} |_{s, \sigma - 1} \leq_s |(\Pi_{N_\nu}^\bot  {\bf R}_\nu)  {\frak D} |_{s, \s-1} +
  N_\nu^{2\t+1} \g^{-1} |{\bf R}_\nu  {\frak D} |_{s, \s-1} | {\bf R}_\nu  {\frak D} |_{ s_{0}, \s-1} \, ,
\end{equation} 
which by the induction hyphothesis leads to
\begin{align}
|{\bf R}_{\nu + 1}{\frak D} |_{s, \sigma - 1} &\stackrel{\eqref{smoothingN} }{\leq_s} N_\nu^{- \beta} |{\bf R}_\nu {\frak D}|_{s + \beta, \sigma - 1} + N_\nu^{2 \tau + 1} \gamma^{- 1} |{\bf R}_\nu {\frak D}|_{s, \sigma - 1} |{\bf R}_\nu {\frak D}|_{s_0, \sigma - 1} \nonumber\\
& \stackrel{\eqref{Rsb}}{\leq} C(s) \big( N_\nu^{- \beta}  N_{\nu - 1} |{\bf R}_0 {\frak D}|_{s + \beta, \sigma - 1} + N_\nu^{2 \tau + 1} \gamma^{- 1} N_{\nu - 1}^{- 2 \alpha} |{\bf R}_0 {\frak D}|_{s + \beta, \sigma - 1}  
|{\bf R}_0 {\frak D}|_{s_0 + \beta , \sigma - 1} \big)\,. \label{rtyyrrer}
\end{align}
In order to insure that $|{\bf R}_{\nu + 1}{\frak D} |_{s, \sigma - 1}$ can be bounded by $|{\bf R}_{0}{\frak D} |_{s, \sigma - 1} N_\nu^{- \alpha}$  we 
need that for any $\nu \geq 0$ 
$$
C(s) N_\nu^{- \beta} N_{\nu - 1} N_\nu^{\alpha} \leq 1/2\, \quad \text{and} \quad  C(s )N_{\nu}^{2 \tau + 1} N_{\nu - 1}^{- 2 \alpha} N_\nu^\alpha |{\bf R}_0 {\frak D}|_{s_0 + \beta , \sigma - 1} \gamma^{- 1} \leq 1/ 2\,.
$$
The latter conditions are fullfilled since by \eqref{alpha beta} $\b = \a +  1$, $\a = 6 \tau + 4$
and by \eqref{piccolezza1}, $N_0^{C_0}|{\bf R}_0 {\frak D}|_{s_0 + \beta , \sigma - 1} \g^{-1}  \le 1$, with $C_0 = 2 \tau + 2 + \alpha$, taking $N_0$ large enough. 
Thus the first inequality of \eqref{Rsb} at the inductive step $ \nu +1 $ is verified.
By \eqref{Rsgen}, applied for $ s + \b $ with $s_0 \leq s \leq s_* - \bar \mu - \beta$, 
 we get
\be \label{sch2}
|{\bf R}_{\nu + 1} {\frak D} |_{s+\b, \s-1} \,
{\leq_{s+\b}} \,  | {\bf R}_\nu {\frak D} |_{s+ \b, \s-1} +  N_\nu^{2\t+1} \g^{-1} 
|{\bf R}_\nu {\frak D} |_{s+\b, \s-1}| {\bf R}_\nu {\frak D} |_{ s_{0}, \s-1}\,. 
\ee
Then \eqref{sch2}, \eqref{Rsb}, \eqref{piccolezza1}, \eqref{alpha beta} imply
the
inequality
$$
|{\bf R}_{\nu + 1} {\frak D} |_{s+\b, \s-1} \leq_{s + \beta}   | {\bf R}_\nu {\frak D} |_{s+ \b, \s-1},
$$
whence by the induction hyphothesis \eqref{Rsb} we get
$$
|{\bf R}_{\nu + 1} {\frak D} |_{s+\b, \sigma - 1}  \leq N_{\nu} |{\bf R}_0 {\frak D} |_{s+ \b, \s-1}
$$
for $ N_0 = N_0 (s_* ,\tau, S) >0 $ in \eqref{piccolezza1} large enough, which is the second inequality of \eqref{Rsb} at the step $ \nu +1 $.

\medskip

\noindent
{\it Proof of $({\bf S2})_{\nu + 1}$:}  For any $k \in S_+^\bot$
\begin{equation}\label{paolino}
\| [{\bf N}^{(1)}_{\nu + 1}]_k^k - [{\bf N}^{(1)}_\nu]_k^k \|^{\gamma {\rm lip}} \stackrel{\eqref{nuovadiagonale}}{\lessdot} |{\bf R}_\nu {\frak D}|_{s_0, \sigma - 1} k^{- 1} \stackrel{\eqref{Rsb}}{\lessdot} N_{\nu - 1}^{- \alpha} |{\bf R}_0 {\frak D}|_{s_0 + \beta, \sigma - 1} k^{- 1}
\end{equation}
where the Lipschitz seminorm is computed on $\Omega^\gamma_{\nu + 1}(\io)$.
By Lemma M.5 in \cite{KP} and its proof, the matrix elements of $[{\bf N}_{\nu}^{\Delta}]_k^k := [{\bf N}^{(1)}_{\nu + 1}]_k^k - [{\bf N}^{(1)}_\nu]_k^k$ can be extended to all of $\Omega_o(\io)$ so that the extension $[\widetilde{\bf N}_{\nu}^{\Delta}]_k^k$ of $[{\bf N}_{\nu}^{\Delta}]_k^k$ is Lipschitz, self-adjoint  and satisfies the estimate \eqref{paolino}. $({\bf S2})_{\nu + 1}$ then follows by setting
$$
 [\widetilde{\bf N}^{(1)}_{\nu + 1}]_k^k := [\widetilde{\bf N}^{(1)}_{\nu }]_k^k +[\widetilde{\bf N}_{\nu}^{\Delta}]_k^k\,. 
$$
This concludes the proof of Theorem \ref{thm:abstract linear reducibility}.


\subsection{$2 \times 2$ block diagonalization of ${\bf L}_0$}
\label{2 times 2 block diagonalization}

In this subsection we study the limit of the sequence of operators ${\bf L}_\nu$,
introduced in Theorem \ref{thm:abstract linear reducibility}, and show that it 
is the $2 \times 2$ block diagonalization of ${\bf L}_0$.
Recall that, for any $k \in S_+^\bot$,  the $2\times 2$ matrices $[\widetilde{\bf N}_\nu^{(1)}]_k^k$, $\nu \geq 1,$  were introduced 
in $({\bf S 2})_\nu$ of Theorem \ref{thm:abstract linear reducibility} 
and that $[\widetilde{\bf N}_0^{(1)}]_k^k$ is given by $[{\bf N}_0^{(1)}]_k^k$.

\begin{lemma}\label{final blocks normal form}
Assume that \eqref{piccolezza1} holds. Then for any $k \in S_+^\bot$, the sequence 
$([\widetilde{\bf N}_\nu^{(1)}]_k^k)_{\nu \geq 0}$ converges in the norm $\| \cdot \|^\Lipg$ to a 
$\vphi$-independent $2 \times 2$ matrix $[{\bf N}_\infty^{(1)}]_k^k$. 
The limit $[{\bf N}_\infty^{(1)}]_k^k$ is self-adjoint and satisfies the estimate
\begin{equation}\label{stime blocchi 2 per 2 finali}
\|[{\bf N}_\infty^{(1)}]_k^k - [\widetilde{\bf N}_\nu^{(1)}]_k^k \|^\Lipg \lessdot  
N_{\nu - 1}^{- \alpha}|{\bf R}_0 {\frak D}|_{s_0 + \beta, \sigma - 1} k^{- 1}\,, \quad \forall \nu \geq 0\, . 
\end{equation}
\end{lemma}

\begin{proof}
Note that for any $k \in S_+^\bot$ and any $ \nu \geq 0 $
\begin{align}
\sum_{n \geq \nu + 1}\| [\widetilde{\bf N}_n^{(1)}]_k^k - [\widetilde{\bf N}_{n - 1}^{(1)}]_k^k \|^{\Lipg} 
& \stackrel{\eqref{closeness extended blocks}}{\lessdot}
\sum_{n \geq \nu +1} |{\bf R}_{n - 1} {\frak D}|_{s_0, \sigma - 1}^{\gamma {\rm lip}} k^{- 1} \nonumber\\ 
& \stackrel{\eqref{Rsb}}{\lessdot} 
|{\bf R}_0 {\frak D}|_{s_0 + \beta, \sigma - 1}^\Lipg k^{- 1}
\sum_{n \geq \nu+1} N_{n - 2}^{- \alpha} \stackrel{\eqref{defN}, \eqref{alpha beta}}{\lessdot} 
N_{\nu - 1}^{- \alpha} |{\bf R}_0 {\frak D}|_{s_0 + \beta, \sigma - 1}^\Lipg k^{- 1} \,\nonumber\,.
\end{align}
Hence the sequence  $[\widetilde{\bf N}_\nu^{(1)}]_k^k $ has a limit, 
denoted by $[{\bf N}_\infty^{(1)}]_k^k $, and  \eqref{stime blocchi 2 per 2 finali} holds. 
Since $ [ {\bf N}_0^{(1)}]_k^k$ (by  \eqref{first diagonal normal form}) 
and $[ \widetilde{\bf N}_\nu^{(1)}]_k^k$  (by $({\bf S 2})_\nu $) are self-adjoint 
so is $[{\bf N}_\infty^{(1)}]_k^k$. 
\end{proof}

In Theorem \ref{teoremadiriducibilita} below we 
 prove that $ {\bf L}_0 $ is conjugated to the normal form Hamiltonian operator 
\begin{equation}\label{bf L infinito esplicito}
{\bf L}_\infty(\omega) := \omega \cdot \partial_\vphi {\mathbb I}_2 + {\bf N}_\infty(\omega) 
\end{equation}
where
\begin{equation}\label{cal D infinito}
{\bf N}_\infty := J \begin{pmatrix}
 {\bf N}_\infty^{(1)} & 0 \\
0 &   \overline{\bf N}_\infty^{(1)}
\end{pmatrix}\,,\quad {\bf N}_\infty^{(1)} := {\rm diag}_{k \in S_+^\bot} [{\bf N}_\infty^{(1)}]_k^k \, .
\end{equation}
To this end we study the compositions of the symplectic transformations $\Phi_\nu$, $\nu \geq 0,$ 
introduced in ${\bf (S1)_\nu}$ of Theorem \ref{thm:abstract linear reducibility}. 
 For any $\nu \geq 0$, we define 
$$
\widetilde \Phi_\nu := \Phi_0 \circ \Phi_1 \circ \ldots \circ \Phi_\nu\,.
$$
\begin{lemma}  {\bf (Composition of $\Phi_\nu$)} \label{composizione trasformazioni KAM}
Assume that \eqref{piccolezza1} holds with $N_0 = N_0(s_*, \tau, |S|) > 0$ sufficiently large. 
Then on the set $ \cap_{\nu \geq 0} \Omega_\nu^\gamma(\io)$, the sequence of symplectic transformations $\widetilde \Phi_\nu$ converges to an invertible 
map $\Phi_\infty$ in the norm $|\cdot|^\Lipg_{s, \sigma'}$ for $\sigma' = \sigma, \sigma - 2$ and $s \in [s_0, s_* - \bar \mu - \beta]$.   
Moreover 
$\Phi_\infty$, $\Phi_\infty^{- 1}$ are symplectic and satisfy the estimates
$$
|\Phi_\infty^{\pm 1} - {\mathbb I}_2 |_{s, \sigma - 2 }^\Lipg\,, \quad |\Phi_\infty^{\pm 1} - {\mathbb I}_2 |_{s, \sigma }^\Lipg \quad   \leq_s \gamma^{- 1} |{\bf R}_0 {\frak D}|_{s + \beta, \sigma - 1}^\Lipg\,.
$$
\end{lemma}
\begin{proof}
To simplify notations we write $|\cdot |_{s, \sigma - 1}$ instead of $| \cdot |_{s, \sigma - 1}^\Lipg$. For any $\nu \geq 0$, write 
$$
 \Phi_\nu = {\mathbb I}_2 + \Psi_\nu^{\Sigma}\,, \quad  \Psi_\nu^{\Sigma} := \sum_{n \geq 1} \frac{\Psi_\nu^n}{n !}\,.
$$
By \eqref{Psinus} and the smallness condition \eqref{piccolezza1}, as specified in \eqref{Psinu0}, we get $C(s_*)|\Psi_\nu {\frak D}|_{s_0, \sigma - 1} \leq 1$, where $C(s)$ denotes the same constant  as
in \eqref{Psinu0}. Hence, for any $s \in [s_0, s_* - \beta]$, we obtain
\begin{equation}\label{stima widetilde Psi n}
| \Psi_\nu^\Sigma {\frak D}|_{s, \sigma - 1} \stackrel{Lemma \,\,\ref{lem:inverti}}{\leq_s}  |\Psi_\nu {\frak D}|_{s, \sigma - 1} \stackrel{\eqref{Psinus}}{\leq} \e_\nu(s)\,, \quad \e_\nu(s) := K(s) \gamma^{- 1} |{\bf R}_0 {\frak D}|_{s + \beta, \sigma - 1} N_{\nu}^{2 \tau + 1} N_{\nu - 1}^{- \alpha}
\end{equation}
for some constant $K(s) \geq C(s)$, chosen to be increasing in $s$. In particular one has 
\begin{equation}\label{estimate Phi nu epsilon nu}
|\Phi_\nu - {\mathbb I}_2 |_{s, \sigma - 1} \leq  \e_\nu(s)\,.
\end{equation}
We claim that for any $\nu \geq 0$ and $s \in [s_0, s_* - \beta]$, 
\begin{equation}\label{locarno 0}
|\widetilde \Phi_\nu - \mathbb I_2|_{s, \sigma - 1} \leq 2 \e_0(s)\,.
\end{equation}
To prove it we argue by induction. For $\nu = 0$, inequality \eqref{locarno 0} follows from \eqref{estimate Phi nu epsilon nu} since $\widetilde \Phi_0 = \Phi_0$. To prove the inductive step from $\nu$ to $\nu + 1$, we write $\widetilde \Phi_{\nu + 1} - {\mathbb I}_2$ as a telescoping sum 
\begin{equation}\label{locarno 1}
\widetilde \Phi_{\nu + 1} - {\mathbb I}_2 = \sum_{k = 0}^\nu (\widetilde \Phi_{k + 1} - \widetilde \Phi_k) + \widetilde \Phi_0 - {\mathbb I}_2\,.
\end{equation}
Using that 
$$
\widetilde \Phi_{k + 1} - \widetilde \Phi_k = (\widetilde \Phi_k - {\mathbb I}_2)(\Phi_{k + 1} - \mathbb I_2) + \Phi_{k + 1} - {\mathbb I}_2 \,,
$$
one has by Lemma \ref{prodest} and by \eqref{estimate Phi nu epsilon nu}
$$
|\widetilde \Phi_{k + 1} - \widetilde \Phi_k|_{s, \sigma- 1} \leq C_{op}(s) |\widetilde \Phi_k - {\mathbb I}_2|_{s_0, \sigma - 1} \e_{k + 1}(s) + C_{op}(s) |\widetilde \Phi_k - {\mathbb I}_2|_{s, \sigma - 1} \e_{k + 1}(s_0) + \e_{k + 1}(s)\,.
$$
By the induction hyphothesis,
$|\widetilde \Phi_k - \mathbb I_2|_{s, \sigma - 1} \leq 2 \e_0(s)$.
Since by \eqref{stima widetilde Psi n} $2 \e_0(s) \e_{k + 1}(s_0) = 2 \e_0(s_0) \e_{k + 1}(s)$
one sees that
$
|\widetilde \Phi_k - \mathbb I_2|_{s, \sigma - 1} \e_{k + 1}(s_0) \leq 2 \e_0(s_0) \e_{k + 1}(s)\,,
$
yielding with $C(s) = 2 C_{op}(s)$ altogether
$$
|\widetilde \Phi_{k + 1} - \widetilde \Phi_k |_{s, \sigma - 1} \leq (2 C(s) \e_0(s_0) + 1) \e_{k + 1}(s)\,.
$$
Substituting this estimate into \eqref{locarno 1} leads to 
$$
|\widetilde \Phi_{\nu + 1} - {\mathbb I}_2|_{s, \sigma - 1} \leq 
(2 C(s) \e_0(s_0) + 1) \sum_{k = 0}^\nu \e_{k + 1}(s) + \e_0(s)\,.
$$
With $N_0$ in \eqref{Psinus} chosen large enough, it follows that $|\widetilde \Phi_{\nu + 1} - {\mathbb I}_2|_{s, \sigma - 1} \leq 2 \e_0(s)$ and hence \eqref{locarno 0} is established. 
Finally for all $\nu_2 > \nu_1 > 0$
\begin{align}
|(\widetilde \Phi_{\nu_2} - \widetilde \Phi_{\nu_1}) {\frak D}|_{s, \sigma - 1} & \leq \sum_{\nu = \nu_1}^{\nu_2 - 1}|(\widetilde \Phi_{\nu + 1} - \widetilde \Phi_\nu) {\frak D}|_{s, \sigma - 1} \nonumber\\
& = \sum_{\nu = \nu_1}^{\nu_2 - 1} |\widetilde \Phi_\nu  \Psi_{\nu + 1}^{\Sigma} {\frak D}|_{s, \sigma - 1} \stackrel{ \eqref{interpm Lip}}{\leq_s} \sum_{\nu = \nu_1}^{\nu_2 - 1} \Big(|\widetilde \Phi_\nu|_{s, \sigma - 1} | \Psi_{\nu + 1}^\Sigma {\frak D}|_{s_0, \sigma - 1} + |\widetilde \Phi_\nu|_{s_0, \sigma - 1} |\Psi_{\nu + 1}^\Sigma {\frak D}|_{s, \sigma - 1}  \Big) \nonumber\\
& \stackrel{\eqref{stima widetilde Psi n}, \eqref{estimate Phi nu epsilon nu} }{\leq_s} \sum_{\nu = \nu_1}^{\nu_2 - 1} \Big( (1 + 2 \e_0(s)) \e_{\nu + 1}(s_0) + (1 + 2 \e_0 (s_0)) \e_{\nu + 1}(s) \Big)\,. \nonumber
\end{align}
Using again $\e_0(s) \e_{\nu + 1}(s_0) = \e_0(s_0) \e_{\nu + 1}(s)$, it then follows from the smallness assumption \eqref{piccolezza1} that
$$
|(\widetilde \Phi_{\nu_2} - \widetilde \Phi_{\nu_1}) {\frak D}|_{s, \sigma - 1} \leq_s \e_{\nu_1}(s ) \leq_s \gamma^{- 1} |{\bf R}_0 {\frak D}|_{s + \beta, \sigma - 1} N_{\nu_1 }^{2 \tau + 1} N_{\nu_1 - 1}^{- \alpha}
$$
Therefore the sequence $((\widetilde \Phi_\nu -{\mathbb I}_2  ){\frak D})_{\nu \geq 0}$ 
is a Cauchy sequence with respect to the norm $|\cdot |_{s, \sigma - 1}$ and hence converges in 
$H^s(\T^S, {\cal L}(h^{\sigma - 1}_\bot \times h^{\sigma - 1}_\bot ))$. It then follows that  
$(\widetilde \Phi_\nu)_{\nu \geq 0}$ is a Cauchy sequence in the space 
$H^s(\T^S, {\cal L}(h^{\sigma - 2}_\bot \times h^{\sigma - 2}_\bot))$ and hence has a limit $\Phi_\infty $ in $ {H^s}(\T^S,{\cal L} ( h^{\sigma - 2}_\bot \times h^{\sigma - 2}_\bot))$. Since $\Phi_\nu^{- 1} = {\rm exp}(\Psi_\nu)$, one can show by the same arguments that the sequence $(\widetilde \Phi_\nu^{- 1})_{\nu \geq 0}$ satisfies the same bounds. Since $\widetilde \Phi_\nu  \widetilde \Phi_\nu^{- 1} = {\mathbb I}_2$ for all $\nu \geq 0$, the limit of $(\widetilde\Phi_\nu^{- 1})_{\nu \geq 0}$ is equal to $\Phi_\infty^{- 1}$.  
By the same arguments one shows that $(\widetilde \Phi_\nu^{\pm 1})_{\nu \geq 0}$ is a Cauchy sequence in $H^s(\T^S, {\cal L}(h^{\sigma }_\bot \times h^{\sigma }_\bot ))$ and hence it also converges in this space to (the restriction of) $\Phi_\infty^{\pm 1}$.
By Theorem \ref{thm:abstract linear reducibility}, the maps $\Phi_\nu$ are symplectic for any $\nu \geq 0$
and hence by the characterization \eqref{condizione simpletticita matrice} of sympletic maps, 
so are  $\widetilde \Phi_\nu$ and in turn $\Phi_\infty^{\pm 1}$. 
\end{proof}

\smallskip

\noindent
For any $\ell \in \Z^S$, $j, k \in S_+^\bot$ and $\omega \in \Omega_o(\io)$, we define
\begin{align}\label{definizione L infinito - seconde Melnikov}
&  L_\infty^{+}(\ell, j, k) \equiv L_\infty^{+}(\ell, j, k; \omega) :=  
\omega \cdot \ell \,\,{\rm Id}_{\C^{2 \times 2}} + M_L([{\bf N}_\infty^{(1)}]_j^j) + 
M_R([\overline{\bf N}_\infty^{(1)}]_k^k)  \\
& \label{definizione L infinito + seconde Melnikov}
L_\infty^{-}(\ell, j, k) \equiv L_\infty^{-}(\ell, j, k; \omega) :=  
\omega \cdot \ell \,\, {\rm Id}_{\C^{2 \times 2}} + M_L([{\bf N}_\infty^{(1)}]_j^j) - M_R([{\bf N}_\infty^{(1)}]_k^k) 
\end{align}
and the set 
\begin{align} 
\Omega_\infty^{2\g}(\io)  & := \big\{ \omega \in \Omega_o(\io) \, :\,\, {({\bf M}_{+, 2 \gamma}^{II})}_\infty,\,\, {({\bf M}_{-, 2 \gamma}^{II})}_\infty \,\,\text{hold} \big\} \label{Omegainfty}
\end{align}
where ${ ({\bf M}^{II}_{+, 2 \gamma})}_\infty$, ${ ({\bf M}^{II}_{-, 2 \gamma})}_\infty$ are the following second order Melnikov conditions:

\smallskip

\noindent
{${({\bf M}_{+, 2 \gamma}^{II})}_\infty$} For any $\ell \in \Z^S$, $j, k \in S_+^\bot$, the operator 
$L_\infty^+(\ell, j, k; \omega)$ is invertible and
\begin{equation}\label{seconde melnikov off diagonali finali matrici}
\|L_\infty^+(\ell, j, k; \omega)^{- 1} \| \leq \frac{\langle \ell\rangle^\tau}{2 \gamma \langle j^2 + k^2 \rangle}\,.
\end{equation}

\noindent
{ ${ ({\bf M}^{II}_{-, 2 \gamma})}_\infty$} For any $\ell \in \Z^S$, $j, k \in S_+^\bot$ with $(\ell,j,k) \neq (0, j, j)$, the operator $L^-_\infty(\ell, j, k; \omega)$ is invertible and
\begin{equation}\label{seconde melnikov diagonali finali matrici}
\| L_\infty^-(\ell, j, k; \omega)^{- 1} \| \leq \frac{\langle \ell \rangle^\tau}{2 \gamma \langle j^2 - k^2 \rangle}\,.
\end{equation}
We remark that the superindex $2\g$ in
$\Omega_\infty^{2\g}(\io)$ stands for the factor $2\g$ in the denominator of the bounds in 
\eqref{seconde melnikov off diagonali finali matrici} and \eqref{seconde melnikov diagonali finali matrici}.
The set can be localized as follows: 
\begin{lemma}\label{inclusion of cantor sets}
If \eqref{piccolezza1} holds, with $N_0 = N_0(s_*, \tau, |S|) > 0$ sufficiently large, then $\Omega_\infty^{2 \gamma}(\io) \subseteq \cap_{\nu \geq 0} \Omega_\nu^\gamma(\io)$. 
\end{lemma}
\begin{proof}
Note that by the definition \eqref{Omega-second-meln}, $(\Omega_\nu^\gamma(\io))_{\nu \geq 0}$ is a decreasing sequence. Hence it suffices to show that for any $\nu \geq 0$, $\Omega_\infty^{2 \gamma}(\io) \subseteq \Omega_\nu^\gamma(\io)$. We argue by induction. Since $\Omega_0^\gamma(\io) = \Omega_o(\io)$ by \eqref{Omega-second-meln}, it follows from the definition \eqref{Omegainfty} that $\Omega_\infty^{2 \gamma}(\io) \subseteq \Omega_0^\gamma(\io)$.
To prove the inductive step from $\nu$ to $\nu+1$ we have to verify that
$\Omega_\infty^{2 \gamma}(\io) \subseteq \Omega_{\nu + 1}^\gamma(\io)$.
Let $\omega \in \Omega_\infty^{2 \gamma}(\io)$. By the induction hyphothesis we know that $\omega \in \Omega_\nu^\gamma(\io)$. Theorem \ref{thm:abstract linear reducibility} then implies that 
the $2 \times 2$ matrices $[{\bf N}^{(1)}_\nu(\omega)]_k^k$, $k \in S_+^\bot$, are well defined and that
$[{\bf N}_\nu^{(1)}(\omega)]_k^k = [\widetilde{\bf N}_\nu^{(1)}(\omega)]_k^k $. 
By the definitions \eqref{L + nu} and \eqref{L - nu}, also the matrices $L_\nu^{\pm}(\ell, j, k; \omega)$ 
are well defined.
Since $\omega \in \Omega_\infty^{2 \gamma}(\io)$, $L^-_\infty(\ell, j, k; \omega)$ is invertible 
and we may write 
$$
L_\nu^-(\ell, j, k; \omega) = L_\infty^- (\ell, j, k; \omega) + L_{\Delta}^-(\ell, j, k; \omega) =  L_\infty^- (\ell, j, k; \omega) 
\big( {\rm Id}_{\C^{2 \times 2}} + 
L_\infty^- (\ell, j, k; \omega)^{- 1}L_{\Delta}^-( j, k; \omega)  \big)  
$$
where
$$
L_{\Delta}^-( j, k; \omega) := M_L\big( [{\bf N}_\nu^{(1)}(\omega) - {\bf N}_\infty^{(1)}(\omega)]_j^j \big) - M_R \big( [{\bf N}_\nu^{(1)}(\omega) - {\bf N}_\infty^{(1)}(\omega)]_k^k \big)\,.
$$
By the estimate \eqref{stime blocchi 2 per 2 finali}
$$
\| L_{\Delta}^-(j, k; \omega) \| \lessdot 
N_{\nu - 1}^{- \alpha}|{\bf R}_0 {\frak D}|_{s_0 + \beta, \sigma - 1} k^{- 1}\,.
$$
By \eqref{seconde melnikov diagonali finali matrici} it then follows that
for any $|\ell| \leq N_\nu$ and $j, k \in S_+^\bot$, with $(\ell, j, k) \neq (0, j, j)$
\begin{equation}\label{copenaghen 1}
\| L_\infty^-(\ell, j, k; \omega)^{- 1} L_{\Delta}^-(\ell, j, k; \omega)  \| \leq C \frac{N_\nu^\tau N_{\nu - 1}^{-\alpha}}{2 \gamma \langle j^2 - k^2 \rangle } |{\bf R}_0 {\frak D}|_{s_0 + \beta, \sigma - 1}  \stackrel{\eqref{alpha beta}, \eqref{piccolezza1}}{\leq} \frac12\,,
\end{equation}
with $N_0 > 0$ in \eqref{piccolezza1} large enough. Hence the $2 \times 2$ matrix $L_\nu^-(\ell, j, k; \omega)$ is invertible, with inverse given by a Neumann series. For all $|\ell| \leq N_\nu$, $j, k \in S_+^\bot$ with $(\ell, j, k) \neq (0, j, j)$
\begin{align*}
\|L_\nu^-(\ell, j, k; \omega)^{- 1} \| & \leq \frac{\| L_\infty^-(\ell, j, k; \omega)^{- 1}\|}{1 - \| L_\infty^-(\ell, j, k; \omega)^{- 1} L_{\Delta}^-( j, k; \omega) \|}  \stackrel{\eqref{copenaghen 1}}{\leq 2} \| L_\infty^-(\ell, j, k; \omega)^{- 1}\| \stackrel{\eqref{seconde melnikov diagonali finali matrici}}{\leq} \frac{\langle \ell \rangle^\tau}{\gamma \langle j^2 - k^2 \rangle}\,.
\end{align*}
By similar arguments, one can prove that, for any $|\ell| \leq N_\nu$ and $j, k \in S_+^\bot$ 
$$
 \|L_\nu^+(\ell, j, k; \omega)^{- 1} \|  \leq 
\frac{\langle \ell \rangle^\tau}{\gamma \langle j^2 + k^2 \rangle}\,.
$$
Hence, by the definition \eqref{Omega-second-meln}, $\omega \in \Omega_{\nu + 1}^\gamma(\io)$ and the inductitive step is proved. 
\end{proof}

As advertised we now prove that  ${\bf L}_0  $ is conjugated to the normal form Hamiltonian operator ${\bf L}_\infty $:

\begin{theorem}\label{teoremadiriducibilita}
{ \bf ($ 2 \times 2 $ diagonalization of  ${\bf L}_0 $)}
There exists $0 < \delta \equiv \delta (|S|, \tau, s_*) < 1$ 
such that for any $ \io : \T^S \times \Omega_o(\io) \to M^\sigma$ with  
\begin{equation}\label{final KAM smallness condition}
 \| \io\|_{s_0 + \bar \mu + \beta}^\Lipg \leq C \e \gamma^{- 2}\,, \qquad  \e \gamma^{- 4} \leq \delta\,,
\end{equation}
where $\bar \mu$ is given as in \eqref{perdita mu dopo prime trasformazioni}, and $\beta$ as in \eqref{alpha beta}, the following holds:

\noindent
$(i)$ For any $\omega \in \Omega_\infty^{2\g}(\io)$ and $s \in [s_0, s_* - \bar \mu - \beta]$, the transformations $\Phi_\infty, \Phi_\infty\inv$ satisfy the estimates
\begin{equation} \label{stima Phi infty}
 \quad |\Phi_\infty^{\pm 1} - {\mathbb I}_2 |_{s, \sigma }^\Lipg\,, \quad |\Phi_\infty^{\pm 1} - {\mathbb I}_2 |_{s, \sigma - 2 }^\Lipg
\leq_s \gamma^{- 1}\big( \e  + \e \gamma^{- 2}\| \io \|_{s + \bar \mu + \b }^{\gamma {\rm lip}} \big)\,.
\end{equation}

\noindent
$(ii)$
For any $\omega \in \Omega_\infty^{2 \gamma}(\io)$ and any $s \in [s_0 +1, s_* - \bar \mu - \beta]$ , the Hamiltonian operator 
$$
{\bf L}_0(\omega) : H^s(\T^S, h^{\sigma}_\bot \times h^{\sigma}_\bot) \to H^{s - 1}(\T^S, h^{\sigma - 2}_\bot \times h^{\sigma - 2}_\bot)
$$ 
in \eqref{L0} is conjugated to the  normal form Hamiltonian  operator ${\bf L}_\infty(\omega)$ in \eqref{bf L infinito esplicito} 
by $\Phi_\infty(\omega)$, 
\be\label{Lfinale}
{\bf L}_{\infty}(\omega)
= \Phi_{\infty}\inv(\omega) {\bf L}_0(\omega)   \Phi_{\infty}(\omega)\,.
\ee
($iii$) For any $ k \in S_+^\bot $, the two eigenvalues of  $ [{\bf N}_\infty^{(1)}]_k^k $ are 
real and of the form 
\begin{align}\label{prima asintotica autovalori}
\omega_{-k}^{nls}(\xi, 0) + c_\e + \frac{r^{(-)}_{\xi, \e}(k)}{k} =  
4 \pi^2 k^2 + c_{\xi, \e} + \frac{\rho^{(-)}_{\xi, \e}(k)}{k}    \,, \\
\quad \omega_{ k}^{nls}(\xi, 0) + c_\e + \frac{r^{(+)}_{\xi, \e}(k)}{k} =
4 \pi^2 k^2 + c_{\xi, \e} + \frac{\rho^{(+)}_{\xi, \e}(k)}{k}   \label{prima asintotica autovalori1}
\end{align}
where 
\be\label{estimates-errors} 
|c_\e|^{\sup} = O(\e \g^{-2})  \, , \
|r^{(\pm)}_{\xi, \e}(k)|^{\sup}  = O(\e \gamma^{- 2}) \, ,  \  | c_{\xi, \e}|^{\rm sup} = O(1) \, ,  \
 \sup_{k \in S^\bot_+}|\rho_{\xi, \e}^{(\pm)}(k)|^{\rm sup}  = O(1) \, .
\ee
When listed
according to size, they are denoted by $\lambda_k^{(\pm)}$, i.e. $\lambda_k^{(-)} \leq \lambda_k^{(+)}$. 
Then $\lambda_k^{(\pm)}$ are Lipschitz continuous and satisfy 
\begin{align}\label{asintotica autovalori finali misura0}
 \sup_{k \in S^\bot_+} | \lambda_k^{(\pm)} |^{\rm lip} =  O(1) \, .
\end{align}
\end{theorem}
\begin{proof}

\noindent
 By the estimate  \eqref{stima R0 riducibilita}, we get 
\begin{equation}\label{asdf}
|{\bf R}_0 {\frak D}|_{s_0 + \beta}^\Lipg \leq_{s_0 + \beta} \e  + \e \gamma^{- 2}  \| \io \|^\Lipg_{s_0 + \bar \mu + \beta} \stackrel{\eqref{final KAM smallness condition}}{\leq_{s_0 + \beta}} \e\,.
\end{equation}
 This together with the smallness condition \eqref{final KAM smallness condition} implies 
that the smallness condition \eqref{piccolezza1} of Theorem \ref{thm:abstract linear reducibility}
holds once $\delta_0$ is chosen so that $\delta_0 \leq_{s_*}  N_0^{- C_0}$ (recall \eqref{regularity s-1}). 
We now prove items $(i)$ and $(ii)$. 

\noindent
$(i)$ Since $\Omega_\infty^{2 \gamma}(\io) \stackrel{Lemma\,\, \ref{inclusion of cantor sets}}{\subseteq} \cap_{\nu \geq 0} \Omega_\nu^\gamma(\io)$, Lemma \,\ref{composizione trasformazioni KAM} implies that
$$
 |\Phi_\infty^{\pm 1} - {\mathbb I}_2 |_{s, \sigma }^\Lipg\,, \quad 
|\Phi_\infty^{\pm 1} - {\mathbb I}_2 |_{s, \sigma - 2 }^\Lipg \quad
\leq_s  \gamma^{- 1} |{\bf R}_0 {\frak D}|_{s + \beta, \sigma - 1}^\Lipg\,.
$$
Furthermore by \eqref{stima R0 riducibilita},  the operator ${\bf R}_0$ in \eqref{L0} satisfies
\begin{equation}\label{dreq}
|{\bf R}_0 {\frak D}|_{s + \beta, \sigma - 1}^{\Lipg} \leq_{s + \beta} \e   + \e \gamma^{- 2}\| \io \|_{s+ \bar \mu + \beta}^{\Lipg}\,,
\end{equation}
yielding the claimed estimates \eqref{stima Phi infty}. 

\noindent
$(ii)$ By \eqref{Lnu+1}, we get 
\begin{equation}\label{spanish}
{\bf L}_\nu = \widetilde{\Phi}_{\nu - 1}^{- 1} {\bf L}_0 \widetilde \Phi_{\nu - 1} = \omega \cdot \partial_\vphi {\mathbb I}_2 + {\bf N}_\nu + {\bf R}_\nu\,, \qquad 
\widetilde  \Phi_\nu = \Phi_0 \circ \cdots \circ \Phi_\nu\,.
\end{equation}
Since 
$
|{\bf N}^{(1)}_\infty - {\bf N}^{(1)}_\nu |_{\sigma - 2}^\Lipg  \leq |({\bf N}^{(1)}_\infty - {\bf N}^{(1)}_\nu) {\frak D} |_{\sigma - 1}^\Lipg \lessdot
\sup_{k \in S_+^\bot} \|[{\bf N}^{(1)}_\infty - {\bf N}^{(1)}_\nu ]_k^k k\|^\Lipg $
one has
$$
|{\bf N}^{(1)}_\infty - {\bf N}^{(1)}_\nu |_{\sigma - 2}^\Lipg \quad 
\stackrel{\eqref{stime blocchi 2 per 2 finali}, \eqref{asdf}}{\leq_{s_0 + \beta}} 
\quad N_{\nu - 1}^{- \alpha} \e  \quad  \stackrel{\nu \to + \infty}{\to} \quad 0
$$
and for any $s \in [s_0, s_* - \bar \mu - \beta]$ 
$$
|{\bf R}_\nu|_{s, \sigma - 2}^\Lipg \lessdot |{\bf R}_\nu {\frak D}|_{s, \sigma - 1}^\Lipg \stackrel{\eqref{Rsb}, \eqref{dreq}}{\lessdot} 
N_{\nu - 1}^{- \alpha} \big(\e   +  \e \gamma^{- 2}\| \io \|_{s+ \bar \mu + \beta}^{\Lipg} \big) 
\quad \stackrel{\nu \to + \infty}{\to} \quad 0\,.
$$
Hence ${\bf L}_\nu - {\bf L}_\infty \stackrel{\nu \to + \infty}{\to}0$ 
with respect to the norm $| \cdot |_{s, \sigma - 2}^\Lipg$ and  
${\bf L}_\nu \,\, \stackrel{\nu \to + \infty}{\to} \,\, {\bf L}_\infty$ in the space of linear, bounded operators from $H^s(\T^S, h^{\sigma}_\bot \times h^{\sigma}_\bot) \to H^{s - 1}(\T^S, h^{\sigma - 2}_\bot \times h^{\sigma - 2}_\bot) $. 
Since by Lemma \ref{composizione trasformazioni KAM}, 
$\widetilde\Phi_\nu \,\, \stackrel{\nu \to + \infty}{\to} \,\, \Phi_\infty$ in the norm $|\cdot |_{s, \sigma}^\Lipg$ and similarly,
$\widetilde \Phi_\nu^{- 1} \,\, \stackrel{\nu \to + \infty}{\to} \,\, \Phi_\infty^{- 1}$ 
in the norm $| \cdot |_{s - 1, \sigma - 2}^\Lipg$  for any $s_0 + 1 \leq s \leq s_* - \bar \mu - \beta$,
formula \eqref{Lfinale} follows by passing to the limit in \eqref{spanish}. 

\noindent
($iii$)
{\it Proof of formula \eqref{prima asintotica autovalori}-\eqref{estimates-errors}:} 
We write
$[{\bf N}_\infty^{(1)}]_k^k = [{\bf N}_0^{(1)}]_k^k + [{\bf N}_\infty^{(1)} - {\bf N}_0^{(1)}]_k^k $ and note that 
\be\label{pert-matr1}
\|[{\bf N}_\infty^{(1)}]_k^k - [{\bf N}_0^{(1)}]_k^k \|^\Lipg \stackrel{\eqref{stime blocchi 2 per 2 finali}}{\lessdot}  
|{\bf R}_0 {\frak D}|_{s_0 + \beta, \sigma - 1} k^{- 1} \stackrel{\eqref{asdf}}{\lessdot } \e k^{- 1}\,.
\ee
By \eqref{first diagonal normal form}, \eqref{espansione asintotica autovalori blocco iniziale}, 
the matrix $[{\bf N}^{(1)}_0]_k^k$ is diagonal and its entries are given by 
\be\label{spectrum-diagonal}
\omega_{-k}^{nls}(\xi, 0) + c_\e + 
\frac{1}{-k}  r_{-k,\xi}\,, \quad \omega_{ k}^{nls}(\xi, 0) + c_\e + 
\frac{1}{k}  r_{ k,\xi}\,, \qquad |c_\e|^{\Lipg}\,,
\, \sup_{k \in S^\bot_+} | r_{\pm k,\xi}|^{\Lipg} \stackrel{\eqref{stime asintotica autovalori iniziali}}{=} O(\e \gamma^{- 2}) \, . 
\ee
By standard perturbation theory for the eigenvalues of self-adjoint $2 \times 2$ matrices,
the estimates  \eqref{pert-matr1} and \eqref{spectrum-diagonal} imply that 
the eigenvalues of $ [{\bf N}_\infty^{(1)}]_k^k $ are 
given by the left hand side of the identities  
\eqref{prima asintotica autovalori}-\eqref{prima asintotica autovalori1} with estimates 
$ |c_\e|^{\sup} = O(\e \g^{-2}) $,  $ |r^{(\pm)}_{\xi, \e}(k)|^{\sup}  = O(\e \gamma^{- 2}) $, cf \eqref{estimates-errors}. 
The right hand side of the identities 
\eqref{prima asintotica autovalori}-\eqref{prima asintotica autovalori1} are obtained by expanding
$ \om_{\pm  k}^{nls}(\xi, 0) $ by Theorem \ref{Corollary 2.2} item $(ii)$.

\noindent
{\it Proof of formula \eqref{asintotica autovalori finali misura0}:} 
The eigenvalues $ \lambda_k^{(\pm)} (\om) $ of the matrix $ [{\bf N}_\infty^{(1)}]_k^k (\om) $ 
are Lipschitz continuous functions of the matrices 
$$
|\lambda_k^{\pm} (\om_2) - \lambda_k^{\pm} (\om_1) | \lessdot 
\|  [{\bf N}_\infty^{(1)}]_k^k (\om_2)  -   [{\bf N}_\infty^{(1)}]_k^k (\om_1) \| \lessdot |\om_2 - \om_1 |  
$$
by \eqref{pert-matr1}, \eqref{spectrum-diagonal} and  Theorem \ref{Corollary 2.2}  item $(ii)$.
\end{proof}

\subsection{Proof of Theorem \ref{invertibility of frak L omega}}\label{proof of invertibility of frak L omega}

By Theorem \ref{teoremadiriducibilita}, the normal form Hamiltonian 
operator ${\bf L}_\infty(\omega) = \omega \cdot \partial_\vphi {\mathbb I}_2 + {\bf N}_\infty(\omega)$ is a $\vphi$-independent $2 \times 2 $ block diagonal operator
 for any $\omega$  in $\Omega_\infty^{2 \gamma}(\io)$, which is defined in \eqref{Omegainfty}. Furthermore, the operator ${\bf L}_\infty$ is conjugated to ${\frak L}_\omega$ introduced in \eqref{definition frak L omega sec 4} 
by the composition of the symplectic transformations 
$\mathtt \Phi_1$, $\mathtt \Phi_2$, $\mathtt \Phi_3$ 
(Section \ref{sec:5}), and $\Phi_\infty$ (Section \ref{2 times 2 block diagonalization}), 
\be\label{coniugazione-frak-Lomega}
{\frak L}_\omega = \mathtt \Phi_1 \mathtt \Phi_2 \mathtt \Phi_3 \Phi_\infty {\bf L}_\infty \Phi_\infty^{- 1} \mathtt \Phi_3^{- 1} \mathtt \Phi_2^{- 1} \mathtt \Phi_1^{- 1}\,.
\ee
This representation allows to prove Theorem  \ref{invertibility of frak L omega}. To this end, introduce
\begin{equation}\label{cantor for invertibility}
\Omega^{2 \gamma}_{\rm Mel}(\io) := \big\{ \omega \in \Omega_\infty^{2 \gamma}(\io) : \omega \,\, {\rm satisfies}\,\, ({\bf M}^{ I}_{2 \gamma})_\infty \big\}\,,
\end{equation}
where {${({\bf M}^{I}_{2 \gamma}})_\infty$} is the following first order Melnikov condition:

\medskip

\noindent
{${({\bf M}^{I}_{2 \gamma}})_\infty$} {\it For any $\ell \in \Z^S$, $ j \in S_+^\bot$, the operator 
$ \omega \cdot \ell \, {\rm Id}_{2} + [{\bf N}_\infty^{(1)}]_j^j$ is invertible and}
\begin{equation}\label{prime melnikov off diagonali finali matrici}
\big\|\big( \omega \cdot \ell \, {\rm Id}_{2} + [{\bf N}_\infty^{(1)}]_j^j \big)^{- 1} \big\| \leq \frac{\langle \ell\rangle^\tau}{ 2 \gamma   j^2 }\,.
\end{equation}

\noindent
Before proving Theorem \ref{invertibility of frak L omega}, we need to establish the following 

\begin{lemma} {\bf (Estimate of $ {\bf L}_\infty^{-1} $)}
For any $\omega \in \Omega^{2 \gamma}_{\rm Mel}(\io)$ and  $g \in H^{s + \tau}(\T^S, h_\bot^{\sigma - 2} \times h_\bot^{\sigma - 2})$ the linear equation ${\bf L}_\infty(\omega) h = g$ has a unique solution $h$ 
in $ H^s(\T^S, h_\bot^{\sigma} \times h_\bot^{\sigma} )$, denoted by ${\bf L}_\infty^{- 1} g$.
Moreover, if $g$ is a Lipschitz family in $H^{s + 2 \tau + 1}(\T^S, h^{\sigma - 2}_\bot \times h^{\sigma - 2}_\bot)$, 
\begin{equation}\label{stima inverso L infinito}
\| {\bf L}_\infty^{- 1} g \|_{s, \sigma}^\Lipg \lessdot \gamma^{- 1} \| g \|_{s + 2 \tau + 1, \sigma - 2}^\Lipg\,.
\end{equation}
\end{lemma}
\begin{proof}
By \eqref{bf L infinito esplicito}, the normal form Hamiltonian operator ${\bf L}_\infty$ can be written as 
$$
{\bf L}_\infty = \begin{pmatrix}
{\bf L}_\infty^{(1)} & 0 \\
0 & \overline{\bf L}_\infty^{(1)}
\end{pmatrix}\,, \qquad {\bf L}_\infty^{(1)} := \omega \cdot \partial_\vphi {\rm I}_2
+ \ii {\bf N}_\infty^{(1)}\,, \quad {\bf N}_\infty^{(1)} := {\rm diag}_{j \in S_+^\bot} [ {\bf N}_\infty^{(1)} ]_j^j\,.
$$
It thus suffices to study the operator ${\bf L}_\infty^{(1)}$.
For any $\omega \in \Omega_{\rm Mel}^{2 \gamma}(\io)$ and 
$g \in H^{s + \tau}(\T^S, h_\bot^{\sigma - 2})$, one has by
\eqref{prime melnikov off diagonali finali matrici} 
$$
\big( {\bf L}_\infty^{(1)}\big)^{- 1} g =  
\sum_{\ell \in \Z^S} \Big( {\bf A}_\infty(\ell, j)^{- 1} \begin{pmatrix}
\hat g_{- j}(\ell) \\
\hat g_{ j}(\ell)
\end{pmatrix} \Big)_{j \in S_+^\bot} e^{\ii \ell \cdot \vphi}\,, \qquad 
A_\infty(\ell, j) \equiv [{\bf A}_\infty(\ell) ]_j^j  := 
\ii \Big( \omega \cdot \ell \, {\rm Id}_2 + [{\bf N}_\infty^{(1)}]_j^j \Big) \,.
$$
In view of Lemma \ref{bound op2} $(i)$ and \eqref{prime melnikov off diagonali finali matrici} one then obtains
\begin{equation}\label{stima sup L infty (1)}
\| \big( {\bf L}_\infty^{(1)}\big)^{- 1} g \|_{s, \sigma} \lessdot \gamma^{- 1} \| g \|_{s + \tau, \sigma - 2}\,.
\end{equation}
\noindent
Concerning the Lipschitz seminorm, 
given any $\omega_1, \omega_2 \in \Omega_{\rm Mel}^{2 \gamma}(\io)$, 
write $(  {\bf L}_\infty^{(1)}(\omega_1) )^{- 1} g_{\omega_1} - 
(  {\bf L}_\infty^{(1)}(\omega_2) )^{- 1} g_{\omega_2 } $ as
\begin{align}
& (  {\bf L}_\infty^{(1)}(\omega_1) )^{- 1} \big( g_{\omega_1} - g_{\omega_2} \big)
+  \big( ({\bf L}_\infty^{(1)}(\omega_1) )^{- 1} - ({\bf L}_\infty^{(1)}(\omega_2) )^{- 1} \big) g_{\omega_2}\,. \label{L infty stima lip 1}
\end{align}
The latter two terms are estimated individually: by \eqref{stima sup L infty (1)}, 
the first term satisfies the estimate 
\begin{equation}\label{L infty stima lip 2}
\| \big(  {\bf L}_\infty^{(1)}(\omega_1) \big)^{- 1} \big( g_{\omega_1} - g_{\omega_2} \big)  \|_{s, \sigma } 
{\lessdot} \gamma^{- 1} \| g \|_{s + \tau, \sigma - 2}^{{\rm lip}} |\omega_1 - \omega_2 |
\end{equation}
whereas the term  
$\big( ({\bf L}_\infty^{(1)}(\omega_1) )^{- 1} - ( {\bf L}_\infty^{(1)}(\omega_2))^{- 1} \big) g_{[\omega_2} $
equals
\begin{equation}\label{mariolino}
 \sum_{\ell \in \Z^S} \Big( 
\big( {A}_\infty(\ell, j ; \omega_1)^{- 1} - {A}_\infty(\ell, j ; \omega_2)^{- 1} \big) \begin{pmatrix} \hat g_{- j}(\ell; \omega_2) \\
\hat g_j(\ell ; \omega_2)
\end{pmatrix} \Big)_{j \in S_+^\bot} e^{\ii \ell \cdot \vphi} \,.
\end{equation}
Since 
\begin{align}
{ A}_\infty(\ell, j ; \omega_1)^{- 1} - {A}_\infty(\ell, j ; \omega_2)^{- 1} & 
=  {A}_\infty(\ell, j ; \omega_2)^{- 1} \big( { A}_\infty(\ell, j ; \omega_2) - { A}_\infty(\ell, j ; \omega_1) \big)  {A}_\infty(\ell, j ; \omega_1)^{- 1}\,, \nonumber
\end{align}
we have  
\begin{align}
\| {A}_\infty(\ell, j ; \omega_1)^{- 1} - { A}_\infty(\ell, j ; \omega_2)^{- 1} \| 
\stackrel{\eqref{prime melnikov off diagonali finali matrici}}{\lessdot} 
\frac{\langle \ell\rangle^{2 \tau}}{ \gamma^2 j^4} 
\| {A}_\infty(\ell, j ; \omega_2) - {A}_\infty(\ell, j ; \omega_1)  \| 
\label{L infty stima lip 3}
\end{align}
with
$\| { A}_\infty(\ell, j ; \omega_2) - { A}_\infty(\ell, j ; \omega_1) \|  
\lessdot |\omega_2 - \omega_1| |\ell| + 
\|[{\bf N}_\infty^{(1)}(\omega_2) - {\bf N}_\infty^{(1)}(\omega_1)]_j^j \| $.
Since $\|[{\bf N}_\infty^{(1)}(\omega_2) - {\bf N}_\infty^{(1)}(\omega_1)]_j^j \|$
is bounded by
$$
\|[{\bf N}_\infty^{(1)}(\omega_2) - {\bf N}_0^{(1)}(\omega_2)]_j^j \|
+ \|[{\bf N}_0^{(1)}(\omega_2) - {\bf N}_0^{(1)}(\omega_1)]_j^j \|
+\|[{\bf N}_0^{(1)}(\omega_1) - {\bf N}_\infty^{(1)}(\omega_1)]_j^j \|
$$
and
$$
\| [{\bf N}_\infty^{(1)} - {\bf N}_0^{(1)}]_j^j \|^{\rm lip} \leq
\gamma^{- 1} \| [{\bf N}_\infty^{(1)} - {\bf N}_0^{(1)}]_j^j \|^{\Lipg} 
\stackrel{\eqref{stime blocchi 2 per 2 finali}}{\lessdot} 
\gamma^{- 1} |{\bf R}_0 {\frak D}|_{s_0 + \beta, \sigma - 1} j^{- 1} 
\stackrel{\eqref{asdf},\,\, \e \gamma^{- 1} \leq 1 \,  }{\lessdot} \,\, 1
$$
one concludes that
\begin{align*}
\|[{\bf N}_\infty^{(1)}(\omega_2) - {\bf N}_\infty^{(1)}(\omega_1)]_j^j \|
\lessdot  |\omega_1 - \omega_2| +
 \| [{\bf N}_0^{(1)}(\omega_2) - {\bf N}_0^{(1)}(\omega_1)]_j^j  \| 
+ |\omega_1 - \omega_2| 
&  \stackrel{ \eqref{first diagonal normal form}, \eqref{estimates of the initial normal form}}{\lessdot} |\omega_1 - \omega_2| \,.
\end{align*} 
We thus have proved that
$$
\| {A}_\infty(\ell, j ; \omega_2) - { A}_\infty(\ell, j ; \omega_1)  \| \lessdot \,
 |\omega_2 - \omega_1| \,\langle \ell \rangle \label{L infty stima lip 4}
$$
and hence \eqref{L infty stima lip 3}, \eqref{L infty stima lip 4} imply that
$$
\| {A}_\infty(\ell, j ; \omega_1)^{- 1} - {A}_\infty(\ell, j ; \omega_2)^{- 1} \| \lessdot 
\frac{\langle \ell\rangle^{2 \tau + 1}}{ \gamma^2 j^4} |\omega_1 - \omega_2|\,.
$$
Applying this estimate to \eqref{mariolino}, one sees that
\begin{align}\label{L infty stima lip 6}
\big\| \big( {\bf L}_\infty^{(1)}(\omega_1) \big)^{- 1} - {\bf L}_\infty^{(1)}(\omega_2) \big)^{- 1} \big) g_{\omega_2} \big\|_{s, \sigma} \lessdot  \gamma^{- 2} \| g \|_{s + 2 \tau + 1, \sigma - 2}\,.
\end{align}
Combining \eqref{L infty stima lip 1}, \eqref{L infty stima lip 2}, and \eqref{L infty stima lip 6} leads to
$$
\| \big( {\bf L}_\infty^{(1)} \big)^{- 1} g \|_{s, \sigma}^{\rm lip} \lessdot 
 \gamma^{- 1} \| g \|_{s + \tau, \sigma - 2}^{\rm lip} +
\gamma^{- 2} \| g \|_{s + 2 \tau + 1, \sigma - 2} 
$$
which, together with \eqref{stima sup L infty (1)}, proves
 \eqref{stima inverso L infinito}. 
\end{proof}

\noindent
{\it Proof of Theorem \ref{invertibility of frak L omega}. } By Lemmata 
\ref{lem:44 Lip}, \ref{stima R2 lip}, \ref{stima R3 lip},
Theorem \ref{teoremadiriducibilita}, and the smallness condition $\e \g^{-4} \le 1$ one gets
\begin{equation}\label{stima tutte le trasformazioni}
 |\mathtt \Phi_j |_{s, \sigma}^\Lipg\,,\, |\Phi_\infty|_{s, \sigma}^\Lipg \leq_s 1 + \e \gamma^{- 3} \| \io \|_{s + \bar \mu + \beta}^\Lipg \leq_s 1 + \| \io \|_{s + \bar \mu + \beta}^\Lipg\,, 
\qquad \forall j \in \{1,2,3 \}\,,
\end{equation}
implying together with \eqref{size of torus} that
$$
|\mathtt \Phi_j |_{s_0, \sigma}^\Lipg\,,\, |\Phi_\infty|_{s_0, \sigma}^\Lipg \lessdot 1\,,
\qquad \forall j \in \{1,2,3 \}\,.
$$

\noindent
It then follows by Lemma \ref{lemma:action-Sobolev} that
\begin{align}
\| \mathtt \Phi_1 \mathtt \Phi_2 \mathtt \Phi_3 \Phi_\infty {\bf L}_\infty^{- 1} g \|_{s, \sigma}^\Lipg & \stackrel{\eqref{stima tutte le trasformazioni}}{\leq_s} 
\| {\bf L}_\infty^{- 1} g \|_s^\Lipg + \|\io \|_{s + \bar \mu + \beta}^\Lipg 
\| {\bf L}_\infty^{- 1} g \|_{s_0}^\Lipg  \nonumber\\
& \stackrel{\eqref{stima inverso L infinito}}{\leq_s} \gamma^{- 1} \big( \| g \|_{s + 2 \tau + 1, \sigma - 2}^\Lipg  +  \| \io \|_{s + \bar \mu + \beta}^\Lipg \| g \|_{s_0 + 2 \tau + 1, \sigma - 2}^\Lipg\big)\,. \nonumber
\end{align}
Similarly one has
$$
\|\Phi_\infty^{- 1} \mathtt \Phi_3^{- 1} \mathtt \Phi_2^{- 1} \mathtt \Phi_1^{- 1} g \|_{s + 2 \tau + 1, \sigma - 2}^\Lipg \leq_s \| g \|_{s + 2 \tau + 1, \sigma - 2}^\Lipg + \| \io\|_{s +\bar \mu + \beta + 2 \tau + 1}^\Lipg \| g \|_{s_0 + 2\tau + 1, \sigma - 2}^\Lipg\,.
$$
Combining the above estimates yield 
$$
\| \mathtt \Phi_1 \mathtt \Phi_2 \mathtt \Phi_3 \Phi_\infty {\bf L}_\infty^{- 1} \Phi_\infty^{- 1} \mathtt \Phi_3^{- 1} \mathtt \Phi_2^{- 1} \mathtt \Phi_1^{- 1} g  \|_{s, \sigma}^\Lipg \leq_s \gamma^{- 1} 
\big( \| g \|_{s + 2 \tau + 1, \sigma - 2}^\Lipg +  
\| \io\|_{s +\bar \mu + \beta + 2 \tau + 1}^\Lipg \| g \|_{s_0 + 2\tau + 1, \sigma - 2}^\Lipg \big) \,,
$$
which, recalling \eqref{coniugazione-frak-Lomega},  is the estimate \eqref{tame inverse} of Theorem \ref{invertibility of frak L omega}, 
with 
\be\label{value mu0}
\mu_0 := \bar \mu + \beta + 2 \tau + 1 
\stackrel{\eqref{perdita mu dopo prime trasformazioni}, \eqref{alpha beta}} {=} 
 4s_0 +10 \tau + 7.
\ee

\subsection{Variation with respect to $\io$}

In this section we provide estimates for the variation of the $2 \times 2$ matrices $[{\bf N}_\nu^{(1)}]_k^k$,
introduced in Theorem \ref{thm:abstract linear reducibility},
with respect to $\io$. They are required in  Section~\ref{sec:measure} for obtaining the measure estimate of Theorem \ref{main theorem}. 
To prove them, we also need such estimates for the remainder terms ${\bf R}_\nu$, $\nu \ge 0$,
of Theorem \ref{thm:abstract linear reducibility}.

\begin{theorem}\label{teorema variazione autovalori}
Let $ \breve \io^{(a)}(\vphi) = (\vphi, 0, 0) + \io^{(a)}(\vphi)$, $a = 1, 2$, be
two Lipschitz families of torus embeddings with $\breve \io^{(a)} \equiv \breve \io_\om^{(a)}$  
defined on $\Omega_o(\io^{(a)})$ where
$\Omega_o(\io^{(2)}) \subseteq \Omega_o(\io^{(1)})$ with  
$\Omega_o(\io^{(1)}) \subseteq \Omega_{2\g, \tau}$ for some given $0 < \g < 1/2$.
Furthermore we assume that $\io^{(1)}$ and $\io^{(2)}$ satisfy  the smallness condition \eqref{final KAM smallness condition} (with $2\g$). Then the following statements hold: 

\begin{itemize}
\item[ ${\bf (S1)_{\nu}}$]
  There exists a constant $C_{\rm var} = C_{\rm var} (\tau, |S|) > 0$ so that for any
$\nu \geq 0$ and any $\g /2 \le \g_1, \g_2  \le 2\g$, the operator 
$ \Delta_{12} {\bf R}_\nu
 := {\bf R}_\nu(\breve \io^{(1)}) - {\bf R}_\nu(\breve \io^{(2)})$, defined for
$\omega \in \Omega_{\nu}^{\g_1}(\io^{(1)}) \cap  \Omega_{\nu}^{\g_2}(\io^{(2)})$
(with $\Omega_{\nu}^{\g_a}(\io^{(a)})$ as in \eqref{Omega-second-meln})
satisfies 
\begin{equation}\label{derivate-R-nu}
|\Delta_{12} {\bf R}_{\nu} {\frak D}|_{ s_{0}, \sigma - 1}\leq 
C_{\rm var}N_{\nu - 1}^{-\alpha}  \Vert \Delta_{12} \io \Vert_{ s_0 + \bar \mu + \beta },\quad
|\Delta_{12} {\bf R}_{\nu} {\frak D}|_{ s_{0} + \beta , \sigma - 1}\leq 
C_{\rm var} N_{\nu - 1}  \Vert \Delta_{12} \io \Vert_{ s_0 + \bar \mu + \beta } \, 
\end{equation}
where $\bar \mu$, $N_\nu$, and $\a$, $\b$ are given in 
\eqref{perdita mu dopo prime trasformazioni}, \eqref{defN}, and \eqref{alpha beta}, respectively. 
Moreover, for any $k \in S_+^\bot$ one has 
\be\label{Delta12 rj}
\| \Delta_{12}  [{\bf N}_\nu^{(1)}]_k^k \| \lessdot   
\| \Delta_{12} \io \|_{s_0 + \bar \mu + \beta} 
\ee 
and, in case $\nu \ge 1$,
\be\label{deltarj12}
\big\|\Delta_{12}\big([{\bf N}_\nu^{(1)} - {\bf N}_{\nu - 1}^{(1)}]_k^k\big) \big\|\lessdot  \vert \Delta_{12} {\bf R}_{\nu - 1} {\frak D} \vert_{s_0, \sigma - 1} k^{- 1}\,. 
\ee

\noindent
\item[ ${\bf (S2)_{\nu}}$]
There exists a constant $C_{\rm var}'= C_{\rm var} '(\tau, |S|) > 0$ so that 
for any given $ 0 < \rho \le \g / 2 $, 
\begin{equation}\label{legno}
 C_{\rm var}'  N_{\nu - 1}^{\tau} \Vert \Delta_{12} \io \Vert_{ s_0 + \bar \mu + \beta }^{\rm sup} \leq \rho 
\quad \Longrightarrow \quad
 \Omega_{\nu }^{\gamma}(\io^{(1)}) \cap \Omega_o(\io^{(2)}) \subseteq 
 \Omega_{\nu}^{\gamma - \rho}(\io^{(2)}) \, .
\end{equation}
\end{itemize}
\end{theorem}
\begin{proof} We argue by induction. First let us prove ${\bf (S1)}_{0} $ and ${\bf (S2)}_{0} $.
Concerning ${\bf (S1)}_{0} $, note that 
by \eqref{estimate Delta 12 R3}, the operator 
${\bf R}_0 = {\frak R}_3$ satisfies for any  
$ \o \in \Omega_o(\io^{(2)})\, (= \Omega_o(\io^{(1)}) )$
$$
|\Delta_{12} {\bf R}_0 {\frak D}|_{s_0 + \beta, \sigma - 1}  \lessdot 
\e \gamma^{- 2} \| \Delta_{12} \io\|_{ s_0 + \beta + 4 s_0 + \tau } + 
{\rm max}_{ s_0 + \beta + 4 s_0 + \tau}(\io) \| \Delta_{12} \io\|_{5 s_0 + \tau}\,,
$$
implying that
$$|\Delta_{12} {\bf R}_0 {\frak D}|_{s_0 + \beta, \sigma - 1} 
 \stackrel{\eqref{perdita mu dopo prime trasformazioni}}{\lessdot}  
(\e \gamma^{- 2} + {\rm max}_{s_0 + \bar \mu + \beta}(\io)) 
\| \Delta_{12} \io\|_{s_0 + \bar \mu + \beta} 
\stackrel{\eqref{final KAM smallness condition}}{\lessdot} 
\| \Delta_{12} \io\|_{s_0 + \bar \mu + \beta}\,.
$$
Since $N_{-1} = 1,$
the estimates \eqref{derivate-R-nu} for $\nu = 0$ then follow 
by choosing  $C_{\rm var} (\tau, |S|) > 0$ large enough. 
Concerning the estimate \eqref{Delta12 rj} for $ \nu = 0 $ 
recall that 
by \eqref{first diagonal normal form}, the matrix element
$({\bf N}_0^{(1)})_k^k $, $k \in S^\bot$, is given by
$[[ \om_k^{nls}  ]]+ \e  [[ q_1 ]]
=4 \pi^2 k^2 +  [[ \Omega_k^{nls} ]] + \e  [[ q_1 ]]$.
By the estimates of $\Delta_{12}\Omega^{nls}$ and $\Delta_{12} q_1$ in
Lemma \ref{lemma variazione i Omega nls} $(i)$ and, respectively, 
Lemma \ref{estimates from the perturbation lip} $(i)$ (valid uniformly on
$\Omega_o(\io^{(2)})$) and using
 the smallness condition \eqref{final KAM smallness condition}, one concludes that for any $k \in S_+^\bot$
$$
\| \Delta_{12} [{\bf N}_0^{(1)}]_k^k\| {\lessdot}\| \Delta_{12} \io\|_{s_0 + \bar \mu + \beta}\,,
$$
which is the estimate \eqref{Delta12 rj} for $\nu = 0$. 
Clearly, ${\bf (S2)}_0 $ holds for any choice of $C_{\rm var}'$ since by assumption, 
$\Omega_o (\io^{(2)}) \subseteq \Omega_o (\io^{(1)})$ and
by \eqref{Omega-second-meln},
 $\Omega_0^{\gamma}( \io^{(a)}) = \Omega_o ( \io^{(a)})$, $a= 1, 2$, implying that $\Omega_0^\gamma(\io^{(1)}) \cap \Omega_o(\io^{(2)}) = \Omega_o(\io^{(2)})$.

Let us now prove the inductive step from $\nu$ to $\nu+1$.
We assume that ${\bf (S1)_{\nu}}$, ${\bf (S2)_{\nu}}$ hold and begin
by showing ${\bf (S1)}_{\nu + 1} $. 
Since the torus embeddings $\breve \io^{(1)}$, $\breve \io^{(2)}$ 
satisfy \eqref{final KAM smallness condition},
it follows from \eqref{stima R0 riducibilita} that the operators ${\bf R}_0(\breve \io^{(a)})$, $a = 1, 2$, 
satisfy 
\be\label{estimate for R0 frakd}
|{\bf R}_0(\breve \io^{(a)}) {\frak D}|_{s_0 + \beta, \sigma - 1} \lessdot \e \g^{-2}\,.
\ee
In particular, the condition \eqref{piccolezza1} of Theorem \ref{thm:abstract linear reducibility} holds
and hence \eqref{Rsb}, combined with \eqref{estimate for R0 frakd}, yields

\begin{equation}\label{Rsb io 1 io 2}
|{\bf R}_\nu(\breve \io^{(a)}) {\frak D}|_{s_0, \sigma - 1} \lessdot \e \g^{-2} N_{\nu - 1}^{- \alpha}\,, \quad 
|{\bf R}_\nu(\breve \io^{(a)}) {\frak D}|_{s_0 + \beta, \sigma - 1} \lessdot \e \g^{-2}  N_{\nu - 1}\,, \quad
 a = 1,2\,.
\end{equation}
We have to estimate $\Delta_{12} {\bf R}_{\nu + 1}$, which according to \eqref{D+} is given by
\begin{equation}\label{recall bf R nu + 1}
\Delta_{12} {\bf R}_{\nu + 1} = 
\Delta_{12} ( \Phi^{-1}_\nu \tilde {\bf R}_\nu) + 
\Delta_{12} ( (\Phi^{-1}_\nu - {\mathbb I}_2 )  {\bf R}^{nf}_\nu)
\end{equation}
where by \eqref{bf R tilde nu}
\begin{equation}\label{recall bf R nu}
\widetilde{\bf R}_\nu = \Pi_{N_\nu}^\bot {\bf R}_\nu + ( \om \cdot \partial_\vphi ) ( \Phi_\nu -{\mathbb I}_2 + \Psi_\nu ) + 
 \left[ {\bf N}_\nu, \Phi_\nu - {\mathbb I}_2 + \Psi_\nu \right] +  {\bf R}_\nu ( \Phi_\nu -{\mathbb I}_2 )\,.
\end{equation}
 We first need to estimate 
$\Delta_{12} { \Psi}_\nu = {\Psi}_\nu(\breve \io^{(1)}) - {\Psi}_\nu(\breve \io^{(2)})$
where ${\Psi}_\nu(\breve \io^{(a)})$, $a = 1, 2$, are the solutions of the homological equation \eqref{eq:homo} with ${\bf R}_\nu = {\bf R}_\nu(\breve \io^{(a)})$:

\begin{lemma}\label{Psi(u1) - Psi(u2)}
 For $s = s_0$ and  $s = s_0 + \beta$, the norms
$|\Delta_{12}\Psi_\nu {\frak D}|_{s, \sigma - 1}$,
$|\Delta_{12}\Psi_\nu |_{s, \sigma}$, and $|\Delta_{12}\Psi_\nu |_{s, \sigma - 2}$
are $\lessdot$ bounded for any $\nu \ge 0$ by
$$
 N_\nu^{2 \tau}\Big(
\gamma^{- 2}|{\bf R}_\nu(\breve \io^{(1)}) {\frak D}|_{s, \sigma - 1} 
\|\Delta_{12} \io \|_{s_0 + \bar \mu + \beta} 
+ \gamma^{- 2}|{\bf R}_\nu(\breve \io^{(2)}) {\frak D}|_{s, \sigma - 1} 
\|\Delta_{12} \io \|_{s_0 + \bar \mu + \beta} 
 +\gamma^{- 1} |\Delta_{12} {\bf R}_\nu {\frak D} |_{s, \sigma - 1} \Big) \,.
$$
\end{lemma}
\begin{proof}
To simplify notations, we drop the index $\nu$ in this proof. 
Since $\Psi_\nu$ is of the form  \eqref{structure Psi nu}, it suffices to prove the estimates corresponding
to  the claimed ones
 for the operators $\Delta_{12} \Psi^{(1)} \lla D \rra $ and $\Delta_{12} \Psi^{(2)} \lla D \rra $. The estimates for
these two operators can be shown in the same way and hence we consider $\Delta_{12} \Psi^{(1)} \lla D \rra $ only. 
Evaluating \eqref{Psi 1 j k ell} at $\io^{(a)}$, one has for any $j, k \in S^\bot_+$ 
and any $\omega$  in  $\Omega_{\nu + 1}^{\gamma_a}(\breve \io^{(a)})$,
$$
[\hat\Psi^{(1)} (\ell)]_j^k =
- \ii L^{-}(\ell, j, k)^{-1}  [{\hat{\bf R}}^{(1)} (\ell)]^k_j\,,\quad \forall \ell \in \Z^S\,, \quad |\ell| \leq N\,, \quad (\ell, j, k) \neq (0, j, j)
$$
and hence for any $\omega \in  
\Omega_{\nu + 1}^{\gamma_1}(\breve \io^{(1)}) \cap \Omega_{\nu + 1}^{\gamma_2} (\breve \io^{(2)})$,
\begin{equation}\label{australia - 1}
\Delta_{12} [\hat\Psi^{(1)} (\ell)]_j^k = - \ii \big( \Delta_{12}L^{-}(\ell, j, k)^{-1} \big)[{\hat{\bf R}}^{(1)} (\ell ; \breve \io^{(1)})]^k_j -  \ii L^{-}(\ell, j, k ; \breve \io^{(2)})^{-1} \big(\Delta_{1 2} [{\hat{\bf R}}^{(1)} (\ell)]^k_j \big)\,.
\end{equation}
Together with
\begin{align*}
\Delta_{12} L^{-}(\ell, j, k)^{- 1} =  - L^{-}(\ell, j, k ; \breve \io^{(2)})^{- 1}  \Delta_{12} L^{-}(\ell, j, k )  L^{-}(\ell, j, k ; \breve \io^{(1)})^{- 1}\,,
\end{align*}
the definition \eqref{L - nu} of $L^{-}(\ell, j, k )$ implies that
$$
\Delta_{12} L^{-}(\ell, j, k )   =  M_L\big(\Delta_{12} [{\bf N}^{(1)} ]_j^j  \big)  - M_R\big(\Delta_{12} [{\bf N}^{(1)}]_k^k \big) \,.
$$
By the induction hypothesis, estimate \eqref{Delta12 rj} holds and hence
$ \| \Delta_{12} L^-(\ell, j, k) \|  \lessdot \| \Delta_{12} \io\|_{s_0 + \bar \mu + \beta} $. 
This together with \eqref{Hyp2} then yields
$$
\|\Delta_{12}  L^-(\ell, j, k)^{- 1}  \| \lessdot  \frac{N^{2 \tau}}{\g_1\g_2 \langle j^2 - k^2 \rangle^2} \| \Delta_{12} \io\|_{s_0 + \bar \mu + \beta}\,.
$$
Hence \eqref{australia - 1} implies that for any $\ell \in \Z^S$, $|\ell| \leq N$, and $j, k \in S_+^\bot$,
$$
\|  \Delta_{12} [\hat\Psi^{(1)} (\ell)]_j^k\| \lessdot  
\frac{N^{2 \tau}}{\g_1 \g_2 \langle j^2 - k^2 \rangle^2} \| \Delta_{12} \io\|_{s_0 + \bar \mu + \beta} \|  [{\hat{\bf R}}^{(1)} (\ell ; \breve \io^{(1)})]^k_j  \| + 
\frac{N^\tau}{\g_2 \langle j^2 - k^2 \rangle}\| \Delta_{1 2} [{\hat{\bf R}}^{(1)} (\ell)]^k_j  \|\,.
$$
Arguing as in the proof of Lemma \ref{lemma:redu} for deriving the estimate of 
$ \| \hat{\Psi}^{(1)}(\ell) \lla D \rra \|_{{\cal L}(h^{\sigma - 1}_\bot)}$ and using the assumption 
$\g_1, \g_2 \ge \g/2$,
one sees that for any $\ell \in \Z^S$, $|\ell| \leq N$,
$$
\| \Delta_{12} \hat\Psi^{(1)}(\ell )  \lla D \rra   \|_{{\cal L}(h^{\sigma - 1}_\bot)}  \lessdot 	
N^{2 \tau } \gamma^{- 2} \| \Delta_{12} \io \|_{s_0 + \bar \mu + \beta}   
\| \hat{\bf R}^{(1)}(\ell ; \breve \io^{(2)}  )\|_{{\cal L}(h^{\sigma - 1}_\bot)} +
{N^\tau}\gamma^{- 1} \| \Delta_{12} \hat{\bf R}^{(1)}(\ell )  \lla D \rra \|_{{\cal L}(h^{\sigma - 1}_\bot)} 
$$
which implies that $|\Delta_{12} \Psi^{(1)} \lla D \rra |_{s, \sigma - 1}$ satisfies the claimed estimate. The one for 
$|\Delta_{12} \Psi^{(1)}|_{s, \sigma}$ follows by similar arguments. Finally,
the  estimate for $|\Delta_{12} \Psi^{(1)} \lla D \rra |_{s, \sigma - 1}$  implies the claimed one
 for $|\Delta_{12} \Psi^{(1)}|_{s, \sigma - 2}$ since 
$|\Delta_{12} \Psi^{(1)}|_{s, \sigma - 2} \leq |\Delta_{12} \Psi^{(1)} \lla D \rra |_{s, \sigma - 1}$. 
\end{proof}

We estimate each term in the expression \eqref{recall bf R nu + 1}
for $\Delta_{12} {\bf R}_{\nu + 1}$ individually. For convenience,  introduce
$$
R_\nu(s) := {\rm max}\{ |{\bf R}_\nu(\breve \io^{(1)}) {\frak D}|_{s, \sigma - 1}, |{\bf R}_\nu(\breve \io^{(2)}) {\frak D}|_{s, \sigma - 1} \}\,, \quad s = s_0, \,  s_0 + \beta \, .
$$
By Lemma \ref{Psi(u1) - Psi(u2)} and then using the induction hypothesis, one sees that
\begin{align}
|\Delta_{12} \Psi_\nu & {\frak D}|_{s_0, \sigma - 1} \lessdot  N_\nu^{2 \tau }\big( \gamma^{- 2}  R_\nu(s_0) \|\Delta_{12} \io \|_{s_0 + \bar \mu + \beta} + \gamma^{- 1}|\Delta_{12}{\bf R}_\nu {\frak D}|_{s_0, \sigma - 1} \big) \nonumber\\
& \stackrel{\eqref{Rsb io 1 io 2}, \, \eqref{derivate-R-nu},\, \,\e \gamma^{- 1} \leq 1}{\lessdot}  \,\,\,
N_\nu^{2 \tau } N_{\nu - 1}^{- \alpha} \gamma^{- 1} \| \Delta_{12} \io \|_{s_0 + \bar \mu + \beta}\, \label{delta 12 Psi nu s0}
\end{align}
and
\begin{align}
|\Delta_{12} \Psi_\nu  {\frak D} &|_{s_0 + \beta, \sigma - 1} \lessdot N_\nu^{2 \tau } 
 \big( \gamma^{- 2} R_\nu(s_0 + \beta) \|\Delta_{12} \io \|_{s_0 + \bar \mu + \beta} + 
\gamma^{- 1}|\Delta_{12}{\bf R}_\nu {\frak D}|_{s_0 + \beta, \sigma - 1} \big) \nonumber\\
& \stackrel{\eqref{Rsb io 1 io 2}, \, \eqref{derivate-R-nu},\,\,  \e \gamma^{- 1} \leq 1}{\lessdot} \,
 N_\nu^{2 \tau } N_{\nu - 1}  \gamma^{- 1} \| \Delta_{12} \io \|_{s_0 + \bar \mu + \beta}\,.
\label{delta 12 Psi nu s0 + beta}
\end{align}
By Lemma \ref{lemma:redu}, 
the operators $\Psi_\nu(\breve \io^{(a)})$, $a = 1, 2$, satisfy the estimates 
\begin{equation}\label{stime Psi nu (i a)}
|\Psi_\nu(\breve \io^{(a)}) {\frak D}|_{s, \sigma - 1},  \,\, 
|\Psi_\nu(\breve \io^{(a)})|_{s, \sigma}, \,\,
|\Psi_\nu(\breve \io^{(a)}) |_{s, \sigma - 2} \, \lessdot \,
N_\nu^{\tau} \gamma^{- 1} R_\nu(s)\,, \quad s= s_0, \, s_0 + \beta\,.
\end{equation}
Taking into account that 
\begin{align}\label{condizione u1 - u2 stima delta 12 Phi}
N_\nu^{\tau} \gamma^{-1} R_\nu(s_0) & \stackrel{\eqref{Rsb io 1 io 2}}{\lessdot} 
N_\nu^{\tau} N_{\nu - 1}^{- \alpha} \e \gamma^{- 3} 
\stackrel{\eqref{alpha beta}, \eqref{final KAM smallness condition}}{\leq}1\, , 
\end{align}
one then concludes from  
\eqref{derivata-inversa-Phi} and \eqref{stime Psi nu (i a)}
  that
\begin{align}
|\Delta_{12} \Phi_\nu^{\pm 1}{\frak D} |_{s_0, \sigma - 1} & \lessdot 
|\Delta_{12} \Psi_\nu {\frak D} |_{s_0, \sigma - 1} \stackrel{\eqref{delta 12 Psi nu s0}}{\lessdot} 
 N_\nu^{2 \tau } N_{\nu - 1}^{- \alpha}  \gamma^{- 1} 
\| \Delta_{12} \io \|_{s_0 + \bar \mu + \beta}   \label{Delta 12 Phi nu norma bassa}
\end{align}
and
\begin{align}
|\Delta_{12} & \Phi_\nu^{\pm 1} {\frak D} |_{s_0 + \beta, \sigma - 1}  \lessdot 
|\Delta_{12} \Psi_\nu {\frak D}|_{s_0 + \beta, \sigma - 1} + 
(|\Psi_\nu(\breve \io^{(1)}) {\frak D} |_{s_0 + \beta, \sigma - 1} + 
|\Psi_\nu(\breve \io^{(2)}) {\frak D}|_{s_0 + \beta, \sigma - 1}) 
|\Delta_{12} \Psi_\nu {\frak D}|_{s_0, \sigma - 1}  \nonumber\\
& \stackrel{\eqref{delta 12 Psi nu s0}, \,\eqref{delta 12 Psi nu s0 + beta}, \, 
\eqref{stime Psi nu (i a)}}{\lessdot} \,\,\,
N_\nu^{2 \tau} N_{\nu - 1}  \gamma^{- 1} \| \Delta_{12} \io \|_{s_0 + \bar \mu + \beta}  
   + N_\nu^{\tau} \gamma^{- 1} R_\nu(s_0 + \beta) N_\nu^{2 \tau} N_{\nu - 1}^{- \alpha}  \gamma^{- 1} \| \Delta_{12} \io \|_{s_0 + \bar \mu + \beta} \nonumber\\
& \stackrel{\eqref{Rsb io 1 io 2}, \, \eqref{alpha beta}, \,\,\e \gamma^{- 3} \leq 1}{\lessdot} \,
N_\nu^{2 \tau} N_{\nu - 1}  \gamma^{- 1}  \| \Delta_{12} \io \|_{s_0 + \bar \mu + \beta}\,.
\label{Delta 12 Phi nu norma alta}
\end{align}

\medskip
\noindent
{\it Estimate of $\Delta_{12} \widetilde{\bf R}_\nu$:}
We begin by estimating the term 
$\Delta_{12} \big( {\bf R}_\nu  (\Phi_\nu -  {\mathbb I}_2) \big)$
in $\Delta_{12}\widetilde{\bf R}_\nu$ (cf \eqref{recall bf R nu}):
\begin{align}
|\Delta_{12}\big( {\bf R}_\nu & (\Phi_\nu -{\mathbb I}_2) \big){\frak D} \big\} |_{s_0, \sigma -1}  \lessdot |\Delta_{12 } {\bf R}_\nu {\frak D}|_{s_0, \sigma - 1} |(\Phi_\nu(\breve \io^{(1)}) - {\mathbb I}_2){\frak D} |_{s_0, \sigma - 1} + |{\bf R}_\nu(\breve \io^{(2)}) {\frak D}|_{s_0, \sigma - 1} |\Delta_{12} \Phi_\nu {\frak D}|_{s_0, \sigma - 1} \nonumber\\
& \stackrel{\eqref{PhINV with D} }{\lessdot} \, 
|\Delta_{12 } {\bf R}_\nu {\frak D}|_{s_0, \sigma - 1} |\Psi_\nu(\breve \io^{(1)}) {\frak D}  |_{s_0, \sigma - 1} + |{\bf R}_\nu(\breve \io^{(2)}) {\frak D} |_{s_0, \sigma - 1} |\Delta_{12} \Phi_\nu {\frak D}|_{s_0, \sigma - 1} \,. \nonumber
\end{align}
Using the induction hypothesis one sees that
\be\label{stima primo pezzo delta 12 R nu + 1 s0}
|\Delta_{12}\big( {\bf R}_\nu  (\Phi_\nu -{\mathbb I}_2) \big){\frak D} \big\} |_{s_0, \sigma - 1}
 \stackrel{\eqref{Delta 12 Phi nu norma bassa}, \eqref{stime Psi nu (i a)}, \eqref{Rsb io 1 io 2}, \eqref{derivate-R-nu}}{\lessdot} N_\nu^{2 \tau} N_{\nu - 1}^{- 2 \alpha} \e \gamma^{- 3} \|\Delta_{12} \io \|_{s_0 + \bar \mu + \beta}\,. 
\ee
Similarly,  $|\Delta_{12} \big( {\bf R}_\nu  (\Phi_\nu - {\mathbb I}_2) \big) 
{\frak D}|_{s_0 + \beta, \sigma -1}$ is $\lessdot$ bounded by
\begin{align}
& |\Delta_{12 } {\bf R}_\nu {\frak D}|_{s_0  + \beta, \sigma - 1} 
|(\Phi_\nu(\breve \io^{(1)}) - {\mathbb I}_2){\frak D} |_{s_0, \sigma - 1} 
 + |\Delta_{12 } {\bf R}_\nu {\frak D}|_{s_0, \sigma - 1 } 
|(\Phi_\nu(\breve \io^{(1)}) - {\mathbb I}_2){\frak D} |_{s_0 + \beta, \sigma - 1} \nonumber\\
& + |{\bf R}_\nu(\breve \io^{(2)}) {\frak D}|_{s_0 + \beta, \sigma - 1} 
|\Delta_{12} \Phi_\nu {\frak D}|_{s_0, \sigma - 1} +  
|{\bf R}_\nu(\breve \io^{(2)}) {\frak D}|_{s_0, \sigma - 1} 
|\Delta_{12} \Phi_\nu {\frak D}|_{s_0 + \beta, \sigma - 1} \nonumber\\
& \stackrel{\eqref{PhINV with D}}{\lessdot} \,
|\Delta_{12 } {\bf R}_\nu {\frak D}|_{s_0  + \beta, \sigma - 1} |\Psi_\nu(\breve \io^{(1)}) {\frak D}|_{s_0, \sigma - 1} + |\Delta_{12 } {\bf R}_\nu {\frak D}|_{s_0, \sigma - 1 } |\Psi_\nu(\breve \io^{(1)}) |_{s_0 + \beta, \sigma - 1} \nonumber\\
& \quad + |{\bf R}_\nu(\breve \io^{(2)}) {\frak D}|_{s_0 + \beta, \sigma - 1} |\Delta_{12} \Phi_\nu {\frak D}|_{s_0, \sigma - 1} +  |{\bf R}_\nu(\breve \io^{(2)}) {\frak D}|_{s_0, \sigma - 1} |\Delta_{12} \Phi_\nu {\frak D}|_{s_0 + \beta, \sigma - 1} \nonumber
\end{align}
which by \eqref{stime Psi nu (i a)}  is $\lessdot$ bounded by
\begin{align}
&|\Delta_{12 } {\bf R}_\nu {\frak D}|_{s_0  + \beta, \sigma - 1} 
N_\nu^{\tau} \gamma^{- 1} R_\nu(s_0) \,
 +  \, |\Delta_{12 } {\bf R}_\nu {\frak D}|_{s_0, \sigma - 1 } 
N_\nu^{\tau} \gamma^{- 1}R_\nu(s_0 + \beta)  \nonumber\\
& + R_\nu(s_0 + \beta)  |\Delta_{12} \Phi_\nu {\frak D}|_{s_0, \sigma - 1} + 
R_\nu(s_0) |\Delta_{12} \Phi_\nu {\frak D}|_{s_0 + \beta, \sigma - 1}\,.\nonumber
\end{align}
Again using the induction hypothesis, one then obtains
by \eqref{Delta 12 Phi nu norma bassa}, \eqref{Delta 12 Phi nu norma alta}, \eqref{condizione u1 - u2 stima delta 12 Phi}, \eqref{Rsb io 1 io 2}, \eqref{derivate-R-nu}
\be\label{stima primo pezzo delta 12 R nu + 1 s0 + beta}
|\Delta_{12} \big( {\bf R}_\nu  (\Phi_\nu - {\mathbb I}_2) \big) 
{\frak D}|_{s_0 + \beta, \sigma - 1} 
\lessdot  \,\,
 N_{\nu - 1}   \| \Delta_{12} \io \|_{s_0 + \bar \mu + \beta}\,. 
\ee
Next we estimate the term 
$ \Delta_{12} \big( \om \cdot \partial_\vphi ) ( \Phi_\nu -{\mathbb I}_2 + \Psi_\nu ) + 
 \left[ {\bf N}_\nu, \Phi_\nu - {\mathbb I}_2 + \Psi_\nu \right]\big)$
in $\Delta_{12}\widetilde{\bf R}_\nu$. 
Since $\Phi_\nu = {\rm exp}(- \Psi_\nu)$, one has
 \begin{equation}\label{varsavia}
 ( \om \cdot \partial_\vphi ) ( \Phi_\nu -{\mathbb I}_2 + \Psi_\nu ) + 
 \left[ {\bf N}_\nu, \Phi_\nu - {\mathbb I}_2 + \Psi_\nu \right] = \sum_{n \geq 2} (- 1)^n \frac{(\omega \cdot \partial_\vphi)(\Psi_\nu^n) + [{\bf N}_\nu, \Psi_\nu^n]}{n !}
 \end{equation}
where by \eqref{pizza margherita}
\begin{equation}\label{alfredo}
(\omega \cdot \partial_\vphi)(\Psi_\nu^n) + [{\bf N}_\nu, \Psi_\nu^n] = 
\sum_{n_1 + n_2 + 1= n} \Psi_\nu^{n_1} (\Pi_{N_\nu} {\bf R}_\nu - {\bf R}_\nu^{n f}) \Psi_\nu^{n_2}\,.
 \end{equation}
Iterating the tame estimates \eqref{interpm} for the composition of operator valued maps 
one sees  that for any $i, k$ with $i + k +1 = n$ ($ \ge 2$), \,
$|\Delta_{12}\big(  \Psi_\nu^i (\Pi_{N_\nu} {\bf R}_\nu - {\bf R}_\nu^{n f}) \, \Psi_\nu^k \big) 
{\frak D} |_{s_0, \sigma - 1} $ is bounded by
$$
\big( C' |\Psi_\nu  {\frak D}|_{s_0, \sigma - 1}\big)^{n-1}
|\Delta_{12} {\bf R}_\nu {\frak D}|_{s_0, \sigma - 1} + \,
(n - 1) C' \big( C' |\Psi_\nu  {\frak D}|_{s_0, \sigma - 1}\big)^{n-2}
|{\bf R}_\nu {\frak D}|_{s_0, \sigma - 1} |\Delta_{12} \Psi_\nu  {\frak D}|_{s_0, \sigma - 1}
$$
where $C' \equiv C'(s_0) := 2 C_{op}(s_0)$ with $C_{op}(s)$ as in \eqref{interpm}.
Using \eqref{stime Psi nu (i a)},  
\eqref{delta 12 Psi nu s0} and increasing $C'$ if necessary, one sees that 
the latter expression is bounded by 
\begin{align}
&\big( C' N_\nu^{\tau} \gamma^{- 1}  R_\nu(s_0) \big)^{n - 1} 
|\Delta_{12} {\bf R}_\nu {\frak D}|_{s_0, \sigma - 1} 
+  (n-1) \, C' \big( C' N_\nu^{\tau} \gamma^{- 1} R_\nu(s_0) \big)^{n - 2 } R_\nu(s_0) \, \gamma^{- 1}  N_\nu^{2\tau} N_{\nu - 1}^{- \alpha} \| \Delta_{12} \io \|_{s_0 + \bar \mu + \beta} \nonumber\\
& \stackrel{\eqref{derivate-R-nu}, \eqref{Rsb io 1 io 2} }{\lessdot}  
n\,  C^{n-1} \big( N_\nu^{\tau} N_{\nu - 1}^{- \alpha} \e \gamma^{- 3}\big)^{n - 2} 
N_\nu^{2\tau} N_{\nu - 1}^{- 2\alpha} \e \gamma^{- 3} 
\| \Delta_{12} \io \|_{s_0 + \bar \mu + \beta} \nonumber
\end{align}
with $C \equiv C(s_0) > C'$ chosen sufficiently large.  
Together with \eqref{condizione u1 - u2 stima delta 12 Phi} this then implies that
\be\label{alfredo n s0}
|\Delta_{12}\big(  \Psi_\nu^i (\Pi_{N_\nu} {\bf R}_\nu - {\bf R}_\nu^{n f}) \,\Psi_\nu^k \big) {\frak D} |_{s_0, \sigma - 1}  \lessdot
n \, C(s_0)^{n-1} N_\nu^{2\tau}  N_{\nu - 1}^{- 2 \alpha} \e \gamma^{- 3} 
\| \Delta_{12} \io \|_{s_0 + \bar \mu + \beta} \,.
\ee
Similarly, using 
\eqref{Rsb io 1 io 2}, the induction hypothesis \eqref{derivate-R-nu}, and 
\eqref{delta 12 Psi nu s0}, \eqref{delta 12 Psi nu s0 + beta},
\eqref{stime Psi nu (i a)},  one sees that for 
$C(s_0 +\b) > 2 C_{op}(s_0 + \b)$ sufficiently large and 
any $i, k$ with $i + k +1 = n$ ($ \ge 2$),
$|\Delta_{12}\big(  \Psi_\nu^i ({\bf R}_\nu - {\bf R}_\nu^{nf}) \, \Psi_\nu^k \big) 
{\frak D}|_{s_0 + \beta, \sigma - 1} $
is bounded by
$$
n^2 C(s_0 + \beta)^{n-1} 
\big(N_\nu^{\tau} N_{\nu - 1}^{- \alpha} \e \gamma^{- 3}\big)^{n - 2}
N_\nu^{2\tau} N_{\nu - 1}^{- \alpha} \e \gamma^{- 3}
N_{\nu - 1} \| \Delta_{12} \io \|_{s_0 + \bar \mu + \beta}
$$
yielding
\be\label{alfredo n s0 + beta}
|\Delta_{12}\big(  \Psi_\nu^i ({\bf R}_\nu - {\bf R}_\nu^{nf}) \, \Psi_\nu^k \big) 
{\frak D}|_{s_0 + \beta, \sigma - 1} 
 \stackrel{\eqref{alpha beta}, \eqref{final KAM smallness condition}}{\lessdot} \,
n^2 \, C(s_0 + \beta)^{n - 1} N_{\nu - 1}  \,
\| \Delta_{12} \io \|_{s_0 + \bar \mu + \beta}\,. 
\ee
Hence by \eqref{varsavia}
\begin{align}
\big| \big(( \om \cdot \partial_\vphi ) & ( \Phi_\nu -{\mathbb I}_2 + \Psi_\nu ) + 
 \left[ {\bf N}_\nu, \Phi_\nu - {\mathbb I}_2 + \Psi_\nu \right] \big) 
{\frak D} \big|_{s_0, \sigma - 1}
\stackrel{\eqref{alfredo}}{\leq} 
\sum_{n \geq 2} \frac{1}{n!} \sum_{i + k +1 = n } 
\big|\Delta_{12}\big(  \Psi_\nu^i ({\bf R}_\nu - {\bf R}_\nu^{nf}) \Psi_\nu^k \big) 
{\frak D}\big|_{s_0, \sigma - 1} \nonumber\\
& \stackrel{\eqref{alfredo n s0}}{\lessdot} N_\nu^{2\tau} N_{\nu - 1}^{- 2 \alpha} 
\e \gamma^{- 3} \| \Delta_{12} \io \|_{s_0 + \bar \mu + \beta} 
\sum_{n \geq 2} \frac{C(s_0)^{n - 1} }{(n - 2) !} \,
 \lessdot \, N_\nu^{2\tau } N_{\nu - 1}^{- 2 \alpha} \e \gamma^{- 3} 
\| \Delta_{12} \io \|_{s_0 + \bar \mu + \beta} \label{stima quarto pezzo delta 12 R nu + 1 s0}\,.
\end{align}
Similarly, 
$\big| \big(( \om \cdot \partial_\vphi ) ( \Phi_\nu -{\mathbb I}_2 + \Psi_\nu ) + 
 \left[ {\bf N}_\nu, \Phi_\nu - {\mathbb I}_2 + \Psi_\nu \right] \big) 
{\frak D} \big|_{s_0 + \beta , \sigma - 1} $
is bounded by 
$$
\sum_{n \geq 2} \frac{1}{n!} \sum_{i + k +1 = n } 
\big|\Delta_{12}\big(  \Psi_\nu^i ({\bf R}_\nu - {\bf R}_\nu^{nf}) \Psi_\nu^k \big) 
{\frak D} \big|_{s_0 + \beta, \sigma - 1} \nonumber\\
 \stackrel{\eqref{alfredo n s0 + beta}}{\lessdot} 
N_{\nu - 1}  \| \Delta_{12} \io \|_{s_0 + \bar \mu + \beta} 
\sum_{n \geq 2} n \frac{C(s_0 + \beta)^{n - 1} }{(n - 2) !} 
$$
leading to the estimate
\be \label{stima quarto pezzo delta 12 R nu + 1 s0 + beta}
\big| \big(( \om \cdot \partial_\vphi ) ( \Phi_\nu -{\mathbb I}_2 + \Psi_\nu ) + 
 \left[ {\bf N}_\nu, \Phi_\nu - {\mathbb I}_2 + \Psi_\nu \right] \big) 
{\frak D} \big|_{s_0 + \beta , \sigma - 1} 
 \lessdot N_{\nu - 1}  \| \Delta_{12} \io \|_{s_0 + \bar \mu + \beta}\,. 
\ee
Finally,  the term 
$\Delta_{12} \Pi_{N_\nu}^\bot {\bf R}_\nu = \Pi_{N_\nu}^\bot \Delta_{12} {\bf R}_\nu $
 in $\Delta_{12}\widetilde{\bf R}_\nu$ (cf \eqref{recall bf R nu})
can be estimated as 
\be\label{estimate first piece delta 12 R nu + 1 s0}
|\Pi_{N_\nu}^\bot \Delta_{12}{\bf R}_\nu {\frak D}|_{s_0, \sigma - 1} \stackrel{\eqref{smoothingN}}{\leq} N_\nu^{- \beta} 
|\Delta_{12} {\bf R}_\nu {\frak D}|_{s_0 + \beta, \sigma - 1}
\stackrel{\eqref{derivate-R-nu}}{\lessdot}  N_\nu^{- \beta}
N_{\nu - 1}  \Vert \Delta_{12} \io \Vert_{ s_0 + \bar \mu + \beta } 
\ee
and
\be\label{estimate first piece delta 12 R nu + 1 s0 + beta}
 |\Pi_{N_\nu}^\bot \Delta_{12}{\bf R}_\nu {\frak D}|_{s_0 + \beta, \sigma - 1} \leq |\Delta_{12} {\bf R}_\nu {\frak D}|_{s_0 + \beta, \sigma - 1}
\stackrel{\eqref{derivate-R-nu}}{\lessdot}  N_{\nu -1}
\Vert \Delta_{12} \io \Vert_{ s_0 + \bar \mu + \beta } \,.
\ee
Combining the estimates \eqref{stima primo pezzo delta 12 R nu + 1 s0}, 
\eqref{stima quarto pezzo delta 12 R nu + 1 s0}, and 
\eqref{estimate first piece delta 12 R nu + 1 s0}
we get 
\begin{align}
|\Delta_{12} \widetilde{\bf R}_{\nu }{\frak D}|_{s_0, \sigma - 1} & \lessdot 
\big(N_{\nu - 1} N_{\nu}^{- \beta}  + 
N_\nu^{2 \tau} N_{\nu - 1}^{- 2 \alpha} \e \gamma^{- 3} \big) 
\| \Delta_{12} \io \|_{s_0 + \bar \mu + \beta}\,,  \label{stima Delta 12 tilde R nu s0}
\end{align}
whereas   \eqref{stima primo pezzo delta 12 R nu + 1 s0 + beta}, 
\eqref{stima quarto pezzo delta 12 R nu + 1 s0 + beta}, and
\eqref{estimate first piece delta 12 R nu + 1 s0 + beta} lead to
\begin{equation}\label{stima Delta 12 tilde R nu s0 + beta}
|\Delta_{12} \widetilde{\bf R}_\nu {\frak D}|_{s_0 + \beta, \sigma - 1} \lessdot N_{\nu - 1} \| \Delta_{12} \io \|_{s_0 + \bar \mu + \beta}\,.
\end{equation}

\noindent
{\it Estimate of $\Delta_{12} {\bf R}_{\nu + 1}$:} 
Arguing as in \eqref{stima primo pezzo delta 12 R nu + 1 s0}, \eqref{stima primo pezzo delta 12 R nu + 1 s0 + beta}, we get 
\begin{align}\label{stima secondo pezzo delta 12 R nu + 1 s0}
& |\Delta_{12} \big( (\Phi_\nu^{- 1} - {\mathbb I}_2) {\bf R}_\nu^{nf}  \big) 
{\frak D}|_{s_0, \sigma - 1} \lessdot 
N_\nu^{2 \tau } N_{\nu - 1}^{- 2 \alpha} \e \gamma^{- 3} 
\|\Delta_{12} \io \|_{s_0 + \bar \mu + \beta}\,,  \\
& \label{stima secondo pezzo delta 12 R nu + 1 s0 + beta}
 |\Delta_{12} \big( (\Phi_\nu^{- 1} - {\mathbb I}_2) {\bf R}_\nu^{nf}  \big) 
{\frak D} |_{s_0 + \beta, \sigma - 1}\lessdot N_{\nu - 1}  
\| \Delta_{12} \io \|_{s_0 + \bar \mu + \beta}\,.
\end{align}
Moreover, by the arguments in the proof of $({\bf S 1})_\nu$ in 
Section \ref{proof of reduction theorem}, 
the operators $\widetilde{\bf R}_\nu(\breve \io^{(a)})$, $a = 1, 2$, satisfy 
$$
|\widetilde {\bf R}_\nu {\frak D}|_{s, \sigma - 1} \leq_s
|\Pi_{N_\nu}^\bot {\bf R}_\nu {\frak D}|_{s, \sigma - 1} + 
N_\nu^{2 \tau + 1} \gamma^{- 1} |{\bf R}_\nu {\frak D}|_{s, \sigma - 1}  
|{\bf R}_\nu {\frak D}|_{s_0, \sigma - 1}\,.
$$
Since $|\Pi_{N_\nu}^\bot {\bf R}_\nu {\frak D}|_{s_0, \sigma - 1} \lessdot
N_\nu^{- \beta} |\Pi_{N_\nu}^\bot {\bf R}_\nu {\frak D}|_{s_0 + \b, \sigma - 1} $
one concludes from \eqref{Rsb io 1 io 2} together with
\eqref{alpha beta}, \eqref{final KAM smallness condition} that
\begin{align}
|\widetilde{\bf R}_\nu(\breve \io^{(a)}) {\frak D}|_{s_0, \sigma - 1} \leq_s
N_{\nu - 1} N_\nu^{- \beta} \e \g^{-2}  + 
N_\nu^{2 \tau + 1} N_{\nu - 1}^{- 2 \alpha} \e \g^{-1}\,, \quad 
|\widetilde{\bf R}_\nu(\breve \io^{(a)}) {\frak D} |_{s_0 + \beta, \sigma - 1} \lessdot N_{\nu - 1} \e \g^{-2} \,. \label{varsavia 1}
\end{align}
Recalling that for $a = 1, 2$, 
\begin{align*}
& |(\Phi_\nu^{- 1}(\breve \io^{(a)}) - {\mathbb I}_2){\frak D}|_{s_0, \sigma - 1} \stackrel{\eqref{PhINV with D}}{\lessdot} |\Psi_\nu(\breve \io^{(a)}) {\frak D}|_{s_0, \sigma - 1} \stackrel{\eqref{stime Psi nu (i a)}, \eqref{Rsb io 1 io 2}}{\lessdot} N_\nu^{\tau} N_{\nu - 1}^{- \alpha} \e \gamma^{- 3}\,, \\
&  |(\Phi_\nu^{- 1}(\breve \io^{(a)}) - {\mathbb I}_2){\frak D}|_{s_0 + \beta, \sigma - 1} \stackrel{\eqref{PhINV with D}}{\lessdot} |\Psi_\nu(\breve \io^{(a)}) {\frak D}|_{s_0 + \beta, \sigma - 1} \stackrel{\eqref{stime Psi nu (i a)}, \eqref{Rsb io 1 io 2}}{\lessdot} N_\nu^{\tau} N_{\nu - 1} \e \gamma^{- 3}\,,
\end{align*}
and using \eqref{Delta 12 Phi nu norma bassa}, \eqref{Delta 12 Phi nu norma alta}, \eqref{stima Delta 12 tilde R nu s0}, \eqref{stima Delta 12 tilde R nu s0 + beta}, \eqref{varsavia 1}, $\e \gamma^{- 3} \leq 1$ (cf \eqref{final KAM smallness condition}) one sees that
\begin{align}\label{varsavia 2}
& \big|\Delta_{12} \big( \Phi_\nu^{- 1} \widetilde{\bf R}_\nu \big) 
{\frak D}\big|_{s_0, \sigma - 1} \lessdot \,
\big(N_{\nu - 1} N_{\nu}^{- \beta}  + N_\nu^{2 \tau + 1} N_{\nu - 1}^{- 2 \alpha} 
\e \gamma^{- 3} \big) \| \Delta_{12} \io \|_{s_0 + \bar \mu + \beta}\,, \\
& \label{varsavia 3}
\big|\Delta_{12} \big( \Phi_\nu^{- 1}  \widetilde{\bf R}_\nu\big) {\frak D}\big|_{s_0 + \beta, \sigma - 1} \lessdot \, N_{\nu - 1} \| \Delta_{12} \io\|_{s_0 + \bar \mu + \beta}\,.
\end{align}
By \eqref{recall bf R nu + 1},
$$
|\Delta_{12} {\bf R}_{\nu + 1} {\frak D}|_{s_0, \sigma - 1}  
\stackrel{\eqref{stima secondo pezzo delta 12 R nu + 1 s0}, \, \eqref{varsavia 2}}{\leq} C(\tau, |S|) \big(N_{\nu - 1} N_{\nu}^{- \beta}  + N_\nu^{2 \tau + 1} N_{\nu - 1}^{- 2 \alpha} 
\e \gamma^{- 3} \big) \| \Delta_{12} \io \|_{s_0 + \bar \mu + \beta} 
$$
for some constant $C(\tau, |S|) > 0$. Hence one has 
$$
|\Delta_{12} {\bf R}_{\nu + 1} {\frak D}|_{s_0, \sigma - 1} 
\leq C_{\rm var} N_\nu^{- \alpha} \| \Delta_{12} \io\|_{s_0 + \bar \mu + \beta}
$$
provided that $C_{\rm var}$ can be chosen such that for any $\nu \ge 0$,
$$
C(\tau, |S|) N_{\nu - 1} N_\nu^{- \beta} N_\nu^{\alpha} \leq C_{\rm var} / 2 
\quad \mbox{and} \quad
 \tilde C(\tau, |S|) N_\nu^{2 \tau + 1} N_\nu^{\alpha} N_{\nu - 1}^{- 2 \alpha} \e \gamma^{- 3} \leq C_{\rm var} / 2\,.
$$
In view of \eqref{alpha beta}, \eqref{final KAM smallness condition} 
this is possible by choosing $N_0$ large enough. Furthermore,
$$
|\Delta_{12} {\bf R}_{\nu + 1} {\frak D} |_{s_0 + \beta, \sigma - 1} \stackrel{\eqref{stima secondo pezzo delta 12 R nu + 1 s0 + beta}, \eqref{varsavia 3}}{\leq}
\widetilde C(\tau, |S|) N_{\nu - 1} \| \Delta_{12} \io\|_{s_0 + \bar \mu + \beta} \,,
$$
for some constant $\widetilde C(\tau, |S|) > 0$, 
implying that by increasing $N_0$, if necessary,  
$$
|\Delta_{12} {\bf R}_{\nu + 1} {\frak D} |_{s_0 + \beta, \sigma - 1}
\leq 
C_{\rm var} N_\nu \| \Delta_{12} \io\|_{s_0 + \overline \mu + \beta}\,.
$$ 
This establishes \eqref{derivate-R-nu} at the inductive step $\nu + 1$. 
Since for any $k \in S_+^\bot$, 
$[{\bf N}_{\nu + 1}^{(1)} - {\bf N}_\nu]_k^k = [\hat{\bf R}^{(1)}_\nu(0)]_k^k$ 
(see \eqref{bf R nf A 1 nu})
the estimate \eqref{deltarj12} follows directly from \eqref{derivate-R-nu} and
implies \eqref{Delta12 rj} by a telescopic argument, using the estimate
\eqref{Delta12 rj} in the case $\nu = 0,$ established at the beginning
of the proof.

\medskip

\noindent
Finally let us turn towards ${\bf (S2)}_{\nu + 1} $. 
Since by the definiton \eqref{Omega-second-meln},  
$\Omega_{\nu + 1}^\gamma(\io^{(1)}) \subseteq \Omega_\nu^\gamma(\io^{(1)})$, by the induction hyphothesis,
$\Omega_\nu^\gamma(\io^{(1)}) \cap \Omega_o(\io^{(2)})\subseteq \Omega_\nu^{\gamma - \rho}(\io^{(2)})$, and $\Omega_\nu^{\gamma - \rho}(\io^{(2)}) \subseteq \Omega_o(\io^{(2)})$,
one has
$$
\Omega_{\nu + 1}^\gamma(\io^{(1)}) \cap \Omega_o(\io^{(2)}) \subseteq \Omega_\nu^{\gamma - \rho}(\io^{(2)}) \,
\stackrel{ 0 < \rho < \g / 2 }{\subseteq} \, \Omega_\nu^{\gamma/2}(\io^{(2)}) \,.
$$
By construction, for any $k \in S_+^\bot$, the $2 \times 2$ matrices 
$[{\bf N}_\nu^{(1)}(\io^{(2)})]_k^k \equiv [{\bf N}_\nu^{(1)}(\omega, \io^{(2)}(\omega))]_k^k$ 
are then defined for $\omega \in \Omega_{\nu + 1}^\gamma(\io^{(1)}) \cap \Omega_o(\io^{(2)})$ and hence 
by the definition \eqref{L - nu}, so are
the operators $L_\nu^-(\ell, j, k; \io^{(a)})$, $a = 1, 2$, for any $\ell \in \Z^S$.
Furthermore, if in addition, 
$|\ell| \leq N_\nu$ and $(\ell, j, k) \neq (0, j, j)$, then 
$L_\nu^-(\ell, j, k; \io^{(1)})$ and $L_\nu^-(\ell, j, k; \io^{(2)})$ are invertible for
any  $\omega \in \Omega_{\nu + 1}^\gamma(\io^{(1)}) \cap \Omega_o(\io^{(2)}) $.
Clearly, it follows from the definition \eqref{L - nu} that
\begin{align}
\| \Delta_{12} L_\nu^-(\ell, j, k) \| & \leq  \|M_L\big( \Delta_{12} [{\bf N}_\nu^{(1)}]_k^k \big) \|  +  \| M_R\big( \Delta_{12} [{\bf N}_\nu^{(1)}]_j^j \big) \| \nonumber\\
& \leq C_{\rm mult} \sup_{\kappa \in S_+^\bot} 
\| \Delta_{12} [{\bf N}_\nu^{(1)}]_\kappa^\kappa \| \stackrel{\eqref{Delta12 rj}}{\leq}  
C_{\rm mult} C_{{\rm lip}}
\| \Delta_{12} \io\|_{s_0 + \bar \mu + \beta} \label{stima Delta 12 L -}
\end{align}
where $C_{\rm mult} > 0$ is an absolute constant related to the multiplication of $2 \times 2$ matrices
and  $C_{{\rm lip}}$ denotes the constant in \eqref{Delta12 rj}, 
implying that for any $\kappa \in S^\bot,$
$\| \Delta_{12}  [{\bf N}_\nu^{(1)}]_\kappa^\kappa \| \leq  C_{\rm lip}  
\| \Delta_{12} \io \|_{s_0 + \bar \mu + \beta}$.
We then define $C_{\rm var}' := C_{\rm mult} C_{{\rm lip}}$ and note 
that by assumption, 
\begin{equation}\label{caltanisetta}
 C_{\rm var}'   N_\nu ^\tau  \| \Delta_{12} \io \|_{s_0 + \bar \mu + \beta} 
\leq \rho \,.
 \end{equation}
It is to show that for any 
$\omega \in \Omega_{\nu + 1}^\gamma(\io^{(1)}) \cap \Omega_o(\io^{(2)})$, 
$L_\nu^-(\ell, j, k ; \io^{(2)}(\omega))$ is invertible and its inverse is bounded by 
$ \frac{ \langle \ell \rangle^\tau }{ ( \gamma - \rho) \langle j^2 - k^2 \rangle}$ 
(cf \eqref{Hyp2}). To this end we write $L_\nu^-(\ell, j, k ; \io^{(2)})$ in the form
\begin{equation}\label{varsavia 10}
L_\nu^-(\ell, j, k ; \io^{(2)}) = L_\nu^-(\ell, j, k ; \io^{(1)}) 
\big( {\rm Id}_2 - L_\nu^-(\ell, j, k ; \io^{(1)})^{- 1} \Delta_{12} L_\nu^-(\ell, j, k)  \big)
\end{equation}
where ${\rm Id}_2$ denotes the $2 \times 2$ identity matrix.
Since for any $\omega \in \Omega_{\nu + 1}^\gamma(\io^{(1)}) \cap \Omega_o(\io^{(2)})$  
\begin{align}
&\| L_\nu^-(\ell, j, k ; \io^{(1)})^{- 1} \Delta_{12} L_\nu^-(\ell, j, k)  \| 
 \leq \| L_\nu^-(\ell, j, k ; \io^{(1)})^{- 1}  \|  \| \Delta_{12} L_\nu^-(\ell, j, k) \| \nonumber\\
& \stackrel{\eqref{stima Delta 12 L -}}{\leq}  C_{\rm var}'
\frac{\langle \ell \rangle^\tau}{\gamma \langle j^2 - k^2 \rangle}  
\| \Delta_{12} \io \|_{s_0 + \bar \mu + \beta} 
\stackrel{|\ell| \leq N_\nu}{\leq} C_{\rm var}'  N_\nu^\tau \gamma^{- 1}   
\| \Delta_{12} \io \|_{s_0 + \bar \mu + \beta} 
\stackrel{\eqref{caltanisetta}}{\leq} \rho \gamma^{- 1} \nonumber
\end{align}
and $\rho \gamma^{- 1} \leq 1/2$ it follows from \eqref{varsavia 10} that  
$L_\nu^-(\ell, j, k; \io^{(2)})$ is invertible by Neumann series and 
\begin{align*}
\| L_\nu^-(\ell, j, k; \io^{(2)})^{- 1}\| & \leq 
\frac{1}{ 1 - \rho \g^{-1}} \| L_\nu^-(\ell, j, k ; \io^{(1)})^{- 1}\| \leq 
\frac{\gamma}{\gamma - \rho} 
\frac{\langle \ell \rangle^\tau}{\gamma \langle j^2 - k^2 \rangle} = 
\frac{\langle \ell \rangle^\tau}{(\gamma -\rho) \langle j^2 - k^2 \rangle}\,.
\end{align*}
Using the same strategy, one can prove that for any 
$\omega \in \Omega_{\nu + 1}^\gamma(\io^{(1)}) \cap \Omega_o(\io^{(2)})$, 
any $\ell \in \Z^S$ with $|\ell| \leq N_\nu$, and any $j, k \in S_+^\bot$, the operator $L_\nu^+(\ell, j, k; \io^{(2)})$ is invertible and satisfies 
$$
\| L_\nu^+(\ell, j, k; \io^{(2)})^{- 1} \| \leq \frac{\langle \ell \rangle^\tau}{(\gamma - \rho) \langle j^2 - k^2 \rangle}\,.
$$
Altogether, we thus have verified ${\bf (S2)}_{\nu + 1}$. 
\end{proof}


\section{Nash-Moser iteration}\label{sec:NM}

In this section we prove Theorem~\ref{main theorem} except for the 
measure estimate \eqref{measure estimate Omega in Theorem 4.1} which is proved in Section~\ref{sec:measure}. 
Recall that in  \eqref{smoothing Sobolev} we introduced the family of smoothing operators $(\Pi_t)_{t \geq 0}$ for the Sobolev spaces $ H^s(\T^S, X)$. 
By a slight abuse of notation, we define, for $ n \geq 0 $,  
$$
\Pi_n \equiv \Pi_{N_n}\,, \quad \Pi_n^\bot = {\rm Id} - \Pi_n\,,  \quad N_n = N_0^{\chi^n}\,, \quad \chi = 3/2  \, , 
$$ 
with  $N_0= N_0(|S|, \tau) > 0$ as is Theorem \ref{iterazione-non-lineare}.
By Lemma \ref{smoothing sobolev lemma}, 
the classical smoothing properties hold: for any $s \geq 0 $, $k \geq 0 $, 
and  any Lipschitz family $\io \equiv \io_\omega \in H^s(\T^S,\, \T^S \times \R^S \times h^{\s'}_\bot)$
with  $\s' \le \s $, we have 
\begin{equation}\label{smoothing-u1}
\|\Pi_{n} \io \|_{s + k, \s'}^\Lipg 
\leq N_{n}^{k} \| \io \|_{s, \s'}^\Lipg \, ,   
\end{equation}
and for any Lipschitz family 
$\io \equiv \io_\omega \in H^{s + k}(\T^S, \, \T^S \times \R^S \times h^{\s'}_\bot)$
\begin{equation}\label{smoothing-u2}
\|\Pi_{n}^\bot \io \|_{s, \s'}^\Lipg 
\leq N_{n}^{- k} \| \io \|_{s + k , \s'}^\Lipg \, .
\end{equation}
Furthermore, introduce for any $n \geq 0$ 
\[
E_n := \big\{ \vphi \mapsto \io (\vphi) = ( \Theta(\vphi), y(\vphi), z(\vphi) )  : \, 
\Theta = \Pi_n \Theta, \ y = \Pi_n y \in U_0, \ z = \Pi_n z \big\} 
\subseteq {\cal C}^\infty(\T^S, M^\sigma)\,, \quad E_{- 1} := \{ 0 \}
\]
with $M^\sigma = \, \T^S \times U_0 \times h^{\s}_\bot$ introduced in  
\eqref{def:M-sigma}. Recall that in Subsection \ref{Hamiltonian setup}, the differential of a possibly $\vphi$-dependent vector field on $M^\sigma$ has been extended to a linear operator on $\R^S \times \R^S \times h^\sigma_\bot \times h^\sigma_\bot$ -- see formula \eqref{differential XF (2)}. This extension turned out to be useful in 
Sections \ref{3. Set up} - \ref{sec:redu} for the construction of an approximate right inverse of $d_{\io, \zeta} F_\omega$. In the sequel, by a slight abuse of notation, we will identify a possibly $\vphi$-dependent vector $(\widehat \theta, \widehat y, \widehat z) \in \R^S \times \R^S \times h^\sigma_\bot$ with the vector $(\widehat \theta, \widehat y, \widehat z, \overline{\widehat z}) \in \R^S \times \R^S \times h^\sigma_\bot \times h^\sigma_\bot$. 

\noindent
 Define the constants
\begin{alignat}{3} \label{costanti nash moser}
 \eta_1 := 6 \mu_1 + 1\,, \qquad & \alpha_1 :=  2 \mu_1 + \frac23\,, \qquad  & 
\kappa_1 := 6 \mu_1 + 1\,,\qquad &
& \beta_1 := 12 \mu_1 + 2   \qquad &
\end{alignat}
where $\mu_1 = \mu_1(|S|, \tau) > 0$ is the integer of Theorem \ref{thm:stima inverso approssimato}. 
Finally, for any $0 < \g < 1/2$, introduce 
\begin{equation}\label{gamma n nash moser}
\gamma_n := \gamma (1 + 2^{-n})\,, \qquad n \geq 0\,,
\end{equation}
let $0 < \delta_1 < 1$ be as in Theorem \ref{thm:stima inverso approssimato}, and recall that  $ \Omega_{ \gamma, \tau} $ denotes the set of
diophantine frequencies, introduced in \eqref{Omega o Omega gamma tau}. Let  $ N_{- 1} := 1 $.

\begin{theorem}\label{iterazione-non-lineare} 
{\bf (Nash-Moser)} Assume that the perturbation $f $ in \eqref{nonlin:f} 
 is ${\cal C}^{\sigma, s_*}$-smooth  with $ s_* \geq s_0 + \b_1 + \mu_1  $ and 
let $\tau \geq 2 |S| + 1 $. Then there exist 
$0 < \delta_2 = \delta_2(|S|, \tau) \le \delta_1 (<1 )$, 
$N_0 = N_0(|S|, \tau) > 0$, and $C_* \ge 1$ so that 
if $\e>0 $, $0 < \gamma < 1/4$ satisfy
\begin{equation}\label{nash moser smallness condition}  
  \e \gamma^{-4} < \d_2\,, 
\end{equation}
then the following holds: for any $ n \geq 0$, there exists a Lipschitz family 
$(\io_{n+1}, \zeta_{n+1}) :\O^{\rm Mel}_{n+1} \to E_{n} \times \R^S$ where 
\begin{equation}\label{def:Gn+1}
\O^{\rm Mel}_{n+1} :=   \Omega_{\rm Mel}^{2 \gamma_n}(\io_n) \quad \quad
\end{equation}
with $\Omega_{\rm Mel}^{2 \gamma_n}(\io_n)$ defined as in \eqref{cantor for invertibility}, \eqref{Omegainfty} by choosing $\Omega_o(\io_n)$ to be $\O^{\rm Mel}_n$  in the case $n \ge 1$ whereas for $n = 0$
\be\label{def:G-0}
\Omega_o(\io_0) \equiv \O^{\rm Mel}_{0} := \Omega_{4 \gamma, \tau} \quad \text{with} 
 \quad (\io_0, \zeta_0):= (0, 0)
\ee
so that the following estimates are valid for any $ n \geq 0 $: 

\noindent
$({NM}1)_{n}$  \emph{(middle norms)}
\begin{equation}\label{ansatz induttivi nell'iterazione}
\| \io_n \|_{s_0 + \mu_1 }^{\Lipg} 
\lessdot \e \gamma^{-2}\,, \quad 
\| F_\omega(\io_n, \zeta_n)\|_{s_0 + \mu_1 , \s-2}^{\Lipg} \lessdot \e  \,. \quad
\end{equation}
 
The difference 
$ \widehat {\io}_n :=  {\io}_n - {\io}_{n - 1} $ (with $ \widehat \io_0 := 0 $) is defined on 
$\O^{\rm Mel}_n$ and one has, in case $n \ge 1$,
\begin{equation}  \label{Hn}
\| \widehat {\io}_n \|_{ s_0 + \mu_1}^{\Lipg} \lessdot \e \gamma^{-2} N_{n - 1}^{-\alpha_1} \,.
\end{equation}

\noindent
$(NM2)_{n}$  \emph{(low norms)}   $ \| F_\omega(\io_n, \zeta_n) \|_{ s_0, \s-2}^{\Lipg} \leq 
C_* \e N_{n - 1}^{- \eta_1}\,,\quad  | \zeta_n |^\Lipg \leq 
C_* \|F_\omega(\io_n, \zeta_n ) \|_{s_0, \s-2}^\Lipg$\,.

\noindent
$({NM}3)_{n}$ \emph{(high norms)} 
\ $  \| \io_n \|_{ s_{0}+ \beta_1}^{\Lipg} \leq C_* \e \gamma^{-2}  N_{n - 1}^{\kappa_1} $\,, \quad  
$ \|F_\omega(\io_n, \zeta_n ) \|_{ s_{0}+\beta_1, \s-2}^{\Lipg} \leq 
C_* \e   N_{n - 1}^{\kappa_1} $.   

\noindent
In  $({NM}1)_{n} - ({NM}3)_{n}$, the $\Lipg$ norms are defined on $ \O^{\rm Mel}_{n} $, namely  
$\| \,\cdot \, \|_s^{\Lipg} = \| \, \cdot \, \|_{s, \O^{\rm Mel}_n}^{\Lipg} $\,.
\end{theorem}

\begin{proof}
The proof of Theorem \ref{iterazione-non-lineare} 
follows the scheme in \cite{BBM3}. Note however that in contrast to the setup in
\cite{BBM3}, the regularity in the space variable is fixed, meaning that $\s$ in $h^\s_\bot$ is kept unchanged along the iteration. The main ingredient for proving the claimed estimates are the tame estimates of the approximate right inverse ${\bf T}$ of Theorem  \ref{thm:stima inverso approssimato}.  
To shorten notation, we write $\| \, \|$ for  $\| \cdot \|^{\Lipg}$ in this proof. 

\smallskip

\noindent
 \emph{Proof of $({NM}1)_0- (NM3)_0$:} Since $\omega^{nls}(\xi, 0) = \omega$ (by the definition of $\xi = \xi(\omega)$)
and $(\io_0 , \zeta_0) = (0, 0)$ (by definition) one has
$X_{H^{nls}} \circ \breve \io_0 = (\omega^{nls}(\xi, 0), 0, 0)$ (cf \eqref{HS-action-angle}),
and hence by the definition \eqref{operatorF} of $F_\omega$,
 $$
 F_\omega(\io_0 , \zeta_0) = - \e X_P \circ \breve \io_0
 $$
where $X_P$ is the Hamiltonian vector field of the Hamiltonian $P$, expressed in the coordinates
$(\theta, y, z) \in M^\s$. 
By \eqref{signXP} we have 
 $$
 \widetilde X_P = (d \Phi   \widetilde X_{\cal P} )_{|\Phi^{-1}} \, , \quad P = {\cal P} \circ \Phi^{- 1} 
$$
where $\Phi = \Phi^{nls}$ is the Birkhoff map of Theorem \ref{Theorem Birkhoff coordinates} and 
$ \widetilde X_P $ is obtained from $X_{\cal P}$ by expressing it in the Birkhoff coordinates $(w_n)_{n \in \Z}$ 
and then adding the complex conjugate as a second component. 
 In this way one sees that for any $s_0 \le s \le s_* -1$
 $$
 \| X_P \circ \breve \io_0\|_{s, \sigma - 2} \leq_s 1\,.
 $$
 Altogether  we proved that 
 \be \label{estimate F 0}
 \|  F_\omega(\io_0 , \zeta_0) \|_{s, \sigma - 2} \leq_s \e \,.
 \ee
Since $N_{- 1} = 1$ (by definition), one sees that the claimed estimates of $(NM1)_0-(NM3)_0$ hold, once $C_* \equiv C_* (s_0 + \b_1)$ is chosen large enough. 

\smallskip

\noindent
 \emph{Proof of inductive step:} Assume that $(NM1)_n - (NM3)_n$ hold for a given $n \geq 0$. 
 Our task is to prove that $(NM1)_{n + 1}-(NM3)_{n + 1}$ hold as well. First we have to make sure that the smallness assumption \eqref{final smallness condition for approximate inverse} of Theorem \ref{thm:stima inverso approssimato} for $(\io_n, \zeta_n)$ is valid with 
$\Omega_o(\io_n)$ given by $\O^{\rm Mel}_n$. Indeed, since \eqref{ansatz induttivi nell'iterazione}  
is satisfied by the induction hypothesis, \eqref{final smallness condition for approximate inverse} holds by
choosing $\delta_2$ in the statement of the theorem sufficiently small.
Hence Theorem \ref{thm:stima inverso approssimato} applies to $(\io_n, \zeta_n)$: 
 by the definition of $\O^{\rm Mel}_{n + 1}$ in \eqref{def:Gn+1} 
 there exists a family of operators $({\bf T}_n(\omega))_{\omega \in \O^{\rm Mel}_{n + 1}}$ so that the estimates \eqref{stima inverso approssimato 1} hold, 
 \begin{align}\label{stima Tn}
\| {\bf T}_n g  \|_{s,\s} 
& \leq_s \gamma^{-2} \big( \| g \|_{s + \mu_1, \s-2} 
+  \| \io_n \|_{s + \mu_1}  \| g \|_{s_0+ \mu_1, \s-2} \big)  \,, \quad \forall s \in [s_0, s_0 + \beta_1]\,,
\end{align}
implying together with \eqref{ansatz induttivi nell'iterazione} and 
\eqref{nash moser smallness condition}   that
\begin{align}
\| {\bf T}_n  g \|_{s_0, \s} 
& \leq_{s_0} \gamma^{-2} \| g \|_{s_0 + \mu_1, \s-2}\,.   \label{stima Tn norma bassa}
\end{align}
Furthermore, denoting by $L_n$ the differential $d_{\io, \zeta} F_\omega(\io_n, \zeta_n)$,
one has  by \eqref{stima inverso approssimato 2}
for any $s$ in $[s_0, s_0 + \beta_1]$,
\begin{align}
\| \big(L_n & \circ  {\bf T}_n -  {\rm Id} \big)  g \|_{s, \s-2} \leq_s 
\gamma^{-3}   \| F_\omega(\io_n, \zeta_n) \|_{s_0 + \mu_1, \s-2} 
\| g\|_{s + \mu_1, \s-2} \, \,+ \nonumber \\
&  \gamma^{- 3}\| F_\omega(\io_n, \zeta_n) \|_{s + \mu_1, \s-2} 
\|  g \|_{s_0 + \mu_1, \s-2} \, + \,
\gamma^{-3} \| \io_n \|_{s + \mu_1} \| F_\omega(\io_n, \zeta_n) \|_{s_0 + \mu_1, \s-2} 
\| g \|_{s_0 + \mu_1, \s-2 }  \,.
 \label{stima Tn inverso approssimato} 
\end{align}
For $s=s_0$, this yields
$\| \big(L_n \circ  {\bf T}_n -{\rm Id} \big) g  \|_{s_0, \s-2} 
 \leq_{s_0} \gamma^{-3} \| F_\omega(\io_n, \zeta_n)\|_{s_0 + \mu_1, \s-2} \| g \|_{s_0 + \mu_1, \s-2} $. Using that
\begin{align}
\|F_\omega(\io_n, \zeta_n) \|_{s_0 + \mu_1, \sigma - 2} & \leq_s \|\Pi_n  F_\omega(\io_n, \zeta_n)\|_{s_0 + \mu_1, \s-2}+  
\|\Pi_n^\bot  F_\omega(\io_n, \zeta_n)\|_{s_0 + \mu_1,\s-2} \nonumber\\
& \stackrel{\eqref{smoothing-u1}, \eqref{smoothing-u2}} \leq N_n^{\mu_1} \| F_\omega(\io_n, \zeta_n) \|_{s_0, \sigma - 2} + N_n^{\mu_1 - \beta_1} \| F_\omega(\io_n, \zeta_n)\|_{s_0 + \beta_1, \sigma - 2} \label{sostitutivo ansatz nash moser}
\end{align}
the above estimate then leads to
\begin{align}
\label{stima Tn inverso approssimato norma bassa}
\| \big(L_n \circ  {\bf T}_n -{\rm Id} \big) g  \|_{s_0, \s-2} 
 \leq_{s_0} & N_{n }^{\mu_1} \gamma^{-3} \| F_\omega(\io_n, \zeta_n)\|_{s_0, \s-2} 
\|  g  \|_{s_0 + \mu_1, \s-2} \nonumber\\
& + N_{n}^{\mu_1 -\beta_1} \gamma^{- 3} \| F_\omega(\io_n, \zeta_n)\|_{s_0 + \beta_1, \s-2} \|  g  \|_{s_0 + \mu_1, \s-2}\, .
\end{align}

 \medskip

\noindent
For convenience we define ${\cal S}_n := (\io_n, \zeta_n)$. As advertised at the beginning of this section, we identify the vectors $(\widehat \theta, \widehat y, \widehat z) \in \R^S \times \R^S \times h^\sigma_\bot$ and $(\widehat \theta, \widehat y, \widehat z, \overline{\widehat z}) \in \R^S \times \R^S \times h^\sigma_\bot \times h^\sigma_\bot$. With this convention the Taylor expansion up to order $1$ of $F_\omega$ at ${\cal S}_n$, reads 
$$
F_\omega({\cal S}_n + \widehat {\cal S}) =  F_\omega({\cal S}_n) + 
L_n \widehat {\cal S} + Q({\cal S}_n, \widehat {\cal S}) \,,
$$
where $\widehat {\cal S} = ( \widehat \io, \widehat \zeta)$ is assumed to be
a sufficiently small element in  $E_n \times \R^S$ and $Q({\cal S}_n, \widehat {\cal S})$
denotes the Taylor remainder term. 
By the Newton-Nash-Moser iteration scheme, we define
${\cal S}_{n + 1}$ as ${\cal S}_n + \widehat {\cal S}_{n + 1}$ with
$\widehat {\cal S}_{n + 1} :=( \widehat \io_{n+1}, \widehat \zeta_{n+1})$ 
chosen to be an approximate solution of the equation 
$ F_\omega({\cal S}_n) + L_n \widehat {\cal S} = 0$.
More precisely, we define ${\cal S}_{n + 1}$ on $\O^{\rm Mel}_{n + 1}$ by
\begin{equation}\label{soluzioni approssimate}
{\cal S}_{n + 1} := {\cal S}_n + \widehat {\cal S}_{n + 1}\,, \quad 
\widehat {\cal S}_{n + 1} :=
 - {\widetilde \Pi}_{n } {\bf T}_n \Pi_{n } F_\omega({\cal S}_n)  
\end{equation}
where  $ {\wtilde \Pi}_n ( \io , \zeta ) :=   ( \Pi_n \io , \zeta ) $.
Arguing as above and using the induction hypothesis, one verifies  that 
${\cal S}_{n + 1}$ and $\widehat {\cal S}_{n + 1}$ are in  $E_n \times \R^S$.
(We choose $C_1$, $N_0$ sufficiently large and $\delta_2$ sufficiently small.)
Then   
\begin{equation}\label{def:Qn}
F_\omega({\cal S}_{n + 1}) =  F_\omega({\cal S}_n) + L_n \widehat {\cal S}_{n + 1} + 
Q_n\,,\quad Q_n := Q({\cal S}_n, \widehat {\cal S}_{n + 1})\,. 
\end{equation} 
Upon substituting the expression for  $\widehat {\cal S}_{n + 1}$ in  \eqref{soluzioni approssimate} 
and writing  $\widetilde \Pi_n$ as ${\rm Id} - \widetilde \Pi_n^\bot$ with
$ {\wtilde \Pi}_n^\bot (\io, \zeta ) := (\Pi_n^\bot \io, 0) $, the identity \eqref{def:Qn} reads  
\begin{align}
F_\omega({\cal S}_{n + 1}) & 
= F_\omega({\cal S}_n) - L_n  {\bf T}_n \Pi_{n } F_\omega({\cal S}_n) +
 L_n  {\wtilde \Pi}_n^\bot  {\bf T}_n \Pi_{n } F_\omega({\cal S}_n) + Q_n\,. \nonumber
\end{align}
The first two terms in the latter expression are split up by applying ${\rm Id} = \Pi_n + \Pi_n^\bot$, yielding
\begin{equation}\label{relazione algebrica induttiva}
F_\omega({\cal S}_{n + 1}) = \Pi_{n }^\bot F_\omega({\cal S}_n) + R_n + Q_n' + Q_n 
\end{equation}
where
\begin{equation}\label{Rn Q tilde n}
R_n := (L_n  {\wtilde \Pi}_n^\bot -  \Pi_n^\bot L_n) {\bf T}_n \Pi_{n }F_\omega({\cal S}_n) \,,
\qquad 
Q_n' := - \Pi_{n } ( L_n {\bf T}_n - {\rm Id} ) \Pi_{n } F_\omega({\cal S}_n)\,.
\end{equation}
We  estimate the terms $ Q_n,$ $Q_n'$, and $R_n$ separately.
\\[1mm]
{\it Estimate of $ Q_n $:}
By \eqref{operatorF}, $\zeta_n$ appears linearly in $F_\omega({\cal S}_n)$, hence for any 
$\widehat {\cal S} = (\widehat \io, \widehat \zeta) \in E_n \times \R^S$, 
$Q({\cal S}_n, \widehat {\cal S})$ is independent of $\zeta_n$ and $\widehat \zeta$. 
By Lemmata \ref{quadraticPart X H nls}, \ref{quadraticPart X P} and
using   \eqref{smoothing-u1}, \eqref{ansatz induttivi nell'iterazione}
we conclude that
\begin{align}
\label{stima parte quadratica norma alta}
&\| Q({\cal S}_n, \widehat {\cal S})  \|_{s,\s-2}  \leq_s 
\| \widehat \io \|_{s } \|  \widehat \io \|_{s_0 } + \| \io_n \|_{s + 2s_0} 
\|  \widehat \io \|_{s_0}^2\,, \quad \forall s \in [s_0, s_0 + \beta_1]\,,  \\
\label{stima parte quadratica norma bassa}
&\| Q({\cal S}_n, \widehat {\cal S}) \|_{s_0, \s-2}  \leq_{s_0} \| \widehat \io \|_{s_0}^2   \, . 
\end{align}
By the definition of $\widehat {\cal S}_{n + 1}$ in \eqref{soluzioni approssimate},
one gets by using first \eqref{smoothing-u1} and then 
\eqref{stima Tn} together with \eqref{ansatz induttivi nell'iterazione}, 
\ref{nash moser smallness condition},
\begin{align}
\|  \widehat \io_{n + 1} \|_{s_0 + \beta_1} 
& \leq N_n^{\mu_1} \|  \widehat \io_{n + 1} \|_{s_0 + \beta_1 - \mu_1} 
\leq_{s_0 + \beta_1} N_n^{\mu_1} 
\big( \gamma^{-2}\| F_\omega({\cal S}_n) \|_{s_0 + \beta_1, \s-2}  
+ \| \io_n \|_{s_0 + \beta_1}\big)\, ,  \label{H n+1 alta} 
\end{align}
and similarly,     
\begin{align} \label{H n+1 bassa}
\|  \widehat \io_{n + 1}\|_{s_0} 
& \stackrel{\eqref{stima Tn norma bassa}}\lessdot 
\gamma^{-2} \|  \Pi_n F_\omega({\cal S}_n)\|_{s_0 + \mu_1, \s-2} 
\stackrel{\eqref{smoothing-u1}} \lessdot
\gamma^{-2}N_{n}^{ \mu_1} \| F_\omega({\cal S}_n)\|_{s_0, \s-2} 
\quad \text{and} \quad 
\|  \widehat \io_{n + 1}\|_{s_0} 
\stackrel{\eqref{ansatz induttivi nell'iterazione} } \lessdot \e \g^{-2} \, .
\end{align}
Hence the term $ Q_n $, defined in \eqref{def:Qn}, satisfies
 by \eqref{stima parte quadratica norma bassa} and \eqref{H n+1 bassa}
\begin{align}
\| Q_n \|_{s_0, \s-2} 
&  \leq_{s_0}  \gamma^{-4}  N_n^{2 \mu_1  }  
\| F_\omega({\cal S}_n) \|_{s_0, \s-2}^2 \label{Qn norma bassa}
\end{align}
and by \eqref{stima parte quadratica norma alta},
\eqref{H n+1 alta}, \eqref{H n+1 bassa} together with \eqref{ansatz induttivi nell'iterazione} 
\begin{align} 
\| Q_n \|_{s_0 + \beta_1, \s-2} 
& \leq_{s_0 + \beta_1}
 N_n^{ \mu_1 }  \e \gamma^{- 2} \big( \g^{-2} \| F_\omega({\cal S}_n)\|_{s_0 + \beta_1, \s-2}  + 
\| \io_n \|_{s_0 + \beta_1} \big) \, .  \label{Qn norma alta} 
\end{align}
{\it Estimate of $ Q_n' $:}
Using  \eqref{stima Tn inverso approssimato norma bassa} and, respectively,
\eqref{smoothing-u1}, \eqref{stima Tn inverso approssimato}, 
 together with
 \eqref{costanti nash moser}, \eqref{ansatz induttivi nell'iterazione} one verifies that
\begin{align}
\label{Qn' norma bassa}
\| Q_n'\|_{s_0, \s-2} & \leq_{s_0}  N_{n }^{2\mu_1 } \gamma^{- 3}\big(\| F_\omega({\cal S}_n) \|_{s_0, \s-2} + 
N_{n}^{- \beta_1} \|F_\omega({\cal S}_n)\|_{s_0 + \beta_1, \s-2} \big) \| F_\omega({\cal S}_n) \|_{s_0, \s-2}\,, \\
\label{Qn' norma alta}
\| Q_n' \|_{s_0 + \beta_1, \s-2} &  \leq 
N_n^{ \mu_1} \| Q_n' \|_{s_0 + \beta_1 - \mu_1, \s-2}
\leq_{s_0 + \beta_1}
N_n^{ \mu_1} \e \gamma^{- 3} \big(
 \| F_\omega({\cal S}_n) \|_{s_0 + \beta_1, \s-2} + \e  \| \io_n \|_{s_0 + \beta_1}\big) \, .
\end{align}

\noindent
{\it Estimate of $ R_n $:} In a first step we estimate the operator 
$L_n  {\wtilde \Pi}_n^\bot -  \Pi_n^\bot L_n$.
For  $ \widehat {\cal S} := (\widehat \io, \widehat \zeta ) $ we have 
\begin{align}
L_n \widehat {\cal S}  & = \omega \cdot \partial_\vphi \widehat \imath - 
d_\io X_{H_\e}(\io_n)[\widehat \imath] + (0, \widehat \zeta, 0, 0) \nonumber\\
& =\omega \cdot \partial_\vphi \widehat \imath - d_\io X_{H^{nls}}(\io_n)[\widehat \io] - 
\e d_\io X_P(\io_n)[\widehat \io] + (0, \widehat \zeta, 0, 0)\,.
\end{align}
Writing  
$ d_\io X_{H^{nls}}(\io_n) =  d_\io X_{H^{nls}}(\io_0) + \big(d_\io X_{H^{nls}}(\io_n) - d_\io X_{H^{nls}}(\io_0)  \big) $
we get
$$
L_n \widehat {\cal S} = L_n^I \widehat {\cal S} + L_n^{II} \widehat {\cal S} + (0, \widehat \zeta, 0, 0)
$$
where 
$$
\qquad L_n^I \widehat {\cal S} := 
\omega \cdot \partial_\vphi \widehat \io - d_\io X_{H^{nls}}(\io_0) [\widehat \imath]\,, \quad
 L_n^{II} \widehat {\cal S} := \big(d_\io X_{H^{nls}}(\io_n) - d_\io X_{H^{nls}}(\io_0)  \big)
[\widehat \io] + \e d_\io X_P(\io_n)[\widehat \io]\,.
$$
Since
$$
d_\io X_{H^{nls}}(\io_0)[\widehat \io] = \Big(
\big( \sum_{k \in S} \partial_{I_k} \omega_n(\xi, 0) \widehat y_k \big)_{n \in S}\,, \,\, 0\, , \, \,
- \ii \big(\omega_n(\xi, 0)\widehat z_n \big)_{n \in S^\bot}, \,\,
\ii \big(\omega_n(\xi, 0)\widehat{\bar z}_n \big)_{n \in S^\bot} \,\Big)\,,
$$
the 'commutator'  $L_n^I \widetilde \Pi_n^\bot - \Pi_n^\bot L_n^I$ 
vanishes, implying that 
$$
L_n  {\wtilde \Pi}_n^\bot -  \Pi_n^\bot L_n = L_n^{II}  {\wtilde \Pi}_n^\bot -  \Pi_n^\bot L_n^{II}\,.
$$
Using Proposition \ref{teorema stime perturbazione}, Corollary \ref{differenziale X H nls imperturbato}, the smallness condition \eqref{ansatz induttivi nell'iterazione}, and the smoothing properties \eqref{smoothing-u1}, \eqref{smoothing-u2},
it follows that for any $\widehat {\cal S}$ in  $E_n \times \R^S$
\begin{align}\label{stima commutatore modi alti norma bassa}
&\| (L_n  {\wtilde \Pi}_n^\bot -  \Pi_n^\bot L_n) \widehat {\cal S} \|_{s_0, \s-2}  \leq_{s_0+ \b_1} 
 N_{n }^{- \beta_1 + \mu_1} \big(\e \gamma^{- 2} \|  \widehat \io \|_{s_0 + \beta_1} + 
\| \io_n \|_{s_0 + \beta_1} \|  \widehat \io \|_{s_0}\big)\,, \\
\label{stima commutatore modi alti norma alta}
 &\| (L_n  {\wtilde \Pi}_n^\bot -  \Pi_n^\bot L_n) \widehat {\cal S} \|_{s_0 + \beta_1, \s-2} 
\leq_{s_0 + \beta_1} 
 N_n^{\mu_1} \big(\e \gamma^{- 2}\|  \widehat \io \|_{s_0 + \beta_1  } + \| \io_n \|_{s_0 + \beta_1  } 
\|  \widehat \io \|_{s_0 } \big)\,.
\end{align}
Hence, applying 
\eqref{stima Tn}, 
\eqref{stima commutatore modi alti norma bassa}, \eqref{stima commutatore modi alti norma alta},  
\eqref{nash moser smallness condition}, 
\eqref{ansatz induttivi nell'iterazione}, 
\eqref{smoothing-u1}, 
the term $R_n$ defined in \eqref{Rn Q tilde n} satisfies
\begin{align} \label{stima Rn norma bassa}
\| R_n\|_{s_0, \s-2} 
& \leq_{s_0 + \beta_1}  N_n^{ 2 \mu_1  - \beta_1} ( \e \gamma^{-4} \| F_\omega({\cal S}_n)\|_{s_0 + \beta_1, \s-2} +
 \e \gamma^{- 2} \| \io_n  \|_{s_0 + \beta_1} )\,, 
\\
\label{stima Rn norma alta}
\| R_n \|_{s_0 + \beta_1, \s-2} 
& \leq_{s_0 + \beta_1} N_n^{ 2 \mu_1 } ( \e \gamma^{-4} \| F_\omega({\cal S}_n)\|_{s_0 + \beta_1, \s-2} 
+ \e\gamma^{- 2} \| \io_n  \|_{s_0 + \beta_1} )\,.
\end{align}
{\it Estimate of $ F_\omega({\cal S}_{n + 1}) $:} 
By the identity \eqref{relazione algebrica induttiva} and 
the estimates \eqref{Qn norma alta}, 
 \eqref{Qn norma bassa}, \eqref{Qn' norma alta}, \eqref{Qn' norma bassa}, 
\eqref{stima Rn norma bassa}, \eqref{stima Rn norma alta},  \eqref{nash moser smallness condition},  
\eqref{ansatz induttivi nell'iterazione}, 
we get
\begin{align}\label{F(U n+1) norma bassa}
& \| F_\omega({\cal S}_{n + 1})\|_{s_0, \s-2} 
\leq_{s_0 + \beta_1}   N_{n }^{2 \mu_1 - \beta_1} ( \| F_\omega({\cal S}_n)\|_{s_0 + \beta_1, \s-2} 
+  \e \gamma^{- 2} \| \io_n \|_{s_0 + \beta_1} ) 
+ N_n^{2 \mu_1}  \gamma^{-4} \| F_\omega({\cal S}_n)\|_{s_0, \s-2}^2\,, 
\\
\label{F(U n+1) norma alta}
& \| F_\omega({\cal S}_{n + 1}) \|_{s_0 + \beta_1, \s-2} 
\leq_{s_0 + \beta_1}  N_n^{2 \mu_1} ( \| F_\omega({\cal S}_n) \|_{s_0 + \beta_1, \s-2} 
+ \e \gamma^{- 2} \| \io_n \|_{s_0 + \beta_1} ) \,.
\end{align}

\noindent 
{\it Estimate of $ \io_{n+1} $:} 
Using  \eqref{H n+1 alta}  the term $ \io_{n+1} = \io_n + \widehat {\io}_{n+1} $ 
can be estimated as follows:
\begin{equation}\label{U n+1 alta}
\| \io_{n + 1}\|_{s_0 + \beta_1} \leq_{s_0 + \beta_1} 
\| \io_{n }\|_{s_0 + \beta_1}  + \| \widehat \io_{n + 1}\|_{s_0 + \beta_1} \leq_{s_0 + \beta_1} 
N_n^{\mu_1} ( \| \io_n\|_{s_0 + \beta_1} +\gamma^{-2} \| F_\omega({\cal S}_n)\|_{s_0 + \beta_1} )\, . 
\end{equation}

\noindent
\emph{Proof of $(NM3)_{n + 1}$:} 
By \eqref{F(U n+1) norma alta}, $(NM3)_n$ we have 
\begin{align}
\| F_\omega({\cal S}_{n + 1}) &\|_{s_0 + \beta_1}   \leq_{s_0 + \beta_1} N_n^{2 \mu_1} 
\| F_\omega({\cal S}_n) \|_{s_0 + \beta_1, \s-2} 
+ N_n^{2 \mu_1} \e \gamma^{- 2} \| \io_n \|_{s_0 + \beta_1}  \nonumber\\
& \leq_{s_0 + \beta_1} N_n^{2 \mu_1} C_* \e N_{n - 1}^{\kappa_1}  + 
\e \gamma^{- 2}  N_n^{2 \mu_1}  C_* \e \gamma^{- 2}  N_{n - 1}^{\kappa_1} 
\stackrel{\e \gamma^{- 4} \leq 1}{\leq} C(s_0 + \beta_1) C_* \e N_n^{2 \mu_1} N_{n - 1}^{\kappa_1}  \,.
\end{align}
Hence 
$\| F_\omega({\cal S}_{n + 1}) \|_{s_0 + \beta_1} \leq C_* \e N_n^{\kappa_1}$
provided that
$$
N_j^{\kappa_1 - 2 \mu_1} N_{j - 1}^{-\kappa_1} \geq C(s_0 + \beta_1)\,, 
\quad \forall j \geq 0\,,
$$
which is satisfied by choosing $\kappa_1$ as in \eqref{costanti nash moser} and 
$N_0$ sufficiently large. The bound for $\| \io_{n + 1}\|_{s_0 + \beta_1}$ is proved similarly, hence $(NM3)_{n + 1}$ is established.

\noindent
\emph{Proof of $(NM2)_{n + 1}$:} 
By \eqref{F(U n+1) norma bassa}, $(NM2)_n$, $(NM3)_n$, 
and $\e \gamma^{- 4} \leq 1$ 
(cf \eqref{nash moser smallness condition}), one has 
$$
\| F_\omega({\cal S}_{n + 1})\|_{s_0, \sigma - 2}  \leq C(s_0 + \beta_1) 
\big( N_n^{2 \mu_1 - \beta_1} N_{n - 1}^{\kappa_1} C_* \e +
 N_n^{2 \mu_1} N_{n - 1}^{- 2 \eta_1} C_*^2 \e^2 \gamma^{- 4} \big) \,.
$$
Hence $\| F_\omega({\cal S}_{n + 1})\|_{s_0, \sigma - 2} \leq C_* \e N_n^{- \eta_1}$
provided that
$$
C(s_0 + \beta_1) N_j^{2 \mu_1 + \eta_1 - \beta_1} N_{j - 1}^{\kappa_1} \leq \frac12\,, \quad C(s_0 + \beta_1) C_* N_j^{2 \mu_1 + \eta_1} N_{j - 1}^{- 2 \eta_1} \e \gamma^{- 4}  \leq \frac12\,, \quad \forall j \geq 0\,.
$$
The latter conditions are fulfilled by  choosing $\eta_1$, $\beta_1$ 
as in \eqref{costanti nash moser}, $N_0$ sufficiently large and $\delta_2$ in \eqref{nash moser smallness condition} sufficiently small.  
Moreover, the claimed estimate for $\zeta_n$ follows from Lemma \ref{zeta = 0}
(no induction needed).
Altogether, this establishes $(NM2)_{n + 1}$. 

\smallskip

\noindent
{\em Proof of estimate \eqref{Hn}:}
The bound \eqref{Hn} for $  \widehat \io_1 $ 
follows by \eqref{soluzioni approssimate} and \eqref{stima Tn} (for $ s = s_0 + \mu_1 $) 
together with the estimate 
$  \| F_\omega ( {\cal S}_0 ) \|_{s_0 + 2 \mu_1, \sigma - 2}$ 
$\leq_{s_0+2\mu_1 } \e $ of \eqref{estimate F 0}.
Similarly, the bound \eqref{Hn} for $ \widehat \io_{n + 1} $ 
is obtained from \eqref{soluzioni approssimate} and \eqref{stima Tn}
(cf \eqref{H n+1 alta}), 
using \eqref{smoothing-u1} and \eqref{costanti nash moser}. 

\smallskip

\noindent
{\em Proof of estimate \eqref{ansatz induttivi nell'iterazione}: }
It remains to prove the inductive step from $n$ to $n + 1$ of 
\eqref{ansatz induttivi nell'iterazione}. We have
$$
\| \io_{n + 1} \|_{s_0 + \mu_1} 
\leq {\mathop \sum}_{k = 1}^{n + 1} \| \widehat {\io}_k \|_{s_0 + \mu_1} 
\lessdot \e\gamma^{-2} {\mathop \sum}_{k \geq 1} N_{k - 1}^{- \alpha_1} 
\lessdot \e \gamma^{-2}\,.
$$
Finally, to prove the claimed estimate for
$\| F_\omega({\cal S}_{n + 1}) \|_{s_0 + \mu_1, \sigma - 2}$
we write $F_\omega({\cal S}_{n + 1})$ as a sum,
$\Pi_n F_\omega({\cal S}_{n + 1}) + \Pi_n^\bot F_\omega({\cal S}_{n + 1})$,
and then use \eqref{smoothing-u1} to get 
$$
\| F_\omega({\cal S}_{n + 1}) \|_{s_0 + \mu_1, \sigma - 2} 
 \leq  N_n^{\mu_1} \|F_\omega({\cal S}_{n + 1}) \|_{s_0, \sigma - 2} + 
N_{n}^{\mu_1 - \beta_1} \| F_\omega({\cal S}_{n + 1}) \|_{s_0 + \beta_1, \sigma - 2} \,.
$$
By $(NM2)_{n +1}$, $(NM3)_{n +1}$,  and \eqref{costanti nash moser} it then follows that
$$
\| F_\omega({\cal S}_{n + 1}) \|_{s_0 + \mu_1, \sigma - 2} 
 \leq C_* \e N_n^{\mu_1 - \eta_1} + C_* \e N_n^{\mu_1 - \beta_1 + \kappa_1}
\lessdot \e\,,
$$
which is the second inequality in \eqref{ansatz induttivi nell'iterazione} at the step $n + 1$.
This finishes the proof ot the inductive step.
\end{proof}

 Theorem \ref{iterazione-non-lineare} leads in a straightforward way to a
proof of Theorem~\ref{main theorem},
except for the measure estimate \eqref{measure estimate Omega in Theorem 4.1}
which is proved in Section~\ref{sec:measure}.
By 
$ (NM1)_n $ the sequence 
$ (\io_n(\, \cdot \, ; \omega))_{n \geq 0}$ 
converges to $ \io_\omega $ in the norm $\| \ \|_{s_0 + \mu_1}^\Lipg   $, while $ (NM2)_n $ implies that  
$ F_\omega(\io_n, \zeta_n) \to 0 $ and $ \zeta_n  \to 0 $. 
Altogether it then follows that $F_\omega(\io_\omega, 0) = 0$. 
The following corollary implies Theorem~\ref{main theorem}
with $s_*$ chosen as in Theorem \ref{iterazione-non-lineare}, $\mu_2$ given by $\mu_1(|S|, \tau)$ with $\tau = 2 |S| + 1$ (cf Section \ref{sec:measure} for this choice of $\tau$) and $ 0 < \e_0 < 1 $
so that for some $0 < a < 1/4$, $\e_0^{1-4a} < \delta_2$ with $\delta_2$ as
in  Theorem \ref{iterazione-non-lineare} (cf Theorem \ref{measure estimate}).

\begin{corollary}\label{corollario finale nash moser}
{\bf (Invariant torus and linear stability)}
Under the same assumptions as in Theorem \ref{iterazione-non-lineare}, the sequence $(\io_n, \zeta_n)$ converges in the norm $\| \cdot \|_{s_0 + \mu_1}^\Lipg$ on the set 
\begin{equation}\label{definizione cal G infty}
\O^{\rm Mel}_\infty := \bigcap_{n \geq 0} \O^{\rm Mel}_n
\end{equation}
to $(\io, 0)$ with $\io \equiv \io_\omega$, $\omega \in \O^{\rm Mel}_\infty$,  satisfying 
$F_\omega(\io_\omega, 0) = 0$ and $ \| \io \|_{s_0 + \mu_1}^\Lipg \lessdot \e \gamma^{- 2}$. The sets $ \O^{\rm Mel}_n $ are defined in \eqref{def:Gn+1}. 
Furthermore, for any $\omega \in \O^{\rm Mel}_\infty$, the torus 
$\breve \io_\omega(\T^S)$ is linearly stable in the sense of Lyapunov:
linearizing the equation 
$
\partial_t \breve \io - X_{H_\e}(\breve \io) = 0
$ 
at the quasi-periodic solution $t \mapsto \io_\omega(\omega t)$ in the coordinates 
provided in Section~\ref{3. Set up}, one obtains 
\begin{equation}\label{lineare stabilita esplicita}
\begin{cases}
\dot{\widehat \psi} = K_{2, 0}(\omega t ) [\widehat \upsilon] + K_{1,1}(\omega t) [\widehat W] \\
\dot{\widehat \upsilon} = 0 \\
\dot{\widehat W } =  - {\mathbb  J_2} K_{0 ,2}(\omega t) [\widehat W] -
 {\mathbb  J_2} (K_{1, 1}(\omega t ))^t [ \widehat \upsilon]
\end{cases}\qquad {\mathbb  J_2} := \ii \begin{pmatrix}
 0 &  {\rm Id}_\bot \\
-  {\rm Id}_\bot & 0
\end{pmatrix}\,.
\end{equation}
For any initial datum $ (\widehat \upsilon_0 , \widehat W_0 ) $ 
the solution of \eqref{lineare stabilita esplicita} satisfies
\be\label{Lyapunov-stability}
\widehat \upsilon (t) = \widehat \upsilon (0) \, , \forall t \in \R \, , \quad 
\sup_{t \in \R} \| \widehat W(t, \cdot)\|_{h^\sigma_\bot \times h^\sigma_\bot} 
\lessdot \| \widehat W (0) \|_{h^\sigma_\bot \times h^\sigma_\bot} + |\widehat \upsilon_0|\,.
\ee

\end{corollary}
\begin{proof}
It remains to prove that $\breve \io_\omega(\T^S)$ is linearly stable for any 
$\omega \in \O^{\rm Mel}_\infty$. By \eqref{splitting d F d X K} and, since $F_\omega(\io_\omega, 0 ) = 0 $ implies that  
$ G_2 = 0 $ by Lemma \ref{stima G2}, we have 
$$
d_{\io, \zeta} F_\omega ( \io_{\rm iso})[\widehat \imath, \widehat \zeta]  = 
d \Gamma(\breve \io_0) \big( \omega \cdot \partial_\vphi  - 
d_{\io, \zeta} X_{K_{\e, \zeta}}(\breve \io_0)\big) 
[d \Gamma(\breve \io_0)^{- 1}[\widehat \imath], \widehat \zeta]\,. 
$$
Since $\breve \io_\o $ is an isotropic 
torus embedding it coincides with $\breve \io_{\rm iso}$, constructed
in Subsection~\ref{Isotropic torus embeddings} 
(cf \eqref{toro isotropico modificato A}, \eqref{formula divergence free part}).
Furthermore recall that by \eqref{error term approximate inverse 3}, and since $ G_3 = 0 $ by Lemma \ref{stima G3},
we have 
$$
\omega \cdot \partial_\vphi - d_{\io, \zeta} X_{K_{\e, \zeta}}(\breve \io_0) = {\frak T}_\omega
$$
where $\frak T_\om$, when expressed in the coordinates $\psi,$ $\upsilon,$ $W$, 
is given by 
$$
\frak T_\om[\widehat \io, 0] = \big(
\omega \cdot \partial_\vphi \widehat \psi  - K_{2, 0}(\vphi) [\widehat \upsilon] - K_{1,1}(\vphi) [\widehat W] \,, \,\, \,\omega \cdot \partial_\vphi \widehat \upsilon\,, \,\,\,
\omega \cdot \partial_\vphi \widehat W +  {\mathbb J}_2  K_{1, 1}(\vphi)^t [ \widehat \upsilon] + 
{\mathbb J}_2K_{0 ,2}(\vphi) [\widehat W] 
\big).
$$
Then \eqref{lineare stabilita esplicita}  follows.
To prove \eqref{Lyapunov-stability}
recall that the operator
$\frak L_\o = \omega \cdot \partial_\vphi  + {\mathbb  J_2} K_{0,2}(\varphi)$,
introduced in \eqref{definition frak L omega sec 4},
is conjugated to the $\vphi$-independent $2 \times 2 $ block diagonal operator
${\bf L}_\infty(\omega) = \omega \cdot \partial_\vphi {\mathbb I}_2 + {\bf N}_\infty(\omega)$, 
defined in \eqref{bf L infinito esplicito}, \eqref{cal D infinito}, 
$${\frak L}_\omega = \mathtt \Phi_1 \mathtt \Phi_2 \mathtt \Phi_3 \Phi_\infty {\bf L}_\infty \Phi_\infty^{- 1} \mathtt \Phi_3^{- 1} \mathtt \Phi_2^{- 1} \mathtt \Phi_1^{- 1}\,,$$
by the composition of the symplectic transformations 
$\mathtt \Phi_1$, $\mathtt \Phi_2$, $\mathtt \Phi_3$ 
(Section \ref{sec:5}) and $\Phi_\infty$ (Subsection \ref{2 times 2 block diagonalization}).
The equation 
$\dot{\widehat W} =  
- {\mathbb  J_2} K_{0 ,2}(\omega t) [\widehat W] - 
{\mathbb  J_2} (K_{1, 1}(\omega t ))^t [ \widehat \upsilon_0]$
then transforms into 
$$
\dot{\widehat V} = - {\bf N}_\infty(\o) \widehat{V} - g_\infty(\omega t) \,, \qquad  
g_\infty(\omega t) :=
 \big( \Phi_\infty(\omega t)^{- 1} \circ \mathtt \Phi_3(\omega t)^{- 1} \circ 
\mathtt \Phi_2(\omega t)^{- 1} \circ \mathtt \Phi_1(\omega t)^{- 1} \big) 
{\mathbb  J_2} (K_{1, 1}(\omega t ))^t [ \widehat \upsilon_0]
$$
where $\widehat V(t)$ is given by
$ \big( \Phi_\infty(\omega t)^{- 1} \circ \mathtt \Phi_3(\omega t)^{- 1} \circ 
\mathtt \Phi_2(\omega t)^{- 1} \circ \mathtt \Phi_1(\omega t)^{- 1} \big) \widehat W(t)$.
Since the coordinate transformations 
$\mathtt \Phi_1(\omega t)^{-1}, \mathtt \Phi_2(\omega t)^{-1}, \mathtt \Phi_3(\omega t)^{-1}$, 
$\Phi_\infty(\omega t)^{-1}: h^\s_\bot \times h^\s_\bot \to h^\s_\bot \times h^\s_\bot$
(see Sections \ref{sec:5}, \ref{sec:redu}) and the operator
$(K_{1, 1}(\omega t ))^t : \R^S \to h^\s_\bot \times h^\s_\bot$
(see Lemma \ref{lemma:Kapponi vari})  are bounded,  uniformly in $ t $, one has
$$
\sup_{t \in \R} \| g_\infty(\omega t) \|_{h^\sigma_\bot \times h^\sigma_\bot} 
\lessdot  |\widehat \upsilon_0|\, .
$$
By the definition of ${\bf N}_\infty$ in  \eqref{cal D infinito} and 
the estimates provided by \eqref{prima asintotica autovalori} - \eqref{estimates-errors}
 in Theorem  \ref{teoremadiriducibilita}  it then follows by
the method of the variation of constants that the solution of
$\dot{\widehat V} = - {\bf N}_\infty \widehat{V} - g_\infty(\omega t) $
with initial datum ${\widehat V}_0$ satisfies 
$$
\sup_{t \in \R} \| \widehat V(t, \cdot)\|_{h^\sigma_\bot \times h^\sigma_\bot} 
\lessdot \| \widehat V_0 \|_{h^\sigma_\bot \times h^\sigma_\bot} + |\widehat \upsilon_0|\, .
$$
Finally, using that the coordinate transformations $\mathtt \Phi_1(\omega t), \mathtt \Phi_2(\omega t), \mathtt \Phi_3(\omega t)$, 
$\Phi_\infty(\omega t)$ are  bounded operators on $ h^\s_\bot \times h^\s_\bot $,  uniformly in $ t $, 
(see Sections \ref{sec:5}, \ref{sec:redu}), one concludes that the corresponding solution $\widehat W(t)$ of 
$\dot{\widehat W} =  
- {\mathbb  J_2} K_{0 ,2}(\omega t) [\widehat W] - 
{\mathbb  J_2} (K_{1, 1}(\omega t ))^t [ \widehat \upsilon_0]$ satisfies \eqref{Lyapunov-stability}.
\end{proof}
Finally we prove the statement of Remark~\ref{remark size of torus}
saying that for most of the $\omega \in \Omega_{\infty}^{\rm Mel}$, the distance of the embedded torus $\breve \io_\o (\T^S)$ to the
standard torus $\breve \io_0 (\T^S)$ is of the order of $\e \g^{-1}$. 
To state our result more precisely, we introduce the  first order 
Melnikov non resonance conditions for the unperturbed equation
\begin{equation}\label{first mel integrable}
\Omega_{\gamma, \tau}^{nls} := \big\{ \omega \in \Omega : \ 
|\omega \cdot \ell + \omega_k^{nls}(\xi(\omega), 0) | \geq 
\frac{\gamma k^2}{\langle \ell \rangle^\tau} \quad \forall (\ell, k) \in \Z^S \times S^\bot \big\}\,.
\end{equation}
Arguing as in Section \ref{sec:measure} (cf Lemmas \ref{limitazioni indici risonanti}, \ref{stima in misura insiemi risonanti}) one shows that
$ {\rm meas}(\Omega \setminus \Omega_{\gamma, \tau}^{nls}) = O(\gamma)$. 
Then the following holds:
\begin{corollary}\label{stima optimal size torus} {\bf (Size of perturbed torus)}
For any $\omega \in \O^{\rm Mel}_\infty \cap \Omega_{\gamma, \tau}^{nls}$, the torus embedding $\breve \io_\omega(\vphi) = (\theta(\vphi), y(\vphi), z(\vphi))$
of Corollary~\ref{corollario finale nash moser} satisfies 
$$
\| y \|_{s_0}\,,\,\,\, \| z \|_{s_0, \sigma} \, \lessdot \, \e \gamma^{- 1}\,.
$$
\end{corollary}

\begin{proof}
The torus embedding $\breve \io(\vphi) = (\theta(\vphi), y(\vphi), z(\vphi))$ 
of Corollary \ref{corollario finale nash moser} satisfies 
the equation $F_\omega (\io, 0) = 0$.  When written componentwise,
the latter equation reads
\begin{equation}\label{pappa pappa}
\begin{cases}
\omega \cdot \partial_\vphi \theta = 
\omega^{nls}(\xi + y, z \bar z) + \e \nabla_y P(\theta, y, z) \\
\omega \cdot \partial_\vphi y = - \e \nabla_\theta P(\theta, y, z) \\
\ii \omega \cdot \partial_\vphi z_k = \omega_k^{nls}(\xi + y, z \bar z) z_k  + 
\e \partial_{\bar z_k} P(\theta, y, z)\,, \qquad k \in S^\bot\,.
\end{cases}
\end{equation}
Furthermore,  $\io(\vphi) = (\Theta(\vphi), y(\vphi), z(\vphi) )$ with 
$\Theta(\vphi) = \theta(\vphi) - \vphi$ can be estimated as follows
$$
\| \io \|_{s_0 + \mu_1}  = \| \Theta \|_{s_0 + \mu_1} \, + \,
\| y \|_{s_0 + \mu_1}\, + \, \| z \|_{s_0 + \mu_1, \sigma} \lessdot \e \gamma^{- 2}
$$
where $\mu_1$ is the integer given in Theorem \ref{thm:stima inverso approssimato}.
Since $\mu_1$ is larger than the integer $\mu_0$ of Theorem \ref{invertibility of frak L omega} and $\mu_0 =  4s_0 +10 \tau + 7$ one has $\mu_1 \geq 2s_0 + \tau$, implying that
\begin{equation}\label{piccolo ansatz}
\| \io \|_{ s_0 + 2s_0 + \tau} \lessdot \e \gamma^{- 2}\,.
\end{equation}  
{\em Estimate of $\| y \|_{s_0}$:}
Since $\omega \in \O^{\rm Mel}_\infty \subset  \Omega_{\gamma, \tau}$,
the solution $y$ of the equation
$\omega \cdot \partial_\vphi y = - \e \nabla_\theta P(\theta, y, z)$, 
$$
y = - \e (\omega \cdot \partial_\vphi)^{- 1} \nabla_\theta P(\theta, y, z)\,,
$$
can be estimated  as follows
$$
\| y \|_{s_0} \stackrel{Lemma\,\ref{om vphi - 1 lip gamma}}{\leq} 
\, \e \gamma^{- 1} \| \nabla_\theta P(\theta, y, z)\|_{s_0 + \tau} 
\stackrel{Prop.\,\ref{teorema stime perturbazione}\,(i)}{\lessdot}  \,
\e \gamma^{- 1}(1 + \| \io\|_{3 s_0+ \tau}) 
\stackrel{\eqref{piccolo ansatz}, \, \eqref{nash moser smallness condition}}{\lessdot}  
\,\e \gamma^{- 1}\,.
$$

\noindent
{\em Estimate of $\| z \|_{s_0, \sigma}$:}
For any $k \in S^\bot$ write $\omega_k^{nls}(\xi  + y, z \overline z) = a_k^{I} + a_k^{II}$ where
\begin{equation}\label{pappa pappa - 1}
a_k^{I} := \omega_k^{nls}(\xi, 0) \qquad 
a_k^{II} := \omega_k^{nls}(\xi  + y, z \overline z) - \omega_k^{nls}(\xi, 0)
\end{equation}
and define the diagonal operators 
\begin{equation}\label{pappa pappa 0}
A^{I} := {\rm diag}_{k \in S^\bot} \, a_k^{I} \,, \qquad 
A^{II} := {\rm diag}_{k \in S^\bot} \, a_k^{II}\,.
\end{equation}
The third equation in \eqref{pappa pappa} can then be rewritten as 
\begin{equation}\label{pappa pappa 1}
B z = A^{II}z + \e \nabla_{\bar z} P(\theta, y, z)\,, \qquad 
B := \ii \omega \cdot \partial_\vphi {\rm Id}_\bot  - A^I\,.
\end{equation}
Since by assumption $\omega \in \Omega_{\gamma, \tau}^{nls}$, the diagonal operator $B$ is invertible and for any 
$g \in H^{s + \tau}(\T^S, \, h^{\sigma -2 }_\bot)$ one has 
$\| B^{- 1} g \|_{s, \sigma} \leq \gamma^{- 1}  \| g \|_{s + \tau, \sigma -2}$. 
Furthermore, the identity \eqref{pappa pappa 1} leads to 
\begin{equation}\label{pappa pappa 8}
z = B^{- 1}A^{II} z + \e B^{- 1} \nabla_{\bar z} P(\theta, y, z)\,.
\end{equation}
The latter two terms are estimated individually:  
\begin{align}
\|  B^{- 1} A^{II} z \|_{s_0, \sigma } & \, \lessdot \,
\gamma^{- 1} \| A^{II} z\|_{s_0 + \tau, \sigma } \,
\stackrel{\eqref{pappa pappa - 1}, \eqref{pappa pappa 0}, \eqref{tame estimates for omega}}{\lessdot} \, \gamma^{- 1} \| \io \|_{3 s_0 + \tau, \sigma } \| z \|_{s_0 + \tau, \sigma} \nonumber\\
&  \stackrel{\eqref{piccolo ansatz}}{\lessdot}\e^2 \gamma^{- 5} \, \lessdot \,
(\e \gamma^{- 1})(\e \gamma^{- 4}) \,
\stackrel{\eqref{nash moser smallness condition}}{\lessdot} 
\, \e \gamma^{- 1} \,.\label{pappa pappa 10}  
\end{align}
The second term on the right hand side of \eqref{pappa pappa 8} can be estimated as 
\begin{align}
\e \| B^{- 1}  \nabla_{\bar z}P(\theta, y, z) \|_{s_0, \sigma } \,  & \lessdot \, \e \gamma^{- 1}  \| \nabla_{\bar z}P(\theta, y, z) \|_{s_0 + \tau, \sigma} \, \stackrel{Prop. \,\ref{teorema stime perturbazione}\,(i)}{\lessdot} \, \e \gamma^{- 1} (1 + \| \io\|_{3 s_0 + \tau }) 
\nonumber\\
&  \stackrel{\eqref{piccolo ansatz}, \eqref{nash moser smallness condition}}{\lessdot} 
\, \e \gamma^{- 1}\,. \label{pappa pappa 11}
\end{align}
The identity \eqref{pappa pappa 8} 
and the estimates \eqref{pappa pappa 10}, \eqref{pappa pappa 11} then yield 
$\| z \|_{s_0, \sigma} \lessdot \e \gamma^{- 1}$. 
\end{proof}


\section{Measure estimate}\label{sec:measure}

The goal of this section is to prove the measure estimate of Theorem \ref{main theorem}.

\begin{theorem}{\bf (Measure estimate)} \label{measure estimate}
Let $ \tau := 2 |S| + 1 $. 
Assume the smallness condition \eqref{nash moser smallness condition} hold with $ \e  $, $ \gamma $ satisfying 
\be\label{prop-def}
0 < \e^{\frak a} <  \frac{1}{64} \, , \quad 
0 < {\frak a} < 1/4 \, , \quad  \gamma = \e^{\frak a}  \, .
\ee
Then there exists $ 0 < \frak b  \le 1/2 $ so that 
the set $\Omega_\e := \O^{\rm Mel}_\infty$ (cf \eqref{definizione cal G infty}), satisfies 
\be\label{estimate-meas-final}
{\rm meas}\big( \Omega \setminus \Omega_\e \big)  = 
O(\e^{\frak a \frak b})\,, \quad \text{as} \quad \e \to 0\,.
\ee
\end{theorem}

The remaining part of this section is devoted to the proof of Theorem \ref{measure estimate}. 
We first choose 
\be\label{def:gamma}
\gamma_* := \gamma^{1/2} = \e^{{\frak a}/2}, \quad 
\tau_* := |S| + 1 \, . 
\ee
Note that, by \eqref{prop-def}, we have $ 8 \gamma  < \gamma_* < 1 $.
Then we consider the set of diophantine frequencies (cf \eqref{Omega o Omega gamma tau})
\begin{equation}\label{diofanteo ausiliare}
\Omega_{\gamma_*, \tau_*} = \big\{ \omega \in \Omega : 
|\omega \cdot \ell| \geq \frac{\gamma_*}{|\ell|^{\tau_*}}\,, 
\quad \forall \ell \in \Z^S \setminus \{ 0 \} \big\}\, . 
\end{equation}
To estimate the Lebesgue measure of the set $\Omega \setminus \O^{\rm Mel}_\infty$, note that
\begin{equation}\label{zero decomposizione risonante}
\Omega \setminus \O^{\rm Mel}_\infty \subseteq (\Omega \setminus \Omega_{\gamma_*, \tau_*}) \cup (\Omega_{\gamma_*, \tau_*} \cap \Omega \setminus \O^{\rm Mel}_\infty)\,.
\end{equation}
Since $\Omega$ is compact and $\tau_* = |S| + 1$, one verifies by a standard  estimate that
\begin{equation}\label{primo pezzo misura}
{\rm meas} (\Omega \setminus \Omega_{\gamma_*, \tau_*} ) = O(\gamma_*)
\stackrel{\eqref{def:gamma}}  = O(\e^{{\frak a}/2}) \, . 
\end{equation}
To deduce  Theorem \ref{measure estimate} 
it thus remains to prove that the measure of 
$ (\Omega \setminus \O^{\rm Mel}_\infty) \cap \Omega_{\gamma_*, \tau_*}$ satisfies 
the estimate  \eqref{estimate-meas-final}. 
Recall that by \eqref{definizione cal G infty},  
$\O^{\rm Mel}_\infty = \cap_{n \geq 0} \O^{\rm Mel}_n$ where,  
according to \eqref{def:Gn+1}-\eqref{def:G-0},  
the sequence of subsets
$(\O^{\rm Mel}_n)_{n \ge 0}$ is defined inductively by
\be\label{Melnikov condition in NM A}
\O^{\rm Mel}_0 = \Omega_{2 \gamma_0, \tau}\,, \quad \text{and} \quad 
\O^{\rm Mel}_{n + 1} = \Omega^{2 \gamma_n}_{\rm Mel}(\io_n)\,, \,\,\, n \ge 0.
\ee 
Here $\gamma_n = \gamma (1 + 2^{- n})$ (hence $\gamma_0 = 2\g$) and 
$\Omega_{\rm Mel}^{2 \gamma_n}(\io_n)$ is defined by \eqref{cantor for invertibility},
\eqref{Omegainfty},
\be\label{Melnikov condition in NM B}
\Omega_{\rm Mel}^{2 \gamma_n}(\io_n) = \big\{ \omega \in \O^{\rm Mel}_n :
 ({{\bf M}^{I}_{2 \gamma_n}})_\infty\,,\,({{\bf M}_{+, 2 \gamma_n}^{II}})_\infty\,,\,({{\bf M}_{-, 2 \gamma_n}^{II}})_\infty\,\, \text{hold} \big\}\,.
\ee
According to \eqref{prime melnikov off diagonali finali matrici},
\eqref{seconde melnikov off diagonali finali matrici}, and
\eqref{seconde melnikov diagonali finali matrici}
the Melnikov conditions 
$({{\bf M}^{I}_{2 \gamma_n}})_\infty$, $({{\bf M}_{+, 2 \gamma_n}^{II}})_\infty$, and $({{\bf M}_{-, 2 \gamma_n}^{II}})_\infty$ 
 for the Lipschitz family $\io_n \equiv \io_n(\, \cdot \,; \omega)$, \, 
$\omega \in \O^{\rm Mel}_n$, are defined as follows:

\noindent
{$ ({{\bf M}^{I}_{2 \gamma_n}})_\infty$} For any $\ell \in \Z^S$, $j \in S_+^\bot$, the linear operator
\begin{equation}\label{matrice prime Melnikov stime in misura}
{ A}_\infty(\ell, j ; \, \omega , \io_n(\omega)) := 
\omega \cdot \ell \, {\rm Id}_2 + [{\bf N}_\infty^{(1)}(\omega, \io_n(\omega))]_j^j\,,
\end{equation}
acting on the vector space $\C^2$ (cf Lemma \ref{final blocks normal form}), is invertible and 
\begin{equation}\label{prime di Melnikov stime in misura}
\| {A}_\infty(\ell, j ; \, \omega , \io_n(\omega))^{- 1}\| \leq 
\frac{\langle \ell \rangle^\tau}{2 \gamma_n  \langle j \rangle^2}\,.
\end{equation}

\noindent
{$({{\bf M}_{+, 2 \gamma_n}^{II}})_\infty$} For any $\ell \in \Z^S$, $j, k \in S_+^\bot$,
the linear operator 
\begin{equation}\label{definizione L infinito + seconde Melnikov stime misura}
L_\infty^{+}(\ell, j, k; \, \omega, \io_n(\omega)) :=  \omega \cdot \ell \,\, {\rm Id}_{\C^{2 \times 2}} + M_L([{\bf N}_\infty^{(1)}(\omega, \io_n(\omega))]_j^j) + M_R([\overline{\bf N}_\infty^{(1)}(\omega, \io_n(\omega))]_k^k)\,,
\end{equation}
acting on the vector space $\C^{2 \times 2}$ of $2 \times 2$ matrices
(cf \eqref{definizione L infinito + seconde Melnikov}), is invertible and 
\begin{equation}\label{seconde melnikov off diagonali finali matrici misura}
\|L_\infty^+(\ell, j, k; \, \omega , \io_n(\omega))^{- 1} \| \leq \frac{\langle \ell\rangle^\tau}{2 \gamma_n \langle j^2 + k^2 \rangle}\,.
\end{equation}

\noindent
{ $({{\bf M}_{-, 2 \gamma_n}^{II}})_\infty$} For any $\ell \in \Z^S$, $j, k \in S_+^\bot$ with 
$(\ell,j,k) \neq (0, j, j)$, the linear operator
\begin{equation}\label{definizione L infinito - seconde Melnikov stime misura}
L_\infty^{-}(\ell, j, k; \, \omega, \io_n(\omega)) :=  \omega \cdot \ell \,\, {\rm Id}_{\C^{2 \times 2}} + M_L([{\bf N}_\infty^{(1)}(\omega, \io_n(\omega))]_j^j) - M_R([{\bf N}_\infty^{(1)}(\omega, \io_n(\omega))]_k^k)\,,
\end{equation}
acting on the vector space $\C^{2 \times 2}$ of $2 \times 2$ matrices
(cf \eqref{definizione L infinito - seconde Melnikov}), is invertible and 
\begin{equation}\label{seconde melnikov diagonali finali matrici misura}
\| L_\infty^-(\ell, j, k; \, \omega, \io_n(\omega))^{- 1} \| \leq \frac{\langle \ell \rangle^\tau}{2 \gamma_n \langle j^2 - k^2 \rangle}\,.
\end{equation}
Since the sequence $\O^{\rm Mel}_n$, $n \ge 0$, is decreasing,
$(\Omega \setminus \O^{\rm Mel}_\infty )\cap \Omega_{\gamma_*, \tau_*}$
can be written as a disjoint union,
\begin{equation}\label{secondo pezzo misura}
(\Omega \setminus \O^{\rm Mel}_\infty )\cap \Omega_{\gamma_*, \tau_*} =  
\Big( \big( \Omega \setminus \O^{\rm Mel}_0 \big) \cap 
\Omega_{\gamma_*, \tau_*} \Big) \cap 
\Big(  \bigcup_{n \geq 0} \big(\O^{\rm Mel}_{n} \setminus \O^{\rm Mel}_{n + 1}  \big) 
\cap \Omega_{\gamma_*, \tau_*} \Big)\,.
\end{equation}
Since $\O^{\rm Mel}_0 = \Omega_{4 \gamma, \tau}$, we have, by a standard  estimate, 
\begin{equation}\label{terzo pezzo misura}
{\rm meas} \big( \Omega \setminus \O^{\rm Mel}_0 \big) 
= O(\gamma)\, . 
\end{equation} 
To estimate the measure of 
$ (\O^{\rm Mel}_n \setminus \O^{\rm Mel}_{n + 1}) \cap \Omega_{\gamma_*, \tau_*}$, write
\begin{equation}\label{espansione risonanti}
\big(\O^{\rm Mel}_{n} \setminus \O^{\rm Mel}_{n + 1}  \big) 
\cap \Omega_{\gamma_*, \tau_*} = 
\Big( \bigcup_{\begin{subarray}{c} \ell \in \Z^S \\ j \in S^\bot_+\end{subarray}} 
Q_{\ell j}(\io_n) \Big) 
\cup 
\Big( \bigcup_{\begin{subarray}{c} \ell \in \Z^S \\ j, k \in S^\bot_+\end{subarray}}
R_{\ell j k}^{+}(\io_n) \Big)
\cup 
\Big( \bigcup_{\begin{subarray}{c} \ell \in \Z^S, \,  j, k \in S^\bot_+ \\  (\ell, j, k) \neq (0, j, j) 
\end{subarray}} R_{\ell j k}^{-}(\io_n) \Big)
\end{equation}
where, by \eqref{prime di Melnikov stime in misura}, 
\eqref{seconde melnikov off diagonali finali matrici misura}, \eqref{seconde melnikov diagonali finali matrici misura}, 
for any $\ell \in \Z^S$, $j, k$ in $S^\bot_+$, and $n \ge 0$,
\begin{align}
\label{risonanti prime Melnikov} Q_{\ell j}(\io_n) & := \Big\{ \omega \in \O^{\rm Mel}_n 
\cap \Omega_{\gamma_*, \tau_*} : \,\, \text{ either \,\,}
\, {A}_\infty(\ell, j ; \omega , \io_n(\omega))\,\, \text{ not\,\,invertible\,\,\,\, or} \\
& \qquad {A}_\infty(\ell, j ; \omega , \io_n(\omega))
\text{\,invertible\,\,and}\, \| { A}_\infty(\ell, j ; \omega , \io_n(\omega))^{- 1}\| > 
\frac{\langle \ell \rangle^\tau}{2 \gamma_n \langle j \rangle^2}  \Big\}\,, 
\nonumber \\
 \label{risonanti seconde di melnikov somma}
R_{\ell j k}^+(\io_n) & := \Big\{ \omega \in \O^{\rm Mel}_n \cap \Omega_{\gamma_*, \tau_*} : 
\,\, \text{ either \,\,}
 \, {L}_\infty^+(\ell, j, k ; \omega , \io_n(\omega))\,\, \text{ not\,\,invertible\,\,\,\,or} \\
& \qquad  {L}_\infty^+(\ell, j, k ; \omega , \io_n(\omega)) \text{\,invertible\,\,and}\, 
\| { L}_\infty^+(\ell, j, k ; \omega , \io_n(\omega))^{- 1}\| > 
\frac{\langle \ell \rangle^\tau}{2\gamma_n \langle j^2 + k^2 \rangle}  \Big\}\,, 
\nonumber \\
 \label{risonanti seconde di melnikov differenza}
R_{\ell j k}^-(\io_n) & := \Big\{ \omega \in \O^{\rm Mel}_n \cap \Omega_{\gamma_*, \tau_*} : 
\,\, \text{ either \,\,}
 \, {L}_\infty^-(\ell, j, k ; \omega , \io_n(\omega))\,\, \text{\, not\,\,invertible\,\,\,\, or} \\
& \qquad 
{L}_\infty^-(\ell, j, k ; \omega , \io_n(\omega))\text{\,\,invertible\,\,and}\, \| { L}_\infty^-(\ell, j, k ; \omega , \io_n(\omega))^{- 1}\| > \frac{\langle \ell \rangle^\tau}{2\gamma_n \langle j^2 - k^2 \rangle}  \Big\}\,.
\nonumber
\end{align}
Actually many of the subsets in \eqref{espansione risonanti} turn out to be empty due to the overlapping 
of $ \O^{\rm Mel}_n $ and $ \O^{\rm Mel}_{n+1} $. 
In order to show this we first prove that the eigenvalues of the normal form $ {\bf N}_\infty^{(1)} $
(cf Lemma \ref{final blocks normal form})
evaluated at two consecutive approximate solutions $ \breve \io_{n}, \breve \io_{n - 1} $
are very close to each other. 

\begin{lemma}
For any $n \geq 1$, 
\begin{equation}\label{marco}
\sup_{j \in S_+^\bot} \big\|   [{\bf N}_\infty^{(1)}( \io_n) - 
{\bf N}_\infty^{(1)}(\io_{n - 1})]_j^j  \big\|
\lessdot \e \gamma^{- 2} N_{n - 1}^{-\a} \, , \quad \forall \omega \in \O^{\rm Mel}_n \, , 
\end{equation}
where  $ \a = 6 \t + 4$ (cf \eqref{alpha beta})
and $[{\bf N}_\infty^{(1)} (\io_n) ]_j^j$ is a short for $ [{\bf N}_\infty^{(1)}(\omega, \io_n(\omega)) ]_j^j$.
\end{lemma}
\begin{proof}
We first task is to 
 show that ${\bf (S2)_{\nu}}$ of Theorem \ref{teorema variazione autovalori}  
with ($\nu,$ $\g$, $\rho$,  $\io^{(1)}$, $\io^{(2)}$) given by 
($n$,  $\g_{n-1} $, $\g 2^{-n} $, $\io_{n-1}$, $\io_n $), applies. 
Since  $ \rho =  \g2^{-n} < \g_{n-1} /2$ and $\g_{n-1} - \rho = \g_n$
it means that
 \be \label{Theorem 7.3 applied}
\Omega_{\nu}^{\g_{n-1}} ( \io_{n-1}) \cap \O^{\rm Mel}_{n} \subseteq 
\Omega_{\nu}^{\g_n} ( \io_n )\,, \quad \forall \, \nu \ge 0\,.
\ee
Since $n \ge 1$ one has  by \eqref{Melnikov condition in NM A}
$\O^{\rm Mel}_{n} = \Omega^{2 \gamma_{n-1}}_{\rm Mel}(\io_{n-1})$ and 
from \eqref{Melnikov condition in NM B} and 
 Lemma \ref{inclusion of cantor sets} one concludes that
$$
\Omega^{2 \gamma_{n-1}}_{\rm Mel}(\io_{n-1}) \subseteq
\Omega_\infty^{2 \gamma_{n-1}}(\io_{n-1}) \subseteq 
\cap_{\nu \geq 0} \Omega_\nu^{\gamma_{n-1}}(\io_{n-1})\,.
$$
In particular, one has $\O^{\rm Mel}_{n} \subseteq \Omega_{n}^{\g_{n-1}} ( \io_{n-1})$
and hence for $\nu = n$, the inclusion  \eqref{Theorem 7.3 applied} becomes
\be\label{primoste}
\O^{\rm Mel}_{n} \subseteq 
\Omega_{n}^{\g_{n-1}} ( \io_{n-1}) \cap \Omega_{n}^{\g_n} ( \io_n ) \, . 
\ee
To justify that ${\bf (S2)_{\nu}}$ of Theorem \ref{teorema variazione autovalori}  
in the situation above applies it remains to verify 
the smallness condition in \eqref{legno} of 
Theorem \ref{teorema variazione autovalori}: To see it, recall that
$\bar \mu = 4s_0 + 2\tau +1$ 
(cf \eqref{perdita mu dopo prime trasformazioni}),
$\b = 6\tau + 5$ (cf \eqref{alpha beta}),
$\mu_0 = 4s_0 + 10 \tau + 7$
(cf remark after Theorem \ref{invertibility of frak L omega}),
and $\mu_0 < \mu_1$ (cf Theorem \ref{thm:stima inverso approssimato}).
Therefore $s_0 + \bar \mu + \beta < s_0 + \mu_0 < s_0 + \mu_1$
 and in turn
$\Vert \io_n - \io_{n - 1} \Vert_{s_0 + \bar \mu + \beta} 
\leq \Vert \io_n - \io_{n - 1} \Vert_{ s_0 + \mu_1} $.
Furthermore, by \eqref{Hn} 
$$
\Vert \io_n - \io_{n - 1}\Vert_{ s_0 + \mu_1} \lessdot  N_{n - 1}^{- \alpha_1}  \e \g^{-2} \,.
$$
Since $\alpha_1 = 2\mu_1 + 2/3 > \tau $ (cf \eqref{costanti nash moser}) one has
$N_{n - 1}^{\tau} N_{n - 1}^{- \alpha_1} \le 1$.
Altogether we proved that for some $C' > 0$,
$
 C_{\rm var}' N_{n - 1}^{\tau} \Vert \io_n - \io_{n - 1} \Vert_{s_0 + \bar \mu + \beta} 
\leq C'   \e \g^{-2} 
$
implying that
$$
 C_{\rm var}' N_{n - 1}^{\tau} \Vert \io_n - \io_{n - 1} \Vert_{s_0 + \bar \mu + \beta} 
 \leq  \g 2^{-n } = \rho
$$
for $\e \gamma^{-3}$ small enough. Hence the smallness condition in \eqref{legno} 
is satisfied and therefore \eqref{primoste} holds.

Since by \eqref{primoste}
$\O^{\rm Mel}_n  \subset \Omega_{n}^{\g_{n-1}}(\io_{n-1}) \cap \Omega_{n}^{\g_n}(\io_n)$
the $2 \times 2$ matrices $ [{\bf N}_{n}^{(1)} (\io_{n-1})]_j^j $ and $ [{\bf N}_{n}^{(1)} (\io_{n})]_j^j $ are defined for any $ \omega \in \O^{\rm Mel}_n$, and
by the estimate \eqref{Delta12 rj} of Theorem \ref{teorema variazione autovalori}
 with $ \nu = n $  one has
\be\label{vicin+1}
\sup_{j \in S_+^\bot} \big\| 
\big[{\bf N}_{n}^{(1)}(\io_n) - {\bf N}_{n }^{(1)}(\io_{n - 1}) \big]_j^j \big\| \stackrel{\eqref{Delta12 rj}} 
\lessdot   \| \io_{n} - \io_{n - 1} \|_{s_0 + \bar \mu + \beta} \lessdot \| \io_{n} - \io_{n - 1} \|_{s_0 + \mu_1} \, .
\ee 
Moreover \eqref{stime blocchi 2 per 2 finali}  (with $ \nu = n $) and \eqref{asdf} imply that for any $j \in S_+^\bot$ 
\begin{eqnarray}\label{diffrkn}
 \big\| \big[{\bf N}_\infty^{(1)}(\io_{n - 1}) - {\bf N}_{n}^{(1)}(\io_{n - 1}) \big]_j^j \big\| \,,  \, \,
\,\, \big\| \big[{\bf N}_\infty^{(1)}(\io_{n}) - {\bf N}_{n}^{(1)}(\io_{n}) \big]_j^j \big\| & \lessdot &
\e  N_{n - 1}^{- \alpha} \,.
\end{eqnarray}
Since  
$\big\| \big[{\bf N}_\infty^{(1)}(\io_n) - {\bf N}_\infty^{(1)}(\io_{n - 1}) \big]_j^j \big\|$
is bounded by
$$
\big\| \big[{\bf N}_{n }^{(1)}(\io_n) - {\bf N}_{n }^{(1)}(\io_{n - 1}) \big]_j^j \big\| 
+ \big\| \big[{\bf N}_\infty^{(1)}(\io_{n - 1}) - {\bf N}_{n }^{(1)}(\io_{n - 1}) \big]_j^j \big\| 
+ \big\| \big[{\bf N}_\infty^{(1)}(\io_{n}) - {\bf N}_{n }^{(1)}(\io_{n}) \big]_j^j \big\| 
$$
one then concludes that  for any $\omega \in \O^{\rm Mel}_{n}$ and any  $ j \in S_+^\bot$,
$$
\big\| \big[{\bf N}_\infty^{(1)}(\io_n) - {\bf N}_\infty^{(1)}(\io_{n - 1}) \big]_j^j \big\|
 \stackrel{\eqref{vicin+1}, \eqref{diffrkn}} \lessdot
 \| \io_n - \io_{n - 1} \|_{ s_0+ \mu_1}  + \e  N_{n - 1}^{-\a}
\stackrel{\eqref{Hn}} \lessdot  \e \g^{-2}  N_{n - 1}^{-\a}  
$$
where for the latter inequality we used that $ \alpha_1 > \a $ since
$\alpha_1 = 2\mu_1 + 2/3$ and $\mu_1 > \bar \mu + \a$
 (cf \eqref{costanti nash moser},  \eqref{alpha beta}). 
The claimed estimate \eqref{marco} is thus established. 
\end{proof}

\begin{lemma}\label{risonanti-1}
For  $ \e  \g^{-4}$ small enough one has for any $n \geq 1$, 
$\ell \in \Z^S$ with $|\ell| \leq N_{n - 1}$, and  $j, k \in S_+^\bot$,
\begin{equation}\label{inclusione-1}
Q_{\ell j}(\io_n) = \emptyset\,, \qquad 
R_{\ell jk}^{+}(\io_{n}) = \emptyset \,,
\end{equation}
and, if in addition $(\ell,j,k) \neq (0, j, j)$, 
\be\label{inclusione-1 for R-}
R_{\ell jk}^{-}(\io_{n}) = \emptyset  \, . 
\ee
\end{lemma}

\begin{proof} Since the proofs of the three stated inclusions are similar
we only prove \eqref{inclusione-1 for R-}. 
For any $n \ge 1$,
$\ell \in \Z^S$ with $|\ell| \leq N_{n - 1}$,
$ j, k \in S_+^\bot$ with $(\ell,j,k) \neq (0, j, j)$, and $\omega \in \O^{\rm Mel}_n$,  
the operator $L_\infty^-(\ell, j, k; \io_{n - 1})$ is invertible and hence we can write
\begin{align*}
L_\infty^-(\ell, j, k; \io_n) & = L_\infty^-(\ell, j, k; \io_{n - 1}) \,
\big( {\rm Id}_{\C^{2 \times 2}} + L_\infty^-(\ell, j, k; \io_{n - 1})^{- 1} \Delta_\infty(j, k, n) \big)
\end{align*}
where 
$$
\Delta_\infty(j, k, n) := 
M_L \big( [{\bf N}_\infty^{(1)}( \io_n) - {\bf N}_\infty^{(1)}(\io_{n - 1})]_j^j \big) - 
M_R \big(  [{\bf N}_\infty^{(1)}(\io_n) - {\bf N}_\infty^{(1)}(\io_{n - 1})]_k^k \big)\,.
$$
Since
\begin{align*}
\big\| L_\infty^-(\ell, j, k; \io_{n - 1})^{- 1} \Delta_\infty(j, k, n) \big\| & \, 
\stackrel{\eqref{seconde melnikov diagonali finali matrici misura}}\leq \,
\frac{\langle \ell \rangle^\tau}{2 \gamma_{n - 1} \langle j^2 - k^2 \rangle} 
\| \Delta_\infty(j, k, n) \| 
\, \stackrel{\eqref{marco}}{\leq} \, C  \e \gamma^{- 3} 
\langle \ell \rangle^\tau N_{n - 1}^{ - \alpha} 
\end{align*}
and $|\ell| \leq N_{n - 1}$ (by assumption), $\alpha > \tau$ (cf \eqref{alpha beta}) 
 it follows that for $\e \gamma^{- 3}$ small enough,
$$
\big\| L_\infty^-(\ell, j, k; \io_{n - 1})^{- 1} \Delta_\infty(j, k, n) \big\|\leq  1/2 \,.
$$
Therefore $L_\infty^-(\ell, j, k; \io_n)$ is invertible by a Neumann series and 
\begin{align*}
\| L_\infty^-(\ell, j, k; \io_n)^{- 1} \| & \leq  
\| L_\infty^-(\ell, j, k; \io_{n - 1})^{- 1} \| 
\big(1 + C \e \gamma^{- 3} N_{n - 1}^{\tau - \alpha} \big) \,
\stackrel{\eqref{seconde melnikov diagonali finali matrici misura}}
\leq \, \frac{\langle \ell \rangle^\tau}{2 \gamma_{n - 1}\langle j^2 - k^2 \rangle} 
\big( 1 + C \e \gamma^{- 3} N_{n - 1}^{\tau - \alpha} \big)\,.
\end{align*}
Choosing $\e \gamma^{- 3}$ sufficiently small one achieves that
$C \e \gamma^{- 3} N_{n - 1}^{\tau - \alpha} \leq \frac{1}{1 + 2^n}$
for any $n\ge 1$.
Since by the definition of $\gamma_n$, 
$\frac{\gamma_{n - 1} - \gamma_n}{\gamma_n} = \frac{1}{1 + 2^n}$
it then follows that
$$
\| L_\infty^-(\ell, j, k; \io_n)^{- 1} \| \leq \frac{\langle \ell \rangle^\tau}{2 \gamma_{n}\langle j^2 - k^2 \rangle}\,.
$$
Hence, recalling \eqref{risonanti seconde di melnikov differenza}, we have proved that $R_{\ell jk}^{-}(\io_{n}) = \emptyset$.  
\end{proof}

As an immediate consequence of Lemma \ref{risonanti-1}, one gets the following
\begin{corollary}\label{excisione parametri io}
For any $n \geq 1$,
\be\label{parametri cattivi}
\big( \O^{\rm Mel}_{n} \setminus\O^{\rm Mel}_{n+1} \big) \cap \Omega_{\gamma_*, \tau_*} \, \stackrel{\eqref{espansione risonanti}}  = \,  
\Big( \bigcup_{\begin{subarray}{c}
|\ell| > N_{n - 1} \\
 j \in S^\bot_+
\end{subarray}} Q_{\ell j}(\io_n) \Big) \cup 
\Big( \bigcup_{\begin{subarray}{c}
|\ell| > N_{n - 1}  \\
j, k \in S^\bot_+
\end{subarray}}R_{\ell j k}^{+}(\io_n) \Big) \cup
\Big( \bigcup_{\begin{subarray}{c}
|\ell| > N_{n - 1} \,  j, k \in S^\bot_+ \\
 (\ell, j, k) \neq (0, j, j)
\end{subarray}} R_{\ell j k}^{-}(\io_n) \Big)  \,.   
\ee
\end{corollary}
\begin{proof}
By definition, $ R_{\ell jk}^{\pm}(\io_n)$, $Q_{\ell j}(\io_n) \subset \O^{\rm Mel}_n $ and, 
by \eqref{inclusione-1}, for any $\ell \in \Z^S$ with $ |\ell| \leq N_{n - 1} $, one has 
$ R_{\ell jk}^{\pm} (\io_n) \subseteq R_{ljk}^{\pm} (\io_{n-1}) $ and $Q_{\ell j}(\io_n) \subseteq Q_{\ell j}(\io_{n - 1})$.
By definition, one also has $ R_{\ell jk}^\pm(\io_{n-1}) \cap \O^{\rm Mel}_n$ and $Q_{\ell j}(\io_{n - 1}) \cap \O^{\rm Mel}_n$ are  empty sets. 
As a consequence, for any $\ell$ with $ |\ell| \leq N_{n - 1} $, 
$ R_{\ell jk}^{\pm} (\io_n)\,,\, Q_{\ell j}(\io_n) = \emptyset $. 
\end{proof}

The next lemma 
is the core of the measure estimates. To prove ($iv$) the key ingredients are the
asymptotic expansion of the dNLS frequencies of Theorem \ref{Corollary 2.2} ($ii$)  and the
one of the eigenvalues of  the normal form $ {\bf N}_\infty^{(1)} $
up to order $- 1$,  obtained in \eqref{prima asintotica autovalori}-\eqref{estimates-errors}.  

\begin{lemma}\label{limitazioni indici risonanti}
For any $n \geq 0$, $\ell \in \Z^S$, and $j, k \in S^\bot_+$, the following statements hold: 

\noindent
$(i)$ If $Q_{\ell j }(\io_n) \neq \emptyset$, then $j^2 \lessdot  \langle\ell \rangle$\,.
$(ii)$ If $R_{\ell j k}^{+}(\io_n) \neq \emptyset$, then $|j^2 + k^2| \lessdot  \langle\ell \rangle$.

\noindent
$(iii)$ If $R_{\ell j k}^{-}(\io_n) \neq \emptyset$ and $j \neq k$ then $|j^2 - k^2| \lessdot \langle\ell \rangle$. 
$ (iv)$ If $R_{\ell j j}^-(\io_n) \neq \emptyset$ and $\ell \ne 0$
 then $|j| \lessdot  \gamma_*^{- 1} \langle\ell \rangle^{\tau_*}$. 

As a consequence, for any $C> 0$ there are finitely many triples $(\ell, j, k) \ne (0, j, j)$
with $|\ell | \le C$ and $j, k \in S^\bot_+$ so that at least one of the sets
$Q_{\ell j }(\io_n)$, $R_{\ell j k}^{+}(\io_n)$, or $R_{\ell j k}^{-}(\io_n)$  is nonempty.
\end{lemma}

\begin{proof} 
We prove item $(iii)$ and $(iv)$ in detail.
Items ($i$) and ($ii$) follow by similar, but simpler arguments as a  less 
precise  asymptotic expansion suffices. 
Since  the operator  $L_\infty^-(\ell, j, k) \in \mathcal L(\C^{2\times 2})$, defined in 
\eqref{definizione L infinito - seconde Melnikov stime misura}, is self-adjoint, 
the norm of  $ L_\infty^-(\ell, j, k)^{-1} $ (when it exists) is given 
by the inverse of the minimum modulus of the four eigenvalues of $L_\infty^-(\ell, j, k) $. 
By Lemma \ref{properties operators matrices}, these eigenvalues  are given by 
$$
\omega \cdot \ell + \lambda_j^{(a)}(\omega) - \lambda_k^{(b)}(\omega)\,, \quad  a, b \in \{+, -\}\,,
$$
where for any $\kappa \in S_+^\bot$, $\lambda_\kappa^{(+)}(\omega)$, $\lambda_\kappa^{(-)}(\omega)$ 
denote the two eigenvalues of the matrix 
$[{\bf N}_\infty^{(1)}(\omega, \io_n(\omega))]_\kappa^\kappa \in \C^{2 \times 2}$. 
By the definition \eqref{risonanti seconde di melnikov differenza}, $ R_{\ell j k}^-(\io_n) $ thus reads
\begin{align}
R_{\ell j k}^-(\io_n) & = \Big\{ \omega \in \O^{\rm Mel}_n \cap \Omega_{\gamma_*, \tau_*} : 
\exists \, a, b \in \{+, -\}  \,\, \text{with} \,\,  |\omega \cdot \ell + \lambda_j^{(a)}(\omega) - \lambda_k^{(b)}(\omega)| 
< \frac{2 \gamma_n \langle j^2 - k^2 \rangle}{\langle \ell \rangle^\tau} \Big\}  \,. \label{risonanti con autovalori}
\end{align}
By  item $(iii)$ of Theorem \ref{teoremadiriducibilita}, we have for $a \in \{ +, -\}$
\begin{align}\label{asintotica autovalori finali misura}
 \lambda_\kappa^{(a)} = 4 \pi^2 \kappa^2 + c_{\e, \xi} + \frac{\rho_{\xi, \e}^{(a)}(\kappa)}{\kappa}\,, \qquad 
|c_{ \xi, \e}| = O(1)\,, \quad  \sup_{\kappa \in S^\bot_+} |\rho_{\xi, \e}^{(a)}(\kappa)| = O(1)\,.
\end{align}

\smallskip

\noindent
{\it Case $j \neq k$:} 
Assume that  $R^-_{\ell j k}(\io_n) \neq \emptyset$. By \eqref{risonanti con autovalori}, given
$\o \in R^-_{\ell j k}(\io_n) $ 
there exist $a, b \in \{ +, - \}$ so that
\begin{equation}\label{pietro il grande}
  | \lambda_j^{(a)}(\omega) - \lambda_k^{(b)}(\omega)| <
 \frac{2 \gamma_n |j^2 - k^2|}{\langle \ell \rangle^\tau} + |\omega| |\ell| \,.
\end{equation}
On the other hand, by \eqref{asintotica autovalori finali misura},
one sees that 
\be\label{primo-caso-lower}
|\lambda_j^{(a)}(\o) - \lambda_k^{(b)}(\o)|  
\geq |j^2 - k^2| - C'
\ee
for some constant $C' > 0$.
Hence \eqref{pietro il grande} and \eqref{primo-caso-lower} imply that 
$$ 
 |\omega| |\ell|  + C' \geq 
\Big(1-   \frac{2 \gamma_n }{\langle \ell \rangle^\tau} \Big)  |j^2 - k^2| 
\geq (1-   2 \gamma_n )  |j^2 - k^2| \geq \frac12  |j^2 - k^2| 
$$
 taking $\gamma$ in $\g_n = \g (1 + 2^{-n})$ so small that  $  \gamma_n \leq 1/ 4 $.
One concludes that $ |j^2 - k^2| \lessdot \langle \ell \rangle$ and 
item $(iii)$ is proved.

 \medskip
 
 \noindent
 {\it Case $j = k$, $\ell \ne 0$:} Assume that $R_{\ell j j}^-(\io_n) \neq \emptyset$. By \eqref{risonanti con autovalori},
given $ \o \in R_{\ell j j}^-(\io_n) $,
there exist $a, b \in \{ +, - \}$ so that
 \be\label{estimateRljj -}
 |\omega \cdot \ell + \lambda_j^{(a)} (\om)  - \lambda_j^{(b)} (\om)| < \frac{2 \gamma_n}{\langle \ell \rangle^\tau}\,.
 \ee
 Assume that $a = b$.  By \eqref{estimateRljj -}
 and since $\omega \in \Omega_{\gamma_*, \tau_*}$ (see \eqref{diofanteo ausiliare}) one has
 $$
 \frac{2 \gamma_n}{\langle \ell \rangle^\tau} >  |\omega \cdot \ell| \geq \frac{\gamma_*}{\langle \ell \rangle^{\tau_*}} 
 >  \frac{2 \gamma_n}{\langle \ell \rangle^{\tau}}   
 $$
 since $\gamma_* > 8\gamma \geq 2 \gamma_n$ and  $\tau > \tau_*$. 
 The assumption $a=b$ thus yields a contradiction.  
 Hence $a \neq b $. 
Using the asymptotics  
\eqref{asintotica autovalori finali misura}, we  get  that,  
for some constant $C' > 0 $,  
\be\label{crucial-point}
|\omega \cdot \ell + \lambda_j^{(a)}(\o) - \lambda_j^{(b)}(\o) |  
\, \geq \,  |\omega \cdot \ell| - \frac{C'}{ |j|} 
\, \stackrel{\eqref{diofanteo ausiliare}}\geq \,
\frac{ \gamma_*}{\langle \ell \rangle^{\tau_*}} -\frac{C'}{ |j|}\,,
\ee
which,  together with \eqref{estimateRljj -} and $\tau > \tau_*$,  implies that 
$$
 \frac{C'}{ |j|} \geq   \frac{\gamma_* - 2 \gamma_n}{\langle \ell \rangle^{\tau_*}}
 \geq   \frac{\gamma_* }{2 \langle \ell \rangle^{\tau_*}} 
$$
because 
$\gamma_n \leq 2 \gamma$ and $8 \gamma < \gamma_* $.
The claimed inequality
$|j| \lessdot \gamma_*^{- 1} \langle \ell \rangle^{\tau_*}$
of item $(iv)$ is proved. 
 \end{proof}

\medskip

Combining Corollary \ref{excisione parametri io} and Lemma \ref{limitazioni indici risonanti}, 
one sees that there exists a constant $C_* > 0$ so that
the identity \eqref{parametri cattivi} for 
$\big(\O^{\rm Mel}_n \setminus \O^{\rm Mel}_{n + 1} \big) \cap 
\Omega_{\gamma_*, \tau_*}$ with $n \ge 1$ becomes
 \be\label{parametri cattivi forma finale}
\Big( \bigcup_{\begin{subarray}{c}
|\ell| > N_{n - 1} \\
 j \in S^\bot_+ \\
 |j| \leq C_* |\ell |^{1/2}
\end{subarray}} Q_{\ell j}(\io_n) \Big) \cup 
\Big( \bigcup_{\begin{subarray}{c}
|\ell| > N_{n - 1}  \\
j, k \in S^\bot_+\\
j^2 + k^2 \leq C_* | \ell |
\end{subarray}}R_{\ell j k}^{+}(\io_n) \Big) \cup
\Big( \bigcup_{\begin{subarray}{c}
|\ell| > N_{n - 1}  \\
 j, k \in S^\bot_+\,,\,j \neq k \\
 |j^2 - k^2| \leq C_* | \ell |
\end{subarray}} R_{\ell j k}^{-}(\io_n) \Big)  \cup 
\Big( \bigcup_{\begin{subarray}{c}
|\ell| > N_{n - 1}\\
 j\in S^\bot_+ \\
  |j| \leq C_*\gamma_*^{- 1} | \ell |^{\tau_*} 
\end{subarray}} R_{\ell j j}^{-}(\io_n) \Big)\,. 
\ee
The measures of these resonant sets are now estimated individually:  
\begin{lemma}\label{stima in misura insiemi risonanti}
There exists a constant $\widetilde C > 0$ so that 
for any $n \geq 0$,
$j, k \in S_+^\bot$, and $\ell \in \Z^S$ with $|\ell| \geq \widetilde C$  
the following holds: \quad
$(i)$ ${\rm meas}\big( Q_{\ell j}(\io_n) \big)
\lessdot \gamma \langle j \rangle^2 \langle \ell \rangle^{- \tau - 1}$;
\quad
$(ii)$ 
${\rm meas}\big( R_{\ell j k}^+(\io_n) \big)
\lessdot \gamma \langle j^2 + k^2 \rangle \langle \ell \rangle^{- \tau - 1}$;

\noindent
$(iii)$ 
${\rm meas}\big( R_{\ell j k}^-(\io_n) \big) \lessdot 
\gamma \langle j^2 - k^2 \rangle \langle \ell \rangle^{- \tau - 1}$.

\end{lemma}
 \begin{proof}
Since  the proofs of the three items are similar, we only prove  item $(iii)$. 
Assume that $j, k \in S^\bot$ and 
$\ell \in \Z^S$ with  $\ell \neq 0$. Consider the straight line in
$\Omega$ of the form 
 $$
 \omega (s) = s \frac{\ell}{|\ell|}  + v \,, \quad v \cdot \ell = 0 
 $$
where $s$ is a real parameter of appropriate range. The four eigenvalues of the operator
$ L_\infty^-\big(\ell, j, k ; \, s\frac{\ell}{|\ell|} + v \big)$ in 
${\cal L}(\C^{2 \times 2})$ are given by
$ \phi_{a, b} (s) := |\ell| s + \widetilde \lambda_j^{(a)}(s) - \widetilde \lambda_k^{(b)}(s)$
where $a, b \in \{+, - \}$ and
$$ \widetilde \lambda_\kappa^{(a)}(s) := \lambda_\kappa^{(a)} 
\big( s \frac{\ell}{|\ell|}  + v \big)\,, 
\quad a \in \{+, - \}, \quad \kappa \in \{ j, k\}.
 $$
Recall that $\lambda_\kappa^{(-)}(\o)$, $\lambda_\kappa^{(+)}(\o)$ 
denote the two eigenvalues of 
$[{\bf N}_\infty^{(1)}(\omega, \io_n(\omega))]_\kappa^\kappa$
(cf \eqref{asintotica autovalori finali misura}), listed according to their size, $\lambda_\kappa^{(-)}(\o) \leq\lambda_\kappa^{(+)}(\o)$.
By 
\eqref{asintotica autovalori finali misura0}, they are Lipschitz continuous and, for any $\kappa \in S^\bot$, $a \in \{+, -\}$,
 $$
| \widetilde \lambda_\kappa^{(a)}(s) |^{\rm lip} \lessdot 1 \, . 
 $$
 Hence for any $a, b \in \{+, -\}$, $\phi_{a, b} (s)$ satisfies the estimate
 $
|\phi_{a, b} (s_1) - \phi_{a, b} (s_2)| \geq \big( |\ell| - C' \big) |s_1 - s_2| 
$
for some constant $C' >0$. Setting $\widetilde C := 2 C'$ it then follows that for any
 $\ell \in \Z^S$ with $|\ell | \geq \widetilde C$,
$$
|\phi_{a, b} (s_1) - \phi_{a, b} (s_2)| \geq \frac{|\ell|}{2} |s_1 - s_2|\,.
$$
Since $\O$ is compact and by  \eqref{risonanti con autovalori}
$$
\{ s \in \R:  \, s \frac{\ell}{|\ell|}  + v \in R_{\ell j k}^-(\io_n) \big\} =
 \Big\{ s \in \R: \,  \exists \,\, a, b \in \{ +, -\}  \,\, \text{with} \,\, |\phi_{a, b} (s) | < 
\frac{2\gamma_n \langle j^2 - k^2 \rangle}{\langle \ell \rangle^{\tau}}
\Big\}
$$
one  sees by a standard  argument that
 $$
 {\rm meas}\big( \big\{ s \in \R:  \, s \frac{\ell}{|\ell|}  + v \in R_{\ell j k}^-(\io_n) \big\} \big) \lessdot \frac{\gamma \langle j^2 - k^2 \rangle}{\langle \ell \rangle^{\tau + 1}}
 $$
which then yields item $(iii)$ using Fubini's theorem. 
 \end{proof}

By choosing $N_0 \geq \widetilde C$, where $\widetilde C$ is the constant given in Lemma \ref{stima in misura insiemi risonanti}, we have estimated in the latter lemma
 the measures of all the
resonant sets appearing in \eqref{parametri cattivi forma finale},
which will allow us to derive measure estimates of 
$\O^{\rm Mel}_{n} \setminus \O^{\rm Mel}_{n + 1}$ for any $n \geq 1$. 
In view of \eqref{secondo pezzo misura},
it then remains to estimate the measure of $ \O^{\rm Mel}_0 \setminus \O^{\rm Mel}_1$.
 Hence taking into account \eqref{espansione risonanti} and 
 Lemma \ref{stima in misura insiemi risonanti}
 we need to estimate the measures of 
 $Q_{\ell j}(\io_0)$, $R_{\ell j k}^+(\io_0)$, $R_{\ell j k}^-(\io_0)$ 
for any  $\ell \in \Z^S$ with $|\ell| \leq \widetilde C$. 
We use the analyticity of the dNLS frequencies to obtain the following: 
 
 \begin{lemma}\label{stima risonanti modi bassi}
 There exists $ {\frak b}' \in (0, 1]$ so that for any $ j, k \in S^\bot_+$ and  
$\ell \in \Z^S$ with $|\ell| \leq \widetilde C$ (with $\widetilde C$ as in Lemma \ref{stima in misura insiemi risonanti} ) 
the following statements hold: \quad 
 $(i)$ 
${\rm meas}\big( Q_{\ell j}(\io_0) \big)  = O(\gamma^{{\frak b}'})$; \quad
 $(ii)$
${\rm meas}\big( R_{\ell j k}^+(\io_0) \big) = O(\gamma^{{\frak b}'})$; \quad

\noindent
 $(iii)$ if in addition $(\ell, j , k) \ne (0, j, j)$ then
${\rm meas}\big( R_{\ell j k}^-(\io_0) \big) = O(\gamma^{{\frak b}'})$.

 \end{lemma}
 \begin{proof}
Since the proofs of the three items are similar, we only consider item $(iii)$.
By  Lemma \ref{limitazioni indici risonanti}  there are finitely many triples
 $(\ell, j, k) \ne (0, j, j)$ in $\Z^S \times S^\bot_+ \times S^\bot_+ $ 
with $|\ell | \leq \widetilde C$ so that $R_{\ell j k}^-(\io_0) \neq \emptyset$.
For these finitely many triples it follows from
 the definition \eqref{risonanti con autovalori} and  
 \eqref{prima asintotica autovalori}-\eqref{estimates-errors}
 that there exists $C' > 0$ so that when choosing $ \e \gamma^{- 3}$ small enough 
 $$
 R_{\ell j k}^-(\io_0) \, \subseteq \, 
\bigcup_{\begin{subarray}{c}
a, b \in \{+, -\}\\
\end{subarray}}
\big\{ \omega \in \O^{\rm Mel}_0 \, \cap \, \Omega_{\gamma_*, \tau_*} : \,
 |\omega \cdot \ell + \omega_{a j}^{nls}(\xi, 0) - \omega_{b k}^{nls}(\xi, 0)| < 
C' \gamma \big\}\,.
 $$
By Theorem \ref{Corollary 2.2},  $\omega \mapsto \xi(\omega)$, being the inverse map of 
$\xi \mapsto (\omega_\kappa^{nls}(\xi, 0))_{\kappa \in S}$, is analytic as are the maps
 $$
 \omega \mapsto \omega \cdot \ell + \omega_{a j}^{nls}(\xi(\omega), 0) - 
\omega_{b k}^{nls}(\xi(\omega), 0)
 $$
are analytic. By Proposition \ref{Proposition 2.3}, none of these maps vanishes identically. The claimed estimate of item $(iii)$ 
then follows by the Weierstrass preparation theorem as used for instance in \cite[Proposition 3.1]{Biasco-Coglitore}. 
\end{proof}
 
 \medskip

Lemma \ref{stima in misura insiemi risonanti} and Lemma \ref{stima risonanti modi bassi}  are now used to prove measure estimates
of $\big(\O^{\rm Mel}_{n} \setminus \O^{\rm Mel}_{n + 1}  \big) 
\cap \Omega_{\gamma_*, \tau_*}$  for any $n \ge 0$.

 \noindent
 \begin{lemma}\label{cal G n - cal G n + 1}
The following estimates hold: 
$$
{\rm meas}\Big( \big(\O^{\rm Mel}_{0} \setminus \O^{\rm Mel}_{1}  \big) \cap \Omega_{\gamma_*, \tau_*} \Big) = 
O(\gamma^{{\frak b}'})\,,
\qquad
{\rm meas}\Big( \big(\O^{\rm Mel}_{n} \setminus \O^{\rm Mel}_{n + 1}  \big) \cap \Omega_{\gamma_*, \tau_*} \Big) =
 O(\gamma \gamma_*^{- 1} N_{n - 1}^{- 1})\,, \ \  \forall n \geq 1\,.
$$
\end{lemma}
\begin{proof}
 To estimate $ {\rm meas}\big( \big(\O^{\rm Mel}_{n} \setminus\O^{\rm Mel}_{n+1} \big) 
\cap \Omega_{\gamma_*, \tau_*} \big)$ for $n \geq 1$, note that by
\eqref{parametri cattivi forma finale} and Lemma \ref{stima in misura insiemi risonanti}, 
it  is $\lessdot $ bounded by
 \begin{align}
&  \sum_{\begin{subarray}{c}
|\ell| > N_{n - 1} \\
 j \in S^\bot_+ \\
 |j| \leq C_* \langle \ell \rangle^{\frac12}
\end{subarray}} \frac{\gamma \langle j \rangle^2}{\langle \ell \rangle^{\tau + 1}}
\,\,+ \,
\sum_{\begin{subarray}{c}
|\ell| > N_{n - 1}  \\
j, k \in S^\bot_+\\
j^2 + k^2 \leq C_* \langle \ell \rangle
\end{subarray}}\frac{\gamma \langle j^2 + k^2 \rangle}{\langle \ell \rangle^{\tau + 1}}
 \,\, + \, \sum_{\begin{subarray}{c}
|\ell| > N_{n - 1}  \\
 j, k \in S^\bot_+\,,\,j \neq k \\
 |j^2 - k^2| \leq C_* \langle \ell \rangle
\end{subarray}} \frac{\gamma \langle j^2 - k^2 \rangle}{\langle \ell \rangle^{\tau + 1}} \,\, +\, \sum_{\begin{subarray}{c}
|\ell| > N_{n - 1} \\
 j\in S^\bot_+ \\
  |j| \leq C_* \gamma_*^{- 1} \langle \ell \rangle^{\tau_*} 
\end{subarray}} \frac{\gamma}{\langle \ell \rangle^{\tau + 1}}  
 \nonumber\\
& \lessdot \,  \gamma \sum_{|\ell| > N_{n - 1}} \frac{1}{\langle \ell \rangle^{\tau - \frac12}} \,\,+ \, \gamma \sum_{|\ell| > N_{n - 1}} \frac{1}{\langle \ell \rangle^{\tau - 1}} 
\, + \,\gamma \sum_{|\ell| > N_{n - 1}} \frac{1}{\langle \ell \rangle^{\tau - 1}}
 \, + \, \gamma \gamma_*^{- 1} 
\sum_{|\ell| > N_{n - 1}} \frac{1}{\langle \ell \rangle^{\tau + 1 - \tau_*}} \,.  \nonumber
\end{align} 
Since by definition, $\tau = 2|S| + 1$ and $\tau_* = |S| + 1$
(cf \eqref{def:gamma}), one has $\t + 1 - \t_*  = |S| + 1$,
yielding the estimate
$$
{\rm meas}\Big( \big(\O^{\rm Mel}_{n} \setminus\O^{\rm Mel}_{n+1} \big) \cap 
\Omega_{\gamma_*, \tau_*} \Big) \,
 \lessdot \, \gamma \gamma_*^{- 1} \sum_{|\ell| > N_{n - 1}} 
\frac{1}{\langle \ell \rangle^{\tau + 1 - \tau_*}} \, \lessdot \,
 \gamma \gamma_*^{- 1} \frac{1}{N_{n - 1}}\,.
$$
The estimate of ${\rm meas} \big(\big( \O^{\rm Mel}_0 \setminus \O^{\rm Mel}_1 \big) 
\cap \Omega_{\gamma_*, \tau_*} \big)$ 
follows by similar arguments, using in addition Lemma \ref{stima risonanti modi bassi}.
 \end{proof}

 \medskip
 
 \noindent
 {\it Proof of Theorem \ref{measure estimate}:} By \eqref{zero decomposizione risonante},
 \eqref{primo pezzo misura},  
 \eqref{terzo pezzo misura} and Lemma \ref{cal G n - cal G n + 1} one has that 
 $$
 {\rm meas}\big( \Omega \setminus \O^{\rm Mel}_\infty \big) \, \leq   \,
O( \gamma_*) +  O(\gamma )  + O(\gamma^{{\frak b}'} ) + 
O(\gamma \gamma_*^{- 1}) \sum_{n \geq 1} \frac{1}{N_{n - 1}} \,
\leq \, O(\gamma^{{\frak b}'}) + O(\gamma_*) + O(\gamma_*^{- 1} \gamma)\,.
  $$
Thanks to our choice of $ \gamma_* $ in \eqref{def:gamma} and $ \g = \e^{\frak a} $, 
we have $ \gamma_* = \gamma_*^{- 1} \gamma = \e^{{\frak a}/2} $
and \eqref{estimate-meas-final}
then follows with 
 $ {\frak b} := \min \{ {\frak b}' , 1/2 \}$.

\vspace{1.0cm}

\noindent
M. Berti
SISSA, Via Bonomea 265, 34136 Trieste, Italy; \\
${}\qquad$ email: berti@sissa.it\\

\noindent
T. Kappeler, 
Institut f\"ur Mathematik, 
Universit\"at Z\"urich, Winterthurerstr 190, CH-8057 Z\"urich;\\
${}\qquad$  email: thomas.kappeler@math.uzh.ch \\

\noindent
R. Montalto, 
Institut f\"ur Mathematik, 
Universit\"at Z\"urich, Winterthurerstr 190, CH-8057 Z\"urich;\\
${}\qquad$ email: riccardo.montalto@math.uzh.ch

\end{document}